\documentclass[reqno,11pt]{preprint}
\usepackage[full]{textcomp}
\usepackage[osf]{newtxtext}

\usepackage[cal=boondoxo]{mathalfa}
\usepackage{microtype}
\usepackage{shuffle}
\usepackage{longtable}

\usepackage{amsmath}
\usepackage{amsfonts}
\usepackage{amssymb}

\usepackage[all]{xy}
\usepackage[utf8]{inputenc} 
\usepackage{comment}

\usepackage{hyperref}
\usepackage{breakurl}

\usepackage{mathrsfs}
\usepackage{enumitem}
\usepackage{booktabs}

\usepackage{tikz}
\usepackage{dsfont}
\usepackage{mhequ} 
\usepackage{mhsymb} 
\usepackage{mhenvs} 
\usepackage{array}

\usepackage{wasysym}
\usepackage{centernot}
\usepackage{mathtools}

\setlist{noitemsep,nolistsep,leftmargin=1.7em}
\usetikzlibrary{snakes}
\usetikzlibrary{decorations}
\usetikzlibrary{positioning}
\usetikzlibrary{shapes}
\usetikzlibrary{external}
\DeclareFontFamily{U}{mathx}{\hyphenchar\font45}
\DeclareFontShape{U}{mathx}{m}{n}{
      <5> <6> <7> <8> <9> <10>
      <10.95> <12> <14.4> <17.28> <20.74> <24.88>
      mathx10
      }{}
\DeclareSymbolFont{mathx}{U}{mathx}{m}{n}
\DeclareMathSymbol{\bigtimes}{1}{mathx}{"91}

\def\s{\mathfrak{s}}
\def\ii{\mathfrak{i}}
\def\M{\mathfrak{M}}
\def\emptyset{{\centernot\ocircle}}

\definecolor{darkred}{rgb}{0.7,0.1,0.1}
\definecolor{darkblue}{rgb}{0.1,0.1,0.8}
\definecolor{darkgreen}{rgb}{0.1,0.7,0.1}
\def\restr{\mathord{\upharpoonright}}
\def\reg{\mathord{\mathrm{reg}}}

\providecommand{\figures}{false}
{ \ifthenelse{\equal{\figures}{false}} {#1}{\[ {\rm Figure \ missing !} \]} }{}
\def\id{\mathrm{id}}

\def\CH{\mathcal{H}}
\def\CP{\mathcal{P}}

\def\CG{\mathcal{G}}
\def\CJ{\mathcal{J}}
\def\CA{\mathcal{A}}
\def\CE{\mathcal{E}}
\def\CC{\mathcal{C}}
\def\CQ{\mathcal{Q}}
\def\CB{\mathcal{B}}
\def\CM{\mathcal{M}}
\def\CT{\mathcal{T}}
\def\RR{\mathfrak{R}}

\def\K{\mathfrak{K}}
\def\C{\mathfrak{C}}

\def\Labe{\mathfrak{e}}

\def\Labn{\mathfrak{n}}
\def\Labo{\mathfrak{o}}

\def\Labhom{\mathfrak{t}}
\def\noise{\mathfrak{l}}
\def\Lab{\mathfrak{L}}
\def\Adm{\mathfrak{A}}

\def\Deltam{\Delta^{\!-}}
\def\Deltamm{\Delta^{\!-}}

\def\Deltap{\Delta^{\!+}}
\def\Deltapp{\Delta^{\!+}}

\def\${|\!|\!|}

\def\DeltaM{\Delta^{\!M}}
\def\DeltaMg{\Delta^{\!M_g}}

\def\br{\penalty0}
\def\proj{\mathfrak{p}}
\def\inj{\mathfrak{i}}

\def\scal#1{{\langle#1\rangle}}

\def\Trees{\mathfrak{T}}
\def\Primitive{\mathfrak{P}}

\def\Forests{\mathfrak{F}}
\def\Units{\mathfrak{U}}

\def\Tra{\mathfrak{F}}

\def\CS{\mathcal{S}}
\def\CR{\mathcal{R}}
\def\CF{\mathcal{F}}
\def\dd{\mathfrak d}

\def\Vec{\mathop{\mathrm{Vec}}}

\def\hPeps{\rlap{$\hat{\phantom{\PPi}}$}\PPi^{(\eps)}}
\def\hPPi{\rlap{$\hat{\phantom{\PPi}}$}\PPi}

\newenvironment{DIFnomarkup}{}{} % see man latexdiff

\newcommand{\gep}{\varepsilon}

\newtheorem{assumption}{Assumption}
 
\theorembodyfont{\rmfamily}
\newtheorem{example}[lemma]{Example}

\newfont{\indic}{bbmss12}

\def\un#1{\hbox{{\indic 1}$_{#1}$}}
\def\ex{\mathrm{ex}}

\def\PPi{\boldsymbol{\Pi}}

\colorlet{symbols}{blue!90!black}
\colorlet{testcolor}{green!60!black}
\colorlet{connection}{red!30!black}
\def\symbol#1{\textcolor{symbols}{#1}}
\def\symbol#1{\textcolor{symbols}{#1}}

\tikzset{
root/.style={circle,fill=black!50,inner sep=0pt, minimum size=3mm},
        circ/.style={circle,fill=white,draw=black,very thin,inner sep=.5pt, minimum size=1.2mm},
        round1/.style={fill=white,outer sep = 0,inner sep=2pt,rounded corners=1mm,draw,text=black,thin,minimum size=1.2mm},
          circ1/.style={circle,fill=red!10,draw=red,very thin,inner sep=.5pt, minimum size=1.2mm},
        rect/.style={fill=white,outer sep = 0,inner sep=2pt,rectangle,draw,text=black,thin,minimum size=1.2mm},
        rect1/.style={fill=white,outer sep = 0,inner sep=2pt,rectangle,draw,text=black,thin,minimum size=1.2mm},
        round2/.style={fill=red!10,outer sep = 0,inner sep=2pt,rounded corners=1mm,draw,text=black,thin,minimum size=1.2mm},
       round3/.style={fill=blue!10,outer sep = 0,inner sep=2pt,rounded corners=1mm,draw,text=black,thin,minimum size=1.2mm}, 
        rect2/.style={fill=black!10,outer sep = 0,inner sep=2pt,rectangle,draw,text=black,thin,minimum size=1.2mm},
        dot/.style={circle,fill=black,inner sep=0pt, minimum size=1.2mm},
        dotred/.style={circle,fill=black!50,inner sep=0pt, minimum size=2mm},
        var/.style={circle,fill=black!10,draw=black,inner sep=0pt, minimum size=3mm},
        kernel/.style={semithick,shorten >=2pt,shorten <=2pt},
         diag/.style={thin,shorten >=4pt,shorten <=4pt},
        kernel1/.style={thick},
        kernels/.style={snake=zigzag,shorten >=2pt,shorten <=2pt,segment amplitude=1pt,segment length=4pt,line before snake=2pt,line after snake=5pt,},
		kernels1/.style={snake=zigzag,segment amplitude=0.5pt,segment length=2pt},
		rho1/.style={densely dotted,semithick},
        rho/.style={densely dashed,semithick,shorten >=2pt,shorten <=2pt},
           testfcn/.style={dotted,semithick,shorten >=2pt,shorten <=2pt},
           visible/.style={draw, circle, fill, inner sep=0.25ex},
        renorm/.style={shape=circle,fill=white,inner sep=1pt},
        labl/.style={shape=rectangle,fill=white,inner sep=1pt},
        xic/.style={very thin,circle,fill=symbols,draw=black,inner sep=0pt,minimum size=1.2mm},
        xi/.style={very thin,circle,fill=blue!10,draw=black,inner sep=0pt,minimum size=1.2mm},
	xib/.style={very thin,circle,fill=blue!10,draw=black,inner sep=0pt,minimum size=1.6mm},
	xie/.style={very thin,circle,fill=green!50!black,draw=black,inner sep=0pt,minimum size=1mm},
	xid/.style={very thin,circle,fill=symbols,draw=black,inner sep=0pt,minimum size=1.6mm},
	edgetype/.style={very thin,circle,draw=black,inner sep=0pt,minimum size=5mm},
	nodetype/.style={very thick,circle,draw=black,inner sep=0pt,minimum size=5mm},
	kernels2/.style={very thick,draw=connection,segment length=12pt},
clean/.style={thin,circle,fill=black,inner sep=0pt,minimum size=1mm},	not/.style={thin,circle,fill=symbols,draw=connection,fill=connection,inner sep=0pt,minimum size=0.8mm},
	>=stealth,
        }

\tikzset{ individus/.style={scale=0.40,draw,circle,thick,fill=black!10},
 individu/.style={scale=0.40,draw,circle,thick,fill=black!50},       } 
\makeatletter
\def\DeclareSymbol#1#2#3{\expandafter\gdef\csname MH@symb@#1\endcsname{\tikz[baseline=#2,scale=0.15,draw=symbols]{#3}}\expandafter\gdef\csname MH@symb@#1s\endcsname{\scalebox{0.7}{\tikz[baseline=#2,scale=0.15,draw=symbols]{#3}}}}
\def\<#1>{\csname MH@symb@#1\endcsname}
\makeatother

 \def\1{\mathbf{\symbol{1}}}

 \def\un#1{\hbox{{\indic 1}$_{#1}$}}
\def\one{\mathbf{1}}
\def\eps{\varepsilon}
\def\gep{\varepsilon}

\def\MHmakebox#1#2#3{\tikz[baseline=#3,line width=#1,cross/.style={path picture={ 
  \draw[black](path picture bounding box.south) -- (path picture bounding box.north) (path picture bounding box.west) -- (path picture bounding box.east);
}}]{\draw[white] (0,0) rectangle (#2,#2);
\draw[black,cross] (0.1em,0.1em) rectangle (#2-0.1em,#2-0.1em);}}

 \def\IXitwo{\begin{tikzpicture}[scale=0.15,baseline=0.1cm]
        \node at (0,0)  [dot,label=below:$  $] (root) {};
         \node at (1,2)  [circ] (right) {\mbox{\tiny $ $}};
         \node at (-1,2)  [circ] (left) {\mbox{\tiny $ $}};
            \draw[kernel1] (right) to
     node [sloped,below] {\small }     (root); \draw[kernel1] (left) to
     node [sloped,below] {\small }     (root);
     \end{tikzpicture}}
     
     \def\XIXitwo{\begin{tikzpicture}[scale=0.15,baseline=0.1cm]
        \node at (0,0)  [circ,label=below:$  $] (root) {\mbox{\tiny $ i $}};
         \node at (1,2)  [circ] (right) {\mbox{\tiny $ $}};
         \node at (-1,2)  [circ] (left) {\mbox{\tiny $ $}};
            \draw[kernel1] (right) to
     node [sloped,below] {\small }     (root); \draw[kernel1] (left) to
     node [sloped,below] {\small }     (root);
     \end{tikzpicture}}
     
      \def\IXitwoIIXithree{\begin{tikzpicture}[scale=0.15,baseline=0.1cm]
        \node at (0,0)  [dot,label=below:$  $] (root) {};
         \node at (2,2)  [circ] (right) {\mbox{\tiny $ $}};
         \node at (0,3)  [dot] (center) {};
          \node at (0,6)  [circ] (centerc) {\mbox{\tiny $ $}};
          \node at (-2,5)  [circ] (centerl) {\mbox{\tiny $ $}};
          \node at (2,5)  [circ] (centerr) {\mbox{\tiny $ $}};
         \node at (-2,2)  [circ] (left) {\mbox{\tiny $ $}};
            \draw[kernel1] (right) to
     node [sloped,below] {\small }     (root); \draw[kernel1] (left) to
     node [sloped,below] {\small }     (root);
     \draw[kernel1] (center) to
     node [sloped,below] {\small }     (root);
     \draw[kernel1] (centerc) to
     node [sloped,below] {\small }     (center);
      \draw[kernel1] (centerr) to
     node [sloped,below] {\small }     (center);
      \draw[kernel1] (centerl) to
     node [sloped,below] {\small }     (center);
     \end{tikzpicture}}

\DeclareMathOperator*{\bplus}{
\tikzexternaldisable
\mathchoice{\MHmakebox{.17ex}{1.5em}{0.48em}}{\MHmakebox{.12ex}{1.1em}{0.28em}}{}{}
\tikzexternalenable
}
\def\hattimes{\mathbin{\hat\otimes}}

\DeclareMathAlphabet{\mathpzc}{OT1}{pzc}{m}{it}

\def\eps{\varepsilon}
\def\gep{\varepsilon}

\def\BB{\mathfrak{B}}
\def\TT{\mathscr{T}}
\def\MM{\mathscr{M}}
\def\JJ{\mathscr{J}}

\def\Deltap{\Delta^{\!+}}

\def\Deltam{\Delta^{\!-}}

\def\hotimes{\mathbin{\hat\otimes}}

\def\PPi{\boldsymbol{\Pi}}

\def\Vec{{\mathrm{Vec}}}

\def\id{\mathrm{id}}

\def\XX{\mathbb X}

\def\bi{{\mathrm{bi}}}

\def\simnot{\stackrel{\vbox to 0.15em{\hbox{\kern0.07em$^\circ$}}}{\sim}}

\begin{document}

\title{Algebraic renormalisation of regularity structures}
\author{Y. Bruned$^1$, M. Hairer$^1$, L. Zambotti$^2$}
\institute{Imperial College London\\
\email{y.bruned@imperial.ac.uk, m.hairer@imperial.ac.uk} \and Laboratoire de Probabilit\'es Statistique et Mod\'elisation, Sorbonne Université, Paris\\
\email{lorenzo.zambotti@upmc.fr}}

\maketitle

\begin{abstract}
We give a systematic description of a canonical renormalisation procedure of stochastic PDEs 
containing nonlinearities involving generalised functions. This theory is based on the construction of a 
new class of regularity structures  
which comes with an explicit and elegant description of a subgroup of their group of automorphisms. This subgroup is
sufficiently large to be able to implement a version of the BPHZ renormalisation prescription in this context. 
This is in stark contrast to previous works where one considered regularity structures 
with a much smaller group of automorphisms, which lead to a much more indirect and convoluted construction
of a renormalisation group acting on the corresponding space of admissible models by continuous
transformations. 

Our construction is based on bialgebras of decorated coloured forests in cointeraction. More precisely, we have two Hopf algebras in cointeraction, coacting jointly on a vector space which represents the generalised functions of the theory. Two twisted antipodes play a fundamental role in the construction and provide a variant of the algebraic Birkhoff factorisation that arises naturally in perturbative quantum field theory.
\end{abstract}

\setcounter{tocdepth}{2}
\tableofcontents

\section{Introduction}

In a series of celebrated papers \cite{Chen54, Chen57, Chen58, Chen71} Kuo-Tsai Chen discovered that, for any
finite alphabet $A$, 
the family of iterated integrals of a smooth path $x:\R_+\to\R^A$ has a number of interesting 
algebraic properties. 
Writing $\CT = T(\R^A)$ for the tensor algebra on $\R^A$, which we identify with the space spanned
by all finite words $\{(a_1\cdots a_n)\}_{n \ge 0}$ with letters in $A$,  
we define the family of functionals $\XX_{s,t}$ on $\CT$ inductively by
\[
\XX_{s,t}()\eqdef 1, \qquad \XX_{s,t}(a_1\cdots a_{n})\eqdef 
\int_s^t \XX_{s,u}(a_1\cdots a_{n-1}) \, \dot{x}_{a_n}(u) \, d u
\]
where $0\leq s\leq t$. Chen showed that this family yields for fixed 
$s,t$ a character on $\CT$ endowed with the {\it shuffle product} $\shuffle$, namely
\begin{equ}[e:shuffle]
\XX_{s,t}(v\shuffle w) = \XX_{s,t}(v) \, \XX_{s,t}(w),
\end{equ}
 which furthermore satisfies the {\it flow relation}
\[
( \XX_{s,r}\otimes \XX_{r,t}) \Delta \tau = \XX_{s,t}\tau, \qquad s\leq r\leq t,
\]
where $\Delta: \CT\to \CT\otimes \CT$ is the {\it deconcatenation coproduct}
\[
\Delta (a_1\cdots a_n) = \sum_{k=0}^n (a_{1}\cdots a_{k}) \otimes (a_{k+1}\cdots a_{n})\;.
\]
In other words, we have a function $(s,t)\mapsto \XX_{s,t}\in \CT^*$ which takes values in the 
characters on the algebra $(\CT,\shuffle)$ and satisfies the {\it Chen relation}
\begin{equ}[e:chen]
\XX_{s,r}\star \XX_{r,t}=\XX_{s,t}, \qquad s\leq r\leq t,
\end{equ}
where $\star$ is the product dual to $\Delta$. Note that $\CT$, endowed with the shuffle product 
and the deconcatenation coproduct, is a Hopf algebra.

These two remarkable properties do not depend explicitly on the differentiability of the path $(x_t)_{t\geq 0}$. They can therefore serve as an important tool if one wants to consider non-smooth paths and still build a consistent calculus.
This intuition was at the heart of Terry Lyons' definition \cite{Lyons98} of a {\it geometric rough path} as a function $(s,t)\mapsto \XX_{s,t}\in \CT^*$ satisfying the two algebraic properties above and with a controlled modulus of continuity, for instance of H\"older type
\begin{equ}[e:holder]
| \XX_{s,t}(a_1 \cdots a_{n})|\leq C |t-s|^{n\gamma},
\end{equ}
with some fixed $\gamma>0$
(although the original definition involved rather a $p$-variation norm, which is 
natural in this context since it is invariant under reparametrisation of the path $x$, 
just like the definition of $\XX$). 
Lyons realised that this setting would allow to build a robust theory of integration and of associated differential equations. For instance, in the case of stochastic differential equations of Stratonovich type
\[
dX_t = \sigma(X_t) \circ dW_t\;,
\]
with $W:\R_+\to\R^d$ a $d$-dimensional Brownian motion and $\sigma:\R^d\to \R^d\otimes\R^d$ smooth, one can build rough paths $\XX$ and ${\mathbb W}$ over $X$, respectively $W$, such that the map ${\mathbb W}\mapsto\XX$ is \textit{continuous}, while in general the map $W\mapsto X$ is simply measurable.

The It\^o stochastic integration was included in Lyons' theory although it can not be described in terms of geometric rough paths. A few years later Massimiliano Gubinelli \cite{Gubinelli2010693} introduced the concept of a {\it branched rough path} as a function $(s,t)\mapsto \XX_{s,t}\in \CH^*$ taking values in the characters of an algebra $(\CH,\cdot)$ of rooted forests, satisfying the analogue of the Chen relation \eqref{e:chen} with respect to the Grossman-Larsson $\star$-product, dual of the {\it Connes-Kreimer coproduct}, and with a regularity condition
\begin{equ}[e:holder2]
| \XX_{s,t}(\tau)|\leq C |t-s|^{|\tau|\gamma}
\end{equ}
where $|\tau|$ counts the number of nodes in the forest $\tau$ and $\gamma>0$ is fixed. Again, this framework allows for a robust theory of integration and differential equations driven by branched rough paths. Moreover $\CH$, endowed with the forest product and Connes-Kreimer coproduct, turns out to be a Hopf algebra.

The theory of {\it regularity structures} \cite{reg}, due to the second named author of this paper, 
arose from the desire to apply the above ideas to (stochastic) partial differential equations (SPDEs)
involving non-linearities of (random) space-time distributions. 
Prominent examples are the KPZ equation \cite{KPZ,Peter,2015arXiv150803877G}, the $\Phi^4$ 
stochastic quantization equation
\cite{Jona,AlbRock91,DPD2,reg,CatCh,Antti}, the continuous parabolic Anderson model \cite{Cyril2,Cyril3,Gub}, 
and the stochastic Navier-Stokes equations \cite{DPD,Zhu}. 

One apparent obstacle to the application of the rough paths framework to such SPDEs is that one would like 
to allow for the analogue of the map $s\mapsto\XX_{s,t}\tau$ to be a space-time distribution for some $\tau\in\CH$. However, 
the algebraic relations discussed above involve {\it products} of such quantities, which are in general ill-defined. 
One of the main ideas of \cite{reg} was to replace the Hopf-algebra structure with a comodule structure: instead of 
a single space $\CH$, we have two spaces $(\CT,\CT_+)$ and a coaction $\Deltap:\CT\to\CT\otimes\CT_+$
such that $\CT$ is a right comodule over the Hopf algebra $\CT_+$. 
In this way, elements in the dual space $\CT^*$ of $\CT$ are used to encode the distributional 
objects which are needed in the theory, 
while elements of $\CT_+^*$ encode continuous functions. Note that $\CT$ 
admits neither a product nor a coproduct in general.

However, the comodule structure allows to define the analogue of a rough path as a \textit{pair}: 
consider a distribution-valued continuous function 
\[
\R^d \ni y \mapsto \Pi_y\in \CT^* \otimes \CD'(\R^d)\;,
\]
as well as a continuous function
\[
\R^d\times\R^d\ni(x,y)\mapsto \gamma_{xy}\in\CT_+^*.
\]
The analogue of the Chen relation \eqref{e:chen} is then given by
\begin{equ}[e:chen2]
\gamma_{xy}\star \gamma_{yz}=\gamma_{xz}\;, \qquad 
\Pi_y \star \gamma_{yz} =  \Pi_z\;, 
\end{equ}
where the first $\star$-product is  the convolution product on $\CT_+^*$, while the
second $\star$-product is given by the dual of the coaction $\Deltap$.
This structure guarantees that all relevant expressions will be linear in the $\Pi_y$, so we never
need to multiply distributions. To compare this expression to \eqref{e:chen}, think of
$(\Pi_y \tau)(\cdot) \in \CD'(\R^d)$ for $\tau \in \CT$ as being the analogue of
$z \mapsto \XX_{z,y}(\tau)$.
Note that the algebraic conditions \eqref{e:chen2} are not enough to provide a useful object: analytic conditions analogous to \eqref{e:holder2} play an essential role in the analytical aspects of the theory. 
Once a \textit{model} $\XX = (\Pi,\gamma)$ has been constructed, it plays a role analogous to that of a 
rough path and allows to construct a robust solution theory for a class of rough (partial) differential equations.

In various specific situations, the theory 
yields a \textit{canonical lift} of any smoothened realisation of the driving noise 
for the stochastic PDE under consideration to a model $\XX^\eps$. 
Another major difference with what one sees in the rough paths setting is the following 
phenomenon: if we remove the regularisation as $\eps\to 0$, neither the canonical model 
$\XX^\eps$ nor the solution to the regularised equation converge in general to a limit. 
This is a structural problem which reflects again the fact that some products are intrinsically 
ill-defined. 

This is where {\it renormalisation} enters the game. It was already recognised in \cite{reg} that one should find a group $\RR$ of transformations on the space of models and elements $M_\eps$ in $\RR$ in such a way that,
when applying $M_\eps$ to the canonical lift $\XX^\eps$, the resulting
sequence of models converges to a limit. Then the theory essentially provides a
\textit{black box}, allowing to build maximal solutions for the stochastic PDE in question. 

One aspect of the theory developed in \cite{reg} that is far from satisfactory is that while one has
in principle a characterisation of $\RR$, this characterisation is very
indirect. The methodology pursued so far has been to first make an educated guess for a sufficiently 
large family of renormalisation maps, then verify \textit{by hand} that these do indeed
belong to $\RR$ and finally show, again \textit{by hand}, that the renormalised models converge to a limit.
Since these steps did not rely on any general theory, they 
had to be performed separately for each new class of stochastic PDEs.

The main aim of the present article is to define an algebraic framework 
allowing to build regularity structures which, on the one hand, extend the ones built in \cite{reg} 
and, on the other hand, admit sufficiently many automorphisms (in the sense of 
\cite[Def.~2.28]{reg}) to cover the renormalisation procedures of all subcritical stochastic PDEs that
have been studied to date.

Moreover our construction is not restricted to the Gaussian setting and applies to {\it any} choice of the driving noise with minimal integrability conditions. In particular this allows to recover all the renormalisation procedures
used so far in applications of the theory
\cite{reg,wong,woKP,2015arXiv150701237H,2016arXiv160204570H,2016arXiv160105724S}. 
It reaches however far beyond this and shows that the BPHZ renormalisation procedure belongs to the
renormalisation group of the regularity structure associated to \textit{any} class of subcritical 
semilinear stochastic PDEs. In particular, this is the case for the generalised KPZ equation which 
is the most natural stochastic evolution on loop space and is (formally!) given in local coordinates by
\begin{equ}[e:chris]
\d_t u^\alpha = \d_x^2 u^\alpha + \Gamma^\alpha_{\beta\gamma}(u) \d_x u^\beta \d_x u^\gamma 
+ \sigma_i^\alpha(u)\,\xi_i\;,
\end{equ}
where the $\xi_i$ are independent space-time white noises, $\Gamma^\alpha_{\beta\gamma}$ are the
Christoffel symbols of the underlying manifold, and the $\sigma_i$ are a collection of vector fields with
the property that $\sum_i L_{\sigma_i}^2 = \Delta$, where $L_{\sigma}$ is the Lie derivative in the direction
of $\sigma$ and $\Delta$ is the Laplace-Beltrami operator.
Another example is given by the stochastic sine-Gordon equation \cite{2014arXiv1409.5724H} close to the
Kosterlitz-Thouless transition. In both of these examples, the relevant group describing the renormalisation procedures
is of very large dimension (about 100 in the first example and arbitrarily large in the second one), so that
the verification ``by hand'' that it does indeed belong to the ``renormalisation group'' as done
for example in \cite{reg,wong}, would be impractical.

In order to describe the renormalisation procedure of SPDEs we introduce a new construction of an associated regularity structure, that will
be called {\it extended} since it contains a new parameter which was not present in \cite{reg}, the {\it extended decoration}. As above, this yields spaces $(\CT^\ex,\CT^\ex_+)$, such that $\CT_+^\ex$ is a Hopf algebra and $\CT^\ex$
a right comodule over $\CT^\ex_+$. 
The renormalisation procedure of distributions coded by $\CT^\ex$ is then described by another Hopf algebra 
$\CT_-^\ex$ and coactions 
$\Deltam_\ex:\CT^\ex\to\CT_-^\ex\otimes\CT^\ex$ and $\Deltam_\ex:\CT^\ex_+\to\CT_-^\ex\otimes\CT_+^\ex$ turning
both $\CT^\ex$ and  $\CT_+^\ex$ into left comodules over $\CT_-^\ex$. This construction is, crucially, 
compatible with the comodule structure of $\CT^\ex$ over $\CT_+^\ex$ in the sense that 
$\Deltam_\ex$ and $\Deltap_\ex$ are in {\it cointeraction} in 
the terminology of \cite{2016arXiv160508310F}, see formulae \eqref{e:intertwine}-\eqref{e:propWanted1} and
Remark \ref{cointera} below. 
Once this structure is obtained, we can define {\it renormalised models} as follows: given a functional $g:\CT_-^\ex\to\R$
and a model $\XX = (\Pi,\gamma)$, we construct a new model $\XX^g$ by setting
\begin{equ}
\gamma_{z\bar z}^g = (g \otimes \gamma_{z\bar z})\Deltam_\ex\;,\qquad 
\Pi_z^g  = (g \otimes \Pi_z)\Deltam_\ex\;.
\end{equ}
The cointeraction property then guarantees that  $\XX^g$ satisfies again the generalised Chen relation \eqref{e:chen2}. 
Furthermore, the action of $\CT_-^\ex$ on $\CT^\ex$ and $\CT_+^\ex$ is such that, crucially, the associated 
analytical conditions automatically hold as well.

All the coproducts and coactions mentioned above are a priori different operators, but we describe them in a unified framework as special cases of a contraction / extraction operation of subforests, as arising
in the BPHZ renormalisation procedure / forest formula \cite{BP,Hepp,Zimmermann,FMRS85}. 
It is interesting to remark that the structure described in this article
is an extension of that previously described in \cite{CHV,MR2657947,MR2803804} in the context of 
the analysis of B-series for numerical ODE solvers, which is itself an extension 
of the Connes-Kreimer Hopf algebra of rooted trees \cite{CK,CKI} arising in
the abovementioned forest formula in perturbative QFT. It is also closely related to incidence Hopf algebras associated to 
families of posets \cite{MR914660,schmitt99}. 

There are however a number of substantial differences with respect to the existing literature. First we propose a new approach based on coloured forests; for instance we shall consider operations like
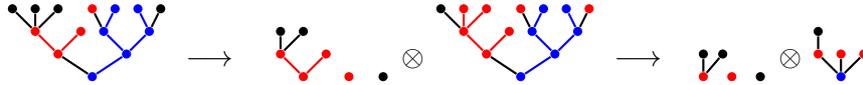
\[
\begin{tikzpicture}[scale=0.15,baseline=0.15cm]
        \node at (0,0)  [dot,blue] (root) {};
         \node at (-7,6)  [dot ] (leftll) {};
          \node at (-5,4)  [dot,red,label=left:] (leftl) {};
          \node at (-3,6)  [dot] (leftlr) {};
          \node at (-5,6)  [dot] (leftlc) {};
      \node at (-3,2)  [dot,color=red] (left) {};
         \node at (-1,4)  [dot,red] (leftr) {};
         \node at (1,4)  [dot,blue,label=above:   ] (rightl) {};
          \node at (0,6)  [dot,red ] (rightll) {};
           \node at (2,6)  [dot,blue] (rightlr) {};
           \node at (5,4)  [dot,blue,label=left:] (rightr) {};
            \node at (4,6)  [dot,blue] (rightrl) {};
        \node at (3,2) [dot,blue] (right) {};
         \node at (6,6)  [dot ] (rightrr) {};
        
        \draw[kernel1] (left) to node [sloped,below] {\small } (root); ;
        \draw[kernel1,red] (leftl) to
     node [sloped,below] {\small }     (left);
     \draw[kernel1] (leftlr) to
     node [sloped,below] {\small }     (leftl); 
     \draw[kernel1] (leftll) to
     node [sloped,below] {\small }     (leftl);
     \draw[kernel1,red] (leftr) to
     node [sloped,below] {\small }     (left);  
        \draw[kernel1,blue] (right) to
     node [sloped,below] {\small }     (root);
      \draw[kernel1] (leftlc) to
     node [sloped,below] {\small }     (leftl); 
     \draw[kernel1,blue] (rightr) to
     node [sloped,below] {\small }     (right);
     \draw[kernel1] (rightrr) to
     node [sloped,below] {\small }     (rightr);
     \draw[kernel1,blue] (rightrl) to
     node [sloped,below] {\small }     (rightr);
     \draw[kernel1,blue] (rightl) to
     node [sloped,below] {\small }     (right);
     \draw[kernel1,blue] (rightlr) to
     node [sloped,below] {\small }     (rightl);
     \draw[kernel1] (rightll) to
     node [sloped,below] {\small }     (rightl);
     \end{tikzpicture} 
     \longrightarrow \quad
     \begin{tikzpicture}[scale=0.15,baseline=0.15cm]
          \node at (-5,2)  [dot,red] (leftl) {};
          \node at (-3,4)  [dot] (leftlr) {};
          \node at (-5,4)  [dot] (leftlc) {};
      \node at (-3,0)  [dot,color=red] (left) {};
         \node at (-1,2)  [dot,red] (leftr) {};
         \node at (1,0)  [dot,red] (o) {};
         \node at (4,0)  [dot] (o) {};
        
        \draw[kernel1,red] (leftl) to
     node [sloped,below] {\small }     (left);
     \draw[kernel1] (leftlr) to
     node [sloped,below] {\small }     (leftl); 
     \draw[kernel1,red] (leftr) to
     node [sloped,below] {\small }     (left);  
      \draw[kernel1] (leftlc) to
     node [sloped,below] {\small }     (leftl); 
     \end{tikzpicture} 
 \     \otimes   
\begin{tikzpicture}[scale=0.15,baseline=0.15cm]
        \node at (0,0)  [dot,blue] (root) {};
         \node at (-7,6)  [dot] (leftll) {};
          \node at (-5,4)  [dot,red] (leftl) {};
          \node at (-3,6)  [dot,red] (leftlr) {};
          \node at (-5,6)  [dot,red] (leftlc) {};
      \node at (-3,2)  [dot,color=red] (left) {};
         \node at (-1,4)  [dot,red] (leftr) {};
         \node at (1,4)  [dot,blue ] (rightl) {};
          \node at (0,6)  [dot,red] (rightll) {};
           \node at (2,6)  [dot,blue] (rightlr) {};
           \node at (5,4)  [dot,blue] (rightr) {};
            \node at (4,6)  [dot,blue] (rightrl) {};
        \node at (3,2) [dot,blue] (right) {};
         \node at (6,6)  [dot,red] (rightrr) {};
        
        \draw[kernel1] (left) to node [sloped,below] {\small } (root); ;
        \draw[kernel1,red] (leftl) to
     node [sloped,below] {\small }     (left);
     \draw[kernel1,red] (leftlr) to
     node [sloped,below] {\small }     (leftl); 
     \draw[kernel1] (leftll) to
     node [sloped,below] {\small }     (leftl);
     \draw[kernel1,red] (leftr) to
     node [sloped,below] {\small }     (left);  
        \draw[kernel1,blue] (right) to
     node [sloped,below] {\small }     (root);
      \draw[kernel1,red] (leftlc) to
     node [sloped,below] {\small }     (leftl); 
     \draw[kernel1,blue] (rightr) to
     node [sloped,below] {\small }     (right);
     \draw[kernel1] (rightrr) to
     node [sloped,below] {\small }     (rightr);
     \draw[kernel1,blue] (rightrl) to
     node [sloped,below] {\small }     (rightr);
     \draw[kernel1,blue] (rightl) to
     node [sloped,below] {\small }     (right);
     \draw[kernel1,blue] (rightlr) to
     node [sloped,below] {\small }     (rightl);
     \draw[kernel1] (rightll) to
     node [sloped,below] {\small }     (rightl);
     \end{tikzpicture}
 \longrightarrow \quad 
     \begin{tikzpicture}[scale=0.15,baseline=0.15cm]
          \node at (-2.5,2)  [dot] (leftlc) {};
          \node at (-.8,2)  [dot] (leftlcc) {};
      \node at (-2.5,0)  [dot,color=red] (left) {};
         \node at (0,0)  [dot,red] (o) {};
         \node at (2.5,0)  [dot] (o) {};
        
      \draw[kernel1] (leftlc) to
     node [sloped,below] {\small }     (left); 
     \draw[kernel1] (leftlcc) to
     node [sloped,below] {\small }     (left); 
     \end{tikzpicture} 
 \     \otimes   
      \begin{tikzpicture}[scale=0.15,baseline=0.15cm]
        \node at (0,0)  [dot,blue] (root) {};
         \node at (-2,4)  [dot ] (leftll) {};
       \node at (-2,2)  [dot,color=red] (left) {};
          \node at (0,2)  [dot,red ] (rightll) {};
         \node at (2,2)  [dot,red ] (rightrr) {};
        
        \draw[kernel1] (left) to node [sloped,below] {\small } (root); ;
     \draw[kernel1] (leftll) to
     node [sloped,below] {\small }     (left);
     \draw[kernel1] (rightrr) to
     node [sloped,below] {\small }     (root);
     \draw[kernel1] (rightll) to
     node [sloped,below] {\small }     (root);
     \end{tikzpicture} 
\]
of colouring, extraction and contraction of subforests. Further, the abovementioned articles 
deal with {\it two} 
spaces in cointeraction, analogous to our Hopf algebras $\CT^\ex_-$ and $\CT^\ex_+$, while our third space $\CT^\ex$ is the crucial 
ingredient which allows for distributions in the analytical part of the theory. Indeed, one of the main 
novelties of regularity structures is that they allow to study random distributional objects in a pathwise sense rather 
than through Feynman path integrals / correlation functions and the space $\CT^\ex$ encodes the fundamental bricks of 
this construction. 
Another important difference is that the structure described here does not consist of simple 
trees / forests, but they are decorated with multiindices
on both their edges and their vertices. These decorations are \textit{not inert} but
transform in a non-trivial way under our coproducts, interacting with other operations like the contraction of sub-forests and the computation of suitable gradings.

In this article, Taylor sums play a very important role, just as in the BPHZ renormalisation 
procedure, and they appear in the coactions of both $\CT^\ex_-$ (the {\it renormalisation}) and $\CT^\ex_+$ (the {\it recentering}). 
In both operations, the group elements used to perform such operations are constructed with the help of a {\it twisted antipode}, providing a variant of the algebraic Birkhoff factorisation that was previously shown to arise naturally
in the context of perturbative quantum field theory, see for
example \cite{Kreimer,CK,CKI,CKII,Birkhoff1,Birkhoff2}. 

In general, the context for a twisted antipode / Birkhoff factorisation is that of
a group $G$ acting on some vector space $A$ which comes with a valuation.
Given an element of $A$, one then wants to renormalise it by acting
on it with a suitable element of $G$ in such a way that its valuation vanishes. In the context of
dimensional regularisation, elements of $A$ assign to each Feynman diagram a Laurent series
in a regularisation parameter $\eps$, and the valuation extracts the pole part of this series.
In our case, the space $A$ consists of stationary random linear maps $\PPi\colon \CT^\ex \to \CC^\infty$
and we have \textit{two} actions on it, by the group of 
characters $\CG^\ex_\pm$ of $\CT^\ex_\pm$, corresponding to two different valuations. 
The {\it renormalisation group} $\CG^\ex_-$ is associated to the valuation that extracts
the value of $\E (\PPi \tau)(0)$ for every homogeneous element $\tau \in \CT^\ex$ of negative degree.
The {\it structure group} $\CG^\ex_+$ on the other hand is associated to the valuations
that extract the values $(\PPi \tau)(x)$ for all homogeneous elements $\tau \in \CT^\ex$ of positive degree.

We show in particular that the twisted antipode related to the action of $\CG^\ex_+$ is intimately 
related to the algebraic properties of Taylor remainders. Also in this respect, regularity structures provide a far-reaching generalisation of 
rough paths, expanding Massimiliano Gubinelli's investigation of the algebraic and analytic properties of increments 
of functions of a real variable achieved in the theory of {\it controlled rough paths} \cite{Gubinelli200486}.

\subsection{A general renormalisation scheme for SPDEs}

Regularity Structures (RS) have been introduced \cite{reg} in order to solve singular SPDEs of the form
\[
\partial_t u = \Delta u +F(u,\nabla u,\xi)
\]
where $u=u(t,x)$ with $t\geq 0$ and $x\in\R^d$, $\xi$ is a random space-time Schwartz distribution
(typically stationary and approximately scaling-invariant at small scales) driving the equation 
and the non-linear term $F(u,\nabla u,\xi)$ 
contains some products of distributions which are not well-defined by classical 
analytic methods. We write this equation in the customary mild formulation
\begin{equation}\label{singularSPDE}
u = G*(F(u,\nabla u,\xi))
\end{equation}
where $G$ is the heat kernel and we suppose for simplicity that $u(0,\cdot)=0$.  

If we regularise the noise $\xi$ by means of a family of smooth mollifiers $(\rho^\varepsilon)_{\varepsilon>0}$, setting $\xi^\varepsilon:=\rho^\gep*\xi$, then 
the regularised PDE 
\[
u^\gep = G*(F(u^\gep,\nabla u^\gep,\xi^\gep))
\]
is well-posed under suitable assumptions on $F$. However, if we want to remove the regularisation by letting $\gep\to 0$, we do not know whether $u^\gep$ converges. The 
problem is that $\xi^\gep\to\xi$ in a space of distributions with negative (say) Sobolev
regularity, and in such spaces the solution map $\xi^\gep\mapsto u^\gep$ is not continuous.

The theory of RS allows to solve this problem for a class of equations, called {\it 
subcritical}. The general approach is as in Rough Paths (RP): the discontinuous 
solution map 
\[
\CD'(\R^d)\ni\xi^\gep\mapsto u^\gep\in \CD'(\R^d)
\] 
is factorised as the composition of two maps:
\[
\CD'(\R^d)\ni\xi^\gep\mapsto {\mathbb X}^\gep\in\MM,
\qquad {\mathbb X}^\gep\mapsto u^\gep=:\Phi({\mathbb X}^\gep)\in\CD'(\R^d), 
\]
where $(\MM,{\rm d})$ is a metric space that we call the {\it space
of models}. The main point is that the map $\Phi:\MM\to\CD'(\R^d)$
 can be chosen in such a way that its is  {\it continuous}, even though $\MM$ is sufficiently
 large to allow for elements exhibiting a local scaling behaviour compatible with that of $\xi$. 
Of course this means that $\xi^\gep\mapsto {\mathbb X}^\gep$ is discontinuous in general. 
In RP, the analogue of the model 
${\mathbb X}^\gep$ is the lift of the driving noise as a rough path, the map $\Phi$ is called 
the It\^o-Lyons map, and its
continuity (due to T.\ Lyons \cite{Lyons98}) is the cornerstone of the theory. The construction of $\Phi:\MM
\to\CD'(\R^d)$ in the general context of subcritical SPDEs is one of the main results of \cite{reg}. 

The construction of $\Phi$, although a very powerful tool, does not solve alone the 
aforementioned problem, since it turns out that the most natural choice of ${\mathbb X}^\gep$, which we call the canonical model, does in general {\it 
not} converge as we remove the regularisation by letting $\gep\to 0$. It is necessary to modify, namely {\it 
renormalise}, the model ${\mathbb X}^\gep$ in order to obtain a family $\hat {\mathbb X}^\gep$ which 
does converge in $\MM$ as $\gep\to 0$ to a limiting
model $\hat{\mathbb X}$. The continuity of $\Phi$ then implies that 
$\hat u^\gep:=\Phi(\hat {\mathbb X}^\gep)$ converges to some limit $\hat u:=\Phi(\hat {\mathbb X})$, which 
we call the {\it renormalised solution} to our equation, see Figure~\ref{figZ}. 
A very important fact is that $\hat u^\gep$ is itself the solution of a {\it
renormalised equation}, which differs from the original equation only by the presence of 
additional {\it local counterterms}, the form of which can be derived explicitly from the
starting SPDE, see \cite{BCCH}. 

\begin{figure}
\centering
\begin{tikzpicture}[scale=0.8]
	\fill[gray!10, draw=black] (3,0) to[out=95, in=0] (0,2) to[out=180, in=85] (-3,0) .. controls (-3,-2) and (-1,0) .. (0,-1) to[bend right, out=0, in=-120] (1,-2) to[bend right, in=-130] cycle;
	\draw (-3,-4) -- (3,-4);
	\node[visible, label={below:$\xi$}] at (-1,-4) {}; 
	\node[visible, label={below:$\xi^\varepsilon$}] at (1,-4) {}; 
	\draw[dashed] (-1,-4) -- (-1,1);
	\node[visible, label={right:$\hat{\mathbb{X}}$}] at (-1,1) {};
	\draw[dashed] (1,-4) -- (1,1);
	\node[visible, label={right:$\hat{\mathbb{X}}^\varepsilon$}] at (1,1) {};
	\node[visible, label={right:$\mathbb{X}^\varepsilon$}] at (1,-1) {};
	\draw (5,-4) -- (11,-4);
%	\node[visible, label={below:$\hat u=\Phi(\hat\mathbb{X})$}] at (6,-4) {}; 
	\node[visible, label={below:$\hat u$}] at (6,-4) {}; 
	\node[visible, label={below:$u^\varepsilon$}] at (9,-4) {}; 
	\node[visible, label={below:$\hat u^\varepsilon$}] at (7.5,-4) {}; 
	\draw[->] (3.5,0) to[bend left] node[midway, above right] {$\Phi$} (8,-3); 
	\draw[->] (1.1,-4.7) to[bend right] (8.7,-4.7);
	\draw[->] (5,-2.5) arc[radius=0.5cm, start angle=0, delta angle=-340];
	\node[label={below:$\CD'({\mathbb R}^d)$}] at (-2.5,-3.8) {};
	\node[label={below:$\CD'({\mathbb R}^d)$}] at (10.5,-3.8) {};
	\node[label={below:$\MM$}] at (-2.7,2) {};
\end{tikzpicture}
\caption{In this figure we show the factorisation of the map $\xi^\gep\mapsto u^\gep$ into
$\xi^\gep\mapsto{\mathbb X}^\gep\mapsto \Phi({\mathbb X}^\gep)=u^\gep$. We also see that in the space of models $\MM$ we have several possible lifts of $\xi^\gep\in{\mathcal 
S}'(\R^d)$, e.g. the canonical model ${\mathbb X}^\gep$ and the renormalised model $\hat 
{\mathbb X}^\gep$; it is the latter that converges to a model $\hat{\mathbb X}$, thus providing a lift of $\xi$. Note that $\hat u^\gep=\Phi(\hat{\mathbb X}^\gep)$ and $\hat u=\Phi(\hat{\mathbb X})$. 
}\label{figZ}
\end{figure}

The transformation ${\mathbb X}^\gep\mapsto\hat{\mathbb X}^\gep$ is described by the 
so-called {\it renormalisation group}. The main aim of this paper is to provide a general
construction of the space of models $\MM$ together with a group of automorphisms   
$\CG_-\ni S:\MM\to\MM$ which allows to describe the renormalised model 
$\hat{\mathbb X}^\gep=S_\gep{\mathbb X}^\gep$ for an appropriate choice of $S_\gep\in 
\CG_-$.

Starting with the $\phi^4_3$ equation and the Parabolic Anderson Model in \cite{reg}, 
several equations have already been successfully renormalised with regularity structures
\cite{CDM,Cyril3,Cyril2,wong,woKP,2014arXiv1409.5724H,2015arXiv150701237H,2016arXiv160105724S}. In all these cases, the construction of the renormalised model and
its convergence as the regularisation is removed are based on {\it ad hoc} arguments
which have to be adapted to each equation. 
The present article, together with the companion ``analytical'' article \cite{Ajay} and the 
work \cite{BCCH}, complete the general 
theory initiated in \cite{reg} by proving that virtually every\footnote{There are some exceptions that can arise 
when one of the driving noises is less regular than white noise. For example, a canonical solution theory for SDEs
driven by fractional Brownian motion can only be given for $H > {1\over 4}$, even though these equations are subcritical for every $H>0$. See in particular the assumptions of \cite[Thm~2.14]{Ajay}.} subcritical equation driven by a stationary 
noise satisfying some natural bounds on its cumulants can be successfully renormalised by means of the following scheme:
\begin{itemize}
\item {\it Algebraic step}: Construction of the space of models $(\MM,{\rm d})$ and 
renormalisation of the canonical model $\MM\ni{\mathbb X}^\gep\mapsto\hat{\mathbb X}^\gep\in\MM$, this article.
\item {\it Analytic step}: Continuity of the solution map $\Phi:\MM\to\CD'(\R^d)$, \cite{reg}.
\item {\it Probabilistic step}: Convergence in probability of the renormalised model $\hat{\mathbb X}^\gep$ to $\hat{\mathbb X}$ in $(\MM,{\rm d})$, \cite{Ajay}.
\item {\it Second algebraic step}: Identification of $\Phi(\hat{\mathbb X}^\gep)$
with the classical solution map for an equation with local counterterms, \cite{BCCH}.
\end{itemize}
%We obtain in this way a {renormalised solution} $\hat u:=\Phi(\hat{\mathbb X})=\lim_{\gep\to 0}\Phi(\hat{\mathbb X}^\gep)$, which is also the unique solution of a fixed point problem and has a physical interpretation as
%a limit of classical solutions for a modified equation.
We stress that this procedure works for very general noises, far beyond the Gaussian case.

\subsection{Overview of results}

We now describe in more detail the main results of this paper. Let us start from the notion
of a {\it subcritical rule}. A {\it rule}, introduced in Definition~\ref{def:rule} below, 
is a formalisation of the notion of a ``class of systems of stochastic PDEs''.
More precisely, given any system of equations of the type \eqref{singularSPDE}, there is a natural way
of assigning to it a rule (see Section~\ref{sec:SPDERules} for an example), which keeps track of which
monomials (of the solution, its derivatives, and the driving noise) appear on the right hand side for each component. 
The notion of a {\it subcritical} rule, see Definition~\ref{def:subcritical}, translates to this general 
context the notion of
subcriticality of equations which was given more informally in \cite[Assumption 8.3]{reg}.

Suppose now that we have fixed a subcritical rule. The first aim is to construct an associated 
space of models 
$\MM^\ex$. The superscript `$\ex$' stands for {\it extended} and is used to distinguish this 
space from the {\it restricted} space of models $\MM$, see Definition~\ref{def:modelspace2}, which
is closer to the original construction of \cite{reg}. 
The space $\MM^\ex$ extends $\MM$ in the sense that there is a canonical 
continuous injection $\MM \hookrightarrow \MM^\ex$, see Theorem~\ref{theo:models}. 
The reason for considering this larger space is that it admits a large group $\CG_-^\ex$ of
automorphisms in the sense of \cite[Def.~2.28]{reg} which can be described in an explicit way.
Our renormalisation procedure then makes use of a suitable subgroup
$\CG_- \subset \CG_-^\ex$ which leaves $\MM$ invariant.
The reason why we do not describe its action on $\MM$ directly is that although it acts by continuous
transformations, it no longer acts by automorphisms, making it much more difficult to describe
without going through $\MM^\ex$.

To define $\MM^\ex$, we construct a regularity structure $(\CT^\ex,\CG_+^\ex)$ in the 
sense of \cite[Def.~2.1]{reg}. This is done in Section~\ref{sec5}, see in 
particular Definitions~\ref{def:CT}-\ref{prop:Hopf-} and Proposition~\ref{TTex}.  
The corresponding {\it structure group} $\CG_+^\ex$ is constructed as the character group of a Hopf 
algebra $\CT_+^\ex$, see \eqref{eq:Hopf}, Proposition~\ref{prop:Hopf+} and Definition~\ref{def:charpm}. 
The vector space $\CT^\ex$ is a right-comodule over $\CT_+^\ex$, namely 
there are linear operators
\[
\Deltap_\ex :\CT^\ex\to \CT^\ex\otimes \CT^\ex_+, \qquad \Deltap_\ex :\CT^\ex_+\to \CT^\ex_+\otimes \CT^\ex_+,
\]
such that the identity
\begin{equation}\label{coass}
(\id\otimes\Deltap_\ex)\Deltap_\ex=(\Deltap_\ex\otimes\id)\Deltap_\ex\;,
\end{equation}
holds both between operators on $\CT^\ex$ and on $\CT_+^\ex$.
The fact that the two operators have the same name but act on different spaces should
not generate confusion since the domain is usually clear from context. When it isn't, as in 
\eqref{coass}, then the identity is assumed by convention to hold for all possible meaningful 
interpretations. 
%Then, there is a right action
%of $\CG_+^\ex$ on the dual space $(\CT^\ex)^*$ of $\CT^\ex$, expressed by
%\[
%h f:=(h\otimes f)\Deltap_\ex \tau, \qquad h\in(\CT^\ex)^*, \ \tau\in \CT^\ex, \, f\in\CG_+^\ex.
%\]

Next, the {\it renormalisation group} $\CG_-^\ex$ is defined as the character group of the 
Hopf algebra $\CT_-^\ex$, see \eqref{eq:Hopf}, Proposition~\ref{prop:Hopf-} and Definition 
\ref{def:charpm}. The vector spaces $\CT^\ex$ and $\CT^\ex_+$ are both left-comodules over 
$\CT_-^\ex$, so that $\CG_-^\ex$ acts on the left on $\CT^\ex$ and on $\CT^\ex_+$. Again, this means that
we have operators
\[
\Deltam_\ex:\CH\to\CT^\ex_-\otimes\CH, \qquad \CH\in\{\CT^\ex,\CT^\ex_+,\CT^\ex_-\}
\]
such that
\[
(\id\otimes\Deltam_\ex)\Deltam_\ex=(\Deltam_\ex\otimes\id)\Deltam_\ex.
\]
The action of $\CG_-^\ex$ on the corresponding dual spaces is given by
\[
(g h)(\tau):=(g\otimes h)\Deltam_\ex\tau, \qquad h\in\CH^*, \, \tau\in\CH, \,  g\in\CG_-^\ex\;.
\]
Crucially, these separate actions satisfy a compatibility condition which can be
expressed as a {\it cointeraction} property, see \eqref{e:propWanted1} in Theorem 
\ref{CTpmexHopf}, which implies the following relation between the two actions above:
\begin{equation}\label{elltaug}
g(h f)=(g h)(g f), \qquad h\in\CH^*, \,  g\in\CG_-^\ex, \, f\in\CG_+^\ex, \, \CH\in\{\CT^\ex,\CT^\ex_+\},
\end{equation}
see Proposition~\ref{left-right} and \eqref{e:semidirect}. This result is the algebraic linchpin of Theorem~\ref{theo:algebra}, where we construct the action of $\CG_-^\ex$ on the space $\MM^\ex$
of models. 

The next step is the construction of the space of {\it smooth} models of the regularity structure 
$(\CT^\ex,\CG_+^\ex)$. This is done in Definition~\ref{def:model}, where we 
follow \cite[Def.~2.17]{reg}, with the additional constraint that we consider
smooth objects. Indeed, we are interested in the canonical model associated to a 
(regularised) smooth noise, constructed in Proposition~\ref{prop:model} and Remark 
\ref{rem:canonical}, and in its renormalised versions, namely its orbit under the 
action of $\CG_-^\ex$, see Theorem~\ref{theo:algebra}. 

Finally, we restrict our attention to a class of models which are {\it random}, {\it
stationary} and have suitable integrability properties, see Definition~\ref{stationary}. 
In this case, we can define a particular
{\it deterministic} element of $\CG_-^\ex$ that gives rise to what we call the {\it BPHZ} renormalisation,
by analogy with the corresponding construction arising in perturbative QFT 
\cite{BP,Hepp,Zimmermann,FMRS85}, see Theorem~\ref{main:renormalisation}. 
We show that the BPHZ construction yields the {\it unique} element 
of $\CG_-^\ex$ such that the associated renormalised model yields a {\it centered} family 
of stochastic processes on the {\it finite} family of elements in $\CT^\ex$ with negative degree.
This is the {\it algebraic} step of the renormalisation procedure.

This is the point where the companion analytical paper \cite{Ajay} starts, and then goes
on to prove that the BPHZ renormalised model does converge in the metric ${\rm d}$ on $\MM$, thus achieving the {\it probabilistic} step mentioned above and thereby completing 
the renormalisation procedure. 

The BPHZ functional is expressed explicitly in terms of an interesting map that
we call {\it negative twisted antipode} by analogy to \cite{1126-6708-1999-09-024}, see Proposition 
\ref{prop:twisted-} and \eqref{e:defHeps}. There is also a {\it positive twisted 
antipode}, see Proposition~\ref{prop:twisted+}, which plays a similarly important role
in \eqref{e:defModel1}. The main point is that these twisted antipodes encode 
in the compact formulae \eqref{e:defModel1}
and \eqref{e:defHeps} a number of nontrivial computations.

How are these spaces and operators defined? Since the analytic theory of \cite{reg} is based on
{\it generalised Taylor expansions} of solutions, the vector space $\CT^\ex$ is generated
by a basis which codes the relevant generalised Taylor monomials, which are defined
iteratively once a rule (i.e.\ a system of equations) is fixed. Definitions~\ref{def:conform}, 
\ref{def_space} and~\ref{def:CT} ensure that $\CT^\ex$ is sufficiently rich to allow 
one to rewrite \eqref{singularSPDE} as a fixed point problem in a space of functions
with values in our regularity structure. Moreover $\CT^\ex$ must also be invariant under the actions 
of $\CG_{\pm}^\ex$. This is the aim of the construction in Sections~\ref{sec2}, 
\ref{sec3} and~\ref{sec4}, that we want now to describe.

The spaces which are constructed in Section~\ref{sec5} depend on the choice of a number
of parameters, like the dimension of the coordinate space, the leading 
differential operator in the equation (the Laplacian being just one of many possible
choices), the non-linearity, the noise. In the previous sections we have built {\it 
universal} objects with nice algebraic properties which depend on none of these choices, but for the dimension of the 
space, namely an (arbitrary) integer number $d$ fixed once for all.

The spaces $\CT^\ex$, $\CT^\ex_+$ and $\CT^\ex_-$ are obtained by considering repeatedly
suitable {\it subsets} and suitable {\it quotients} of two initial spaces, called $\Tra_1$ 
and $\Tra_2$ and defined in and after Definition~\ref{def:admspecific}; more precisely, 
$\Tra_1$ is the ancestor of $\CT^\ex$ and $\CT^\ex_-$, while $\Tra_2$ is the ancestor of 
$\CT^\ex_+$. In Section~\ref{sec4} we represent these spaces as linearly generated by a 
collection of decorated forests, on which we can define suitable algebraic operations like
a product and a coproduct, which are later inherited by $\CT^\ex$, $\CT^\ex_+$ and 
$\CT^\ex_-$ (through other intermediary spaces which are called $\CH_\circ$, $\CH_1$ and 
$\hat\CH_2$). An important difference between $\CT^\ex_-$ and $\CT^\ex_+$ is that the
former is linearly generated by a family of forests, while the latter is linearly generated by a 
family of trees; this difference extends to the algebra structure: $\CT^\ex_-$ is endowed with
a {\it forest product} which corresponds to the disjoint union, while $\CT^\ex_+$ is endowed
with a {\it tree product} whereby one considers a disjoint union and then identifies the roots. 

The content of Section~\ref{sec4} is based on a specific definition of the spaces $\Tra_1$ 
and $\Tra_2$. In Sections~\ref{sec2} and~\ref{sec3} however we present a number of results
on a family of spaces $(\Tra_i)_{i\in I}$ with $I\subset\N$, which are supposed to satisfy
a few assumptions; Section~\ref{sec4} is therefore only a particular example of a more
general theory, which is outlined in Sections~\ref{sec2} and~\ref{sec3}. In this general 
setting we consider spaces $\Tra_i$ of decorated forests, and vector spaces 
$\scal{\Tra_i}$ of infinite series of such forests. Such series are not arbitrary but 
adapted to a grading, see Section~\ref{sec:bigraded}; this is needed since our abstract 
coproducts of Definition~\ref{def:maps} contain infinite series and might be ill-defined 
if were to work on arbitrary formal series.

The family of spaces $(\Tra_i)_{i\in I}$ are introduced in Definition~\ref{def:tra_i}
on the basis of families of admissible forests $\Adm_i$, $i\in I$. If $(\Adm_i)_{i\in I}$ 
satisfy Assumptions~\ref{ass:1}, \ref{ass:2}, \ref{ass:3}, \ref{ass:4}, \ref{ass:5} and 
\ref{ass:6}, then the coproducts $\Delta_i$ of Definition~\ref{def:maps} are coassociative 
and moreover $\Delta_i$ and $\Delta_j$ for $i<j$ are in {\it cointeraction}, see 
\eqref{prop:doublecoass}. As already mentioned, the cointeraction property is the 
algebraic formula behind the fundamental relation \eqref{elltaug} between the actions of 
$\CG^\ex_+$ and $\CG^\ex_-$ on $\CT^\ex_+$.
Appendix~\ref{sec:diagrams} contains a summary of the relations between the most important spaces
appearing in this article, while Appendix~\ref{sec:index} contains a symbolic index.

\subsection*{Acknowledgements}

{\small
We are very grateful to Christian Brouder, Ajay Chandra, Alessandra Frabetti, 
Dominique Manchon and Kurusch Ebrahimi-Fard for many interesting discussions and pointers to the literature.
MH gratefully acknowledges support by the Leverhulme Trust and by an ERC consolidator grant, project 615897 (Critical). LZ gratefully acknowledges support by the Institut Universitaire de France and the project of the Agence Nationale de la Recherche ANR-15-CE40-0020-01 grant LSD. The authors thank the organisers and the participants of a workshop
held in Bergen in April 2017, where the results of this paper were presented and discussed in detail.
}

\section{Rooted forests and bigraded spaces}

\label{sec2}

Given a finite set $S$ and a map $\ell \colon S \to \N$, we write
\begin{equ}
\ell! \eqdef \prod_{x \in S} \ell(x)!\;,
\end{equ}
and we define the corresponding binomial coefficients accordingly.
Note that if $\ell_1$ and $\ell_2$ have disjoint supports, then 
$(\ell_1 + \ell_2)! = \ell_1!\,\ell_2!$. Given a map 
$\pi\colon S \to \bar S$, we also define
$\pi_\star \ell \colon \bar S \to \N$ by $\pi_\star \ell(x) = \sum_{y \in \pi^{-1}(x)} \ell(y)$. 

For $k,\ell:S \to \N$ we define
\[
\binom{k}{\ell}\eqdef \prod_{x \in S}\binom{k(x)}{\ell(x)}\;,
\]
with the convention
$\binom{k}{\ell} = 0$ unless $0\le \ell\leq k$, which will be used
throughout the paper.
With these definitions at hand, one has
the following slight reformulation of the classical Chu-Vandermonde
identity.

\begin{lemma}[Chu-Vandermonde]
For every $k \colon S \to \N$, one has the identity
\begin{equ}
\sum_{\ell\,:\, \pi_\star \ell}\binom{k}{\ell} = \binom{\pi_\star k}{\pi_\star \ell}\;,
\end{equ}
where the sum runs over all possible choices of 
$\ell$ such that $\pi_\star \ell$ is fixed.\qed
\end{lemma}

\begin{remark}
These notations are also consistent with the case where the maps $k$ and $\ell$ are multi-index valued
under the natural identification of a map $S \to \N^d$ with a map 
$S \times \{1,\ldots,\infty\} \to \N$
given by $\ell(x)_i \leftrightarrow \ell(x,i)$. 
\end{remark}

\subsection{Rooted trees and forests}
\label{sec:rooted}

Recall that a rooted tree $T$ is a finite tree (a finite connected simple graph 
without cycles) with a distinguished vertex, $\rho=\rho_T$, called {\it the root}. 
Vertices of $T$, also called nodes, are denoted by $N=N_T$ and edges by $E=E_T\subset N^2$.
Since we want our trees to be rooted, they need to have at least one node,
so that we do not allow for trees with $N_T = \emptyset$. We do however allow
for the \textit{trivial} tree consisting of an empty edge set and a vertex set with only
one element. This tree will play a special role in the sequel and will be denoted by $\bullet$.
We will always assume that our trees are combinatorial meaning that there is no particular order 
imposed on edges leaving any given vertex.

Given a rooted tree $T$, we also endow $N_T$ with the partial order $\le$ where $w \le v$ if and only
if $w$ is on the unique path connecting $v$ to the root, and we orient edges in $E_T$ 
so that if $(x,y) = (x \to y) \in E_T$, then $x \le y$. In this way, we can always view a tree
as a directed graph.

Two rooted trees $T$ and $T'$ are \textit{isomorphic} if there exists a bijection
$\iota \colon E_T \to E_{T'}$ which is coherent
in the sense that there exists a bijection $\iota_N \colon N_T \to N_{T'}$ such that
$\iota(x,y) = (\iota_N(x),\iota_N(y))$ for any edge $(x,y) \in e$ and such that the
roots are mapped onto each other. 

We say that a rooted tree is \textit{typed} if it is furthermore endowed
with a function $\Labhom\colon E_T \to \Lab$, where $\Lab$\label{lab} is some finite set of types.
We think of $\Lab$ as being fixed once and for all and will sometimes omit to mention it
in the sequel. In particular, we will never make explicit the dependence on the choice
of $\Lab$ in our notations.  Two typed trees $(T,\Labhom)$ and $(T',\Labhom')$
are isomorphic if $T$ and $T'$ are isomorphic and $\Labhom$ is
pushed onto $\Labhom'$ by the corresponding isomorphism $\iota$ in the sense 
that $\Labhom'\circ \iota= \Labhom$.

Similarly to a tree, a \textit{forest} $F$ is a finite simple graph 
(again with nodes $N_F$ and edges $E_F \subset N_F^2$) 
without cycles. A forest $F$ is \textit{rooted}
if every connected component $T$ of $F$ is a rooted tree with root $\rho_T$. 
As above, we will consider forests that are typed in the sense that they are endowed with
a map $\Labhom\colon E_F \to \Lab$, and we consider the same notion of isomorphism between
typed forests as for typed trees. Note that while a tree is non-empty by definition, a forest
can be empty. We denote the empty forest by either $\one$ or $\emptyset$.

Given a typed forest $F$, a subforest $A \subset F$ consists of subsets $E_A \subset E_F$ and $N_A \subset N_F$ such that if $(x,y) \in E_A$ then $\{x,y\}\subset N_A$. Types in $A$ are inherited from $F$. A connected component of $A$ is a tree whose root is defined to be the minimal node in the partial order inherited from $F$.
We say that subforests $A$ and $B$ are disjoint, and write $A \cap B = \emptyset$, if
one has $N_A \cap N_B = \emptyset$ (which also implies that $E_A\cap E_B = \emptyset$).
Given two typed forests $F,G$, we write $F\sqcup G$ for the typed forest obtained by
taking the disjoint union (as graphs) of the two forests $F$ and $G$ and adjoining to it
the natural typing inherited from $F$ and $G$.
If furthermore $A \subset F$ and $B \subset G$ are subforests,
then we write $A \sqcup B$ for the corresponding subforest of 
$F \sqcup G$.

We fix once and for all an integer $d\geq 1$, dimension of the parameter-space $\R^d$. We also denote by $\Z(\Lab)$ the free abelian group generated by $\Lab$.

\subsection{Coloured and decorated forests}

Given a typed forest $F$, we want now to consider families of {\it disjoint subforests} of $F$, denoted 
by $(\hat F_i, i>0)$. It is convenient for us to code this family with a single function 
$\hat F:E_F \sqcup N_F \to \N$ as given by the next definition.

\begin{definition}\label{def:colour}
A {\it coloured forest} is a pair $(F,\hat F)$ such that
\begin{enumerate}
\item  $F = (E_F,N_F,\Labhom)$ is a typed rooted forest
\item $\hat F \colon E_F \sqcup N_F \to \N$ 
is such that if $\hat F(e) \neq 0$ for $e=(x,y) \in E_F$ then $\hat F(x) = \hat F(y) = \hat F(e)$.
\end{enumerate}
We say that $\hat F$ is a {\it colouring} of $F$. For $i > 0$, we define the subforest of $F$
\[
\hat F_i = (\hat E_i, \hat N_i), \qquad \hat E_i =\hat F^{-1}(i)\cap E_F, \quad \hat N_i =\hat F^{-1}(i)\cap N_F,
\]
as well as  $\hat E = \bigcup_{i > 0} \hat E_i$. We denote by $\mathfrak{C}$ the set of coloured forests.
\end{definition}

The condition on $\hat F$ guarantees that every $\hat F_i$ is indeed a subforest of $F$
for $i>0$ and 
that they are all disjoint. On the other hand, $\hat F^{-1}(0)$ is not supposed to have 
any particular structure and $0$ is not counted as a colour.
\begin{example}\label{ex:colours0}
This is an example of a forest with two colours: red  for $ 1 $ and blue for $ 2 $ (and black for $0$)
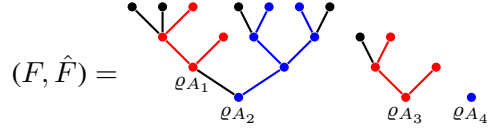
\begin{equ}
 (F,\hat F) =   
\begin{tikzpicture}[scale=0.2,baseline=0.2cm]
        \node at (0,0)  [dot,blue,label={[label distance=-0.2em]below: \scriptsize  $ \rho_{A_2} $}] (root) {};
         \node at (-7,6)  [dot] (leftll) {};
          \node at (-5,4)  [dot,red,label=left:] (leftl) {};
          \node at (-3,6)  [dot,red] (leftlr) {};
          \node at (-5,6)  [dot] (leftlc) {};
      \node at (-3,2)  [dot,color=red,label={[label distance=-0.2em]below: \scriptsize $ \rho_{A_1} $ }] (left) {};
         \node at (-1,4)  [dot,red] (leftr) {};
         \node at (1,4)  [dot,blue,label=above:   ] (rightl) {};
          \node at (0,6)  [dot] (rightll) {};
           \node at (2,6)  [dot,blue] (rightlr) {};
           \node at (5,4)  [dot,blue,label=left:] (rightr) {};
            \node at (4,6)  [dot,blue] (rightrl) {};
        \node at (3,2) [dot,blue] (right) {};
         \node at (6,6)  [dot] (rightrr) {};
        
        \draw[kernel1] (left) to node [sloped,below] {\small } (root); ;
        \draw[kernel1,red] (leftl) to
     node [sloped,below] {\small }     (left);
     \draw[kernel1,red] (leftlr) to
     node [sloped,below] {\small }     (leftl); 
     \draw[kernel1] (leftll) to
     node [sloped,below] {\small }     (leftl);
     \draw[kernel1,red] (leftr) to
     node [sloped,below] {\small }     (left);  
        \draw[kernel1,blue] (right) to
     node [sloped,below] {\small }     (root);
      \draw[kernel1] (leftlc) to
     node [sloped,below] {\small }     (leftl); 
     \draw[kernel1,blue] (rightr) to
     node [sloped,below] {\small }     (right);
     \draw[kernel1] (rightrr) to
     node [sloped,below] {\small }     (rightr);
     \draw[kernel1,blue] (rightrl) to
     node [sloped,below] {\small }     (rightr);
     \draw[kernel1,blue] (rightl) to
     node [sloped,below] {\small }     (right);
     \draw[kernel1,blue] (rightlr) to
     node [sloped,below] {\small }     (rightl);
     \draw[kernel1] (rightll) to
     node [sloped,below] {\small }     (rightl);
\end{tikzpicture} 
\begin{tikzpicture}[scale=0.2,baseline=0.2cm]
          \node at (-5,2)  [dot,red,label=left:] (leftl) {};
          \node at (-4,4)  [dot,red] (leftlr) {};
          \node at (-6,4)  [dot] (leftlc) {};
      \node at (-3,0)  [dot,color=red,label={[label distance=-0.2em]below: \scriptsize  $ \rho_{A_3} $}] (left) {};
         \node at (-1,2)  [dot,red] (leftr) {};
        
        \draw[kernel1,red] (leftl) to
     node [sloped,below] {\small }     (left);
     \draw[kernel1,red] (leftlr) to
     node [sloped,below] {\small }     (leftl); 
     \draw[kernel1,red] (leftr) to
     node [sloped,below] {\small }     (left);  
      \draw[kernel1] (leftlc) to
     node [sloped,below] {\small }     (leftl); 
\end{tikzpicture} 
\begin{tikzpicture}[scale=0.2,baseline=0.2cm]
      \node at (-3,0)  [dot,color=blue,label={[label distance=-0.2em]below: \scriptsize  $ \rho_{A_4} $}] (left) {};
\end{tikzpicture} 
\end{equ}
We then have $\hat F_1= \hat F^{-1}(1) = A_1\sqcup A_3 $ and $\hat F_2=\hat F^{-1}(2) = A_2\sqcup A_4$. 
\end{example}

The set $\mathfrak{C}$ is a commutative monoid under the {\it forest product}
\begin{equ}[cdotC]
(F,\hat F)
\cdot (G,\hat G)
= (F\sqcup G,\hat F+ \hat G)\;,
\end{equ}
where colouringss defined on one of the forests are extended to the disjoint union by
setting them to vanish on the other forest. The neutral element for this associative product is the empty coloured
forest $\one$.

We add now {\it decorations} on the nodes and edges of a coloured forest. 
For this, we fix throughout this article 
an arbitrary ``dimension'' $d \in \N$ and we give the following definition.

\begin{definition}\label{def:decoration}
We denote by $\Tra$ the set of all $5$-tuples $(F,\hat F,\Labn,\Labo,\Labe)$ such that
\begin{enumerate}
\item  $(F,\hat F)\in\C $ is a coloured forest in the sense of Definition~\ref{def:colour}.
\item One has $\Labn\colon N_F \to \N^d$ 
\item One has $\Labo \colon N_F \to \Z^d\oplus\Z(\Lab)$ with $\supp \Labo \subset \supp\hat F$.
\item One has $\Labe \colon E_F \to \N^d$ with $\supp\Labe\subset \{e\in E_F :\,  \hat F(e)=0\}=E_F\setminus\hat E$.
\end{enumerate}

\begin{remark}\label{whylabo}
The reason why $\Labo$ takes values in the space $\Z^d\oplus\Z(\Lab)$ will become apparent 
in \eqref{e:defhatn} below when we define the contraction of coloured subforests and
its action on decorations.
\end{remark}

We identify $(F,\hat F,\Labn,\Labo,\Labe)$ and $(F',\hat F',\Labn',\Labo',\Labe')$
whenever $F$ is isomorphic to $F'$, the corresponding isomorphism maps
$\hat F$ to $\hat F'$ and pushes the three decoration functions onto their counterparts.
We call elements of $\Tra$ {\it decorated forests}. We will also sometimes use the notation
$(F,\hat F)^{\Labn,\Labo}_\Labe$ instead of $(F,\hat F,\Labn,\Labo,\Labe)$.
\end{definition}

\begin{example}\label{ex:decorations} 
Let consider the  decorated forest $ (F,\hat F,\Labn,\Labo,\Labe) $ given by   
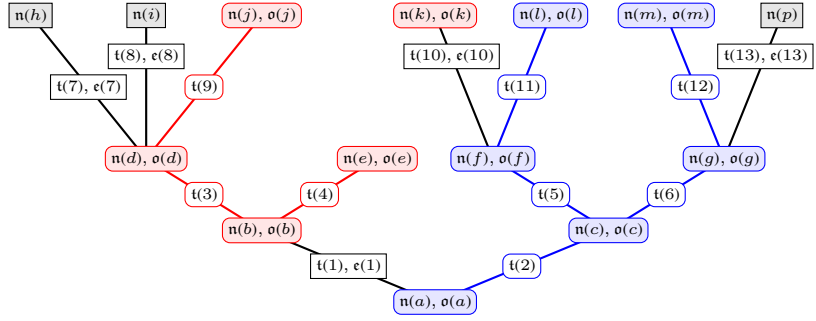
\begin{equs}   
\begin{tikzpicture}[scale=0.19,baseline=2cm]
        \node at (0,10)  (a) {}; %a
         \node at (-28,30)  (h) {}; %h
          \node at (-20,20) (d) {}; %d
          \node at (-12,30)  (j) {}; %j
          \node at (-20,30) (i) {}; %i
      \node at (-12,15) (b) {}; %b
         \node at (-4,20)  (e) {}; %e
         \node at (4,20)  (f) {}; %f
          \node at (0,30)  (k) {}; %k
           \node at (8,30) (l) {}; %l
           \node at (20,20)  (g) {}; %g
            \node at (16,30) (m) {};  %m
        \node at (12,15) (c) {}; %c
         \node at (24,30) (p) {}; %p

     \draw[kernel1,black] (h)node[rect2] {\tiny $ \Labn({h})$}  -- node [rect1] {\tiny$\Labhom(7),\Labe(7)$}   (d) --  node [near end, rect1]{\tiny $ \Labhom(8),\Labe(8) $ } (i) node[rect2] {\tiny $ \Labn({i})$};

     \draw[kernel1,red] (b) -- node [round1] {\tiny$\Labhom(3)$}   (d) node[round2] {\tiny $ \Labn(d),\Labo(d) $} --  node [round1]{\tiny $ \Labhom(9) $ } (j) node[round2] {\tiny $ \Labn({j}),\Labo({j}) $};
     \draw[kernel1,red] (b) -- node [round1] {\tiny$\Labhom(4)$}  (e) node[round2] {\tiny $ \Labn(e),\Labo(e) $} ;
    \draw[kernel1,black] (a) -- node [rect1] {\tiny $\Labhom(1),\Labe(1)$}  (b) node[red,round2] {\tiny $ \Labn(b),\Labo(b) $} ; 
    \draw[kernel1,blue] (a) -- node [round1] {\tiny $\Labhom(2)$}  (c) ; 
     \draw[kernel1,blue] (c) -- node [round1] {\tiny $\Labhom(5)$}  (f) ; 
     \draw[kernel1,blue] (c) -- node [round1] {\tiny $\Labhom(6)$}  (g) ; 
     \draw[kernel1,black] (f) -- node [near end, rect1] {\tiny $\Labhom({10}), \Labe({10})$}  (k) ; 
     \draw[kernel1,blue] (f) -- node [round1] {\tiny $\Labhom({11}) $}  (l) ; 
     \draw[kernel1,blue] (g) -- node [round1] {\tiny $\Labhom({12}) $}  (m) ; 
     \draw[kernel1,black] (g) -- node [near end, rect1] {\tiny $\Labhom({13}),\Labe({13}) $}  (p) ;
     
     \draw (p) node [rect2] {\tiny $ \Labn({p}) $}  ;
    \draw (m) node [blue,round3] {\tiny $ \Labn({m}), \Labo({m}) $}  ;
    \draw (l) node [blue,round3] {\tiny $ \Labn({l}), \Labo({l}) $}  ;
     \draw (k) node [red,round2] {\tiny $ \Labn(k), \Labo(k) $}  ;
    \draw (g) node [blue,round3] {\tiny $ \Labn(g),\Labo(g) $}  ;
    \draw (f) node [blue,round3] {\tiny $ \Labn(f),\Labo(f) $}  ;
    \draw (c) node [blue,round3] {\tiny $ \Labn(c),\Labo(c) $}  ;
    \draw (a) node [blue,round3] {\tiny $ \Labn(a),\Labo(a) $}  ;
\end{tikzpicture} 
\end{equs}
In this figure, the edges in $E_F$ are labelled with the numbers from $1$ to $13$ and the nodes in $N_F$ with the letters $\{a,b,c,f,e,f,g,h,i,j,k,l,m,p\}$. We set
$\hat F^{-1}(1)=\{b,d,e,j,k\}\sqcup\{3,4,9\}$ (red subforest), $\hat F^{-1}(2)=\{a,c,f,g,l,m\}\sqcup\{2,5,6,11,12\}$ (blue subforest), and on all remaining (black) nodes 
and edges $\hat F$ is set equal to $0$. Every edge has a type $\Labhom\in\Lab$, but only 
black 
edges have a possibly non-zero decoration $\Labe\in\N^d$. All nodes
have a decoration $\Labn\in\N^d$, but only coloured nodes have a possibly non-zero 
decoration $\Labo\in\Z^d\oplus\Z(\Lab)$. 
\end{example}
Example~\ref{ex:decorations} is continued in Examples~\ref{ex:colours}, \ref{ex:decor1} 
and~\ref{ex:decor2}.

\begin{definition}\label{def:contrac0}
For any coloured forest $(F,\hat F)$, we define an equivalence relation 
$\sim$ on the node set $N_F$
by saying that $x \sim y$ if 
$x$ and $y$ are connected in $\hat E$;
this is the smallest equivalence relation for which $x \sim y$ whenever $(x,y) \in \hat E$. 
\end{definition}
Definition~\ref{def:contrac0} will be extended to a decorated forest $(F,\hat F,\Labn,\Labo,\Labe)$ in Definition~\ref{CKop} 
below.

\begin{remark}\label{rem:intuition1}
We want to show the intuition behind decorated forests. We think of each $\tau=(F, \hat F,\Labn,\Labo,\Labe)$ 
as defining a function on $(\R^d)^{N_F}$ in the following way. 
We associate to each type $\Labhom\in\Lab$ a kernel $\varphi_\Labhom:\R^d\to\R$ and
we define the domain 
\[
U_F\eqdef \{x\in (\R^d)^{N_F} : \ x_v = x_w \quad \text{if} \quad v\sim w\}\;,
\]
where $\sim$ is the equivalence relation of Definition~\ref{def:contrac0}.
Then we set $H_\tau\in \CC^\infty(U_F)$,
\begin{equ}[function]
H_\tau(x_v, v\in N_{F}) \eqdef \prod_{v\in N_F} (x_v)^{\Labn(v)} \prod_{e=(u,v)\in E_F\setminus \hat E} \partial^{\Labe(e)}\varphi_{\Labhom(e)}(x_u-x_v),
\end{equ}
where, for $x=(x^1,\ldots,x^d)\in\R^d$, $n=(n^1,\ldots,n^d)\in\N^d$ and $\varphi\in \CC^\infty(\R^d)$
\[
(x)^n\eqdef\prod_{j=1}^d (x^j)^{n^j}, \qquad 
\partial^{n}\varphi= \partial^{n^1}_{x^1}\cdots\partial^{n^d}_{x^d}\varphi\in \CC^\infty(\R^d)\;.
\]
In this way,
a decorated forest encodes a function: every node in $N_F/\sim$ represents a variable in $\R^d$, every uncoloured edge of a 
certain type $\Labhom$ a function $ \varphi_{\Labhom(e)}$ of the difference of the two variables sitting at 
each one of its nodes; the decoration $\Labn(v)$ gives a power of $x_v$ and $\Labe(e)$ a derivative of the kernel $\varphi_{\Labhom(e)}$. 

In this example the decoration $\Labo$ plays no role; we shall see below that it allows to 
encode some additional information relevant for the various algebraic manipulations we wish to subject
these functions to, see Remarks~\ref{rem:intuition2}, \ref{rem:explanation},  \ref{explanation} and~\ref{rem:fails} below for further discussions.
\end{remark}

\begin{remark}\label{rem:connected}
Every forest $F=(N_F,E_F)$ has a unique decomposition into non-empty connected components. This property naturally extends to decorated forests $(F, \hat F,\Labn,\Labo,\Labe)$, by considering the connected components of the underlying forest $F$ and restricting the colouring $\hat F$ and the decorations $\Labn,\Labo,\Labe$.
\end{remark}

\begin{remark}
Starting from Section~\ref{sec4} we are going to consider a specific situation where there
are only two colours, namely $\hat F\to\{0,1,2\}$; all examples throughout the paper
are in this setting. However the results of Sections~\ref{sec2} and~\ref{sec3} are stated 
and proved in the more general setting $\hat F\to\N$ without any additional difficulty.
\end{remark}

\subsection{Bigraded spaces and triangular maps}
\label{sec:bigraded}

It will be convenient in the sequel to consider a particular category of  
\textit{bigraded spaces} as follows.
\begin{definition}\label{categ}
For a collection of vector spaces $\{V_n\,:\, n \in \N^2\}$, we define the vector space
\begin{equ}
V = \bplus_{n \in \N^2} V_n\;,
\end{equ}
as the space of all formal sums $\sum_{n \in \N^2} v_n$ with
$v_n \in V_n$ and such that there exists $k \in \N$ such that $v_n= 0$ as soon as
$n_2 > k$.
Given two bigraded spaces $V$ and $W$, we write $V \hattimes W$ for the bigraded space
\begin{equ}\label{newtensor}
V \hattimes W \eqdef \bplus_{n \in \N^2} \left[ \,
\bigoplus_{m+\ell = n}(V_m \otimes W_\ell)\right]. 
\end{equ}
\end{definition}
One has a canonical inclusion $V \otimes W \subset V \hattimes W$ given by
\[
\left(\sum_m v_m\right)\otimes\left(\sum_\ell w_\ell\right)\mapsto \sum_n\left(
\sum_{m+\ell=n} v_m\otimes w_\ell\right), \qquad v_m\in V_m, \ w_\ell\in W_\ell.
\]
However in general $V \hattimes W$ is strictly larger since its generic element has the form
\[
 \sum_n\left(
\sum_{m+\ell=n} v_m^n\otimes w_\ell^n\right),\qquad v_m^n\in V_m, \ w_\ell^n\in W_\ell.
\] 
Note that all tensor products we consider are algebraic.

\begin{definition}\label{def:triang}
We introduce a partial order on $\N^2$ by
\begin{equ}
(m_1, m_2) \ge (n_1, n_2)  \qquad\Leftrightarrow\qquad m_1 \ge n_1 \;\&\; m_2 \le n_2\;.
\end{equ}
Given two such bigraded spaces $V$ and $\bar V$, a family $\{A_{mn}\}_{m,n \in \N^2}$ of linear maps $A_{mn}:V_n\to\bar V_m$  is called \textit{triangular} if $A_{mn} = 0$ unless $m\ge n$.
\end{definition}
\begin{lemma}\label{lem:triang}
Let $V$ and $\bar V$ be two bigraded spaces and $\{A_{mn}\}_{m,n \in \N^2}$ a triangular family of linear maps $A_{mn}:V_n\to\bar V_m$. Then the map
\[
Av\eqdef \sum_m\Big(\sum_n A_{mn}v_n\Big)\in\bplus_{m \in \N^2} \bar V_m, \qquad v=\sum_nv_n\in \bplus_{n \in \N^2} V_n
\]
is well defined from $V$ to $\bar V$ and linear. We call $A:V\to V$ a triangular map.
\end{lemma}
\begin{proof}
Let $v=\sum_nv_n\in V$ and $k\in\N$ such that $v_n=0$ whenever $n_2>k$. 

First we note that, for fixed $m\in\N^2$, the family $(A_{mn}v_n)_{n\in\N^2}$ is zero unless $n\in[0,m_1]\times[0,k]$; indeed if $n_2>k$ then $v_n=0$, while if $n_1>m_1$ then $A_{mn}=0$. Therefore the sum $\sum_n A_{mn}v_n$ is well defined and equal to some $\bar v_m\in \bar V_m$.

We now prove that $\bar v_m=0$ whenever $m_2>k$, so that indeed $\sum_m\bar v_m\in \bplus_{m \in \N^2} \bar V_m$. Let $m_2> k$; for $n_2>k$, $v_n$ is $0$, while for $n_2\leq k$ we have $n_2< m_2$ and therefore $A_{nm}=0$ and this proves the claim. 
\end{proof}
A linear function $A:V\to\bar V$ which can be obtained as in Lemma~\ref{lem:triang} is called \textit{triangular}.
The family $(A_{mn})_{m,n\in\N^2}$ defines an infinite lower triangular matrix and composition of triangular maps is then simply given by formal matrix multiplication, which only ever involves finite sums thanks to the triangular structure of these matrices.

\begin{remark}\label{rem:triang}
The notion of bigraded spaces as above is useful for at least two reasons: 
\begin{enumerate}
\item The operators $\Delta_i$ built in \eqref{def:Deltabar} below turn out to be triangular in 
the sense of Definition~\ref{def:triang} and are therefore well-defined thanks to Lemma~\ref{lem:triang}, see Remark
\ref{rem:triang} below. This is not completely trivial since we are dealing with spaces of infinite formal series.
\item Some of our main tools below will be spaces of multiplicative functionals, see Section~\ref{sec:char} below.
Had we simply considered spaces of arbitrary
infinite formal series, their dual would be too small to contain any non-trivial multiplicative functional at all. 
Considering instead spaces of finite series would cure this problem, but unfortunately
the coproducts $\Delta_i$ do not make sense there.
The notion of bigrading introduced here provides the best of both worlds by 
considering bi-indexed series that are infinite in the first index and finite in the second.  
This yields spaces that are sufficiently large to contain our coproducts and whose dual is still
sufficiently large to contain enough multiplicative linear functionals for our purpose.
\end{enumerate}
\end{remark}

\begin{remark}\label{rem:dual}
One important remark is that this construction behaves quite nicely under duality in the
sense that if $V$ and $W$ are two bigraded spaces, then it is still the case that one
has a canonical inclusion $V^* \otimes W^* \subset (V \hattimes W)^*$, see e.g. \eqref{e:defG} 
below for the applications we have in mind.
Indeed, the dual $V^*$ consists of formal sums $\sum_n v^*_n$ with $v_n^* \in V_n^*$
such that, for every $k \in \N$ there exists $f(k)$ such that $v^*_n = 0$
for every $n \in \N^2$ with $n_1 \ge f(n_2)$.
\end{remark}

The set ${\Forests}$, see Definition~\ref{def:decoration}, admits a number of
different useful gradings and bigradings. One bigrading that is well adapted 
to the construction we give below is 
\begin{equ}[e:grading0]
|(F,\hat F)^{\Labn,\Labo}_\Labe|_\bi \eqdef  (|\Labe|,|F \setminus (\hat F \cup \rho_F)|)\;,
\end{equ}
where 
\[
|\Labe| = \sum_{e \in E_F} |\Labe(e)|, \qquad |a|=\sum_{i=1}^d a_i, \qquad \forall \ a\in\N^d,
\]
and $|F \setminus (\hat F \cup \rho_F)|$ denotes the number of edges and vertices on 
which $\hat F$ vanishes that aren't roots of $F$.

For any subset $A\subseteq\Forests$ let now $\scal{A}$ denote the space built from $A$ with this
grading, namely
\begin{equ}[e:defForests]
\scal{A} \eqdef \bplus_{n \in \N^2}\Vec\{\CF \in A\,:\, |\CF|_\bi = n\}\;,
\end{equ} 
where $\Vec\, S$ denotes the free vector space generated by a set $S$. 
Note that in general $\scal{M}$ is larger than $\Vec\, M$.

The following simple fact will be used several times in the sequel.
Here and throughout this article, we use
as usual the notation $f \restr A$\label{frestr} for the restriction of a map $f$ to some subset $A$
of its domain.

\begin{lemma}\label{lem:PP=P}
Let $V = \bplus_n V_n$ be a bigraded space and let $P:V\to V$ be a triangular map preserving 
the bigrading of $V$ (in the sense that there exist linear maps $P_n \colon V_n \to V_n$ such that
$P \restr V_n = P_n$ for every $n$) and satisfying $P\circ P=P$.
Then, the quotient space $\hat V = V / \ker P$ is again bigraded
and one has canonical identifications
\[
\hat V = \bplus_n (V_n / \ker P_n) = \bplus_n (P_n V_n)\;.
\]
\end{lemma}

\section{Bialgebras, Hopf algebras and comodules of decorated forests}
\label{sec3}

In this section we want to introduce a general class of operators on spaces of decorated forests and show that, under suitable assumptions, one can construct in this way bialgebras, Hopf algebras and comodules.

We recall that $(H,\CM,\one,\Delta,\one^\star)$ is a {\it bialgebra} if:
\begin{itemize}
\item $H$ is a vector space over $\R$
\item there are a linear map $\CM:H\otimes H\to H$ (product) and an element $\one\in H$ 
(identity) such that $(H,\CM,\eta)$ is a unital associative algebra, where $\eta:\R\to H$
is the map $r\mapsto r\one$ (unit)
\item there are linear maps $\Delta:H\to H\otimes H$ (coproduct) and $\one^\star:H\to\R$ (counit), 
such that $(H,\Delta,\one^\star)$ is a counital coassociative coalgebra, namely
\begin{equation}\label{e:coasso}
(\Delta\otimes\id)\Delta=(\id\otimes\Delta)\Delta, \qquad (\one^\star\otimes\id)\Delta=
(\id\otimes\one^\star)\Delta=\id
\end{equation}
\item the coproduct and the counit are homomorphisms of algebras (or, equivalently, multiplication and unit are homomorphisms of coalgebras).
\end{itemize}

A {\it Hopf algebra} is a bialgebra $(H,\CM,\one,\Delta,\one^\star)$ endowed with a linear map
$\CA:H\to H$ such that
\begin{equation}\label{anti}
\CM(\id\otimes \CA)\Delta = \CM(\CA\otimes\id)\Delta= \one^\star\one.
\end{equation}

A {\it left comodule} over a bialgebra $(H,\CM,\one,\Delta,\one^\star)$ is a pair $(M,\psi)$ where $M$ is a vector space and $\psi:M\to H\otimes M$ is a linear map such that
\[
(\Delta\otimes\id)\psi=(\id\otimes\psi)\psi, \qquad (\one^\star\otimes\id)\psi=\id.
\]
Right comodules are defined analogously.

For more details on the theory of coalgebras, bialgebras, Hopf algebras and comodules we 
refer the reader to \cite{Molnar, Cartier}.

\subsection{Incidence coalgebras of forests}

Denote by $\mathcal P$ the set of all pairs $(G;F)$ such that $F$ is a typed forest and $G$ is a subforest of $F$ and by $\Vec({\mathcal P})$ the free vector space generated by $\mathcal P$. Suppose that for all $(G;F)\in{\mathcal P}$ we are given a (finite) collection $\Adm(G;F)$ of subforests $A$ of $F$ such that $G\subseteq A\subseteq F$. Then we define the linear map $\Delta:\Vec({\mathcal P})\to\Vec({\mathcal P})\otimes\Vec({\mathcal P})$ by
\begin{equ}[firstDelta]
\Delta(G;F) \eqdef \sum_{A\in\Adm(G;F)} (G;A)\otimes(A;F).
\end{equ}
We also define the linear functional $\one^\star:\Vec({\mathcal P})\to\R$ by $\one^\star(G;F):=\un{(G=F)}$. 
If $\Adm(G;F)$ is equal to the set of all subforests $A$ of $F$ containing $G$, then it is a simple exercise to show that $(\Vec({\mathcal P}),\Delta,\one^\star)$ is a coalgebra, namely
\eqref{e:coasso} holds. 
In particular, since the inclusion $G\subseteq F$ endows the set of typed forests 
with a partial order, 
$(\Vec({\mathcal P}),\Delta,\one^\star)$ is an example of an {\it incidence coalgebra}, see \cite{MR914660,schmitt99}. However, if $\Adm(F;G)$ is a more general class of subforests, then coassociativity is not granted in general and holds only under certain assumptions. 

Suppose now that, given a typed forest $F$, we want to consider not one but several disjoint subforests $G_1,\ldots,G_n$ of $F$. A natural way to code $(G_1,\ldots,G_n;F)$ is to use a coloured forest $(F,\hat F)$ where 
\[
\hat F(x)=\sum_k k\,\un{x\in G_k}, \qquad x\in N_F\sqcup E_F.
\]
Then, in the notation of Definition~\ref{def:colour}, we have $\hat F_i=G_i$ for $i>0$ and $\hat F^{-1}(0)=F\setminus(\cup_i G_i)$.

In order to define a generalisation of the operator $\Delta$ of formula \eqref{firstDelta} to this setting, we fix $i>0$ and assume the following.
\begin{assumption}\label{ass:1}
Let $i>0$.
For each coloured forest $(F, \hat F)$ as in Definition~\ref{def:colour}
we are given a collection
$\Adm_i(F,\hat F)$ of subforests of $F$ such that for every $A \in \Adm_i(F,\hat F)$
\begin{enumerate}
\item  $\hat F_i \subset A$ and $\hat F_{j} \cap A=\emptyset$ for every $j>i$,
\item for all $0<j<i$ and every connected component $T$ of $\hat F_{j}$, one has either $T \subset A$ or $T \cap A=\emptyset$.
\end{enumerate}
We also assume that $\Adm_i$ is compatible with the equivalence relation $\sim$
given by forest isomorphisms described above in the sense that if $A \in \Adm_i(F,\hat F)$ and
$\iota\colon (F,\hat F) \to (G, \hat G)$ is a forest isomorphism, then $\iota(A) \in \Adm_i(G,\hat G)$.
\end{assumption}

It is important to note that colours are denoted by positive integer numbers and are therefore ordered, so that the forests $\hat F_j$, $\hat F_i$ and $\hat F_k$ can play different roles in Assumption~\ref{ass:1} if $j<i<k$. This becomes crucial in our construction below, see Proposition~\ref{prop:doublecoass} and Remark~\ref{rem:i<j}.

\begin{lemma}\label{lem:new_coloured_forest}
Let $(F, \hat F)\in\C$ be a coloured forest and $A \in \Adm_i(F,\hat F)$. Write
\begin{itemize}
\item $\hat F\restr A$ for the restriction of $\hat F$ to $N_A\sqcup E_A$ 
\item $\hat F \cup_i A$ for the function on $E_F\sqcup N_F$ given by
\begin{equ}
(\hat F \cup_i A)(x) = 
\left\{\begin{array}{cl}
	i & \text{if $x \in E_A \sqcup N_A$,} \\
	\hat F(x) & \text{otherwise.}
\end{array}\right.
\end{equ} 
\end{itemize}
Then, under Assumption~\ref{ass:1}, $(A,\hat F\restr A)$ and $(F, \hat F \cup_i A)$ are coloured forests.
\end{lemma}
\begin{proof}
The claim is elementary for $(A,\hat F\restr A)$; in particular, setting $\hat G\eqdef\hat F\restr A$, we have $\hat G_j=\hat F_j\cap A$ for all $j>0$. We prove it now for $(F, \hat F \cup_i A)$. 
We must prove that, setting $\hat G\eqdef \hat F \cup_i A$, the sets $\hat G_j\eqdef\hat G^{-1}(j)$ define subforests of $F$ for all $j>0$. We have by the definitions
\[
\hat G_i=\hat F_i\cup A, \qquad \hat G_j=\hat F_j\backslash A, \qquad j\ne i, \ j>0,
\]
and these are subforests of $F$ by the properties 1 and 2 of Assumption~\ref{ass:1}.
\end{proof}

We denote by $\Vec(\mathfrak{C})$ the free vector space generated by all coloured forests.
This allows to define the following operator for fixed $i>0$, $\Delta_i:\Vec(\mathfrak{C})\to\Vec(\mathfrak{C})\otimes\Vec(\mathfrak{C})$
\begin{equ}[firstDelta_i]
\Delta_i(F,\hat F)\eqdef\sum_{A \in \Adm_i(F,\hat F)} (A,\hat F\restr A)\otimes (F,\hat F \cup_i A).
\end{equ}
Note that if $i=1$ and $\hat F\leq 1$ then we can identify
\begin{itemize}
\item the coloured forest $(F,\hat F)$ with the pair of subforests $(\hat F_1;F)\in{\mathcal P}$, 
\item $\Adm(\hat F_1;F)$ with $\Adm_1(F,\hat F)$
\item $\Delta$ in \eqref{firstDelta} with $\Delta_1$ in \eqref{firstDelta_i}. 
\end{itemize}

\begin{example}\label{ex:colours}
Let us continue Example~\ref{ex:decorations}, forgetting the decorations but keeping the
same labels for the nodes and in particular for the leaves. We recall that
$\hat F$ is equal to $ 1 $ on the red subforest, to $ 2 $ on the blue subforest and to $0$
elsewhere. Then
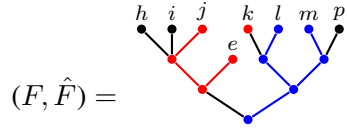
\begin{equ}
 (F,\hat F) =   
\begin{tikzpicture}[scale=0.2,baseline=0.2cm]
        \node at (0,0)  [dot,blue] (root) {};
         \node at (-7,6)  [dot,label={[label distance=-0.2em]above: \scriptsize  $ h $} ] (leftll) {};
          \node at (-5,4)  [dot,red,label=left:] (leftl) {};
          \node at (-3,6)  [dot,red,label={[label distance=-0.2em]above: \scriptsize  $ j $}] (leftlr) {};
          \node at (-5,6)  [dot,label={[label distance=-0.2em]above: \scriptsize  $ i $}] (leftlc) {};
      \node at (-3,2)  [dot,color=red] (left) {};
         \node at (-1,4)  [dot,red,label={[label distance=-0.2em]above: \scriptsize  $ e $}] (leftr) {};
         \node at (1,4)  [dot,blue,label=above:   ] (rightl) {};
          \node at (0,6)  [dot,red,label={[label distance=-0.2em]above: \scriptsize  $ k $} ] (rightll) {};
           \node at (2,6)  [dot,blue,label={[label distance=-0.2em]above: \scriptsize  $ l $}] (rightlr) {};
           \node at (5,4)  [dot,blue,label=left:] (rightr) {};
            \node at (4,6)  [dot,blue,label={[label distance=-0.2em]above: \scriptsize  $ m $}] (rightrl) {};
        \node at (3,2) [dot,blue,label={[label distance=-0.4em]below right: \scriptsize }] (right) {};
         \node at (6,6)  [dot,label={[label distance=-0.2em]above: \scriptsize  $ p $} ] (rightrr) {};
        
        \draw[kernel1] (left) to node [sloped,below] {\small } (root); ;
        \draw[kernel1,red] (leftl) to
     node [sloped,below] {\small }     (left);
     \draw[kernel1,red] (leftlr) to
     node [sloped,below] {\small }     (leftl); 
     \draw[kernel1] (leftll) to
     node [sloped,below] {\small }     (leftl);
     \draw[kernel1,red] (leftr) to
     node [sloped,below] {\small }     (left);  
        \draw[kernel1,blue] (right) to
     node [sloped,below] {\small }     (root);
      \draw[kernel1] (leftlc) to
     node [sloped,below] {\small }     (leftl); 
     \draw[kernel1,blue] (rightr) to
     node [sloped,below] {\small }     (right);
     \draw[kernel1] (rightrr) to
     node [sloped,below] {\small }     (rightr);
     \draw[kernel1,blue] (rightrl) to
     node [sloped,below] {\small }     (rightr);
     \draw[kernel1,blue] (rightl) to
     node [sloped,below] {\small }     (right);
     \draw[kernel1,blue] (rightlr) to
     node [sloped,below] {\small }     (rightl);
     \draw[kernel1] (rightll) to
     node [sloped,below] {\small }     (rightl);
     \end{tikzpicture} 
\end{equ}
A valid example 
 of $ A \in \Adm_{2}(F,\hat F) $ could be such that
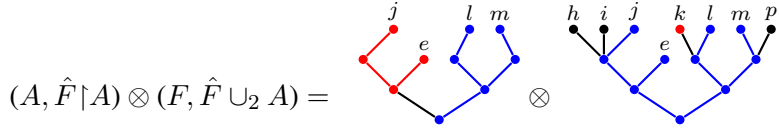
\begin{equs}     
  (A,\hat F\restr A) \otimes  (F,\hat F \cup_{2} A)  =  
\begin{tikzpicture}[scale=0.2,baseline=0.2cm]
        \node at (0,0)  [dot,blue        ] (root) {};
          \node at (-5,4)  [dot,red,label=left:] (leftl) {};
          \node at (-3,6)  [dot,red,label={[label distance=-0.2em]above: \scriptsize  $ j $}] (leftlr) {};
      \node at (-3,2)  [dot,color=red] (left) {};
         \node at (-1,4)  [dot,red,label={[label distance=-0.2em]above: \scriptsize  $ e $}] (leftr) {};
         \node at (1,4)  [dot,blue,label=above:   ] (rightl) {};
           \node at (2,6)  [dot,blue,label={[label distance=-0.2em]above: \scriptsize  $ l $}] (rightlr) {};
           \node at (5,4)  [dot,blue,label=left:] (rightr) {};
            \node at (4,6)  [dot,blue,label={[label distance=-0.2em]above: \scriptsize  $ m $}] (rightrl) {};
        \node at (3,2) [dot,blue,label={[label distance=-0.4em]below right: \scriptsize  $  $}] (right) {};
        
        \draw[kernel1] (left) to node [sloped,below] {\small } (root); 
        \draw[kernel1,red] (leftl) to
     node [sloped,below] {\small }     (left);
     \draw[kernel1,red] (leftlr) to
     node [sloped,below] {\small }     (leftl); 
     \draw[kernel1,red] (leftr) to
     node [sloped,below] {\small }     (left);  
        \draw[kernel1,blue] (right) to
     node [sloped,below] {\small }     (root); 
     \draw[kernel1,blue] (rightr) to
     node [sloped,below] {\small }     (right);
     \draw[kernel1,blue] (rightrl) to
     node [sloped,below] {\small }     (rightr);
     \draw[kernel1,blue] (rightl) to
     node [sloped,below] {\small }     (right);
     \draw[kernel1,blue] (rightlr) to
     node [sloped,below] {\small }     (rightl);
     \end{tikzpicture} 
       \otimes  
\begin{tikzpicture}[scale=0.2,baseline=0.2cm]
        \node at (0,0)  [dot,blue ] (root) {};
         \node at (-7,6)  [dot,label={[label distance=-0.2em]above: \scriptsize  $ h $} ] (leftll) {};
          \node at (-5,4)  [dot,blue,label=left:] (leftl) {};
          \node at (-3,6)  [dot,blue,label={[label distance=-0.2em]above: \scriptsize  $ j $}] (leftlr) {};
          \node at (-5,6)  [dot,label={[label distance=-0.2em]above: \scriptsize  $ i $}] (leftlc) {};
      \node at (-3,2)  [dot,color=blue,label={[label distance=-0.2em]below: \scriptsize  }] (left) {};
         \node at (-1,4)  [dot,blue,label={[label distance=-0.2em]above: \scriptsize  $ e $}] (leftr) {};
         \node at (1,4)  [dot,blue,label=above:   ] (rightl) {};
          \node at (0,6)  [dot,red,label={[label distance=-0.2em]above: \scriptsize  $ k $} ] (rightll) {};
           \node at (2,6)  [dot,blue,label={[label distance=-0.2em]above: \scriptsize  $ l $}] (rightlr) {};
           \node at (5,4)  [dot,blue,label=left:] (rightr) {};
            \node at (4,6)  [dot,blue,label={[label distance=-0.2em]above: \scriptsize  $ m $}] (rightrl) {};
        \node at (3,2) [dot,blue,label={[label distance=-0.4em]below right: \scriptsize  }] (right) {};
         \node at (6,6)  [dot,label={[label distance=-0.2em]above: \scriptsize  $ p $} ] (rightrr) {};
        
        \draw[kernel1,blue] (left) to node [sloped,below] {\small } (root); 
        \draw[kernel1,blue] (leftl) to
     node [sloped,below] {\small }     (left);
     \draw[kernel1,blue] (leftlr) to
     node [sloped,below] {\small }     (leftl); 
     \draw[kernel1] (leftll) to
     node [sloped,below] {\small }     (leftl);
     \draw[kernel1,blue] (leftr) to
     node [sloped,below] {\small }     (left);  
        \draw[kernel1,blue] (right) to
     node [sloped,below] {\small }     (root);
      \draw[kernel1] (leftlc) to
     node [sloped,below] {\small }     (leftl); 
     \draw[kernel1,blue] (rightr) to
     node [sloped,below] {\small }     (right);
     \draw[kernel1] (rightrr) to
     node [sloped,below] {\small }     (rightr);
     \draw[kernel1,blue] (rightrl) to
     node [sloped,below] {\small }     (rightr);
     \draw[kernel1,blue] (rightl) to
     node [sloped,below] {\small }     (right);
     \draw[kernel1,blue] (rightlr) to
     node [sloped,below] {\small }     (rightl);
     \draw[kernel1] (rightll) to
     node [sloped,below] {\small }     (rightl);
     \end{tikzpicture}
  \end{equs}
Note that in this example, one has $\hat F_2 \subset A$, so that $A\notin\Adm_{1}(F,\hat F) $ since $A$ violates
the first condition of Assumption~\ref{ass:1}. A valid example of $ B \in \Adm_{1}(F,\hat F) $ 
could be such that 
 \begin{equs}     
  (B,\hat F\restr B) \otimes  (F,\hat F \cup_{1} B)  =  
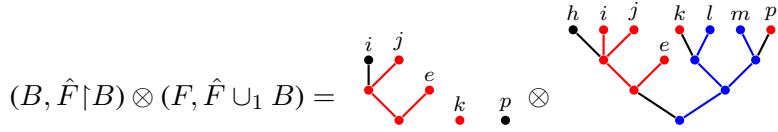
\begin{tikzpicture}[scale=0.2,baseline=0.2cm]
          \node at (-5,2)  [dot,red,label=left:] (leftl) {};
          \node at (-3,4)  [dot,red,label={[label distance=-0.2em]above: \scriptsize  $ j $}] (leftlr) {};
          \node at (-5,4)  [dot,label={[label distance=-0.2em]above: \scriptsize  $ i $}] (leftlc) {};
      \node at (-3,0)  [dot,color=red] (left) {};
         \node at (-1,2)  [dot,red,label={[label distance=-0.2em]above: \scriptsize  $ e $}] (leftr) {};
         \node at (1,0)  [dot,red,label={[label distance=-0.2em]above: \scriptsize  $ k $}] (o) {};
         \node at (4,0)  [dot,label={[label distance=-0.2em]above: \scriptsize  $ p $}] (o) {};
        
        \draw[kernel1,red] (leftl) to
     node [sloped,below] {\small }     (left);
     \draw[kernel1,red] (leftlr) to
     node [sloped,below] {\small }     (leftl); 
     \draw[kernel1,red] (leftr) to
     node [sloped,below] {\small }     (left);  
      \draw[kernel1] (leftlc) to
     node [sloped,below] {\small }     (leftl); 
     \end{tikzpicture} 
      \otimes   
\begin{tikzpicture}[scale=0.2,baseline=0.2cm]
        \node at (0,0)  [dot,blue] (root) {};
         \node at (-7,6)  [dot,label={[label distance=-0.2em]above: \scriptsize  $ h $} ] (leftll) {};
          \node at (-5,4)  [dot,red,label=left:] (leftl) {};
          \node at (-3,6)  [dot,red,label={[label distance=-0.2em]above: \scriptsize  $ j $}] (leftlr) {};
          \node at (-5,6)  [dot,red,label={[label distance=-0.2em]above: \scriptsize  $ i $}] (leftlc) {};
      \node at (-3,2)  [dot,color=red] (left) {};
         \node at (-1,4)  [dot,red,label={[label distance=-0.2em]above: \scriptsize  $ e $}] (leftr) {};
         \node at (1,4)  [dot,blue,label=above:   ] (rightl) {};
          \node at (0,6)  [dot,red,label={[label distance=-0.2em]above: \scriptsize  $ k $} ] (rightll) {};
           \node at (2,6)  [dot,blue,label={[label distance=-0.2em]above: \scriptsize  $ l $}] (rightlr) {};
           \node at (5,4)  [dot,blue,label=left:] (rightr) {};
            \node at (4,6)  [dot,blue,label={[label distance=-0.2em]above: \scriptsize  $ m $}] (rightrl) {};
        \node at (3,2) [dot,blue,label={[label distance=-0.4em]below right: \scriptsize }] (right) {};
         \node at (6,6)  [dot,red,label={[label distance=-0.2em]above: \scriptsize  $ p $} ] (rightrr) {};
        
        \draw[kernel1] (left) to node [sloped,below] {\small } (root); ;
        \draw[kernel1,red] (leftl) to
     node [sloped,below] {\small }     (left);
     \draw[kernel1,red] (leftlr) to
     node [sloped,below] {\small }     (leftl); 
     \draw[kernel1] (leftll) to
     node [sloped,below] {\small }     (leftl);
     \draw[kernel1,red] (leftr) to
     node [sloped,below] {\small }     (left);  
        \draw[kernel1,blue] (right) to
     node [sloped,below] {\small }     (root);
      \draw[kernel1,red] (leftlc) to
     node [sloped,below] {\small }     (leftl); 
     \draw[kernel1,blue] (rightr) to
     node [sloped,below] {\small }     (right);
     \draw[kernel1] (rightrr) to
     node [sloped,below] {\small }     (rightr);
     \draw[kernel1,blue] (rightrl) to
     node [sloped,below] {\small }     (rightr);
     \draw[kernel1,blue] (rightl) to
     node [sloped,below] {\small }     (right);
     \draw[kernel1,blue] (rightlr) to
     node [sloped,below] {\small }     (rightl);
     \draw[kernel1] (rightll) to
     node [sloped,below] {\small }     (rightl);
     \end{tikzpicture}
  \end{equs}
     \end{example}
     
In the rest of this section we state several assumptions on the family $\Adm_i(F,\hat F)$ yielding nice properties for the operator $\Delta_i$ such as coassociativity, see e.g. Assumption~\ref{ass:2}. 
However, one of the main results of this article is the fact that such properties then automatically
also hold at the level of decorated forests with a non-trivial action on the decorations which 
will be defined in the next subsection.

\subsection{Operators on decorated forests}

The set $\Tra$, see Definition~\ref{def:decoration}, is a commutative monoid under the {\it forest product}
\begin{equ}[cdot]
(F,\hat F,\Labn,\Labo,\Labe)
\cdot (G,\hat G,\Labn',\Labo',\Labe')
= (F\sqcup G,\hat F+ \hat G,\Labn + \Labn',\Labo + \Labo',\Labe + \Labe')\;,
\end{equ}
where decorations defined on one of the forests are extended to the disjoint union by
setting them to vanish on the other forest. This product is the natural extension of the product \eqref{cdotC}
on coloured forests and its identity element is the empty forest $\one$. 

Note that 
\begin{equ}[e:prodgrad]
|\CF \cdot \CG|_\bi = |\CF|_\bi + |\CG|_\bi\;,
\end{equ}
for any $\CF, \CG \in \Forests$, where $|\cdot|_\bi$ is the bigrading defined in \eqref{e:grading0} above.
Whenever $M$ is a submonoid of $\Forests$, 
as a consequence of \eqref{e:prodgrad} the forest product
$\cdot$ defined in \eqref{cdot} can be interpreted as a triangular linear map from $\scal{M}\hattimes \scal{M}$
into $\scal{M}$, thus
turning $(\scal{M},\cdot)$ into an algebra in the category of bigraded spaces as in Definition~\ref{categ}; this is in particular the case for $M=\Forests$. We recall that $\scal{M}$ is defined in \eqref{e:defForests}.

We generalise now the construction \eqref{firstDelta_i} to decorated forests.
\begin{definition}\label{def:maps}
The triangular linear maps $\Delta_i \colon \scal{\Tra} \to \scal{\Tra}\hattimes \scal{\Tra}$ are given
for $\tau = (F, \hat F,\Labn,\Labo,\Labe)$ by
\begin{equs}\label{def:Deltabar}
\Delta_i \tau &= 
 \sum_{A \in \Adm_i(F,\hat F)} \sum_{\varepsilon_A^F,\Labn_A} \frac1{\varepsilon_A^F!}
\binom{\Labn}{\Labn_A}
 (A,\hat F\restr A,\Labn_A+\pi\varepsilon_A^F, \Labo \restr N_A,\Labe \restr E_A) 
 \\& \qquad  \otimes(F,\hat F \cup_i A,\Labn - \Labn_A,\ \Labo+\Labn_A+\pi(\varepsilon_A^F-\Labe_\emptyset^A), \Labe_A^F + \varepsilon_A^F)\;, 
\end{equs}
where 
\begin{itemize}
\item[a)] For $A\subseteq B \subseteq F$ and $f:E_F\to\N^d$, we use the notation $f_A^B \eqdef f\,\un{E_B\setminus E_A}$. 
\item[b)] The sum over $\Labn_A$ runs over all maps $\Labn_A:N_F \to \N^d$ with $\supp \Labn_A\subset N_A$.
\item[c)] The sum over $\varepsilon_A^F$ runs over all $\varepsilon_A^F:E_F \to \N^d$ supported on 
the set of edges 
\begin{equ}\label{part0}
\d(A,F)  \eqdef \left\{  (e_+,e_-) \in E_F\setminus E_{A}  \,:\, 
 e_+ \in N_A \right\},
\end{equ}
that we call the {\it boundary of $A$ in $F$}. This notation is consistent with point a).
\item[d)] For all $\varepsilon: E_F\to\Z^d$ we denote
\[
\pi\varepsilon:N_F\to\Z^d, \qquad \pi\varepsilon(x)\eqdef \sum_{e=(x,y)\in E_F} \varepsilon(e).
\]
\end{itemize}
\end{definition}
We will henceforth use these notational conventions for sums over node / edge decorations
without always spelling them out in full.

%\begin{example}%\label{ex:decorations} 
%Let $(F,\hat F)$ and $A$ as in Example~\ref{ex:colours}. Then the boundary of $A$ in $F$ 
%is given by $\d(A,F)=\{(x,\ell_1),(x,\ell_2),(y,\ell_5),(z,\ell_8)\}$ for some $x,y,z\in N_A$.
%\end{example}
%
\begin{example}\label{ex:decor1}
We continue Examples~\ref{ex:decorations} and~\ref{ex:colours}, by showing how 
decorations are modified by $\Delta_i$. We consider first $i=2$, corresponding to
a blue subforest $A\in\Adm_2(F,\hat F)$. Then we have that
$ (A,\hat F\restr A,\Labn_A+\pi\varepsilon_A^F, \Labo \restr N_A,\Labe \restr E_A) $
is equal to

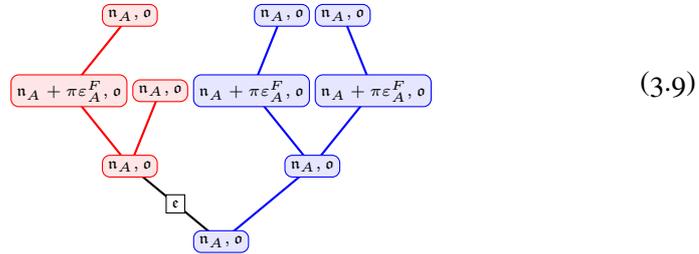
\begin{equ}\label{ex:tra2_1}   
\begin{tikzpicture}[scale=0.10,baseline=2cm]
        \node at (0,0)  (a) {}; %a
         \node at (-28,30)  (h) {}; %h
          \node at (-20,20) (d) {}; %d
          \node at (-12,30)  (j) {}; %j
          \node at (-20,30) (i) {}; %i
      \node at (-12,10) (b) {}; %b
         \node at (-8,20)  (e) {}; %e
         \node at (4,20)  (f) {}; %f
          \node at (0,30)  (k) {}; %k
           \node at (8,30) (l) {}; %l
           \node at (20,20)  (g) {}; %g
            \node at (16,30) (m) {};  %m
        \node at (12,10) (c) {}; %c
         \node at (24,30) (p) {}; %p

%     \draw[kernel1,black] (h)node[rect2] {\tiny $ \Labn$}  -- node [rect1] {\tiny$\Labhom,\Labe$}   (d) --  node [near end, rect1]{\tiny $ \Labhom,\Labe $ } (i) node[rect2] {\tiny $ \Labn$};

     \draw[kernel1,red] (b) --    (d) node[round2] {\tiny $ \Labn_A+\pi\varepsilon_A^F,\Labo $} --   (j) node[round2] {\tiny $ \Labn_A,\Labo $};
     \draw[kernel1,red] (b) --   (e) node[round2] {\tiny $ \Labn_A,\Labo $} ;
    \draw[kernel1,black] (a) -- node [rect1] {\tiny $\Labe$}  (b) node[red,round2] {\tiny $ \Labn_A,\Labo $} ; 
    \draw[kernel1,blue] (a) --   (c) ; 
     \draw[kernel1,blue] (c) --   (f) ; 
     \draw[kernel1,blue] (c) --   (g) ; 
 %    \draw[kernel1,black] (f) -- node [rect1] {\tiny $\Labhom, \Labe$}  (k) ; 
     \draw[kernel1,blue] (f) --   (l) ; 
     \draw[kernel1,blue] (g) --   (m) ; 
%     \draw[kernel1,black] (g) -- node [near end, rect1] {\tiny $\Labhom,\Labe $}  (p) ;
     
%     \draw (p) node [rect2] {\tiny $ \Labn_A $}  ;
    \draw (m) node [blue,round3] {\tiny $ \Labn_A, \Labo $}  ;
    \draw (l) node [blue,round3] {\tiny $ \Labn_A, \Labo $}  ;
%     \draw (k) node [red,round2] {\tiny $ \Labn_A, \Labo $}  ;
    \draw (g) node [blue,round3] {\tiny $ \Labn_A+\pi\varepsilon_A^F,\Labo $}  ;
    \draw (f) node [blue,round3] {\tiny $ \Labn_A+\pi\varepsilon_A^F,\Labo $}  ;
    \draw (c) node [blue,round3] {\tiny $ \Labn_A,\Labo $}  ;
    \draw (a) node [blue,round3] {\tiny $ \Labn_A,\Labo $}  ;
\end{tikzpicture} 
\end{equ}
while 
$(F,\hat F \cup_2 A,\Labn - \Labn_A,\ \Labo+\Labn_A+\pi(\varepsilon_A^F-\Labe_\emptyset^A),\Labe_A^F + \varepsilon_A^F)$ becomes
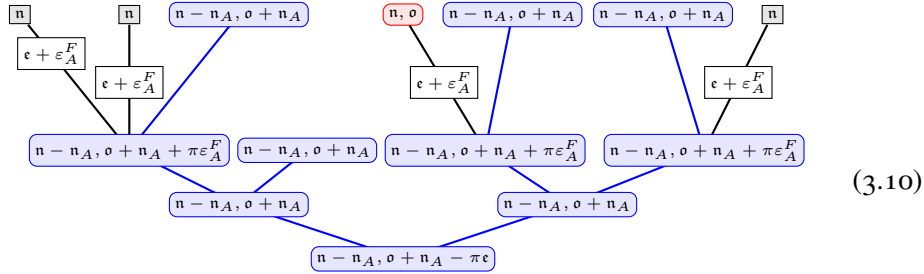
\begin{equ}\label{ex:tra2_2}   
\begin{tikzpicture}[scale=0.18,baseline=2cm]
        \node at (0,6)  (a) {}; %a
         \node at (-28,24)  (h) {}; %h
          \node at (-20,14) (d) {}; %d
          \node at (-12,24)  (j) {}; %j
          \node at (-20,24) (i) {}; %i
      \node at (-12,10) (b) {}; %b
         \node at (-7,14)  (e) {}; %e
         \node at (6,14)  (f) {}; %f
          \node at (0,24)  (k) {}; %k
           \node at (8,24) (l) {}; %l
           \node at (22,14)  (g) {}; %g
            \node at (19,24) (m) {};  %m
        \node at (12,10) (c) {}; %c
         \node at (27,24) (p) {}; %p

     \draw[kernel1,black] (h)node[rect2] {\tiny $ \Labn$}  -- node [near start, rect1] {\tiny$\Labe+\varepsilon_A^F$}   (d) --  node [ rect1]{\tiny $ \Labe+\varepsilon_A^F $ } (i) node[rect2] {\tiny $ \Labn$};

     \draw[kernel1,blue] (b) --   (d) node[round3] {\tiny $ \Labn-\Labn_A,\Labo+\Labn_A +\pi\varepsilon_A^F$} --   (j) node[round3] {\tiny $ \Labn-\Labn_A,\Labo+\Labn_A $};
     \draw[kernel1,blue] (b) --  (e) node[round3] {\tiny $ \Labn-\Labn_A,\Labo+\Labn_A $} ;
    \draw[kernel1,blue] (a) --   (b) node[blue,round3] {\tiny $ \Labn-\Labn_A,\Labo+\Labn_A $} ; 
    \draw[kernel1,blue] (a) --   (c) ; 
     \draw[kernel1,blue] (c) --   (f) ; 
     \draw[kernel1,blue] (c) --   (g) ; 
     \draw[kernel1,black] (f) -- node [rect1] {\tiny $ \Labe+\varepsilon_A^F$}  (k) ; 
     \draw[kernel1,blue] (f) --   (l) ; 
     \draw[kernel1,blue] (g) --   (m) ; 
    \draw[kernel1,black] (g) -- node [rect1] {\tiny $\Labe+\varepsilon_A^F $}  (p) ;
     
     \draw (p) node [rect2] {\tiny $ \Labn$}  ;
    \draw (m) node [blue,round3] {\tiny $ \Labn-\Labn_A, \Labo+\Labn_A $}  ;
    \draw (l) node [blue,round3] {\tiny $ \Labn-\Labn_A, \Labo+\Labn_A $}  ;
     \draw (k) node [red,round2] {\tiny $ \Labn, \Labo $}  ;
    \draw (g) node [blue,round3] {\tiny $ \Labn-\Labn_A,\Labo+\Labn_A +\pi\varepsilon_A^F$}  ;
    \draw (f) node [blue,round3] {\tiny $ \Labn-\Labn_A,\Labo+\Labn_A +\pi\varepsilon_A^F$}  ;
    \draw (c) node [blue,round3] {\tiny $ \Labn-\Labn_A,\Labo+\Labn_A $}  ;
    \draw (a) node [blue,round3] {\tiny $ \Labn-\Labn_A,\Labo+\Labn_A-\pi\Labe $}  ;
\end{tikzpicture} 
\end{equ}
Note that $\varepsilon_A^F$ is
supported by $\partial(A,F)=\{7,8,10,13\}$, where we refer to the labelling of edges and nodes fixed in the Example~\ref{ex:decorations}, and
\[
\pi\varepsilon_A^F(d)=\varepsilon_A^F(7)+\varepsilon_A^F(8), \qquad\pi\varepsilon_A^F(f)=\varepsilon_A^F(10), \qquad \pi\varepsilon_A^F(g)=\varepsilon_A^F(13). 
\]
Note that the edge $1$ was black in $(F,\hat F)$ and becomes blue in $(F,\hat F \cup_2 A)$; accordingly, in $(F,\hat F \cup_2 A,\Labn - \Labn_A,\ \Labo+\Labn_A+\pi(\varepsilon_A^F-\Labe_\emptyset^A),\Labe_A^F + \varepsilon_A^F)$ the value of $\Labe$ on $1$ is set to $0$ and $\Labe(1)$ 
is subtracted from $\Labo(a)$. In accordance with Assumption~\ref{ass:1}, $A\in\Adm_2(F,\hat F)$ contains one of the two connected components
of $\hat F_1$ and is disjoint from the other one. 
\end{example}

\begin{example}\label{ex:decor2}
We continue Example~\ref{ex:decor1} for the choice of $B$ made in Example~\ref{ex:colours} 
and for $i=1$, corresponding to a red subforest $B\in\Adm_1(F,\hat F)$. Then
$ (B,\hat F\restr B,\Labn_B+\pi\varepsilon_B^F, \Labo \restr N_B,\Labe \restr E_B) $
is equal to
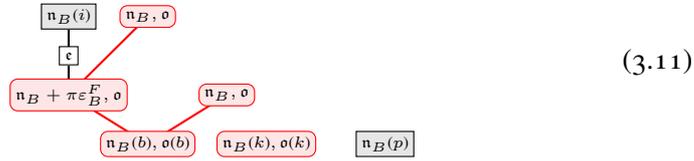
\begin{equ}\label{ex:tra1_1}   
\begin{tikzpicture}[scale=0.13,baseline=1cm]
%        \node at (0,0)  (a) {}; %a
%         \node at (-28,30)  (h) {}; %h
          \node at (-20,5) (d) {}; %d
          \node at (-12,13)  (j) {}; %j
          \node at (-20,13) (i) {}; %i
      \node at (-12,0) (b) {}; %b
         \node at (-4,5)  (e) {}; %e
%         \node at (4,20)  (f) {}; %f
          \node at (0,0)  (k) {}; %k
%           \node at (8,30) (l) {}; %l
%           \node at (20,20)  (g) {}; %g
%            \node at (16,30) (m) {};  %m
%        \node at (12,10) (c) {}; %c
         \node at (12,0) (p) {}; %p

     \draw[kernel1,black] (i)node[rect2] {\tiny $ \Labn_B(i)$}  -- node [rect1] {\tiny$\Labe$}   (d);% --  node [near end, rect1]{\tiny $ \Labhom,\Labe $ } (i) node[rect2] {\tiny $ \Labn$};

     \draw[kernel1,red] (b) --  (e) node[round2] {\tiny $ \Labn_B,\Labo $} ;
     \draw[kernel1,red] (b)node[red,round2] {\tiny $ \Labn_B(b),\Labo(b) $} --   (d) node[round2] {\tiny $ \Labn_B+\pi\varepsilon_B^F,\Labo $} --  (j) node[round2] {\tiny $ \Labn_B,\Labo $};

\draw (p) node [rect2] {\tiny $ \Labn_B(p) $}  ;
\draw (k) node [red,round2] {\tiny $ \Labn_B(k),\Labo(k) $}  ;
\end{tikzpicture} 
\end{equ}
while
$(F,\hat F \cup_1 B,\Labn - \Labn_B,\ \Labo+\Labn_B+\pi(\varepsilon_B^F-\Labe_\emptyset^B),\Labe_B^F + \varepsilon_B^F)$
becomes
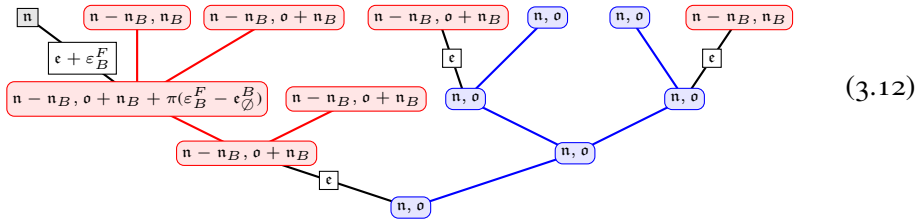
\begin{equ}\label{ex:tra2_3}   
\begin{tikzpicture}[scale=0.18,baseline=1.5cm]
\def\frst{4}\def\scnd{8}\def\thrd{14}
        \node at (0,0)  (a) {}; %a
         \node at (-28,\thrd)  (h) {}; %h
          \node at (-20,\scnd) (d) {}; %d
          \node at (-10,\thrd)  (j) {}; %j
          \node at (-20,\thrd) (i) {}; %i
      \node at (-12,\frst) (b) {}; %b
         \node at (-4,\scnd)  (e) {}; %e
         \node at (4,\scnd)  (f) {}; %f
          \node at (2,\thrd)  (k) {}; %k
           \node at (10,\thrd) (l) {}; %l
           \node at (20,\scnd)  (g) {}; %g
            \node at (16,\thrd) (m) {};  %m
        \node at (12,\frst) (c) {}; %c
         \node at (24,\thrd) (p) {}; %p

     \draw[kernel1,black] (h)node[rect2] {\tiny $ \Labn$}  -- node [rect1] {\tiny$\Labe+\varepsilon_B^F$}   (d);
     \draw[kernel1,red] (d)--  (i) node[round2] {\tiny $ \Labn-\Labn_B,\Labn_B$};

     \draw[kernel1,red] (b) --   (d) node[round2] {\tiny $ \Labn-\Labn_B,\Labo +\Labn_B+\pi(\varepsilon_B^F-\Labe_\emptyset^B)$} --  (j) node[round2] {\tiny $ \Labn-\Labn_B,\Labo+\Labn_B $};
     \draw[kernel1,red] (b) --  (e) node[round2] {\tiny $ \Labn-\Labn_B,\Labo+\Labn_B $} ;
    \draw[kernel1,black] (a) -- node [rect1] {\tiny $\Labe$}  (b) node[red,round2] {\tiny $ \Labn-\Labn_B,\Labo+\Labn_B $} ; 
    \draw[kernel1,blue] (a) --   (c) ; 
     \draw[kernel1,blue] (c) --   (f) ; 
     \draw[kernel1,blue] (c) --   (g) ; 
    \draw[kernel1,black] (f) -- node [rect1] {\tiny $ \Labe$}  (k) ; 
     \draw[kernel1,blue] (f) --  (l) ; 
     \draw[kernel1,blue] (g) --  (m) ; 
     \draw[kernel1,black] (g) -- node [rect1] {\tiny $\Labe $}  (p) ;
     
     \draw (p) node [red,round2] {\tiny $ \Labn-\Labn_B,\Labn_B$}  ;
    \draw (m) node [blue,round3] {\tiny $ \Labn, \Labo $}  ;
    \draw (l) node [blue,round3] {\tiny $ \Labn, \Labo $}  ;
     \draw (k) node [red,round2] {\tiny $ \Labn-\Labn_B,\Labo+\Labn_B$}  ;
    \draw (g) node [blue,round3] {\tiny $ \Labn,\Labo $}  ;
    \draw (f) node [blue,round3] {\tiny $ \Labn,\Labo $}  ;
    \draw (c) node [blue,round3] {\tiny $ \Labn,\Labo $}  ;
    \draw (a) node [blue,round3] {\tiny $ \Labn,\Labo $}  ;
\end{tikzpicture} 
\end{equ}
Here we have that $\partial(B,F)=\{7\}$, where we refer to the labelling of edges and nodes fixed in the Example~\ref{ex:decorations}. Therefore 
$\pi\varepsilon_B^F(d)=\varepsilon_B^F(7)$. 
Note that the edge $8$ was black in $(F,\hat F)$ and becomes red in $(F,\hat F \cup_1 B)$; 
accordingly, in $(F,\hat F \cup_1 B,\Labn - \Labn_B,\ \Labo+\Labn_B+\pi(\varepsilon_B^F-\Labe_\emptyset^B),\Labe_B^F + \varepsilon_B^F)$ the value of $\Labe$ on $8$ is set to $0$ 
and $\Labe(8)$ is subtracted from $\Labo(d)$. In accordance with Assumption~\ref{ass:1},  $B\in\Adm_1(F,\hat F)$ is disjoint from the blue
subforest $\hat F_2$ and, accordingly, all decorations on $\hat F_2$ are unchanged.
Finally, note that the edge $1$ is not in $\partial(B,F)$ since it is equal to $(a,b)$ with $b\in B$ and $a\notin B$.
\end{example}

\begin{remark}\label{rem:convention}
From now on, in expressions like \eqref{def:Deltabar} we are going to use the simplified notation
\[
(A,\hat F\restr A,\Labn_A+\pi\varepsilon_A^F, \Labo \restr N_A,\Labe \restr E_A)=:
(A,\hat F\restr A,\Labn_A+\pi\varepsilon_A^F, \Labo,\Labe),
\]
namely the restrictions of $\Labo$ and $\Labe$ will not be made explicit. This should generate no
confusion, since by Definition~\ref{def:decoration} in $(A,\hat A,\Labn', \Labo',\Labe')$ 
we have $\Labo' \colon N_A \to \Z^d\oplus\Z(\Lab)$ and $\Labe' \colon E_A \to \N^d$. 
%with $\supp \Labo' \subset \supp\hat F$. with $\supp\Labe\subset \{e\in E_F :\,  \hat F(e)=0\}$.
On the other hand, the notation $\hat F\restr A$ refers to a  slightly less standard operation, see Lemma~\ref{lem:new_coloured_forest} above, and will therefore be made explicitly throughout.
Note also that $\Labn_A$ is \emph{not} defined as the restriction of $\Labn$ to $N_A$.
\end{remark}

\begin{remark}\label{rem:intuition2}
It may not be obvious why Definition~\ref{def:maps} is natural, so let us try to offer an intuitive 
explanation of where it comes from. First note that \eqref{def:Deltabar} 
reduces to \eqref{firstDelta_i} if we drop the decorations and the combinatorial coefficients. 

If we go back to Remark~\ref{rem:intuition1}, and we recall that a decorated forest encodes a function of 
a set of variables in $\R^d$ indexed by the nodes of the underlying forest, then we can realise that the 
operator $\Delta_i$ in \eqref{def:Deltabar} is naturally motivated by Taylor expansions. 

Let us consider first the particular case of $\tau=(F,\hat F,0,\Labo,\Labe)$. Then $\Labn_A$ has to vanish 
because of the constraint $0\leq\Labn_A\leq\Labn$ and \eqref{def:Deltabar} becomes
\begin{equs}\label{def:Deltabar2}
\Delta_i \tau &= 
 \sum_{A \in \Adm_i(F,\hat F)} \sum_{\varepsilon_A^F} \frac1{\varepsilon_A^F!}
 (A,\hat F\restr A,\pi\varepsilon_A^F, \Labo ,\Labe ) 
 \\& \qquad  \otimes(F,\hat F \cup_i A,0,\ \Labo+\pi(\varepsilon_A^F-\Labe_\emptyset^A), \Labe_A^F + \varepsilon_A^F)\;.
\end{equs}
Consider a single term in this sum and fix an edge $e=(v,w)\in\partial(A,F)$. Then, in the expression
\[
(F,\hat F \cup_i A,0,\ \Labo+\pi(\varepsilon_A^F-\Labe_\emptyset^A), \Labe_A^F + \varepsilon_A^F)\;,
\]
the decoration of $e$ is changing from $\Labe(e)$ to $\Labe(e)+\varepsilon_A^F(e)$. Recalling \eqref{function}, this 
should be interpreted as differentiating $\varepsilon_A^F(e)$ times the kernel encoded by the edge $e$. 
At the same time, in the expression
\[
(A,\hat F\restr A,\pi\varepsilon_A^F, \Labo ,\Labe ) \;,
\]
the term $\pi\varepsilon_A^F(v)$ is a sum of several contributions, among which $\varepsilon_A^F(e)$. If we take into account the factor $1/\varepsilon_A^F(e)!$, we recognise a (formal) Taylor sum
\[
\sum_{k\in\N^d} \frac{(x_v)^k}{k!} \partial^{\Labe(e)+k}_{x_v}\varphi_{\Labhom(e)}(x_v-x_w), \qquad e=(v,w)\in\partial(A,F).
\]
If $\Labn$ is not zero, then we have a similar Taylor sum given by
\[
\sum_{k\in\N^d} \frac{(x_v)^k}{k!} \partial^{\Labe(e)+k}_{x_v}\left[ (x_v)^{\Labn(v)}\varphi_{\Labhom(e)}(x_v-x_w)\right], \qquad e=(v,w)\in\partial(A,F).
\]
The role of the decoration $ \Labo $ is still mysterious at this stage: we ask the reader to wait until the Remarks~\ref{rem:explanation},  \ref{explanation} and~\ref{rem:fails} below for an explanation. The connection between our construction and Taylor expansions (more precisely, Taylor {\it remainders}) will be made clear in Lemma~\ref{lem:algpropr} and Remark~\ref{rem:remainders} below.
\end{remark}     

\begin{remark}\label{rem:triang2}
Note that, in \eqref{def:Deltabar}, for each fixed $A$ the decoration $\Labn_A$ runs over a finite set because of the constraint $0\leq \Labn_A\leq\Labn$. 

On the other hand, $\varepsilon_A^F$ runs over an infinite set, but the sum is nevertheless
well defined as an element of $\scal{\Tra}\hattimes \scal{\Tra}$, even though it does \textit{not} belong
to the algebraic tensor product $\scal{\Tra}\otimes \scal{\Tra}$. Indeed, since 
$|\Labe \restr E_A| + |\Labe_A^F + \eps_A^F| = |\Labe| + |\eps_A^F| \ge |\Labe|$ and 
\begin{equ}
|A\setminus \bigl((\hat F\restr A)\cup \rho_A\bigr)| + |F \setminus \bigl((\hat F\cup_i A) \cup \rho_F\bigr)| \le |F\setminus (\hat F\cup \rho_F)|\;,
\end{equ}
it is the case that if $|\tau|_\bi = n$, then the degree of each term appearing 
on the right hand side of \eqref{def:Deltabar} is of the type $(n_1+k_1,n_2-k_2)$ with $k_i \ge 0$.
Since furthermore the sum is finite for any given value of $|\eps_A^F|$, this is 
indeed a triangular map on $\scal{\Tra}$, see Remark~\ref{rem:triang} above.

There are many other ways of bigrading $\Tra$ to make the $\Delta_i$ triangular,
but the one chosen here
has the advantage that it behaves nicely with respect to the various quotient operations
of Sections~\ref{sec:Hopf} and~\ref{sec:join} below.
\end{remark}

\begin{remark}
The coproduct $\Delta_i$ defined in \eqref{def:Deltabar} does not look like that of a combinatorial Hopf algebra
since for $\eps_A^F$ the coefficients are not necessarily integers. This could in principle be rectified easily
by a simple change of basis: if we set
\begin{equ}
(F, \hat F,\Labn,\Labo,\Labe)_\circ \eqdef \frac1{\Labe!}(F, \hat F,\Labn,\Labo,\Labe)\;,
\end{equ} 
then we can write \eqref{def:Deltabar} equivalently as
\begin{equs}
\Delta_i \tau &= 
 \sum_{A \in \Adm_i(F,\hat F)} \sum_{\varepsilon_A^F,\Labn_A} \binom{\Labe + \varepsilon_A^F}{\varepsilon_A^F}
\binom{\Labn}{\Labn_A}
 (A,\hat F\restr A,\Labn_A+\pi\varepsilon_A^F,\Labo,\Labe)_\circ
 \\& \qquad  \otimes(F,\hat F \cup_i A,\Labn - \Labn_A,\ \Labo+\Labn_A+\pi(\varepsilon_A^F-\Labe_\emptyset^A), \Labe_A^F + \varepsilon_A^F)_\circ\;, 
\end{equs}
for $\tau = (F, \hat F,\Labn,\Labo,\Labe)_\circ$. Note that with this notation it is still 
the case that
\begin{equ}
(F,\hat F,\Labn,\Labo,\Labe)_\circ
\cdot (G,\hat G,\Labn',\Labo',\Labe')_\circ
= (F\sqcup G,\hat F+ \hat G,\Labn + \Labn',\Labo + \Labo',\Labe + \Labe')_\circ\;.
\end{equ}
However, since this lengthens some expressions, does not seem to create
any significant simplifications, and completely destroys compatibility with the notations of \cite{reg}, 
we prefer to stick to \eqref{def:Deltabar}.
\end{remark}

\begin{remark}
As already remarked, the grading $|\,\cdot\,|_\bi$ defined in \eqref{e:prodgrad} is not preserved by the $\Delta_i$.
This should be considered a feature, not a bug! Indeed, the fact that the first component
of our bigrading is not preserved is precisely what allows us to have an infinite sum in \eqref{def:Deltabar}.
A more natural integer-valued grading in that respect would have been given for example by
\begin{equ}
|(F,\hat F)^{\Labn,\Labo}_\Labe|_- = |E_F| - |\hat E| + |\Labn| - |\Labe|\;,
\end{equ}
which would be preserved by
both the forest product $\cdot$ and $\Delta_i$. However, since $\Labe$ can take arbitrarily
large values, this grading is no longer positive.
A grading very similar to this will play an important role later on, see Definition~\ref{homog} below.
\end{remark}

\subsection{Coassociativity}

\begin{assumption}\label{ass:2}
For each coloured forest $(F, \hat F)$ as in Definition~\ref{def:colour},
the collection 
$\Adm_i(F,\hat F)$ of subforests of $F$ satisfies the following properties.
\begin{enumerate}
\item One has
\begin{equ}[e:prop_mult]
\Adm_i(F\sqcup G,\hat F + \hat G)
= \{C \sqcup D \,:\, C \in \Adm_i(F,\hat F)\;\&\; D \in \Adm_i(G,\hat G)\}\;.
\end{equ}
\item One has
\minilab{e:prop}
\begin{equ}[e:prop2]
A \in \Adm_i(F,\hat F)\quad\&\quad B \in \Adm_i(F,\hat F \cup_i A)\;,
\end{equ}
if and only if
\minilab{e:prop}
\begin{equ}[e:prop1]
B \in \Adm_i(F,\hat F)\quad\&\quad A \in \Adm_i(B,\hat F\restr B).
\end{equ}
\end{enumerate}
\end{assumption}

Assumption~\ref{ass:2} is precisely what is required so that 
the ``undecorated'' versions of the maps $\Delta_i$, as defined in \eqref{firstDelta_i},
are both multiplicative and coassociative. The next proposition shows that the definition
\eqref{def:Deltabar} is such that this automatically carries over to the
``decorated'' counterparts.

\begin{proposition}\label{prop:coassoc}
Under Assumptions~\ref{ass:1} and  \ref{ass:2},
the maps $\Delta_i$ are coassociative and multiplicative
on $\scal{\Forests}$, namely the identities
\minilab{mult_Deltaa}
\begin{equs}\label{mult_Deltaa1}
(\Delta_i \otimes \id)\Delta_i\CF &= (\id \otimes \Delta_i)\Delta_i\CF\;,\\
\Delta_i (\CF\cdot \CG)&=(\Delta_i \CF)\cdot(\Delta_i \CG)\;,
\label{mult_Deltaa2}
\end{equs}
hold for all $\CF,\CG\in\scal{\Forests}$.
\end{proposition}

\begin{proof}
The multiplicativity property \eqref{mult_Deltaa2} is an immediate consequence of property 1 in 
Assumption~\ref{ass:2} 
and the fact that the factorial factorises for functions with disjoint supports, so we only 
need to verify \eqref{mult_Deltaa1}.

Applying the definition \eqref{def:Deltabar} twice yields the identity
\begin{equs}
{} & (\Delta_i \otimes \id)\Delta_i (F, \hat F,\Labn,\Labo,\Labe) =
\\ &= \sum_{B \in \Adm_i(F,\hat F)}\sum_{\varepsilon_B^F,\Labn_B}
\sum_{A \in \Adm_i(B,\hat F\restr B)} \sum_{\varepsilon_{A}^B,\Labn_{A}}\frac1{\varepsilon_B^F!}
\binom{\Labn}{\Labn_B}\frac1{\varepsilon_{A}^B!}
\binom{\Labn_B+\pi\varepsilon_B^F}{\Labn_{A}} \\
&\qquad (A,\hat F\restr A,\Labn_{A}+\pi\varepsilon_{A}^B,\Labo,\Labe) \otimes \label{e:firstprod}\\
&\qquad (B, (\hat F \restr B)\cup_i A,\Labn_B+\pi\varepsilon_B^F - \Labn_{A},\ \Labo+\Labn_A+\pi(\varepsilon_A^B-\Labe_\emptyset^A),\Labe_A^B  + \varepsilon_{A}^B)
\otimes \\
&\qquad (F,\hat F \cup_i B,\Labn - \Labn_B, \ \Labo+\Labn_B+\pi(\varepsilon_B^F-\Labe_\emptyset^B),\Labe_B^F + \varepsilon_B^F)\;.
\end{equs}
Note that we should write for instance $(A,\hat F\restr A,\Labn_{A}+\pi\varepsilon_{A}^B,\Labo\restr N_A,\Labe\restr E_A)$ rather than $(A,\hat F\restr A,\Labn_{A}+\pi\varepsilon_{A}^B,\Labo,\Labe)$, but in this as in other cases we prefer the lighter notation if there is no risk of confusion. 
Analogously, one has
\begin{equs}
{}& (\id \otimes \Delta_i)\Delta_i (F, \hat F,\Labn,\Labo,\Labe) =
\\&= \sum_{A \in \Adm_i(F,\hat F)} \sum_{\varepsilon_A^F,\Labn_A}
\sum_{C \in \Adm_i(F,\hat F \cup_i A)} \sum_{\varepsilon_C^F,\Labn_C}\frac1{\varepsilon_A^F!}
\binom{\Labn}{\Labn_A}\frac1{\varepsilon_C^F!}
\binom{\Labn-\Labn_A}{\Labn_C}
\\ & (A, \hat F \restr A,\Labn_A+\pi\varepsilon_A^F,\Labo,\Labe) \otimes \label{e:second}\\
& (C,(\hat F \cup_i A)\restr C,\Labn_C+\pi\varepsilon_C^F ,  \Labo+\Labn_A+\pi(\varepsilon_A^F-\Labe_\emptyset^A),\Labe_A^C + (\varepsilon_A^F)_A^C) \otimes \\
&  (F,\hat F \cup_i C,\Labn\!-\! \Labn_A\!-\!\Labn_C,  \Labo\!+\!\Labn_A\!+\!\Labn_C+\pi((\varepsilon_A^F)_C^F+\varepsilon_C^F-\Labe_\emptyset^C),\Labe_C^F+ (\varepsilon_A^F)_C^F + \varepsilon_C^F),
\end{equs}
where we recall that, by Definition~\ref{def:maps}, for $A\subseteq B \subseteq F$ and $f:E_F\to\N^d$, we use the notation $f_A^B \eqdef f\,\un{E_B\setminus E_A}$; in particular
\begin{equ}[e:names]
(\varepsilon_A^F)_C^F \eqdef \varepsilon_A^F\,\un{E_F\setminus E_C}\;,\qquad
(\varepsilon_A^F)_A^C \eqdef \varepsilon_A^F\,\un{E_C}\;.
\end{equ}
By this definition it is clear that $(\varepsilon_A^F)_C^F$ and $(\varepsilon_A^F)_A^C$ have disjoint supports and moreover
\[
(\varepsilon_A^F)_C^F+(\varepsilon_A^F)_A^C=\varepsilon_A^F.
\]
This is the reason, in particular, why the term $\pi((\varepsilon_A^F)_C^F)$ appears in the last line of \eqref{e:second}. In the proof of \eqref{e:second} we also make use of the fact that, since $A \subset C$, 
one has
\begin{equ}
(\hat F \cup_i A)\cup_i C = \hat F \cup_i C\;.
\end{equ}

We now make the following changes of variables. First, we set
\begin{equ}[e:directMap]
\bar\varepsilon_{C}^F \eqdef(\varepsilon_A^F)_C^F + \varepsilon_C^F\;, \qquad 
\bar\varepsilon_A^C\eqdef(\varepsilon_A^F)_A^C\;,
\qquad \bar \eps_{A,C}^F\eqdef\bar \eps_C^F-\eps_C^F=(\varepsilon_A^F)_C^F
\end{equ}
with the naming conventions \eqref{e:names}. Note that the support of $\bar \eps_{A,C}^F$ is contained in
$\d(A,F)\cap\d(C,F)$. 
Now the map
\begin{equ}
(\eps_A^F, \eps_C^F) \mapsto (\bar\varepsilon_{C}^F,\bar\varepsilon_{A}^C, \bar \eps_{A,C}^F)
\end{equ}
given by \eqref{e:directMap} is invertible on its image, with inverse given by
\begin{equ}[e:inverseMap]
(\bar\varepsilon_{C}^F,\bar\varepsilon_{A}^C, \bar \eps_{A,C}^F) \mapsto 
(\eps_A^F, \eps_C^F)=(\bar \eps_A^C  + \bar \eps_{A,C}^F, \bar\eps_C^F-\bar \eps_{A,C}^F).
\end{equ}
Furthermore, the only restriction on its image besides the constraints on the supports is
the fact that $\bar \eps_{A,C}^F \le \bar \eps_C^F$, which is required to guarantee 
that, with $\eps_C^F=\bar \eps_C^F  - \bar \eps_{A,C}^F$ as in \eqref{e:inverseMap}, one has $\eps_C^F \ge 0$.

Now, the supports of $\bar \eps_A^C$ and $\bar \eps_{A,C}^F$ are disjoint, since 
\[
\supp\bar \eps_A^C\subset \partial(A,F)\cap E_C, \qquad \supp\bar \eps_{A,C}^F\subset \partial(A,F)\setminus E_C.
\]
Since the factorial factorises for functions with disjoint supports,
we can rewrite the combinatorial prefactor as
\begin{equs}
\frac1{\varepsilon_A^F!}\frac1{\varepsilon_C^F!}
&=
\frac1{\bar \eps_A^C! \bar \eps_{A,C}^F!} \frac1{(\bar \eps_C^F-\bar \eps_{A,C}^F)!}= {1\over  \bar \eps_A^C! \bar \eps_C^F!} \binom{\bar \eps_C^F}{\bar \eps_{A,C}^F}\;.\label{e:calc1}
\end{equs}
In this way, the constraint $\bar \eps_{A,C}^F \le \bar \eps_C^F$ is automatically enforced
by our convention for binomial coefficients, so that \eqref{e:second} can be written as
\begin{equs}
{}& (\id \otimes \Delta_i)\Delta_i (F, \hat F,\Labn,\Labo,\Labe) =
\\&= \sum_{A \in \Adm_i(F,\hat F)} 
\sum_{C \in \Adm_i(F,\hat F \cup_i A)} \sum_{\bar\eps_A^C,\bar\eps_C^F,\bar\eps_{A,C}^F}\sum_{\Labn_A,\Labn_C}{1\over \bar \eps_C^F! \bar \eps_A^C!} \binom{\bar \eps_C^F}{\bar \eps_{A,C}^F}
\binom{\Labn}{\Labn_A}
\binom{\Labn-\Labn_A}{\Labn_C}
\\ & (A, \hat F \restr A,\Labn_A+\pi \eps_A^F,\Labo,\Labe) \otimes \label{e:third}\\
& (C,(\hat F \cup_i A)\restr C,\Labn_C+\pi \eps_C^F,  \Labo+\Labn_A+\pi(\varepsilon_A^F-\Labe_\emptyset^A),\Labe_A^C + \bar \eps_A^C) \otimes \\
&  (F,\hat F \cup_i C,\Labn\!-\! \Labn_A\!-\!\Labn_C,  \Labo\!+\!\Labn_A\!+\!\Labn_C+\pi(\bar\eps_C^F-\Labe_\emptyset^C),\Labe_C^F + \bar\eps_C^F)\;,
\end{equs}
where $\eps_A^F$ and $\eps_C^F$ are determined by \eqref{e:inverseMap}.

We now make the further change of variables
\begin{equ}
\bar\Labn_C =\Labn_A+\Labn_C\;,\qquad
\bar \Labn_A = \Labn_A + \pi \bar \eps_{A,C}^F\;.
\end{equ}
It is clear that, given $\bar \eps_{A,C}^F$, 
this is again a bijection onto its image and that the latter 
is given by those functions with the relevant supports such that furthermore
\begin{equ}[e:supportN]
\bar \Labn_A \ge \pi \bar \eps_{A,C}^F\;.
\end{equ}
With these new variables, \eqref{e:inverseMap} immediately yields
\begin{equ}[e:goodexpr]
\Labn_A+\pi \eps_A^F = \bar \Labn_A + \pi \bar \eps_A^C\;, \qquad
\Labn_C+\pi \eps_C^F = \bar \Labn_C - \bar \Labn_A + \pi \bar \eps_C^F\;.
\end{equ}
Furthermore, we have
\begin{equ}[e:calc2]
\binom{\Labn}{\Labn_A}
\binom{\Labn-\Labn_A}{\Labn_C}
=
\binom{\Labn}{\Labn_A+\Labn_C}
\binom{\Labn_A+\Labn_C}{\Labn_A}
=
\binom{\Labn}{\bar\Labn_C}
\binom{\bar\Labn_C}{\bar \Labn_A - \pi \bar \eps_{A,C}^F}\;.
\end{equ}
Rewriting the combinatorial factor in this way, our convention on binomial coefficients 
once again enforces the condition \eqref{e:supportN}, so that \eqref{e:third} can be written as
\begin{equs}\label{e:fourth}
{}& (\id \otimes \Delta_i)\Delta_i (F, \hat F,\Labn,\Labo,\Labe) =
\\&= \sum_{A \in \Adm_i(F,\hat F)} 
\sum_{C \in \Adm_i(F,\hat F \cup_i A)} \sum_{\bar\eps_A^C,\bar\eps_C^F,\bar\eps_{A,C}^F}\sum_{\bar\Labn_A,\bar\Labn_C}
 \frac1{\bar \eps_C^F! \bar \eps_A^C!} 
\binom{\Labn}{\bar\Labn_C}
\binom{\bar \eps_C^F}{\bar \eps_{A,C}^F}
\binom{\bar\Labn_C}{\bar \Labn_A - \pi \bar \eps_{A,C}^F}
\\ & (A, \hat F \restr A,\bar \Labn_A + \pi \bar \eps_A^C,\Labo,\Labe) \otimes \\
& (C,(\hat F \cup_i A)\restr C,\bar \Labn_C - \bar \Labn_A + \pi \bar \eps_C^F,  \Labo+\bar \Labn_A + \pi \bar \eps_A^C-\pi\Labe_\emptyset^A,\Labe_A^C + \bar \eps_A^C) \otimes \\
&  (F,\hat F \cup_i C,\Labn\!-\bar \Labn_C,  \Labo\!+\!\bar\Labn_C+\pi(\bar\eps_C^F-\Labe_\emptyset^C),\Labe_C^F + \bar\eps_C^F)\;,
\end{equs}
with the summation only restricted by the conditions on the supports implicit in the notations. 
At this point, we note that the right hand side depends on $\bar \eps_{A,C}^F$ only
via the combinatorial factor and that, as a consequence of Chu-Vandermonde, one has
\begin{equs}
\sum_{\bar \eps_{A,C}^F}
&\binom{\bar \eps_C^F}{\bar \eps_{A,C}^F}
\binom{\bar\Labn_C}{\bar \Labn_A -  \pi \bar \eps_{A,C}^F}
=
%\sum_{\bar \eps_{A,C}^F}
%\binom{\bar \eps_C^F}{\bar \eps_{A,C}^F}
%\binom{\bar\Labn_C}{\bar \Labn_A - \pi \bar \eps_{A,C}^F}\\
%&=
\sum_{\pi\bar \eps_{A,C}^F}
\binom{\pi\bar \eps_C^F}{\pi\bar \eps_{A,C}^F}
\binom{\bar\Labn_C}{\bar \Labn_A - \pi \bar \eps_{A,C}^F} \\
&=
\binom{\bar\Labn_C+\pi\bar \eps_C^F}{\bar \Labn_A }.
\label{e:binomfinal}
\end{equs}
Inserting \eqref{e:binomfinal} into \eqref{e:fourth}, using the fact that
$(\hat F \restr C)\cup_i A = (\hat F \cup_i A)\restr C$ 
 and comparing to \eqref{e:firstprod} (with $B$ replaced by $C$) completes the proof.
\end{proof}

\subsection{Bialgebra structure}
\label{sec:Bialg}

Fix throughout this section $i>0$.
\begin{definition}\label{def:tra_i}
For $\Adm_i$ a family satisfying Assumptions~\ref{ass:1} and~\ref{ass:2}, we set 
\begin{equs}
\begin{split}
\C_i \ & \ \eqdef \{(F,\hat F) \in \C \,:\, \hat F \le i\;\ \&\;\ \{F,\hat F_i\}\subset\Adm_i(F,\hat F)\}\;,
\\ \Tra_i \ & \ \eqdef \{(F,\hat F, \Labn, \Labo, \Labe) \in \Tra \,:\, \hat F \le i\;\ \&\;\ \{F,\hat F_i\}\subset\Adm_i(F,\hat F)\}\;.
\end{split}
\end{equs}
We also define the set $\Units_i$ of all $(F,i, 0, \Labo, 0) \in \Tra_i$, where $(F,i)$ denotes the coloured forest $(F,\hat F)$ such that either $F$ is empty or $\hat F\equiv i$ on the whole forest $F$. In particular, one has $|\tau|_\bi = 0$ for every $\tau\in \Units_i$. Finally we define $\one_i^\star \colon \Tra \to \R$ by setting 
\begin{equ}[e:counit]
\one_i^\star(\tau)  \eqdef \un{(\tau\in\Units_i)}.
\end{equ}
\end{definition}
For instance, the following forest belongs to $\Units_1$ where $1$ corresponds to red:
\begin{equ}   \label{exunits_i}
\begin{tikzpicture}[scale=0.1,baseline=0.7cm]
\def\frst{8}\def\scnd{16}
%        \node at (0,0)  (a) {}; %a
%         \node at (-28,30)  (h) {}; %h
          \node at (-20,\frst) (d) {}; %d
          \node at (-12,\scnd)  (j) {}; %j
          \node at (-20,\scnd) (i) {}; %i
      \node at (-12,0) (b) {}; %b
         \node at (-12,\frst)  (e) {}; %e
%         \node at (4,20)  (f) {}; %f
        \node at (-4,\frst) (q) {}; %p
          \node at (0,0)  (k) {}; %k
        \node at (4,\frst) (r) {}; %p
%           \node at (8,30) (l) {}; %l
%           \node at (20,20)  (g) {}; %g
%            \node at (16,30) (m) {};  %m
%        \node at (12,10) (c) {}; %c
         \node at (12,0) (p) {}; %p

     \draw[kernel1,red] (i)node[red,round2] {\tiny $ \Labo$}  -- node [round1] {\tiny$\Labhom$}   (d);% --  node [near end, rect1]{\tiny $ \Labhom,\Labe $ } (i) node[rect2] {\tiny $ \Labn$};

     \draw[kernel1,red] (b)node[red,round2] {\tiny $\Labo $} -- node [round1] {\tiny$\Labhom$}   (d) node[round2] {\tiny $ \Labo $} --  node [round1]{\tiny $ \Labhom $ } (j) node[round2] {\tiny $\Labo $};
     \draw[kernel1,red] (b) -- node [round1] {\tiny$\Labhom$}  (e) node[round2] {\tiny $ \Labo $} ;

\draw (p) node [red,round2] {\tiny $ \Labo $}  ;
     \draw[kernel1,red] (q)node[red,round2] {\tiny $ \Labo$}  -- node [round1] {\tiny$\Labhom$}   (k)    node[red,round2] {\tiny $ \Labo$}   (k) --  node [round1]{\tiny $ \Labhom$ } (r) node[red,round2] {\tiny $ \Labo$};

\end{tikzpicture} 
\end{equ}
We also define $\one_i^\star:\C\to\R$ as $\one_i^\star(F,\hat F)=\un{(\hat F\equiv i)}$.
\begin{assumption}\label{ass:3}
For every coloured forest $(F,\hat F)$ such that $\hat F_i \in \Adm_i(F,\hat F)$ and for all $A\in\Adm_i(F,\hat F)$, 
we have 
\begin{enumerate}
\item $\{A,\hat F_i\}\subset\Adm_i(A,\hat F\restr A)$
\item if $\hat F \le i$ then $\{F, A\}\subset\Adm_i(F,\hat F\cup_i A)$.
\end{enumerate}
\end{assumption}
Under Assumptions~\ref{ass:1} and~\ref{ass:3} it immediately follows from \eqref{def:Deltabar} that, setting 
\begin{equ}
\scal{\Forests_i} = \bplus_{n \in \N^2}\Vec\{\CF \in \Forests_i\,:\, |\CF|_\bi = n\}
\end{equ} 
as in \eqref{e:defForests}, $\Delta_i$ maps $\scal{\Tra_i}$
into $\scal{\Tra_i} \hattimes \scal{\Tra_i}$.

\begin{lemma}\label{lem:counit}
Under Assumptions~\ref{ass:1}, \ref{ass:2} and~\ref{ass:3},
\begin{itemize}
\item $(\Vec(\C_i), \cdot,\Delta_i, \one, \one_i^\star)$
is a bialgebra
\item $(\scal{\Tra_i}, \cdot,\Delta_i, \one, \one_i^\star)$
is a bialgebra in the category of bigraded spaces as in Definition~\ref{categ}.
\end{itemize}
\end{lemma}
\begin{proof}
We consider only $(\scal{\Tra_i}, \cdot,\Delta_i, \one, \one_i^\star)$, since the other case follows in the same way.
By the first part of Assumption~\ref{ass:2}, 
$\Tra_i$ is closed under the forest product, so that $(\scal{\Tra_i}, \cdot, \one)$ is indeed an algebra.

Since we already argued that $\Delta_i\colon \scal{\Tra_i} \to \scal{\Tra_i} \hattimes \scal{\Tra_i}$ and since $\Delta_i$ is coassociative by \eqref{mult_Deltaa1}, in order to show that $(\scal{\Tra_i}, \Delta_i, \one_i^\star)$ is a coalgebra, 
it remains to show that
\begin{equ}
(\one_i^\star \otimes \id)\Delta_i =  (\id \otimes \one_i^\star)\Delta_i = \id,
\quad \text{on} \quad \scal{\Tra_i}\;.
\end{equ}
For $A\in\Adm_i(F,\hat F)$, we have $(A,\hat F\restr A,\Labn',\Labo',\Labe')\in\Units_i$ if and only if $\hat F\equiv i$ on $A$, i.e. $A\subseteq \hat F_i$; since $\hat F_i\subseteq A$ by Assumption~\ref{ass:1}, then the only possibility is $A=\hat F_i$. Analogously, we have $(F,\hat F \cup_i A,\Labn',\Labo', \Labe')\in\Units_i$ if and only if $A=F$. The definition \eqref{def:Deltabar} of $\Delta_i$ yields the result.

The required compatibility between the algebra and coalgebra structures is given by
\eqref{mult_Deltaa2}, thus concluding the proof.
\end{proof}

\subsection{Contraction of coloured subforests and Hopf algebra structure}
\label{sec:Hopf}

The bialgebra $(\scal{\Tra_i}, \cdot,\Delta_i, \one, \one_i^\star)$ does not admit an 
antipode. Indeed, for any $\tau=(F,i,0,\Labo,0)\in\Units_i$, see Definition~\ref{def:tra_i},
with $F$ non-empty, satisfies by \eqref{def:Deltabar2}
\begin{equation}\label{noantipode}
\Delta_i \tau = \tau\otimes \tau.
\end{equation}
In other words $\tau$ is {\it grouplike}.
If a linear map $A:\scal{\Tra_i}\to\scal{\Tra_i}$ must satisfy \eqref{anti}, then
\[
\tau \cdot A\tau = \one_i^\star(\tau) \, \one = \one
\]
by \eqref{e:counit}, which is impossible since $F$ is non-empty while $\one$ is the
empty decorated forest. A way of turning $\scal{\Tra_i}$ into a Hopf algebra
(again in the category of bigraded spaces as in Definition~\ref{categ}) is to take a 
suitable quotient in order to eliminate elements which do not admit an antipode, and this is what we are going to show now.

To formalise this, we introduce a {\it contraction operator} on coloured forests. Given a coloured forest $(F,\hat F)$, we recall that $\hat E$, defined in Definition~\ref{def:colour}, is the union of all edges in $\hat F_j$ over all $j>0$. 
\begin{definition}\label{def:contrac}
For any coloured forest $(F,\hat F)$, we write
$\CK_{\hat F} F$ for the typed forest obtained in the following way. We use the equivalence relation $\sim$ on the node set $N_F$ defined in Definition~\ref{def:contrac0}, namely $x \sim y$ if $x$ and $y$ are connected in $\hat E$. Then
$\CK_{\hat F} F$ is the quotient graph of $(N_F, E_F \setminus \hat E)$ by $\sim$.
By the definition of $\sim$, each equivalence class is connected so that $\CK_{\hat F} F$ is again a typed forest. 
Finally, $\hat F$ is constant on equivalence classes with respect to $\sim$, so that the coloured forest $(\CK_{\hat F} F,\hat F)$ is well defined and we denote it by
\[
\CK(F,\hat F)\eqdef (\CK_{\hat F} F,\hat F).
\]
If $G:=\CK_{\hat F} F$, then there is a canonical projection $\pi:N_F\to N_G$. This allows to define a canonical map $\CK_{\hat F}^\sharp$ from subforests of $\CK_{\hat F} F$ to subforests of $F$ as follows: if $A=(N_A,E_A)$ is 
a subforest of $\CK_{\hat F} F$, then $\CK_{\hat F}^\sharp A:=(N_B,E_B)$ where
$N_B$ is $\pi^{-1}(N_A)$ and $E_B$ is the set of all $(x,y)\in E_F$ such that either $\pi(x)=\pi(y)\in N_A$ or $(\pi(x),\pi(y))\in E_A$.
\end{definition}
Note that in $(\CK_{\hat F} F,\hat F)$ all non-empty coloured subforests are reduced to single nodes.

We are going to restrict our attention to collections $\Adm_i$ satisfying the following assumption.

\begin{assumption}\label{ass:4}
For all coloured forests $(F,\hat F)$, the map $\CK_{\hat F}^\sharp$ is a 
bijection between $\Adm_i(\CK_{\hat F} F,\hat F)$ and $\Adm_i(F,\hat F)$.
\end{assumption}
We recall that we have defined in \eqref{firstDelta_i} the operator acting on linear 
combinations of coloured forests $(F,\hat F)\mapsto \Delta_i(F,\hat F)$. Then we have
\begin{lemma}\label{lem:ass4}
If $\Adm_i$ satisfies Assumption~\ref{ass:4}, then
\[
(\CK\otimes\CK)\Delta_i=(\CK\otimes\CK)\Delta_i\CK \qquad \text{on }\quad \Vec(\C).
\]
\end{lemma}
\begin{proof}
It is enough to check that for all $A\in\Adm_i(\CK_{\hat F}F,\hat F)$, setting $A'=\CK_{\hat F}^\sharp A$,
\[
\CK(A',\hat F\restr A') =\CK(A,\hat F\restr A), 
\]
\[
\CK(F,\hat F \cup_{i} A')=
\CK(\CK_{\hat F} F,\hat F\cup_i A),
\]
which follow from the definitions.
\end{proof}

\begin{example}\label{ex:colours2}
For the tree of Example~\ref{ex:colours}, we have
\begin{equ}
 (F,\hat F) =   
\begin{tikzpicture}[scale=0.2,baseline=0.2cm]
        \node at (0,0)  [dot,blue,label={[label distance=-0.2em]above: \scriptsize  $ a $}] (root) {};
         \node at (-7,6)  [dot,label={[label distance=-0.2em]above: \scriptsize  $ h $} ] (leftll) {};
          \node at (-5,4)  [dot,red,label=left:] (leftl) {};
          \node at (-3,6)  [dot,red,label={[label distance=-0.2em]above: \scriptsize  $ j $}] (leftlr) {};
          \node at (-5,6)  [dot,label={[label distance=-0.2em]above: \scriptsize  $ i $}] (leftlc) {};
      \node at (-3,2)  [dot,color=red,label={[label distance=-0.2em]above: \scriptsize  $ b $}] (left) {};
         \node at (-1,4)  [dot,red,label={[label distance=-0.2em]above: \scriptsize  $ e $}] (leftr) {};
         \node at (1,4)  [dot,blue,label=above:   ] (rightl) {};
          \node at (0,6)  [dot,red,label={[label distance=-0.2em]above: \scriptsize  $ k $} ] (rightll) {};
           \node at (2,6)  [dot,blue,label={[label distance=-0.2em]above: \scriptsize  $ l $}] (rightlr) {};
           \node at (5,4)  [dot,blue,label=left:] (rightr) {};
            \node at (4,6)  [dot,blue,label={[label distance=-0.2em]above: \scriptsize  $ m $}] (rightrl) {};
        \node at (3,2) [dot,blue,label={[label distance=-0.4em]below right: \scriptsize }] (right) {};
         \node at (6,6)  [dot,label={[label distance=-0.2em]above: \scriptsize  $ p $} ] (rightrr) {};
        
        \draw[kernel1] (left) to node [sloped,below] {\small } (root); ;
        \draw[kernel1,red] (leftl) to
     node [sloped,below] {\small }     (left);
     \draw[kernel1,red] (leftlr) to
     node [sloped,below] {\small }     (leftl); 
     \draw[kernel1] (leftll) to
     node [sloped,below] {\small }     (leftl);
     \draw[kernel1,red] (leftr) to
     node [sloped,below] {\small }     (left);  
        \draw[kernel1,blue] (right) to
     node [sloped,below] {\small }     (root);
      \draw[kernel1] (leftlc) to
     node [sloped,below] {\small }     (leftl); 
     \draw[kernel1,blue] (rightr) to
     node [sloped,below] {\small }     (right);
     \draw[kernel1] (rightrr) to
     node [sloped,below] {\small }     (rightr);
     \draw[kernel1,blue] (rightrl) to
     node [sloped,below] {\small }     (rightr);
     \draw[kernel1,blue] (rightl) to
     node [sloped,below] {\small }     (right);
     \draw[kernel1,blue] (rightlr) to
     node [sloped,below] {\small }     (rightl);
     \draw[kernel1] (rightll) to
     node [sloped,below] {\small }     (rightl);
     \end{tikzpicture} 
     , \qquad (\CK_{\hat F} F,\hat F)=
   \begin{tikzpicture}[scale=0.2,baseline=0.2cm]
        \node at (0,0)  [dot,blue,label={[label distance=-0.2em]left: \scriptsize  $ a $} ] (root) {};
          \node at (-2,2)  [dot,red,,label={[label distance=-0.2em]left: \scriptsize  $ b $} ] (leftl) {};
         \node at (-3,4)  [dot,label={[label distance=-0.2em]above: \scriptsize  $ h $} ] (leftll) {};
          \node at (-1,4)  [dot,label={[label distance=-0.2em]above: \scriptsize  $ i $}] (leftlc) {};
           \node at (1,2)  [dot,red,label={[label distance=-0.2em]above: \scriptsize  $ k $} ] (rightll) {};
         \node at (3,2)  [dot,label={[label distance=-0.2em]above: \scriptsize  $ p $} ] (rightrr) {};
        
       \draw[kernel1] (leftl) to node [sloped,below] {\small } (root); ;
        \draw[kernel1] (leftll) to node [sloped,below] {\small } (leftl); ;
        \draw[kernel1] (leftlc) to
     node [sloped,below] {\small }     (leftl);
     \draw[kernel1] (rightll) to
     node [sloped,below] {\small }     (root); 
     \draw[kernel1] (rightrr) to
     node [sloped,below] {\small }     (root);
     \end{tikzpicture}   
\end{equ}
Moreover for the choice $ A' \in \Adm_{2}(\CK_{\hat F} F,\hat F) $ given by
\[
A' = \begin{tikzpicture}[scale=0.2,baseline=0.1cm]
        \node at (-1,0)  [dot,label={[label distance=-0.2em]left: \scriptsize  $ a $}         ] (root) {};
          \node at (-1,2)  [dot,label={[label distance=-0.2em]left: \scriptsize  $ b $} ] (left) {};
        
        \draw[kernel1] (left) to node [sloped,below] {\small } (root); 
     \end{tikzpicture} \quad \Longrightarrow \quad
(A',\hat F\restr A') = \begin{tikzpicture}[scale=0.2,baseline=0.1cm]
        \node at (-1,0)  [dot,blue,label={[label distance=-0.2em]left: \scriptsize  $ a $}  ] (root) {};
          \node at (-1,2)  [dot,red,label={[label distance=-0.2em]left: \scriptsize  $ b $} ] (left) {};
        
        \draw[kernel1] (left) to node [sloped,below] {\small } (root); 
     \end{tikzpicture} , \qquad (\CK_{\hat F} F,\hat F\cup_2 A')=
   \begin{tikzpicture}[scale=0.2,baseline=0.2cm]
        \node at (0,0)  [dot,blue,label={[label distance=-0.2em]left: \scriptsize  $ a $} ] (root) {};
          \node at (-2,2)  [dot,blue,label={[label distance=-0.2em]left: \scriptsize  $ b $} ] (leftl) {};
         \node at (-3,4)  [dot,label={[label distance=-0.2em]above: \scriptsize  $ h $} ] (leftll) {};
          \node at (-1,4)  [dot,label={[label distance=-0.2em]above: \scriptsize  $ i $}] (leftlc) {};
           \node at (1,2)  [dot,red,label={[label distance=-0.2em]above: \scriptsize  $ k $} ] (rightll) {};
         \node at (3,2)  [dot,label={[label distance=-0.2em]above: \scriptsize  $ p $} ] (rightrr) {};
        
       \draw[kernel1,blue] (leftl) to node [sloped,below] {\small } (root); ;
        \draw[kernel1] (leftll) to node [sloped,below] {\small } (leftl); ;
        \draw[kernel1] (leftlc) to
     node [sloped,below] {\small }     (leftl);
     \draw[kernel1] (rightll) to
     node [sloped,below] {\small }     (root); 
     \draw[kernel1] (rightrr) to
     node [sloped,below] {\small }     (root);
     \end{tikzpicture}   
     \]
     we obtain that $A=\CK_{\hat F}^\sharp A'$ is such that
\begin{equs}     
  (A,\hat F\restr A) \otimes  (F,\hat F \cup_{2} A)  =  
\begin{tikzpicture}[scale=0.2,baseline=0.2cm]
        \node at (0,0)  [dot,blue        ] (root) {};
          \node at (-5,4)  [dot,red,label=left:] (leftl) {};
          \node at (-3,6)  [dot,red,label={[label distance=-0.2em]above: \scriptsize  $ j $}] (leftlr) {};
      \node at (-3,2)  [dot,color=red] (left) {};
         \node at (-1,4)  [dot,red,label={[label distance=-0.2em]above: \scriptsize  $ e $}] (leftr) {};
         \node at (1,4)  [dot,blue,label=above:   ] (rightl) {};
           \node at (2,6)  [dot,blue,label={[label distance=-0.2em]above: \scriptsize  $ l $}] (rightlr) {};
           \node at (5,4)  [dot,blue,label=left:] (rightr) {};
            \node at (4,6)  [dot,blue,label={[label distance=-0.2em]above: \scriptsize  $ m $}] (rightrl) {};
        \node at (3,2) [dot,blue,label={[label distance=-0.4em]below right: \scriptsize  $  $}] (right) {};
        
        \draw[kernel1] (left) to node [sloped,below] {\small } (root); 
        \draw[kernel1,red] (leftl) to
     node [sloped,below] {\small }     (left);
     \draw[kernel1,red] (leftlr) to
     node [sloped,below] {\small }     (leftl); 
     \draw[kernel1,red] (leftr) to
     node [sloped,below] {\small }     (left);  
        \draw[kernel1,blue] (right) to
     node [sloped,below] {\small }     (root); 
     \draw[kernel1,blue] (rightr) to
     node [sloped,below] {\small }     (right);
     \draw[kernel1,blue] (rightrl) to
     node [sloped,below] {\small }     (rightr);
     \draw[kernel1,blue] (rightl) to
     node [sloped,below] {\small }     (right);
     \draw[kernel1,blue] (rightlr) to
     node [sloped,below] {\small }     (rightl);
     \end{tikzpicture} 
       \otimes  
\begin{tikzpicture}[scale=0.2,baseline=0.2cm]
        \node at (0,0)  [dot,blue ] (root) {};
         \node at (-7,6)  [dot,label={[label distance=-0.2em]above: \scriptsize  $ h $} ] (leftll) {};
          \node at (-5,4)  [dot,blue,label=left:] (leftl) {};
          \node at (-3,6)  [dot,blue,label={[label distance=-0.2em]above: \scriptsize  $ j $}] (leftlr) {};
          \node at (-5,6)  [dot,label={[label distance=-0.2em]above: \scriptsize  $ i $}] (leftlc) {};
      \node at (-3,2)  [dot,color=blue,label={[label distance=-0.2em]below: \scriptsize  }] (left) {};
         \node at (-1,4)  [dot,blue,label={[label distance=-0.2em]above: \scriptsize  $ e $}] (leftr) {};
         \node at (1,4)  [dot,blue,label=above:   ] (rightl) {};
          \node at (0,6)  [dot,red,label={[label distance=-0.2em]above: \scriptsize  $ k $} ] (rightll) {};
           \node at (2,6)  [dot,blue,label={[label distance=-0.2em]above: \scriptsize  $ l $}] (rightlr) {};
           \node at (5,4)  [dot,blue,label=left:] (rightr) {};
            \node at (4,6)  [dot,blue,label={[label distance=-0.2em]above: \scriptsize  $ m $}] (rightrl) {};
        \node at (3,2) [dot,blue,label={[label distance=-0.4em]below right: \scriptsize  }] (right) {};
         \node at (6,6)  [dot,label={[label distance=-0.2em]above: \scriptsize  $ p $} ] (rightrr) {};
        
        \draw[kernel1,blue] (left) to node [sloped,below] {\small } (root); 
        \draw[kernel1,blue] (leftl) to
     node [sloped,below] {\small }     (left);
     \draw[kernel1,blue] (leftlr) to
     node [sloped,below] {\small }     (leftl); 
     \draw[kernel1] (leftll) to
     node [sloped,below] {\small }     (leftl);
     \draw[kernel1,blue] (leftr) to
     node [sloped,below] {\small }     (left);  
        \draw[kernel1,blue] (right) to
     node [sloped,below] {\small }     (root);
      \draw[kernel1] (leftlc) to
     node [sloped,below] {\small }     (leftl); 
     \draw[kernel1,blue] (rightr) to
     node [sloped,below] {\small }     (right);
     \draw[kernel1] (rightrr) to
     node [sloped,below] {\small }     (rightr);
     \draw[kernel1,blue] (rightrl) to
     node [sloped,below] {\small }     (rightr);
     \draw[kernel1,blue] (rightl) to
     node [sloped,below] {\small }     (right);
     \draw[kernel1,blue] (rightlr) to
     node [sloped,below] {\small }     (rightl);
     \draw[kernel1] (rightll) to
     node [sloped,below] {\small }     (rightl);
     \end{tikzpicture}
  \end{equs}
  Then in accordance with Lemma~\ref{lem:ass4} we have 
%  we have $\CK_{\hat F\restr A'}A'\in\A'dm_2(G,\hat G)$ and 
\[
  \CK(A,\hat F\restr A) \otimes  \CK(F,\hat F \cup_{2} A)  =  
  \CK(A',\hat F\restr A') \otimes  \CK(\CK_{\hat F}F,\hat F \cup_{2} A') \;,
  \]
  and both are equal to $
\begin{tikzpicture}[scale=0.2,baseline=0.1cm]
        \node at (-1,0)  [dot,blue,label={[label distance=-0.2em]left: \scriptsize  $ a $} ] (root) {};
          \node at (-1,2)  [dot,red,label={[label distance=-0.2em]left: \scriptsize  $ b $}] (left) {};
        
        \draw[kernel1] (left) to node [sloped,below] {\small } (root); 
     \end{tikzpicture} \
       \otimes  
   \begin{tikzpicture}[scale=0.2,baseline=0.1cm]
        \node at (0,0)  [dot,blue] (root) {};
          \node at (-3,2)  [dot,label={[label distance=-0.2em]above: \scriptsize  $ h $} ] (leftll) {};
          \node at (-1,2)  [dot,label={[label distance=-0.2em]above: \scriptsize  $ i $}] (leftlc) {};
           \node at (1,2)  [dot,red,label={[label distance=-0.2em]above: \scriptsize  $ k $} ] (rightll) {};
         \node at (3,2)  [dot,label={[label distance=-0.2em]above: \scriptsize  $ p $} ] (rightrr) {};
        
        \draw[kernel1] (leftll) to node [sloped,below] {\small } (root); ;
        \draw[kernel1] (leftlc) to
     node [sloped,below] {\small }     (root);
     \draw[kernel1] (rightll) to
     node [sloped,below] {\small }     (root); 
     \draw[kernel1] (rightrr) to
     node [sloped,below] {\small }     (root);
     \end{tikzpicture}$.
For the choice of $ B' \in \Adm_{1}(\CK_{\hat F} F,\hat F) $ given by
$
B'=\begin{tikzpicture}[scale=0.2,baseline=0.1cm]
           \node at (-2,2)  [dot,label={[label distance=-0.2em]left: \scriptsize  $ i $}] (leftlc) {};
      \node at (-2,0)  [dot,label={[label distance=-0.2em]left: \scriptsize  $ b $} ] (left) {};
         \node at (1,.7)  [dot,label={[label distance=-0.2em]above: \scriptsize  $ k $}] (o) {};
         \node at (4,.7)  [dot,label={[label distance=-0.2em]above: \scriptsize  $ p $}] (o) {};
        
        \draw[kernel1] (leftlc) to
     node [sloped,below] {\small }     (left);
     \end{tikzpicture}$ so that
 \[
(B',\hat F\restr B') =
     \begin{tikzpicture}[scale=0.2,baseline=0.1cm]
           \node at (-2,2)  [dot,label={[label distance=-0.2em]left: \scriptsize  $ i $}] (leftlc) {};
      \node at (-2,0)  [dot,color=red,label={[label distance=-0.2em]left: \scriptsize  $ b $} ] (left) {};
         \node at (1,.7)  [dot,red,label={[label distance=-0.2em]above: \scriptsize  $ k $}] (o) {};
         \node at (4,.7)  [dot,label={[label distance=-0.2em]above: \scriptsize  $ p $}] (o) {};
        
        \draw[kernel1] (leftlc) to
     node [sloped,below] {\small }     (left);
     \end{tikzpicture}\;,\qquad
      (\CK_{\hat F} F,\hat F\cup_1 B')=
   \begin{tikzpicture}[scale=0.2,baseline=0.2cm]
        \node at (0,0)  [dot,blue,label={[label distance=-0.2em]left: \scriptsize  $ a $} ] (root) {};
          \node at (-2,2)  [dot,red,label={[label distance=-0.2em]left: \scriptsize  $ b $} ] (leftl) {};
         \node at (-3,4)  [dot,label={[label distance=-0.2em]above: \scriptsize  $ h $} ] (leftll) {};
          \node at (-1,4)  [dot,red,label={[label distance=-0.2em]above: \scriptsize  $ i $}] (leftlc) {};
           \node at (1,2)  [dot,red,label={[label distance=-0.2em]above: \scriptsize  $ k $} ] (rightll) {};
         \node at (3,2)  [dot,red,label={[label distance=-0.2em]above: \scriptsize  $ p $} ] (rightrr) {};
        
       \draw[kernel1] (leftl) to node [sloped,below] {\small } (root); ;
        \draw[kernel1] (leftll) to node [sloped,below] {\small } (leftl); ;
        \draw[kernel1,red] (leftlc) to
     node [sloped,below] {\small }     (leftl);
     \draw[kernel1] (rightll) to
     node [sloped,below] {\small }     (root); 
     \draw[kernel1] (rightrr) to
     node [sloped,below] {\small }     (root);
     \end{tikzpicture}   \;,
\]
     we obtain that $B=\CK_{\hat F}^\sharp B'$ is such that
 \begin{equs}     
  (B,\hat F\restr B) \otimes  (F,\hat F \cup_{1} B)  =  
\begin{tikzpicture}[scale=0.2,baseline=0.1cm]
          \node at (-5,2)  [dot,red,label={[label distance=-0.2em]left: \scriptsize  $ b $}] (leftl) {};
          \node at (-3,4)  [dot,red,label={[label distance=-0.2em]above: \scriptsize  $ j $}] (leftlr) {};
          \node at (-5,4)  [dot,label={[label distance=-0.2em]above: \scriptsize  $ i $}] (leftlc) {};
      \node at (-3,0)  [dot,color=red] (left) {};
         \node at (-1,2)  [dot,red,label={[label distance=-0.2em]above: \scriptsize  $ e $}] (leftr) {};
         \node at (1,0)  [dot,red,label={[label distance=-0.2em]above: \scriptsize  $ k $}] (o) {};
         \node at (4,0)  [dot,label={[label distance=-0.2em]above: \scriptsize  $ p $}] (o) {};
        
        \draw[kernel1,red] (leftl) to
     node [sloped,below] {\small }     (left);
     \draw[kernel1,red] (leftlr) to
     node [sloped,below] {\small }     (leftl); 
     \draw[kernel1,red] (leftr) to
     node [sloped,below] {\small }     (left);  
      \draw[kernel1] (leftlc) to
     node [sloped,below] {\small }     (leftl); 
      \end{tikzpicture} 
      \otimes   
\begin{tikzpicture}[scale=0.2,baseline=0.2cm]
        \node at (0,0)  [dot,blue] (root) {};
         \node at (-7,6)  [dot,label={[label distance=-0.2em]above: \scriptsize  $ h $} ] (leftll) {};
          \node at (-5,4)  [dot,red,label=left:] (leftl) {};
          \node at (-3,6)  [dot,red,label={[label distance=-0.2em]above: \scriptsize  $ j $}] (leftlr) {};
          \node at (-5,6)  [dot,red,label={[label distance=-0.2em]above: \scriptsize  $ i $}] (leftlc) {};
      \node at (-3,2)  [dot,color=red] (left) {};
         \node at (-1,4)  [dot,red,label={[label distance=-0.2em]above: \scriptsize  $ e $}] (leftr) {};
         \node at (1,4)  [dot,blue,label=above:   ] (rightl) {};
          \node at (0,6)  [dot,red,label={[label distance=-0.2em]above: \scriptsize  $ k $} ] (rightll) {};
           \node at (2,6)  [dot,blue,label={[label distance=-0.2em]above: \scriptsize  $ l $}] (rightlr) {};
           \node at (5,4)  [dot,blue,label=left:] (rightr) {};
            \node at (4,6)  [dot,blue,label={[label distance=-0.2em]above: \scriptsize  $ m $}] (rightrl) {};
        \node at (3,2) [dot,blue,label={[label distance=-0.4em]below right: \scriptsize }] (right) {};
         \node at (6,6)  [dot,red,label={[label distance=-0.2em]above: \scriptsize  $ p $} ] (rightrr) {};
        
        \draw[kernel1] (left) to node [sloped,below] {\small } (root); ;
        \draw[kernel1,red] (leftl) to
     node [sloped,below] {\small }     (left);
     \draw[kernel1,red] (leftlr) to
     node [sloped,below] {\small }     (leftl); 
     \draw[kernel1] (leftll) to
     node [sloped,below] {\small }     (leftl);
     \draw[kernel1,red] (leftr) to
     node [sloped,below] {\small }     (left);  
        \draw[kernel1,blue] (right) to
     node [sloped,below] {\small }     (root);
      \draw[kernel1,red] (leftlc) to
     node [sloped,below] {\small }     (leftl); 
     \draw[kernel1,blue] (rightr) to
     node [sloped,below] {\small }     (right);
     \draw[kernel1] (rightrr) to
     node [sloped,below] {\small }     (rightr);
     \draw[kernel1,blue] (rightrl) to
     node [sloped,below] {\small }     (rightr);
     \draw[kernel1,blue] (rightl) to
     node [sloped,below] {\small }     (right);
     \draw[kernel1,blue] (rightlr) to
     node [sloped,below] {\small }     (rightl);
     \draw[kernel1] (rightll) to
     node [sloped,below] {\small }     (rightl);
%\draw [brown,use as bounding box] (current bounding box.south west) rectangle ([shift={(0mm,-20mm)}]current bounding box.north east);
     \end{tikzpicture}
  \end{equs}
  Then in accordance with Lemma~\ref{lem:ass4} we have 
%  we have $\CK_{\hat F\restr A}A\in\Adm_2(G,\hat G)$ and 
\[
  \CK(B,\hat F\restr B) \otimes  \CK(F,\hat F \cup_{1} B)  =  
  \CK(B',\hat F\restr B') \otimes  \CK(\CK_{\hat F}F,\hat F \cup_{1} B')  
  \]
  and both are equal to
  \begin{equs}     
%  (\CK_{\hat F\restr B}B,\hat F\restr B) \otimes  (\CK_{\hat F\cup_{1} B}F,\hat F \cup_{1} B)  =  
\begin{tikzpicture}[scale=0.2,baseline=0.2cm]
           \node at (-2,2)  [dot,label={[label distance=-0.2em]above: \scriptsize  $ i $}] (leftlc) {};
      \node at (-2,0)  [dot,color=red] (left) {};
         \node at (1,0)  [dot,red,label={[label distance=-0.2em]above: \scriptsize  $ k $}] (o) {};
         \node at (4,0)  [dot,label={[label distance=-0.2em]above: \scriptsize  $ p $}] (o) {};
        
        \draw[kernel1] (leftlc) to
     node [sloped,below] {\small }     (left);
     \end{tikzpicture} 
      \otimes   
\begin{tikzpicture}[scale=0.2,baseline=0.2cm]
        \node at (0,0)  [dot,blue] (root) {};
          \node at (-2,2)  [dot,red,label=left:] (leftl) {};
         \node at (-2,4)  [dot,label={[label distance=-0.2em]above: \scriptsize  $ h $} ] (leftll) {};
             \node at (0,2)  [dot,red,label={[label distance=-0.2em]above: \scriptsize  $ k $} ] (rightll) {};
         \node at (2,2)  [dot,red,label={[label distance=-0.2em]above: \scriptsize  $ p $} ] (rightrr) {};
        
       \draw[kernel1] (leftl) to node [sloped,below] {\small } (root); ;
        \draw[kernel1] (leftll) to node [sloped,below] {\small } (leftl); ;
      \draw[kernel1] (rightll) to
     node [sloped,below] {\small }     (root); 
     \draw[kernel1] (rightrr) to
     node [sloped,below] {\small }     (root);
     \end{tikzpicture}
  \end{equs}
    \end{example}
Contraction of couloured subforests leads us closer to a Hopf algebra, but there is still a missing element. Indeed, an 
element like 
$(F,\hat F)=(\bullet\sqcup\bullet,1)$, namely two red isolated roots with no edge, is grouplike since it satisfies 
$\Delta_1(F,\hat F)=(F,\hat F)\otimes(F,\hat F)$
and therefore it can not admit an antipode, see the discussion after \eqref{noantipode} above. 

We recall that $\C_i$ has been introduced in Definition~\ref{def:tra_i}.
We define first the factorisation of $\C_i\ni\tau=\mu\cdot\nu$ where the forest product $\cdot$ has been defined 
in \eqref{cdotC} and
\begin{itemize}
\item  $\nu\in\C_i$ is the disjoint union of all non-empty 
connected componens of $\tau$ of the form $(A,i)$
\item $\mu\in\C_i$ is the unique element such that $\tau=\mu\cdot\nu$.
\end{itemize}
For instance
\[
\tau=
\begin{tikzpicture}[scale=0.2,baseline=0.2cm]
          \node at (-5,4)  [dot,red] (leftll) {};
          \node at (-5,2)  [dot,red] (leftlr) {};
      \node at (-4,0)  [dot] (left) {};
         \node at (-3,2)  [dot,red] (leftr) {};
        
        \draw[kernel1,red] (leftll) to     node [sloped,below] {\small }     (leftlr);
        \draw[kernel1] (leftlr) to     node [sloped,below] {\small }     (left);
     \draw[kernel1] (leftr) to     node [sloped,below] {\small }     (left);  
     \end{tikzpicture} 
\begin{tikzpicture}[scale=0.2,baseline=0.2cm]
          \node at (-5,2)  [dot,red] (leftl) {};
          \node at (-4,4)  [dot,red] (leftlr) {};
          \node at (-6,4)  [dot,red] (leftlc) {};
      \node at (-4,0)  [dot,color=red] (left) {};
         \node at (-3,2)  [dot,red] (leftr) {};
        
        \draw[kernel1,red] (leftl) to
     node [sloped,below] {\small }     (left);
     \draw[kernel1,red] (leftlr) to
     node [sloped,below] {\small }     (leftl); 
     \draw[kernel1,red] (leftr) to
     node [sloped,below] {\small }     (left);  
      \draw[kernel1,red] (leftlc) to
     node [sloped,below] {\small }     (leftl); 
     \end{tikzpicture} 
\begin{tikzpicture}[scale=0.2,baseline=0.2cm]
           \node at (-2,2)  [dot,label] (leftlc) {};
      \node at (-2,0)  [dot,color=red] (left) {};
         \node at (1,0)  [dot,red] (o) {};
         \node at (4,0)  [dot] (o) {};
        
        \draw[kernel1] (leftlc) to
     node [sloped,below] {\small }     (left);
     \end{tikzpicture} 
 \quad    \Longrightarrow \quad 
 \nu = \begin{tikzpicture}[scale=0.2,baseline=0.2cm]
          \node at (-5,2)  [dot,red] (leftl) {};
          \node at (-4,4)  [dot,red] (leftlr) {};
          \node at (-6,4)  [dot,red] (leftlc) {};
      \node at (-4,0)  [dot,color=red] (left) {};
         \node at (-3,2)  [dot,red] (leftr) {};
        
        \draw[kernel1,red] (leftl) to
     node [sloped,below] {\small }     (left);
     \draw[kernel1,red] (leftlr) to
     node [sloped,below] {\small }     (leftl); 
     \draw[kernel1,red] (leftr) to
     node [sloped,below] {\small }     (left);  
      \draw[kernel1,red] (leftlc) to
     node [sloped,below] {\small }     (leftl); 
     \end{tikzpicture}  \ 
\begin{tikzpicture}[scale=0.2,baseline=0.2cm]
         \node at (0,0)  [dot,red] (o) {};
     node [sloped,below] {\small }     (left);
     \end{tikzpicture} 
\quad 
 \mu=
\begin{tikzpicture}[scale=0.2,baseline=0.2cm]
          \node at (-5,4)  [dot,red] (leftll) {};
          \node at (-5,2)  [dot,red] (leftlr) {};
      \node at (-4,0)  [dot] (left) {};
         \node at (-3,2)  [dot,red] (leftr) {};
        
        \draw[kernel1,red] (leftll) to     node [sloped,below] {\small }     (leftlr);
        \draw[kernel1] (leftlr) to     node [sloped,below] {\small }     (left);
     \draw[kernel1] (leftr) to     node [sloped,below] {\small }     (left);  
     \end{tikzpicture} 
\begin{tikzpicture}[scale=0.2,baseline=0.2cm]
           \node at (-2,2)  [dot,label] (leftlc) {};
      \node at (-2,0)  [dot,color=red] (left) {};
         \node at (1,0)  [dot] (o) {};
        
        \draw[kernel1] (leftlc) to
     node [sloped,below] {\small }     (left);
     \end{tikzpicture} 
\]
Note that by the first part of Assumption~\ref{ass:2}, we know that if $\tau = \mu\cdot\nu \in \C_i$,
then $\mu \in \C_i$ and $\nu\in \C_i$. Then, we know by Assumption~\ref{ass:4} that if $\mu \in \C_i$, then
$\CK(\mu)\in \C_i$. Then, using this factorisation, we define $\CK_i:\Vec(\C_i)\to\Vec(\C_i)$ as the linear operator such  
that 
\begin{equ}\label{e:cki}
\CK_i(\tau)=\CK(\mu).
\end{equ}
For example
\[
\CK_1 \left(\begin{tikzpicture}[scale=0.2,baseline=0.2cm]
          \node at (-5,4)  [dot,red] (leftll) {};
          \node at (-5,2)  [dot,red] (leftlr) {};
      \node at (-4,0)  [dot] (left) {};
         \node at (-3,2)  [dot,red] (leftr) {};
        
        \draw[kernel1,red] (leftll) to     node [sloped,below] {\small }     (leftlr);
        \draw[kernel1] (leftlr) to     node [sloped,below] {\small }     (left);
     \draw[kernel1] (leftr) to     node [sloped,below] {\small }     (left);  
     \end{tikzpicture} 
\begin{tikzpicture}[scale=0.2,baseline=0.2cm]
          \node at (-5,2)  [dot,red] (leftl) {};
          \node at (-4,4)  [dot,red] (leftlr) {};
          \node at (-6,4)  [dot,red] (leftlc) {};
      \node at (-4,0)  [dot,color=red] (left) {};
         \node at (-3,2)  [dot,red] (leftr) {};
        
        \draw[kernel1,red] (leftl) to
     node [sloped,below] {\small }     (left);
     \draw[kernel1,red] (leftlr) to
     node [sloped,below] {\small }     (leftl); 
     \draw[kernel1,red] (leftr) to
     node [sloped,below] {\small }     (left);  
      \draw[kernel1,red] (leftlc) to
     node [sloped,below] {\small }     (leftl); 
     \end{tikzpicture} 
\begin{tikzpicture}[scale=0.2,baseline=0.2cm]
           \node at (-2,2)  [dot,label] (leftlc) {};
      \node at (-2,0)  [dot,color=red] (left) {};
         \node at (1,0)  [dot,red] (o) {};
         \node at (4,0)  [dot] (o) {};
        
        \draw[kernel1] (leftlc) to
     node [sloped,below] {\small }     (left);
     \end{tikzpicture} 
 \right) \
= \ \begin{tikzpicture}[scale=0.2,baseline=0.2cm]
          \node at (-5,2)  [dot,red] (leftlr) {};
      \node at (-4,0)  [dot] (left) {};
         \node at (-3,2)  [dot,red] (leftr) {};
        
        \draw[kernel1] (leftlr) to     node [sloped,below] {\small }     (left);
     \draw[kernel1] (leftr) to     node [sloped,below] {\small }     (left);  
     \end{tikzpicture} 
\begin{tikzpicture}[scale=0.2,baseline=0.2cm]
           \node at (-2,2)  [dot,label] (leftlc) {};
      \node at (-2,0)  [dot,color=red] (left) {};
         \node at (1,0)  [dot] (o) {};
        
        \draw[kernel1] (leftlc) to
     node [sloped,below] {\small }     (left);
     \end{tikzpicture} 
\]
Then 
\begin{proposition}\label{prop:combi}
Under Assumptions~\ref{ass:1}--\ref{ass:4}, the space
$ \CI_i\eqdef\ker \CK_i$ is a bialgebra ideal of $\Vec(\C_i)$, i.e. 
\begin{equ}
\Vec(\C_i)\cdot \CI_i\subset \CI_i, \qquad 
\Delta_i  \CI_i \subset \CI_i \otimes \Vec(\C_i) + \Vec(\C_i) \otimes\CI_i\;.
\end{equ}
Moreover setting $\BB_i:=\Vec(\C_i)/\CI_i$, the bialgebra $(\BB_i,\cdot,\Delta_i,\one_i,\one_i^\star)$ is a Hopf algebra, where $\one_i\eqdef\one+\CI_i$.
\end{proposition}
\begin{proof}
The first assertion follows from the fact that $\CK_i$ is an algebra morphism, and from Lemma~\ref{lem:ass4}.

For the second assertion, we note that the vector space $\BB_i$ is isomorphic to 
$\Vec(C_i)$, where $C_i=\{\tau\in\C_i: \CK_i\tau=\tau\}=\CK_i\C_i$. 
Moreover $\Vec(\C_i)/\CI_i$ as a bialgebra
is isomorphic to $(\Vec(C_i),\CK_i\CM,(\CK_i\otimes\CK_i)\Delta_i,\one_i,\one_i^\star)$, where $\CM$ denotes the forest product. The latter space is a Hopf algebra since it is a connected 
graded bialgebra with respect to the grading $|(F,\hat F)|_i\eqdef |F\setminus\hat F_i|$, namely the number of nodes and edges which are not coloured with $i$.
\end{proof}

We now extend the above construction to decorated forests.
\begin{definition}\label{CKop}
Let $\CK: \scal{\Tra}\to\scal{\Tra}$ be the triangular map given by
\[
\CK(F, \hat F)_\Labe^{\Labn,\Labo}\eqdef(\CK_{\hat F} F,\hat F)_{[\Labe]}^{[\Labn],[\Labo]}\;, \qquad (F, \hat F)_\Labe^{\Labn,\Labo}\in\Tra,
\]
where the decorations $[\Labn]$, $[\Labo]$ and $[\Labe]$ are defined as follows:
\begin{itemize}
\item if $x$ is an equivalence class of $\sim$ as in Definition~\ref{def:contrac}, then
$[\Labn](x) = \sum_{y \in x} \Labn(y)$. 
\item $[\Labe]$ is defined by simple restriction of $\Labe$ on $E_F\setminus \hat E$. 
\item $[\Labo](x)$ is defined by
\begin{equ}[e:defhatn]
[\Labo](x) \eqdef \sum_{y \in x} \Labo(y) + \sum_{e\in E_F\cap x^2}  \Labhom(e) . 
\end{equ}
\end{itemize}
\end{definition}
The definition \eqref{e:defhatn} explains why $\Labo$ is defined as a function taking 
values in $\Z^d\oplus\Z(\Lab)$, see Remark~\ref{whylabo} above.

\begin{remark}\label{rem:explanation}
The contraction of a subforest entails a loss of information. We use the decoration $\Labo$ in order to retain part of the lost information, namely the types of the edges which are contracted. This plays an important role in the degree $|\cdot|_+$ introduced in Definition~\ref{homog} below and is the key to one of the main results of this paper, see Remark~\ref{explanation}.
\end{remark}

\begin{example}
If $(F, \hat F)_\Labe^{\Labn,\Labo}$ is
\begin{equ}\label{excontr1}  
\begin{tikzpicture}[scale=0.13,baseline=1cm]
\def\frst{8}\def\scnd{16}
         \node at (-28,\scnd)  (i) {}; %h
          \node at (-20,\frst) (d) {}; %d
          \node at (-12,\scnd)  (j) {}; %j
          \node at (-20,\scnd) (h) {}; %i
      \node at (-12,0) (b) {}; %b
         \node at (-4,\frst)  (e) {}; %e
         \node at (-4,\scnd)  (f) {}; %f
          \node at (4,0)  (k) {}; %k
           \node at (4,\frst) (l) {}; %l
          \node at (12,0)  (s) {}; %s
           \node at (12,\frst) (t) {}; %t
         \node at (20,0) (p) {}; %p

      \draw[kernel1,black] (h)node[rect2] {\tiny $ \Labn$}  -- node [rect1] {\tiny$\Labhom,\Labe$}   (d);
          \draw[kernel1,black] (f)node[rect2] {\tiny $ \Labn$}  -- node [rect1] {\tiny$\Labhom,\Labe$}   (e);
    \draw[kernel1,red] (i)node[red,round2] {\tiny $ \Labn,\Labo$}  -- node [round1] {\tiny$\Labhom$}   (d);

     \draw[kernel1,black] (b)node[blue,round3] {\tiny $ \Labn_1,\Labo_1 $} -- node [rect1] {\tiny$\Labhom,\Labe$}   (d);
     \draw[kernel1,red] (d) node[round2] {\tiny $ \Labn,\Labo $} --  node [round1]{\tiny $ \Labhom $ } (j) node[round2] {\tiny $ \Labn,\Labo $};
     \draw[kernel1,blue] (b) -- node [round1] {\tiny$\Labhom$}  (e) node[round3] {\tiny $ \Labn_1,\Labo_1 $} ;

    \draw[kernel1,blue] (s)node[round3] {\tiny $ \Labn_3,\Labo_3 $} -- node [round1] {\tiny$\Labhom$}   (t) node[round3] {\tiny $ \Labn_3,\Labo_3 $} ;

    \draw[kernel1,black] (l)node[red,round2] {\tiny $ \Labn,\Labo $} -- node [rect1] {\tiny$\Labhom,\Labe$}   (k) node[blue,round3] {\tiny $ \Labn_2,\Labo_2 $} ;

\draw (p) node [blue,round3] {\tiny $ \Labn_4,\Labo_4$}  ;
\end{tikzpicture} 
\end{equ}
then $\CK(F, \hat F)_\Labe^{\Labn,\Labo}$ is
\begin{equ}\label{excontr2}
\begin{tikzpicture}[scale=0.13,baseline=1cm]
%        \node at (0,0)  (a) {}; %a
\def\frst{8}\def\scnd{16}
         \node at (-28,\scnd)  (i) {}; %h
          \node at (-20,\frst) (d) {}; %d
          \node at (-12,\scnd)  (j) {}; %j
          \node at (-20,\scnd) (h) {}; %i
      \node at (-12,0) (b) {}; %b
         \node at (-4,\frst)  (e) {}; %e
         \node at (-12,\frst)  (f) {}; %f
          \node at (-4,0)  (k) {}; %k
           \node at (-4,\frst) (l) {}; %l
          \node at (4,0)  (s) {}; %s
           \node at (4,\frst) (t) {}; %t
         \node at (12,0) (p) {}; %p

      \draw[kernel1,black] (h)node[rect2] {\tiny $ \Labn$}  -- node [rect1] {\tiny$\Labhom,\Labe$}   (d);
          \draw[kernel1,black] (f)node[rect2] {\tiny $ \Labn$}  -- node [rect1] {\tiny$\Labhom,\Labe$}   (b);% --  node [near end, rect1]{\tiny $ \Labhom,\Labe $ } (i) node[rect2] {\tiny $ \Labn$};
 %   \draw[kernel1,red] (i)node[red,round2] {\tiny $ \Labn,\Labo$}  -- node [round1] {\tiny$\Labhom$}   (d);% --  node [near end, rect1]{\tiny $ \Labhom,\Labe $ } (i) node[rect2] {\tiny $ \Labn$};

     \draw[kernel1,black] (b)node[blue,round3] {\tiny $ [\Labn_1],[\Labo_1] $} -- node [rect1] {\tiny$\Labhom,\Labe$}   (d);
     \draw[kernel1,red] (d) node[round2] {\tiny $ [\Labn],[\Labo] $};
 % --  node [round1]{\tiny $ \Labhom $ } (j) node[round2] {\tiny $ \Labn,\Labo $};
%     \draw[kernel1,blue] (b) -- node [round1] {\tiny$\Labhom$}  (e) node[round3] {\tiny $ \Labn,\Labo $} ;

    \draw[kernel1,black] (l)node[red,round2] {\tiny $ \Labn,\Labo $} -- node [rect1] {\tiny$\Labhom,\Labe$}   (k) node[blue,round3] {\tiny $ \Labn_2,\Labo_2 $} ;

    \draw[kernel1,blue]  (s) node[round3] {\tiny $ [\Labn_3],[\Labo_3]  $} ;

\draw (p) node [blue,round3] {\tiny $ \Labn_4,\Labo_4$}  ;
\end{tikzpicture} 
\end{equ}
Note that the types $\Labhom$ of edges which are erased by the contraction are stored 
inside the decoration $[\Labo]$ of the corresponding node.
\end{example}

Let now $\M_i \subset \Tra_i$ be the set
of decorated forests which are of type $(F,i, \Labn, \Labo, 0)$. This includes 
the case $F = \emptyset$ so that $\Units_i \subset \M_i$, where $\Units_i$ is defined in Definition~\ref{def:tra_i}. 
For example, the following decorated forest belongs to $\M_1$
\begin{equ}   \label{exM_i}
\begin{tikzpicture}[scale=0.13,baseline=1cm]
\def\frst{8}\def\scnd{16}
%        \node at (0,0)  (a) {}; %a
%         \node at (-28,30)  (h) {}; %h
          \node at (-20,\frst) (d) {}; %d
          \node at (-12,\scnd)  (j) {}; %j
          \node at (-20,\scnd) (i) {}; %i
      \node at (-12,0) (b) {}; %b
         \node at (-12,\frst)  (e) {}; %e
%         \node at (4,20)  (f) {}; %f
        \node at (-4,\frst) (q) {}; %p
          \node at (0,0)  (k) {}; %k
        \node at (4,\frst) (r) {}; %p
%           \node at (8,30) (l) {}; %l
%           \node at (20,20)  (g) {}; %g
%            \node at (16,30) (m) {};  %m
%        \node at (12,10) (c) {}; %c
         \node at (12,0) (p) {}; %p

     \draw[kernel1,red] (i)node[red,round2] {\tiny $ \Labn,\Labo$}  -- node [round1] {\tiny$\Labhom$}   (d);% --  node [near end, rect1]{\tiny $ \Labhom,\Labe $ } (i) node[rect2] {\tiny $ \Labn$};

     \draw[kernel1,red] (b)node[red,round2] {\tiny $\Labn,\Labo $} -- node [round1] {\tiny$\Labhom$}   (d) node[round2] {\tiny $ \Labn,\Labo $} --  node [round1]{\tiny $ \Labhom $ } (j) node[round2] {\tiny $\Labn,\Labo $};
     \draw[kernel1,red] (b) -- node [round1] {\tiny$\Labhom$}  (e) node[round2] {\tiny $ \Labn,\Labo $} ;

\draw (p) node [red,round2] {\tiny $ \Labn,\Labo $}  ;
     \draw[kernel1,red] (q)node[red,round2] {\tiny $ \Labn,\Labo$}  -- node [round1] {\tiny$\Labhom$}   (k)    node[red,round2] {\tiny $ \Labn,\Labo$}   (k) --  node [round1]{\tiny $ \Labhom$ } (r) node[red,round2] {\tiny $ \Labn,\Labo$};

\end{tikzpicture} 
\end{equ}
Compare this forest with that in \eqref{exunits_i}, which belongs to $\Units_1$; in \eqref{exM_i} the decoration $\Labn$ can 
be non-zero, while it has to be identically zero in \eqref{exunits_i}.

We define then an operator $k_i:\M_i\to\M_i$\label{Mi} by setting
\[
k_i(\nu)\eqdef(\bullet,i,\Sigma_\nu\Labn,0,0)\;,
\]
for any $\nu=(F,i, \Labn, \Labo, 0)$ with $\Sigma_\nu\Labn\eqdef\sum_{N_F}\Labn$. For 
instance, the forest in \eqref{exunits_i} is mapped by $k_1$ to $(\bullet,1,0,0,0)$, while
the forest $\nu$ in \eqref{exM_i} is mapped by $k_1$ to $(\bullet,1,\Sigma\Labn,0,0)$.
%, where
%\[
%(\bullet,1,0,0,0) = %{\color{red}\bullet}
%\begin{tikzpicture}[scale=0.2]
%        \node at (0,0)  (a) {}; %a
%     \draw[kernel1,red] (a)node[red,round2] {\tiny $0$ };  \end{tikzpicture},
%     \qquad  \qquad
%(\bullet,1,\Sigma_\nu\Labn,0,0) =
%\begin{tikzpicture}[scale=0.2]
%        \node at (0,0)  (a) {}; %a
%     \draw[kernel1,red] (a)node[red,round2] {\tiny $ \Sigma_\nu\Labn$};  \end{tikzpicture}
%\]

We define first the factorisation of $\Tra_i\ni\tau=\mu\cdot\nu$ where the forest product $\cdot$ has been defined 
in \eqref{cdot} and
\begin{itemize}
\item  $\nu\in\M_i$ is the disjoint union of all non-empty 
connected componens of $\tau$ of the form $(A,i, \Labn, \Labo, \Labe)$
\item $\mu\in\Tra_i$ is the unique element such that $\tau=\mu\cdot\nu$.
\end{itemize}
%Note that, for every $\tau\in\Tra_i$, there is a unique couple $(\mu,\nu)$ such that 
%$\tau = \mu\cdot \nu$, $\nu\in\M_i$, and such that every connected component 
%of $\mu = (G,\hat G, \Labn, \Labo, \Labe)$ in the sense of Remark~\ref{rem:connected} contains at least one node $x$ on which $\hat G(x) \neq i$. 
For instance, in \eqref{excontr1} and \eqref{excontr2}, we have two forests in $\Tra_2$; in both cases we 
have $\tau=\mu\cdot\nu$ as above, where 
$\mu$ is the product of the first two trees (from left to right) and $\nu\in\M_2$ is
the product of the two remaining trees. 

By the first part of Assumption~\ref{ass:2}, we know that if $\tau = \mu\cdot\nu \in \Tra_i$,
then $\mu \in \Tra_i$ and $\nu\in \Tra_i$. We also know by Assumption~\ref{ass:4} that if $\mu \in \Tra_i$, then
$\CK(\mu)\in \Tra_i$. Therefore, using this factorisation, we define $\Phi_i: \Tra_i\to\Tra_i$ by
\begin{equ}[Phi]
\Phi_i(\tau)\eqdef\mu\cdot k_i(\nu)\;.
\end{equ}
In \eqref{excontr1} and \eqref{excontr2}, the action of $\Phi_2$ corresponds to merging 
the third and fourth tree into a single decorated node $(\bullet,2,\Sigma\Labn_3+\Sigma\Labn_4,0,0)$
with all other components remaining unchanged. 

We also define $\hat \Phi_i: \Tra_i\to\Tra_i$ by $\hat \Phi_i = \hat P_i\circ \Phi_i = \Phi_i \circ \hat P_i$, 
where $\hat P_i (G,\hat G, \Labn, \Labo, \Labe)$ sets $\Labo$ to $0$ on 
every connected component of $\hat G_i$ that contains a root of $G$. For instance, the
action of $\hat P_2$ on the forests in   
\eqref{excontr1} and \eqref{excontr2} is to set to $0$ the decoration $\Labo$ of
all blue nodes. On the other hand, we have
\begin{equ}\label{exTra1_2}
\hat P_1 \left( \quad
\begin{tikzpicture}[scale=0.13,baseline=.5cm]
          \node at (-4,0)  (k) {}; %k
           \node at (-4,10) (l) {}; %l

    \draw[kernel1,black] (l)node[red,round2] {\tiny $ \Labn,\Labo $} -- node [rect1] {\tiny$\Labhom,\Labe$}   (k) node[red,round2] {\tiny $ \Labn,\Labo $} ;
\end{tikzpicture} 
\quad \right)=\quad 
\begin{tikzpicture}[scale=0.13,baseline=.5cm]
          \node at (-4,0)  (k) {}; %k
           \node at (-4,10) (l) {}; %l

    \draw[kernel1,black] (l)node[red,round2] {\tiny $ \Labn,\Labo $} -- node [rect1] {\tiny$\Labhom,\Labe$}   (k) node[red,round2] {\tiny $ \Labn $} ;
\end{tikzpicture} 
\end{equ}
namely the red node which is not in the red connected component of the root is left unchanged.
 
Finally, we define $\CK_i, \hat \CK_i:\Tra_i\to\Tra_i$
\begin{equ}[e:bi-ideal]
\CK_i\eqdef\Phi_i\circ\CK\;,\qquad \hat \CK_i\eqdef\hat \Phi_i\circ\CK\;. 
\end{equ}
For instance, if $\tau$ is the forest of \eqref{excontr1} and $\sigma=\CK(\tau)$ is 
that of \eqref{excontr2}, then 
\begin{equs}\label{exCK_2}
\CK_2(\tau)&= %\Phi_2(\CK(\tau))=
\Phi_2(\sigma)= \qquad
\begin{tikzpicture}[scale=0.13,baseline=1cm]
\def\frst{8}\def\scnd{16}
          \node at (-20,\frst) (d) {}; %d
          \node at (-12,\scnd)  (j) {}; %j
          \node at (-20,\scnd) (h) {}; %i
      \node at (-12,0) (b) {}; %b
         \node at (-4,\frst)  (e) {}; %e
         \node at (-12,\frst)  (f) {}; %f
          \node at (-2,0)  (k) {}; %k
           \node at (-2,\frst) (l) {}; %l
          \node at (8,0)  (s) {}; %s
 
      \draw[kernel1,black] (h)node[rect2] {\tiny $ \Labn$}  -- node [rect1] {\tiny$\Labhom,\Labe$}   (d);
          \draw[kernel1,black] (f)node[rect2] {\tiny $ \Labn$}  -- node [rect1] {\tiny$\Labhom,\Labe$}   (b);
     \draw[kernel1,black] (b)node[blue,round3] {\tiny $ [\Labn_1],[\Labo_1] $} -- node [rect1] {\tiny$\Labhom,\Labe$}   (d);
     \draw[kernel1,red] (d) node[round2] {\tiny $ [\Labn],[\Labo] $};

    \draw[kernel1,black] (l)node[red,round2] {\tiny $ \Labn,\Labo $} -- node [rect1] {\tiny$\Labhom,\Labe$}   (k) node[blue,round3] {\tiny $ \Labn_2,\Labo_2 $} ;

    \draw[kernel1,blue]  (s) node[round3] {\tiny $ [\Labn_3]+\Labn_4  $} ;
\end{tikzpicture} \\
\hat\CK_2(\tau)&= %\hat\Phi_2(\CK(\tau))=
\hat\Phi_2(\sigma)=\qquad
\begin{tikzpicture}[scale=0.13,baseline=1cm]
\def\frst{8}\def\scnd{16}
          \node at (-20,\frst) (d) {}; %d
          \node at (-12,\scnd)  (j) {}; %j
          \node at (-20,\scnd) (h) {}; %i
      \node at (-12,0) (b) {}; %b
         \node at (-4,\frst)  (e) {}; %e
         \node at (-12,\frst)  (f) {}; %f
          \node at (-2,0)  (k) {}; %k
           \node at (-2,\frst) (l) {}; %l
          \node at (8,0)  (s) {}; %s

      \draw[kernel1,black] (h)node[rect2] {\tiny $ \Labn$}  -- node [rect1] {\tiny$\Labhom,\Labe$}   (d);
          \draw[kernel1,black] (f)node[rect2] {\tiny $ \Labn$}  -- node [rect1] {\tiny$\Labhom,\Labe$}   (b);

     \draw[kernel1,black] (b)node[blue,round3] {\tiny $ [\Labn_1] $} -- node [rect1] {\tiny$\Labhom,\Labe$}   (d);
     \draw[kernel1,red] (d) node[round2] {\tiny $ [\Labn],[\Labo] $};
 % --  node [round1]{\tiny $ \Labhom $ } (j) node[round2] {\tiny $ \Labn,\Labo $};
%     \draw[kernel1,blue] (b) -- node [round1] {\tiny$\Labhom$}  (e) node[round3] {\tiny $ \Labn,\Labo $} ;

    \draw[kernel1,black] (l)node[red,round2] {\tiny $ \Labn,\Labo $} -- node [rect1] {\tiny$\Labhom,\Labe$}   (k) node[blue,round3] {\tiny $ \Labn_2 $} ;

    \draw[kernel1,blue]  (s) node[round3] {\tiny $[\Labn_3]+\Labn_4  $} ;

%\draw (p) node [blue,round3] {\tiny $ \Labn,\Labo$}  ;
\end{tikzpicture} 
\end{equs}
Note that in $\CK_2(\tau)$ the roots of the connected components which do not belong to 
$\M_2$ may have a non-zero $\Labo$ decoration, while the unique connected component in 
$\M_2$ (reduced to a blue root with a possibly non-zero $\Labn$ decoration) always has a zero 
$\Labo$ decoration. In $\hat\CK_2(\tau)$ all roots have zero $\Labo$ decoration.

Since $\CK$ commutes with $\Phi_i$ (as well as with $\hat \Phi_i$), 
is multiplicative, and is the identity on the image of $k_i$ in $\M_i$, 
it follows that for $\tau=\mu\cdot\nu$ as above, we have
\[
\CK_i(\tau)=\CK(\mu)\cdot k_i(\nu)\;.
\]
Moreover $\CK_i$ and $\hat \CK_i$ are idempotent and extend to triangular maps on 
$\scal{\Tra_i}$ since $\CK$, $\Phi_i$ and $\hat \Phi_i$ are all idempotent
and preserve our bigrading. We then have the following result.

\begin{lemma}\label{lem:biideal} Under Assumptions~\ref{ass:1}--\ref{ass:4}, the spaces 
$ \CI_i = \ker \CK_i$ and $\hat \CI_i = \ker \hat \CK_i$ are bialgebra ideals, i.e. 
\begin{equ}
\scal{\Tra_i}\cdot \CI_i \subset \CI_i, \qquad 
\Delta_i \CI_i \subset \CI_i \hattimes \scal{\Tra_i} + \scal{\Tra_i} \hattimes \CI_i\;,
\end{equ}
and similarly for $\hat \CI_i$.
\end{lemma}
\begin{proof}
Although $\CK_i$ is not quite an algebra morphism of $(\scal{\Tra_i},\cdot)$, it has the property $\CK_i(a\cdot b)=\CK_i(a\cdot\CK_i(b))$ for all $a,b\in\Tra_i$, from which the first property follows
for $\CI_i$. Since $\hat P_i$ is an algebra morphism, the same holds for $\hat \CI_i$.
To show the second claim, we first recall 
that for all coloured forests $(F,\hat F)$, the map $\CK_{\hat F}^\sharp$ defined in Definition~\ref{def:contrac} is, by the Assumption~\ref{ass:4}, a bijection between $\Adm_i(\CK_{\hat F} F,\hat F)$ and $\Adm_i(F,\hat F)$. Combining this with
Chu-Vandermonde, one can show that $\CK$ satisfies
\begin{equ}[e:CKCK]
(\CK\otimes\CK)\Delta_i \CK= (\CK\otimes\CK)\Delta_i\;.
\end{equ}
The same can easily be verified for $\Phi_i$ and $\hat P_i$, so that it also holds for $\CK_i$
and $\hat \CK_i$,
whence the claim follows.
\end{proof} 

If we define
\begin{equ}[e:defHi]
\CH_i\eqdef\scal{\Tra_i} / \CI_i, \qquad \one_i\eqdef \one+\CI_i\in\CH_i, 
\end{equ}
then, as a consequence of Lemma~\ref{lem:biideal}, $( \CH_i,\cdot,\Delta_i,\one_i,\one_i^\star)$ defines a bialgebra. 
\begin{remark}\label{rem:H_i}
Using Lemma~\ref{lem:PP=P}, we have a canonical isomorphism
\begin{equ}
( \CH_i,\CM,\Delta_i,\one_i,\one_i^\star)
\quad\longleftrightarrow\quad (\scal{H_i}, \CK_i \CM, (\CK_i\otimes\CK_i)\Delta_i, \one_i, \one_i^\star)\;,
\end{equ}
where $H_i=\{\CF\in\Tra_i : \CK_i\CF=\CF\}=\CK_i\Tra_i$ and $\CM$ denotes the forest product. 
This can be useful if one wants to work with explicit representatives rather than 
with equivalence classes. Note that $H_i$ can be characterised as the set of
all $(F,\hat F,\Labn,\Labo,\Labe)\in\Tra_i$ such that
\begin{enumerate}
\item the coloured subforests $\hat F_k$, $0<k\leq i$, contain no edges, namely $\hat E=\emptyset$,
\item there is one and only one connected component of $F$ which has the form $(\bullet,i,\Labn,\Labo,0)$ 
and moreover $\Labo(\bullet)=0$.
\end{enumerate}
For example, the forest in \eqref{exCK_2} is an element of $H_2$. 
\end{remark}

\begin{proposition}\label{prop:Hopf}
Under Assumptions~\ref{ass:1}--\ref{ass:4}, the space $(\CH_i, \cdot,\Delta_i, \one_i, \one_i^\star)$ is a Hopf algebra.
\end{proposition}
\begin{proof}
By Lemma~\ref{lem:counit}, $\one_i^\star$ is a counit in $\CH_i$. 
We only need now to show that this space admits an antipode $\CA_i$, that we are going to
construct recursively. 

For $k \in \N^d$, we denote by $X^{k} \in \CH_i$ the equivalence class of the element $(\bullet,i,k,0,0)$. It then follows from 
\begin{equ}[e:DeltaX]
\Delta_i X^k=\sum_{j\in \N^d} \binom{k}{j}X^j\otimes X^{k-j}
\end{equ} 
that the subspace spanned by $(X^{k}, k \in \N^d)$ is isomorphic to the Hopf algebra of 
polynomials in $d$ commuting variables, provided that we set
\begin{equ}[e:defAbasic2]
\CA_i X^{k} = (-1)^{|k|} X^{k}\;.
\end{equ}
For any $\tau = (F,\hat F, \Labn, \Labo, \Labe) \in \Tra_i$, let $|\tau|_i = |F \setminus \hat F_i|$ and recall the definition \eqref{e:grading0} of the bigrading $|\tau|_\bi$. Note that $|\CK_i\tau|_i=|\tau|_i$ and, as we have already remarked, $|\CK_i\tau|_\bi=|\tau|_\bi$, so that both these gradings make sense on $\CH_i$. We now extend $\CA_i$ to $\CH_i$ by induction on $|\tau|_i$.

If $|\tau|_i = 0$ then, by definition, one has $\tau \in \M_i$ so that $\tau = X^k$ for some $k$ and \eqref{e:defAbasic2} defines $\CA_i\tau$. 
Let now $N > 0$ and assume that 
$\CA_i \tau$ has been defined for all $\tau \in \CH_i$ with $|\tau|_i < N$.
Assume also that it is such that if $|\tau|_\bi = m$, then 
$(\CA_i \tau)_n \neq 0$ only if $n \ge m$,
which is indeed the case for \eqref{e:defAbasic2} since all the terms appearing
there have degree $(0,0)$. (This latter condition is required if we want $\CA_i$
to be a triangular map.)

For $\tau = (F,\hat F, \Labn, \Labo, \Labe)$ and 
$k \colon N_F \to \N^d$, we define
$R_k \tau \eqdef (F,\hat F, k, \Labo, \Labe)$.
For such a $\tau$ with $|\tau|_i = N$ and $|\tau|_\bi = M$, we then note that one has
%(with the obvious modification of Sweedler's notation)
\begin{equ}
\Delta_i \tau = \sum_{k\le \Labn} \binom{\Labn}{k} R_k \tau \otimes X^{\Sigma(\Labn-k)}
+ \sum_{\ell+m \ge M} \tau_{(1)}^\ell\otimes \tau_{(2)}^m\;,
\end{equ}
where $\Sigma(\Labn-k):=\sum_{x\in F}(\Labn-k)(x)$ and for $\ell\in\N^2$
\begin{equs}
&\tau_{(1)}^\ell \in \Vec{\{\sigma \in \CH_i: |\sigma|_\bi=\ell, |\sigma|_i < N\}},
\\
&\tau_{(2)}^\ell \in \Vec{\{\sigma \in \CH_i: |\sigma|_\bi=\ell, |\sigma|_i \leq N\}}.
\end{equs}
Note that the first term in the right hand side above corresponds to the choice of $A=F$, while the second term contains the sum over all possible $A\ne F$.
Here, the property $|\tau_{(1)}^\ell|_i < N$ holds because these terms come
from terms with $A \neq F$ in \eqref{def:Deltabar}.
Since for $\tau \neq \one_i$ we want to have
\begin{equ}
\CM(\CA_i \otimes \id)\Delta_i \tau = 0\;,
\end{equ}
this forces us to choose $\CA_i \tau$ in such a way that
\begin{equ}[e:defAirec]
\CA_i \tau = - \sum_{k \neq \Labn} \binom{\Labn}{k} \CA_i(R_k \tau) \cdot X^{\Sigma(\Labn-k)}
- \sum_{\ell+m \ge M} \CA_i(\tau_{(1)}^\ell)\cdot  \tau_{(2)}^m\;.
\end{equ}
In the case $\Labn = 0$, this uniquely defines $\CA_i \tau$ by the induction hypothesis
since every one of the terms $\tau_{(1)}^\ell$ appearing in this expression 
satisfies $|\tau_{(1)}^\ell|_i < N$. 

In the case where $\Labn \neq 0$, $\CA_i \tau$ is also easily seen to be uniquely defined
by performing a second inductive step over $|\Labn| \in \N$.
All terms appearing in the right hand side of \eqref{e:defAirec} do indeed satisfy that their
total $|\cdot|_\bi$-degree is at least $M$ by using the induction hypothesis.
Furthermore, our definition
immediately guarantees that $\CM(\CA_i \otimes \id)\Delta_i = \one_i\one_i^\star$.
It remains to verify that one also has $\CM(\id\otimes \CA_i)\Delta_i = \one_i\one_i^\star$.
For this, it suffices to verify that $\CA_i$ is multiplicative,
whence the claim follows by mimicking the proof of the fact that a semigroup with
left identity and left inverse is a group.

Multiplicativity of $\CA_i$ also follows by induction over $N=|\tau|_i$. Indeed,
it follows from \eqref{e:defAbasic2} that it is the case for $N=0$. It is also
easy to see from \eqref{e:defAirec} that if $\tau$ is of the form $\tau' \cdot X^{k}$ for some
$\tau'$ and some $k >0$, then one has $\CA_i \tau = (\CA_i \tau')\cdot (\CA_i X^{k})$.
Assuming that 
it is the case for all values less than some $N$, it therefore suffices to verify
that $\CA_i$ is multiplicative for elements of the type 
$\tau = \sigma\cdot \bar \sigma$ with $|\sigma|_i\wedge |\bar \sigma|_i > 0$.
If we extend $\CA_i$ multiplicatively to elements of this type then, as a consequence of the 
multiplicativity of $\Delta_i$, one has
\begin{equ}
\CM(\CA_i \otimes \id)\Delta_i \tau = (\CM(\CA_i \otimes \id)\Delta_i \sigma)\cdot
(\CM(\CA_i \otimes \id)\Delta_i \bar\sigma) = 0\;,
\end{equ}
as required. Since the map $\CA_i$ satisfying this property was uniquely defined
by our recursion, this implies that $\CA_i$ is indeed multiplicative.
\end{proof} 

\subsection{Characters group}\label{sec:char}

Recall that an element $g\in\CH_i^*$ is a character if 
$g(\tau \cdot \bar \tau) = g(\tau)g(\bar \tau)$ for any $\tau, \bar \tau \in \CH_i$.
Denoting by $\CG_i$ the set of all such characters, the Hopf algebra
structure described above turns $\CG_i$ into a group by 
\begin{equ}[e:defG]
(f\circ g)(\tau) = (f\otimes g)\,\Delta_i\tau\;,\qquad g^{-1}(\tau) = g(\CA_i\tau)\;,
\end{equ}
where the former operation is guaranteed to make sense by Remark~\ref{rem:dual}.
\begin{definition}\label{def:primitive}
Denote by $\Primitive_i$ the set of elements $\CF=(F,\hat F,\Labn,\Labo,\Labe) \in H_i$ as in Remark~\ref{rem:H_i}, such that 
\begin{itemize}
\item $F$ has exactly one connected component 
\item either $\hat F$ is not identically equal to $i$ or $\CF=(\bullet,i,\delta_n,0,0)$ for some $n \in \{1,\ldots,d\}$, where $(\delta_n(\bullet))_j=\delta_{nj}$.
\end{itemize}
\end{definition}
It is then easy to see that for every $\tau \in H_i$ there exists a unique
(possibly empty) collection $\{\tau_1,\ldots,\tau_N\} \subset \Primitive_i$ such that
$\tau = \CK_i(\tau_1 \cdot \ldots\cdot \tau_N)$.
As a consequence, a multiplicative functional on $\CH_i$ is uniquely determined by
the collection of values $\{g(\tau)\,:\, \tau \in \Primitive_i\}$.
The following result gives a complete characterisation of the class of functions 
$g \colon \Primitive_i \to \R$ which can be extended in this way to a multiplicative functional on $\CH_i$. 

\begin{proposition}
A function $g \colon \Primitive_i \to \R$ determines an element of $\CG_i$ 
as above if
and only if there exists $m \colon \N \to \N$ such that 
$g(\tau) = 0$ for every $\tau\in \Primitive_i$ with $|\tau|_\bi = n$ such that $n_1 > m(n_2)$.
\end{proposition}

\begin{proof}
We first show that, under this condition, the unique multiplicative extension of $g$ defines an element of $\CH_i^*$. By Remark~\ref{rem:dual}, we thus need to show
that there exists a function $\tilde m\colon \N \to \N$ such that
$g(\tau) = 0$ for every $\tau\in H_i$ with $|\tau|_\bi = n$ and $n_1 > \tilde m(n_2)$.

If $\sigma=(F,\hat F,\Labn,\Labo,\Labe) \in \Primitive_i$ satisfies $n_2 =0$, then $\hat F$ is nowhere equal to 0 on $F$ by the definition \eqref{e:grading0}; by property 2 in Definition~\ref{def:colour}, $\hat F$ is constant on $F$, since we also assume that $F$ has a single connected component; in this case $\Labe\equiv 0$ by property 3 in Definition~\ref{def:decoration}; therefore, if $n_2=0$ then $n_1=0$ as well. Therefore we can set $\tilde m(0)=0$.

Let now $k\geq 1$. We claim that $\tilde m(k) \eqdef k\sup_{1\le\ell \le k} m(\ell)$ has the required property. 
Indeed, for $\tau = \CK_i(\tau_1 \cdot \ldots\cdot \tau_N)$,
one has  $g(\tau) = 0$ unless $g(\tau_j) \neq 0$ for every $j$; in this case, setting $n^j=(n^j_1,n^j_2)=|\tau_j|_\bi$, we have $m(n^j_2)\geq n^j_1$ for all $j=1,\ldots,N$.
Since $n=(n_1,n_2)\eqdef|\tau|_\bi = \sum_j |\tau_j|_\bi$, this implies that $n_k=\sum_jn^j_k$, $k=1,2$. Then \[
\tilde m(n_2)\geq n_2\max_{1\leq\ell\leq n_2}m(\ell)\geq  n_2\max_{1\leq\ell\leq N}n^j_1\geq n_1.
\]

The converse is elementary.
\end{proof}

\subsection{Comodule bialgebras}
Let us fix throughout this section $0<i < j$. We want now to study the possible interaction between the structures given by the operators $\Delta_i$ and $\Delta_j$. For the definition
of a comodule, see the beginning of Section \ref{sec3}.

\begin{assumption}\label{ass:5}
Let $0<i < j$.
For every coloured forest $(F,\hat F)$ such that $\hat F \le j$ and $\{F,\hat F_j\} \subset \Adm_j(F,\hat F)$, one has $\hat F_i \in \Adm_i(F,\hat F)$.
\end{assumption}

\begin{lemma}\label{prop:comodule}
Let $0<i < j$. Under Assumptions~\ref{ass:1}--\ref{ass:4} for $i$ and under Assumption~\ref{ass:5} we have
\begin{equ}
\Delta_i \colon \scal{\Tra_{j}} \to \scal{\Tra_i} \hattimes \scal{\Tra_{j}}\;,\qquad
(\one_i^\star \otimes \id)\Delta_i = \id\;,
\end{equ}
which endows $\scal{\Tra_{j}}$ with the structure of a left comodule over the bialgebra $\scal{\Tra_i}$.
\end{lemma}
\begin{proof}
Let $(F,\hat F, \Labn, \Labo, \Labe) \in \Tra_j$ and $A\in\Adm_i(F,\hat F)$; 
by Definition~\ref{def:tra_i}, we have $\hat F\leq j$ and $\{F,\hat F_j\}\subset\Adm_j(F,\hat F)$, so that by Assumption~\ref{ass:5} we have $\hat F_i\in\Adm_i(F,\hat F)$. Then, by property 1 in Assumption~\ref{ass:3}, we have $\hat F_i\cap A=\hat F_i\in\Adm_i(A,\hat F\restr A)$. Now, since $A\cap \hat F_j=\emptyset$ by property 1 in Assumption~\ref{ass:1}, we have $(\hat F\cup_i A)_j=\hat F_j\setminus A=\hat F_j\in\Adm_j(F,\hat F\cup_i A)$ by the Definition~\ref{def:tra_i} of $\Tra_j$; all this shows that $\Delta_i \colon \scal{\Tra_{j}} \to \scal{\Tra_i} \hattimes \scal{\Tra_{j}}$.

For $A\in\Adm_i(F,\hat F)$, we have $(A,\hat F\restr A,\Labn',\Labo',\Labe')\in\Units_i$ if and only if $\hat F\equiv i$ on $A$, i.e. $A\subseteq \hat F_i$; since $\hat F_i\subseteq A$ by Assumption~\ref{ass:1}, then the only possibility is $A=\hat F_i$. By  Assumption~\ref{ass:5} we have $\hat F_i \in \Adm_i(F,\hat F)$ and therefore $(\one_i^\star \otimes \id)\Delta_i = \id$.

Finally, the co-associativity \eqref{mult_Deltaa1} of $\Delta_i$ on $\Tra$ shows the required compatibility between the coaction $\Delta_i \colon \scal{\Tra_{j}} \to \scal{\Tra_i} \hattimes \scal{\Tra_{j}}$ and the coproduct $\Delta_i \colon \scal{\Tra_{i}} \to \scal{\Tra_i} \hattimes \scal{\Tra_{i}}$.
\end{proof}

We now introduce an additional structure which will yield as a consequence the {\it cointeraction} property \eqref{e:intertwine} between the maps $\Delta_i$ and $\Delta_j$, see 
Remark \ref{cointera}.

\begin{assumption}\label{ass:6}
Let $0<i < j$.
For every coloured forest $(F,\hat F)$, one has
\minilab{e:compat}
\begin{equ}[e:compat1]
A \in \Adm_i(F,\hat F)\qquad\& \qquad B \in \Adm_j(F, \hat F \cup_i A)\;,
\end{equ}
if and only if
\minilab{e:compat}
\begin{equ}[e:compat2]
B \in \Adm_j(F,\hat F)\qquad\&\qquad A \in \Adm_i(F, \hat F \cup_j B) \sqcup \Adm_i(B, \hat F \restr B)\;,
\end{equ}
where $\Adm \sqcup \bar \Adm$ is a shorthand for $\{A\sqcup \bar A\,:\ A\in \Adm \;\&\; \bar A\in \bar \Adm\}$. 
\end{assumption}

We then have the following crucial result. 
\begin{proposition}\label{prop:doublecoass}
Under Assumptions~\ref{ass:1} and~\ref{ass:6} for some $0 < i<j$, the identity
\begin{equ}[e:intertwine]
\CM^{(13)(2)(4)} \bigl(\Delta_i \otimes \Delta_i\bigr)\Delta_j = (\id \otimes \Delta_j)\Delta_i
\end{equ}
holds on $\Tra$, where we used the notation
\begin{equ}[e:def1324]
 \mathcal{M}^{(13)(2)(4)} (\tau_1 \otimes \tau_2 \otimes \tau_3 \otimes \tau_4) =
 ( \tau_1\cdot \tau_3 \otimes \tau_2 \otimes \tau_4  )\;.
\end{equ}
\end{proposition}

\begin{proof}
The proof is very similar to that of Proposition~\ref{prop:coassoc}, but using
\eqref{e:compat} instead of \eqref{e:prop}. 
Using \eqref{e:compat} and our definitions, for $\tau=(F,\hat F, \Labn, \Labo, \Labe)\in\Tra$ one has
\begin{equs}
{} &
\CM^{(13)(2)(4)} \bigl(\Delta_i \otimes \Delta_i\bigr)\Delta_j \,\tau =
\\ & = \sum_{B \in \Adm_j(F,\hat F)} 
\sum_{A_1 \in \Adm_i(B,\hat F \restr B)} 
\sum_{A_2 \in \Adm_i(F,\hat F \cup_j B)} 
\sum_{\eps_B^F, \eps_{A_1}^B,\eps_{A_2}^F}\sum_{\Labn_B,\Labn_{A_1},\Labn_{A_2}} \\
& {1\over \eps_B^F! \eps_{A_1}^B!\eps_{A_2}^F!}
\binom{\Labn}{\Labn_B} 
\binom{\Labn-\Labn_B}{\Labn_{A_2}} 
\binom{\Labn_B + \pi\eps_B^F}{\Labn_{A_1}} \label{com_fact0} \\
&(A_1 \sqcup A_2, \hat F \restr A, \Labn_{A_1}+\Labn_{A_2}+ \pi(\eps_{A_1}^B+\eps_{A_2}^F), \Labo ,\Labe) \\
&\otimes (B, (\hat F\restr B)\cup_i A_1, \Labn_B + \pi\eps_B^F - \Labn_{A_1}, \Labo + \Labn_{A_1} + \pi(\eps_{A_1}^B - \Labe_\emptyset^{A_1}), \Labe_{A_1}^B + \eps_{A_1}^B)\\
&\otimes (F, (\hat F \cup_j B) \cup_i A_2, \Labn - \Labn_B - \Labn_{A_2}, \Labo +\Labn_B + \Labn_{A_2}
+ \pi(\eps_B^F + \eps_{A_2}^F - \Labe_\emptyset^{A_2\sqcup B})\\
&\qquad, \Labe_{A_2\sqcup B}^F + (\eps_B^F)_{A_2}^F + \eps_{A_2}^F)\;.
\end{equs}
We claim that $A_2 \cap B = \emptyset$. Indeed, as noted in the proof of Lemma~\ref{lem:new_coloured_forest}, since $B\in\Adm_j(F,\hat F)$ one has $(\hat F \cup_j B)^{-1}(j)=B$ and since $A_2\in \Adm_i(F,\hat F \cup_j B)$ one has $A_2\cap (\hat F \cup_j B)^{-1}(j)=\emptyset$ by property 1 in Assumption~\ref{ass:1}. This implies that
\[
(\eps_B^F)_{A_2}^F=\eps_B^F\;,
\]
since $\eps_B^F$ has support in $\partial(B,F)$ which is disjoint from $E_{A_2}$. 
This is because, for $e=(e_+,e_-)\in\partial(B,F)$ we have by definition $e_+\in N_B\subset N_F\setminus N_{A_2}$ and therefore $e\notin E_{A_2}$. 

Similarly, one has
\begin{equs}
(\id &\otimes \Delta_j)\Delta_i \tau= 
\\ & = 
 \sum_{A \in \Adm_i(F,\hat F)} 
\sum_{C \in \Adm_j(F,\hat F \cup_i A)}
 \sum_{\eps_C^F, \eps_A^F}\sum_{\Labn_C,\Labn_A} 
 {1\over \eps_A^F! \eps_C^F!}\binom{\Labn}{\Labn_A}\binom{\Labn-\Labn_A}{\Labn_C}
 \label{com_fact} \\
& (A, \hat F\restr A, \Labn_A + \pi \eps_A^F, \Labo, \Labe) \\
&\otimes (C, (\hat F\cup_i A) \restr C, \Labn_C + \pi \eps_C^F, \Labo + \Labn_A + \pi(\eps_A^F - \Labe_\emptyset^{C}), \Labe_A^F + \eps_A^F)\\
&\otimes (F, (\hat F \cup_i A) \cup_j C, \Labn - \Labn_A-\Labn_C, \Labo + \Labn_A + \Labn_C 
+ \pi((\eps_A^F)_C^F + \eps_C^F - \Labe_\emptyset^{C\cup A})\\
&\qquad , \Labe_{C\cup A}^F + (\eps_A^F)_C^F + \eps_C^F)\;.
\end{equs}
By Assumption~\ref{ass:6}, there is a bijection between the outer sums of \eqref{com_fact0}
and \eqref{com_fact} given by
$(A,C) \leftrightarrow (A_1 \sqcup A_2,B)$, with inverse $(A_1,A_2,B) \leftrightarrow (A \cap C,A\backslash C,C)$. Then one then has
indeed $(\hat F \restr B)\cup_i A_1 = (\hat F\cup_i A) \restr C$. Similarly, since
$i < j$  and $A_2\cap C=\emptyset$,
one has $(\hat F \cup_j B) \cup_i A_2 = (\hat F \cup_i A) \cup_j C$, so we only need to 
consider the decorations and the combinatorial factors.

For this purpose, we define
\begin{equs}
\bar \eps_{A_1}^C = \eps_A^F\, \un{E_C}\;,\qquad
& \bar \eps_{A_2}^F = (\eps_A^F)\, \un{\partial({A_2},F)}\;,
\\ \bar \eps_{A_1,C}^F = \eps_A^F\,\un{\partial(C,F)}\;,\qquad
&\bar \eps_C^F = \eps_C^F + \bar \eps_{A_1,C}^F\;, 
\end{equs}
as well as
\begin{equ}
\bar \Labn_{A_1} = (\Labn_A \restr C) + \pi \bar \eps_{A_1,C}^F\;,\quad
\bar \Labn_{A_2} = \Labn_A \restr (F\setminus C)\;,\quad
\bar \Labn_C = \Labn_C + (\Labn_A \restr C)\;.
\end{equ}
As before, the supports of these functions are consistent with our notations, with the particular case of $\bar \eps_{A_1,C}^F$ whose support is contained in $\partial(A,F)\cap\partial(C,F)=\partial(A_1,F)\cap\partial(C,F)$, where we use again the fact that $A_2\cap C=\emptyset$. Moreover the map 
\[
(\eps_A^F,\eps_C^F,\Labn_A,\Labn_C)\mapsto(\bar \eps_{A_1}^C,\bar \eps_{A_2}^C,\bar \eps_{A_1,C}^F, \bar \eps_C^F,\bar \Labn_{A_1},\bar \Labn_{A_2},\bar \Labn_C)
\]
is invertible on its image, given by the functions with the correct supports and the additional constraint
\[
\bar \Labn_{A_1}\geq \pi\bar \eps_{A_1,C}^F\;.
\]
Its inverse is given by
\begin{equs}[2]
\eps_C^F&=\bar \eps_C^F-\bar \eps_{A_1,C}^F\;,&\quad
\eps_A^F&=\bar \eps_{A_1}^C+\bar \eps_{A_2}^F+\bar \eps_{A_1,C}^F\;,
\\ \Labn_A&=\bar \Labn_{A_1}+\bar \Labn_{A_2}-\pi \bar \eps_{A_1,C}^F\;,&\quad
 \Labn_C&=\bar \Labn_C-\bar \Labn_{A_1}+\pi \bar \eps_{A_1,C}^F\;.
\end{equs}

Following a calculation virtually identical to \eqref{e:calc1} and \eqref{e:calc2}, combined
with the fact that $\Labn_A + \Labn_C = \bar \Labn_C + \bar \Labn_{A_2}$, we see that
\begin{equs}
{1\over \eps_A^F! \,\eps_C^F!}&=
\frac1{\bar \eps_{A_1}^C!\,\bar \eps_{A_2}^F!\,\bar \eps_{A_1,C}^F!} \frac1{(\bar \eps_C^F-\bar \eps_{A_1,C}^F)!}=
{1\over \bar \eps_C^F! \,\bar \eps_{A_1}^C! \, \bar \eps_{A_2}^F!}
\binom{\bar \eps_C^F}{\bar \eps_{A_1,C}^F}\;, \\
\binom{\Labn}{\Labn_A}\binom{\Labn-\Labn_A}{\Labn_C}
&=
\binom{\bar \Labn_C + \bar \Labn_{A_2}}{\bar \Labn_{A_1} + \bar \Labn_{A_2} - \pi \bar \eps_{A_1,C}^F} 
\binom{\Labn}{\bar \Labn_C + \bar \Labn_{A_2}} \;.
\end{equs}
Since $A_2\cap C=\emptyset$ and $A_1\subset C$, we can simplify this expression further and obtain
\[
\binom{\bar \Labn_C + \bar \Labn_{A_2}}{\bar \Labn_{A_1} + \bar \Labn_{A_2} - \pi \bar \eps_{A_1,C}^F} =
\binom{\bar \Labn_C }{\bar \Labn_{A_1}  - \pi \bar \eps_{A_1,C}^F} \;.
\]
Following the same argument as \eqref{e:binomfinal}, we conclude that
\begin{equ}
\sum_{\bar \eps_{A_1,C}^F} 
\binom{\bar \eps_C^F}{\bar \eps_{A_1,C}^F}
\binom{\bar \Labn_C }{\bar \Labn_{A_1}  - \pi \bar \eps_{A_1,C}^F} 
=
\binom{\bar \Labn_C + \pi\bar \eps_C^F}{\bar \Labn_{A_1}}\;,
\end{equ}
so that \eqref{com_fact} can be rewritten as
\begin{equs}[e:finalPrefact]
(\id &\otimes \Delta_j)\Delta_i \tau= 
\sum_{C \in \Adm_j(F,\hat F)}
 \sum_{A \in \Adm_i(F,\hat F\restr C)} 
 \sum_{A \in \Adm_i(F,\hat F\cup_j C)} 
 \sum_{\bar \eps_{A_1}^C,\bar \eps_{A_2}^C,\bar \eps_C^F}\sum_{\bar \Labn_{A_1},\bar \Labn_{A_2},\bar \Labn_C} 
 \\ & {1\over \bar \eps_C^F! \bar \eps_{A_1}^C!\bar \eps_{A_2}^F!}
\binom{\Labn}{\bar \Labn_C + \bar \Labn_{A_2}} 
\binom{\bar \Labn_C + \pi\bar \eps_C^F}{\bar \Labn_{A_1}} 
\\ &
 (A_1\sqcup A_2, \hat F\restr A, \bar\Labn_{A_1}+\bar\Labn_{A_2}+ \pi(\bar\eps_{A_1}^C+\bar\eps_{A_2}^F), \Labo, \Labe) \\
&\otimes (C, (\hat F\restr C)\cup_i A_1, \bar\Labn_C + \pi \bar\eps_C^F-\bar\Labn_{A_1}, \Labo + \bar\Labn_{A_1} + \pi(\bar\eps_{A_1}^C - \Labe_\emptyset^{A_1}), \Labe_{A_1}^C + \bar\eps_{A_1}^C)\\
&\otimes (F, (\hat F \cup_j C) \cup_i A_2, \Labn - \bar\Labn_C - \bar\Labn_{A_2}, \Labo +\bar\Labn_C + \bar\Labn_{A_2}
+ \pi(\bar\eps_C^F + \bar\eps_{A_2}^F - \Labe_\emptyset^{A_2\sqcup C})\\
&\qquad, \Labe_{A_2\sqcup C}^F + \bar\eps_C^F + \bar\eps_{A_2}^F)\;.
\end{equs}
We have also used the fact that
\[
(\pi\eps_A^F)\restr N_C = \pi(\eps_A^F\un{E_C})+\pi(\eps_A^F\un{\partial(C,F)})=\pi\bar\eps_{A_1}^C+\pi\bar\eps^F_{A_1,C}.
\]

On the other hand, since $A_2$ and $B$ are disjoint, one has
\begin{equ}
\binom{\Labn}{\Labn_B} 
\binom{\Labn-\Labn_B}{\Labn_{A_2}} 
=
{\Labn! \over \Labn_B!\, \Labn_{A_2}! \, (\Labn - \Labn_B - \Labn_{A_2})!}
=
\binom{\Labn}{\Labn_B + \Labn_{A_2}}\;,
\end{equ}
so that \eqref{com_fact0} can be rewritten as
\begin{equs}[e:finalPrefact2]
{} &
\CM^{(13)(2)(4)} \bigl(\Delta_i \otimes \Delta_i\bigr)\Delta_j \,\tau =
\\ & = \sum_{B \in \Adm_j(F,\hat F)} 
\sum_{A_1 \in \Adm_i(B,\hat F \restr B)} 
\sum_{A_2 \in \Adm_i(F,\hat F \cup_j B)} 
\sum_{\eps_B^F, \eps_{A_1}^B,\eps_{A_2}^F}\sum_{\Labn_B,\Labn_{A_1},\Labn_{A_2}} \\
& {1\over \eps_B^F! \eps_{A_1}^B!\eps_{A_2}^F!}
\binom{\Labn}{\Labn_B + \Labn_{A_2}}
\binom{\Labn_B + \pi\eps_B^F}{\Labn_{A_1}} \\
&(A_1 \sqcup A_2, \hat F \restr A, \Labn_{A_1}+\Labn_{A_2}+ \pi(\eps_{A_1}^B+\eps_{A_2}^F), \Labo ,\Labe_\emptyset^{A_1\sqcup A_2}) \\
&\otimes (B, (\hat F\restr B)\cup_i A_1, \Labn_B + \pi\eps_B^F - \Labn_{A_1}, \Labo + \Labn_{A_1} + \pi(\eps_{A_1}^B - \Labe_\emptyset^{A_1}), \Labe_{A_1}^B + \eps_{A_1}^B)\\
&\otimes (F, (\hat F \cup_j B) \cup_i A_2, \Labn - \Labn_B - \Labn_{A_2}, \Labo +\Labn_B + \Labn_{A_2}
+ \pi(\eps_B^F + \eps_{A_2}^F - \Labe_\emptyset^{A_2}),\\
&\qquad \Labe_{A_2}^F + \eps_B^F + \eps_{A_2}^F)\;.
\end{equs}
Comparing this with \eqref{e:finalPrefact} we obtain the desired result.
\end{proof}

\begin{remark}\label{cointera}
Let $0<i < j$. If Assumptions~\ref{ass:1}--\ref{ass:6} hold, then the space $\scal{\Tra_{j}}$ is a 
comodule bialgebra over the bialgebra $\scal{\Tra_i}$ with coaction $\Delta_i$, in the sense of \cite[Def~2.1(e)]{Molnar}. 
In the terminology of \cite[Def. 1]{2016arXiv160508310F},
 $\scal{\Tra_{j}}$ and $\scal{\Tra_{i}}$ are in {\it cointeraction}.
\end{remark}

\begin{remark}\label{rem:i<j}
Note that the roles of $i$ and $j$ are asymmetric for $0<i < j$: 
$\scal{\Tra_{i}}$ is in general \textit{not} a
comodule bialgebra over $\scal{\Tra_j}$.
This is a consequence of the asymmetry between the roles played by $i$ and $j$ in Assumption~\ref{ass:1}. In particular, every $A\in\Adm_i(F,\hat F)$ has empty intersection with $\hat F_j$, while any $B\in\Adm_j(F,\hat F)$ can contain connected components of $\hat F_i$.
\end{remark}

\subsection{Skew products and group actions}
\label{sec:skew}
We assume throughout this subsection that $0<i<j$ and that 
Assumptions~\ref{ass:1}--\ref{ass:6} hold. Following \cite{Molnar},
we define a space $\CH_{ij} = \CH_i \ltimes \CH_j$ as follows.
As a vector space, we set $\CH_{ij} = \CH_i \hattimes \CH_j$, and we
endow it with the product and coproduct
\begin{equs}
(a \otimes b)\cdot (\bar a \otimes \bar b)
&= (a\cdot \bar a) \otimes (b\cdot \bar b)\;,\label{deltaij}\\
 \Delta_{ij}(a \otimes b) &= \CM^{(14)(3)(2)(5)}(\id \otimes \id \otimes \id \otimes \Delta_i)(\Delta_i \otimes \Delta_j)(a\otimes b)\;. 
\end{equs}
We also define $\one_{ij}\eqdef\one_i\otimes\one_j$, 
$\one_{ij}^\star\eqdef\one_i^\star\otimes \one_j^\star$. 

\begin{proposition}
The 5-tuple $(\CH_{ij},\cdot,\Delta_{ij},\one_{ij},\one_{ij}^\star)$ is a Hopf algebra.
\end{proposition}
\begin{proof}
We first note that, for every $\tau \in \M_j$, one has $\Delta_i \tau = \one \otimes \tau$ since 
one has $\Adm_i(F,j) = \{\emptyset\}$ by Assumptions~\ref{ass:1} and~\ref{ass:5}.
It follows that one has the identity
\begin{equ}
(\CK_i \otimes \CK_j) \Delta_i = (\CK_i \otimes \CK_j) \Delta_i \CK_j\;,
\end{equ}
see also \eqref{e:CKCK}. Combining this with Lemma~\ref{prop:comodule}, we conclude that one can indeed view $\Delta_i$ as a map $\Delta_i \colon \CH_j \to \CH_i \hattimes \CH_j$,
so that \eqref{deltaij} is well-defined.

By Proposition~\ref{prop:doublecoass}, $\Delta_{ij}$ is coassociative, and it is multiplicative
with respect to the product, see also \cite[Thm~2.14]{Molnar}.
Note also that on $\CH_j$ one has the identity
\begin{equ}
(\id \otimes \one_j^\star)\Delta_i = \one_i \,\one_j^\star\;,
\end{equ}
where $\one_i$ is the unit in $\CH_i$. As a consequence, $\one_{ij}^\star$
is the counit for $\CH_{ij}$, and one can verify that
\begin{equ}
\CA_{ij} = (\CA_i \CM \otimes \CA_j)(\id \otimes \Delta_i)\;,
\end{equ}
is the antipode turning $\CH_{ij}$ into a Hopf algebra. 
\end{proof}
Let us recall that $\CG_i$ denotes the character group of $\CH_i$.
\begin{lemma}
Let us set for $g\in \CG_i$, $f\in \CG_j$, the element $g f\in \CH_j^*$
\[
(g f)\tau \eqdef (g\otimes f)\, \Delta_i \tau, \qquad  \tau\in \CH_j .
\]
Then this defines a left action of $\CG_i$ onto $\CG_j$ by group automorphisms.
\end{lemma}
\begin{proof}
The dualization of the cointeraction property \eqref{e:intertwine} yields that
$g(f_1f_2) = (g f_1)(g f_2)$,
which means that this is indeed an action. 
\end{proof}
\begin{proposition}
The semi-direct product $\CG_{ij}\eqdef \CG_i\ltimes \CG_j$, with group multiplication 
\begin{equ}\label{groupm}
(g_1,f_1)(g_2,f_2)=(g_1g_2, f_1(g_1f_2)), \qquad g_1,g_2\in \CG_i, \ f_1,f_2\in \CG_j,
\end{equ}
defines a sub-group of the group of characters of $\CH_{ij}$.
\end{proposition}
\begin{proof}
Note that \eqref{groupm} is the dualisation of $\Delta_{ij}$ in \eqref{deltaij}.
The inverse is given by
\begin{equ}
(g,f)^{-1} = (g^{-1}, g^{-1}f^{-1})\;,
\end{equ}
since
$(g, f)\cdot (g^{-1}, g^{-1}f^{-1}) = (gg^{-1}, f (gg^{-1}f^{-1})) = (\one_i^\star,\one_j^\star)$.
\end{proof}

\begin{proposition}\label{left-right}
Let $V$ be a vector space such that $\CG_i$ acts on $V$ on the left and $\CG_j$ acts on $V$ on the right, and we assume that
\begin{equ}\label{actduali}
g(h f) = (g h)(g f)\;, \qquad g\in \CG_i, \ f\in \CG_j, \ h\in V.
\end{equ}
Then $\CG_{ij}$ acts on the left on $V$ by
\begin{equ}\label{actsleft}
(g,f)h = (g h)f^{-1}\;, \qquad g\in \CG_i, \ f\in \CG_j, \ h\in V.
\end{equ}
\end{proposition}
\begin{proof}
Now we have
\begin{equs}
(g_1,f_1)&\bigl((g_2,f_2)h\bigr) 
= (g_1,f_1) \bigl((g_2h)f_2^{-1}\bigr) 
= \bigl(g_1\bigl((g_2h)f_2^{-1}\bigr) \bigr)f_1^{-1}
\\ & = \bigl(g_1g_2 h\bigr)\bigl(g_1f_2^{-1}\bigr)f_1^{-1}=(g_1g_2,f_1(g_1f_2))h
= \bigl((g_1,f_1)(g_2,f_2)\bigr) h\;,
\end{equs}
which is exactly what we wanted. 
\end{proof}
For instance, we can choose as $V$ the dual space $\CH_j^*$ of $\CH_j$. For all $h\in\CH_j^*$, $g\in \CG_i$ 
and $f\in \CG_j$ we can set
\[
\CH_j^*\ni g h:=(g\otimes h)\Delta_i, \qquad \CH_j^*\ni hf:=(h\otimes f)\Delta_j.
\]
In this case \eqref{actduali} is the dualisation of the cointeraction property \eqref{e:intertwine}. 
The space $\CH_j$ is a left comodule over $\CH_{ij}$ with coaction 
given by $\beta_{ij} \colon \CH_j \to \CH_{ij} \otimes \CH_j$ with 
\begin{equ}[e:lact]
\beta_{ij}  = \sigma^{(132)}(\Delta_i \otimes \CA_j)\Delta_j\;,
\end{equ} 
where $\sigma^{(132)}(a\otimes b\otimes c)\eqdef a\otimes c\otimes b$. 
Note that \eqref{actsleft} is the dualisation of \eqref{e:lact}. 

\section{A specific setting suitable for renormalisation}
\label{sec4}

We now specialise the framework described in the previous section to the situation 
of interest to us. We define two collections $\Adm_1$ and $\Adm_2$ 
as follows.

\begin{definition}\label{def:admspecific}
For any coloured forest $(F,\hat F)$  as in Definition~\ref{def:colour} we define the collection
$\Adm_1(F,\hat F)$ of all subforests $A$ of $F$
such that $\hat F_1 \subset A$ and $\hat F_2 \cap A=\emptyset$.
We also define $\Adm_2(F,\hat F)$ to consist of all subforests $A$ of $F$
with the following properties:
\begin{enumerate}
\item $A$ contains $\hat F_2$
\item for every non-empty connected component $T$ of $F$, $T\cap A$ is connected and contains the root of $T$
\item for every connected component $S$ of
$\hat F_1$, one has either $S \subset A$ or $S \cap A = \emptyset$.
\end{enumerate}
\end{definition}
The images in Examples~\ref{ex:colours} and~\ref{ex:colours2} above are compatible with these definitions.
We recall from Definition~\ref{def:tra_i} that $\C_i$ and $\Tra_i$ are given for $i=1,2$ by
\begin{equs}
\begin{split}
\C_i \ & \ = \{(F,\hat F) \in \C \,:\, \hat F \le i\;\ \&\;\ \{F,\hat F_i\}\subset\Adm_i(F,\hat F)\}\;,
\\ \Tra_i \ & \ = \{(F,\hat F, \Labn, \Labo, \Labe) \in \Tra \,:\, (F,\hat F)\in\C_i\}\;.
\end{split}
\end{equs}

\begin{lemma}\label{lem:Fspecific}
For $\tau = (F,\hat F)\in\C$ we have
\begin{itemize}
\item $\tau \in \C_1$ if and only if $\hat F \le 1$
\item $\tau \in \C_2$ if and only if $\hat F\leq 2$ and, for every non-empty connected component $T$ of $F$, $\hat F_2\cap T$ is a subtree of $T$ containing the root of $T$. 
\end{itemize}
\end{lemma}
\begin{proof}
Let $(F,\hat F)\in\C$. If $\hat F\leq 1$ then $\hat F_2=\emptyset$ and therefore $F\in\Adm_1(F,\hat F)$; moreover $A=\hat F_1$ clearly satisfies $\hat F_1\subset A$ and $A\cap \hat F_2=\emptyset$, so that $\hat F_1\in\Adm_1(F,\hat F)$ and therefore $(F,\hat F)_\Labe^{\Labn,\Labo}\in\C_1$. The converse is obvious.

Let us suppose now that $\hat F\leq 2$ and for every connected component $T$ of $F$, $\hat F_2\cap T$ is a subtree of $T$ containing the root of $T$. Then $A=F$ clearly satisfies the properties 1-3 of Definition~\ref{def:admspecific}. If now $A=\hat F_2$, then $A$ satisfies the properties 1 and 2 since for every non-empty connected component $T$ of $F$, $\hat F_2\cap T$ is a subtree of $T$ containing the root of $T$, while property 3 is satisfied since $\hat F_1\cap\hat F_2=\emptyset$. The converse is again obvious.
\end{proof}
\begin{example}\label{F1F2}
As in previous examples, red stands for $1$ and blue for $2$ (and black for $0$):
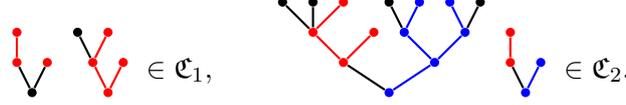
\[
\begin{tikzpicture}[scale=0.2,baseline=0.2cm]
          \node at (-5,4)  [dot,red] (leftll) {};
          \node at (-5,2)  [dot,red] (leftlr) {};
      \node at (-4,0)  [dot] (left) {};
         \node at (-3,2)  [dot,red] (leftr) {};
        
        \draw[kernel1,red] (leftll) to     node [sloped,below] {\small }     (leftlr);
        \draw[kernel1] (leftlr) to     node [sloped,below] {\small }     (left);
     \draw[kernel1] (leftr) to     node [sloped,below] {\small }     (left);  
     \end{tikzpicture} 
\begin{tikzpicture}[scale=0.2,baseline=0.2cm]
          \node at (-5,2)  [dot,red] (leftl) {};
          \node at (-4,4)  [dot,red] (leftlr) {};
          \node at (-6,4)  [dot] (leftlc) {};
      \node at (-4,0)  [dot,color=red] (left) {};
         \node at (-3,2)  [dot,red] (leftr) {};
        
        \draw[kernel1,red] (leftl) to
     node [sloped,below] {\small }     (left);
     \draw[kernel1,red] (leftlr) to
     node [sloped,below] {\small }     (leftl); 
     \draw[kernel1,red] (leftr) to
     node [sloped,below] {\small }     (left);  
      \draw[kernel1] (leftlc) to
     node [sloped,below] {\small }     (leftl); 
     \end{tikzpicture} 
     \in \C_1, \qquad 
\begin{tikzpicture}[scale=0.2,baseline=0.2cm]
        \node at (0,0)  [dot,blue] (root) {};
         \node at (-7,6)  [dot] (leftll) {};
          \node at (-5,4)  [dot,red] (leftl) {};
          \node at (-3,6)  [dot,red] (leftlr) {};
          \node at (-5,6)  [dot] (leftlc) {};
      \node at (-3,2)  [dot,color=red] (left) {};
         \node at (-1,4)  [dot,red] (leftr) {};
         \node at (1,4)  [dot,blue,label=above:   ] (rightl) {};
          \node at (0,6)  [dot ] (rightll) {};
           \node at (2,6)  [dot,blue] (rightlr) {};
           \node at (5,4)  [dot,blue,label=left:] (rightr) {};
            \node at (4,6)  [dot,blue] (rightrl) {};
        \node at (3,2) [dot,blue] (right) {};
         \node at (6,6)  [dot ] (rightrr) {};
        
        \draw[kernel1] (left) to node [sloped,below] {\small } (root); ;
        \draw[kernel1,red] (leftl) to
     node [sloped,below] {\small }     (left);
     \draw[kernel1,red] (leftlr) to
     node [sloped,below] {\small }     (leftl); 
     \draw[kernel1] (leftll) to
     node [sloped,below] {\small }     (leftl);
     \draw[kernel1,red] (leftr) to
     node [sloped,below] {\small }     (left);  
        \draw[kernel1,blue] (right) to
     node [sloped,below] {\small }     (root);
      \draw[kernel1] (leftlc) to
     node [sloped,below] {\small }     (leftl); 
     \draw[kernel1,blue] (rightr) to
     node [sloped,below] {\small }     (right);
     \draw[kernel1] (rightrr) to
     node [sloped,below] {\small }     (rightr);
     \draw[kernel1,blue] (rightrl) to
     node [sloped,below] {\small }     (rightr);
     \draw[kernel1,blue] (rightl) to
     node [sloped,below] {\small }     (right);
     \draw[kernel1,blue] (rightlr) to
     node [sloped,below] {\small }     (rightl);
     \draw[kernel1] (rightll) to
     node [sloped,below] {\small }     (rightl);
     \end{tikzpicture} 
\begin{tikzpicture}[scale=0.2,baseline=0.2cm]
          \node at (-5,4)  [dot,red] (leftll) {};
          \node at (-5,2)  [dot,red] (leftlr) {};
      \node at (-4,0)  [dot,blue] (left) {};
         \node at (-3,2)  [dot,blue] (leftr) {};

        \draw[kernel1,red] (leftll) to     node [sloped,below] {\small }     (leftlr);
        \draw[kernel1] (leftlr) to     node [sloped,below] {\small }     (left);
     \draw[kernel1,blue] (leftr) to     node [sloped,below] {\small }     (left);  
     \end{tikzpicture} 
     \in\C_2.
\]
On the other hand, 
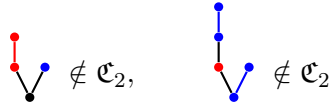
\[
\begin{tikzpicture}[scale=0.2,baseline=0.2cm]
          \node at (-5,4)  [dot,red] (leftll) {};
          \node at (-5,2)  [dot,red] (leftlr) {};
       \node at (-4,0)  [dot] (left) {};
         \node at (-3,2)  [dot,blue] (leftr) {};

        \draw[kernel1,red] (leftll) to     node [sloped,below] {\small }     (leftlr);
        \draw[kernel1] (leftlr) to     node [sloped,below] {\small }     (left);
     \draw[kernel1] (leftr) to     node [sloped,below] {\small }     (left);  
     \end{tikzpicture} 
     \notin\C_2, \qquad
\begin{tikzpicture}[scale=0.2,baseline=0.2cm]
          \node at (-5,6)  [dot,blue] (q) {};
         \node at (-5,4)  [dot,blue] (leftll) {};
          \node at (-5,2)  [dot,red] (leftlr) {};
      \node at (-4,0)  [dot,blue] (left) {};
         \node at (-3,2)  [dot,blue] (leftr) {};

        \draw[kernel1,blue] (leftll) to     node [sloped,below] {\small }     (q);
        \draw[kernel1] (leftll) to     node [sloped,below] {\small }     (leftlr);
        \draw[kernel1] (leftlr) to     node [sloped,below] {\small }     (left);
     \draw[kernel1,blue] (leftr) to     node [sloped,below] {\small }     (left);  
     \end{tikzpicture} 
     \notin\C_2
\]
because $\hat F_2$ does not contain the root in the first case, and in the second
$\hat F_2$ has two disjoint connected components inside a connected component of $F$. The 
decorated forests \eqref{ex:tra1_1}, \eqref{exunits_i}, \eqref{exM_i} and \eqref{exTra1_2} 
are in $\C_1$, while the decorated forests in \eqref{ex:tra2_1}, \eqref{ex:tra2_2}, 
\eqref{ex:tra2_3}, \eqref{excontr1} and \eqref{excontr2} are in $\C_2$.
\end{example}
\begin{lemma}
Let $\Adm_1$ and $\Adm_2$ be given by Definition~\ref{def:admspecific}.
\begin{itemize}
\item $\Adm_1$ satisfies Assumptions~\ref{ass:1}, \ref{ass:2}, \ref{ass:3} and~\ref{ass:4}.
\item $\Adm_2$ satisfies Assumptions~\ref{ass:1}, \ref{ass:2}, \ref{ass:3} and~\ref{ass:4}.
\item The pair $(\Adm_1,\Adm_2)$ satisfies Assumptions~\ref{ass:5} and~\ref{ass:6}.
\end{itemize}
\end{lemma}
\begin{proof}
The first statement concerning $\Adm_1$ is elementary. The only non-trivial property to be checked about $\Adm_2$ is \eqref{e:prop}; note that $\Adm_2$ has the stronger property that for any two subtrees $B \subset A \subset F$,
one has $A \in \Adm_2(F,\hat F)$ if and only if  $A \in \Adm_2(F,\hat F \cup_2 B)$
and $B \in \Adm_2(F,\hat F)$ if and only if $B \in \Adm_2(A,\hat F \restr A)$,
so that property \eqref{e:prop} follows at once. 

Assumption~\ref{ass:5} is easily seen to hold, since for every coloured forest $(F,\hat F)$ such that $\hat F \le 2$ and $\{F,\hat F_2\} \subset \Adm_2(F,\hat F)$, for $A\eqdef\hat F_1$ one has $\hat F_1\subset A$ and $\hat F_2\cap A=\emptyset$, so that $\hat F_1 \in \Adm_1(F,\hat F)$.

We check now that $\Adm_1$ and $\Adm_2$ satisfy Assumption~\ref{ass:6}. Let $A \in \Adm_1(F,\hat F)$ and $B \in \Adm_2(F, \hat F \cup_1 A)$; then $A\cap \hat F_2=\emptyset$ and therefore $B \in \Adm_2(F,\hat F)$; moreover every connected component of $A$ is contained in a connected component of $\hat F_1$ and therefore is either contained in $B$ or disjoint from $B$, i.e. $A \in \Adm_1(F, \hat F \cup_2 B) \sqcup \Adm_1(B, \hat F \restr B)$. Conversely, let $B \in \Adm_2(F,\hat F)$ and $A \in \Adm_1(F, \hat F \cup_2 B) \sqcup \Adm_1(B, \hat F \restr B)$; then $\hat F_1=(\hat F \cup_2 B)_1\sqcup (\hat F \restr B)_1$ and $\hat F_2\subset(\hat F \cup_2 B)_2$ so that $A$ contains $\hat F_1$ and is disjoint from $\hat F_2$ and therefore $A\in\Adm_1(F,\hat F)$; moreover $(\hat F \cup_1 A)_2\subseteq \hat F_2$ so that $B$ contains $(\hat F \cup_1 A)_2$; finally $(\hat F \cup_1 A)_1=A$ and by the assumption on $A$ we have that every connected component of  $(\hat F \cup_1 A)_1$ is either contained in $B$ or disjoint from $B$. The proof is complete.
\end{proof}

In view of Propositions~\ref{prop:combi}, \ref{prop:Hopf} and~\ref{prop:doublecoass}, we have the following result.

\begin{corollary}\label{lem:preCEFM}
Denoting by $\CM$ the forest product, we have:
\begin{enumerate}
\item The space $(\BB_2,\CM,\Delta_2,\one_2,\one^\star_2)$ is a Hopf algebra and a comodule bialgebra over the Hopf algebra $(\BB_1,\CM,\Delta_1,\one_1,\one^\star_1)$ with coaction $\Delta_1$ and counit $\one^\star_1$.
\item The space $(\CH_2,\CM,\Delta_2,\one_2,\one^\star_2)$ is a Hopf algebra and a comodule bialgebra over the Hopf algebra $(\CH_1,\cdot,\Delta_1,\one_1,\one^\star_1)$ with coaction $\Delta_1$ and counit $\one^\star_1$.
\end{enumerate}
\end{corollary}
We note that $\BB_1$ can be canonically identified with $\Vec(C_1)$, where $C_1=\CK_1\C_1$, see the definition of 
$\CK_i$ before Proposition~\ref{prop:combi}, and $C_1$ is the set of (possibly empty) coloured forests $(F,\hat F)$ such that $\hat F\leq 1$ 
and $\hat F_1$ is a collection of isolated nodes, namely $E_1=\emptyset$. For instance
\begin{equs}
\begin{tikzpicture}[scale=0.2,baseline=0.2cm]
          \node at (-5,4)  [dot,red] (leftll) {};
          \node at (-5,2)  [dot,red] (leftlr) {};
      \node at (-4,0)  [dot] (left) {};
         \node at (-3,2)  [dot,red] (leftr) {};
        
        \draw[kernel1] (leftll) to     node [sloped,below] {\small }     (leftlr);
        \draw[kernel1] (leftlr) to     node [sloped,below] {\small }     (left);
     \draw[kernel1] (leftr) to     node [sloped,below] {\small }     (left);  
     \end{tikzpicture} 
\begin{tikzpicture}[scale=0.2,baseline=0.2cm]
          \node at (-5,2)  [dot,red] (leftl) {};
          \node at (-4,4)  [dot,red] (leftlr) {};
          \node at (-6,4)  [dot] (leftlc) {};
      \node at (-4,0)  [dot,color=red] (left) {};
         \node at (-3,2)  [dot,red] (leftr) {};
        
        \draw[kernel1] (leftl) to
     node [sloped,below] {\small }     (left);
     \draw[kernel1] (leftlr) to
     node [sloped,below] {\small }     (leftl); 
     \draw[kernel1] (leftr) to
     node [sloped,below] {\small }     (left);  
      \draw[kernel1] (leftlc) to
     node [sloped,below] {\small }     (leftl); 
     \end{tikzpicture} 
 \in C_1, \qquad
   \begin{tikzpicture}[scale=0.2,baseline=0.2cm]
          \node at (-5,4)  [dot,red] (leftll) {};
          \node at (-5,2)  [dot,red] (leftlr) {};
      \node at (-4,0)  [dot] (left) {};
         \node at (-3,2)  [dot,red] (leftr) {};
        
        \draw[kernel1,red] (leftll) to     node [sloped,below] {\small }     (leftlr);
        \draw[kernel1] (leftlr) to     node [sloped,below] {\small }     (left);
     \draw[kernel1] (leftr) to     node [sloped,below] {\small }     (left);  
     \end{tikzpicture} 
\begin{tikzpicture}[scale=0.2,baseline=0.2cm]
          \node at (-5,2)  [dot,red] (leftl) {};
          \node at (-4,4)  [dot,red] (leftlr) {};
          \node at (-6,4)  [dot] (leftlc) {};
      \node at (-4,0)  [dot,color=red] (left) {};
         \node at (-3,2)  [dot,red] (leftr) {};
        
        \draw[kernel1,red] (leftl) to
     node [sloped,below] {\small }     (left);
     \draw[kernel1,red] (leftlr) to
     node [sloped,below] {\small }     (leftl); 
     \draw[kernel1,red] (leftr) to
     node [sloped,below] {\small }     (left);  
      \draw[kernel1] (leftlc) to
     node [sloped,below] {\small }     (leftl); 
     \end{tikzpicture} 
 \notin C_1.
  \end{equs}
Analogously, $\BB_2$ can be canonically identified with $\Vec(C_2)$, where $C_2=\CK_2\C_2$, and $C_2$ is the set of 
non-empty coloured forests $(F,\hat F)$ such that $\hat F\leq 2$,  $\hat F_1$ is a collection of isolated nodes, 
namely $E_1=\emptyset$, and $\hat F_2$ coincides with the set of roots of $F$. For instance
\begin{equs}
\begin{tikzpicture}[scale=0.2,baseline=0.2cm]
          \node at (-5,4)  [dot,red] (leftll) {};
          \node at (-5,2)  [dot,red] (leftlr) {};
      \node at (-4,0)  [dot,blue] (left) {};
         \node at (-3,2)  [dot,red] (leftr) {};
        
        \draw[kernel1] (leftll) to     node [sloped,below] {\small }     (leftlr);
        \draw[kernel1] (leftlr) to     node [sloped,below] {\small }     (left);
     \draw[kernel1] (leftr) to     node [sloped,below] {\small }     (left);  
     \end{tikzpicture} 
\begin{tikzpicture}[scale=0.2,baseline=0.2cm]
          \node at (-5,2)  [dot,red] (leftl) {};
          \node at (-4,4)  [dot,red] (leftlr) {};
          \node at (-6,4)  [dot] (leftlc) {};
      \node at (-4,0)  [dot,color=blue] (left) {};
         \node at (-3,2)  [dot,red] (leftr) {};
        
        \draw[kernel1] (leftl) to
     node [sloped,below] {\small }     (left);
     \draw[kernel1] (leftlr) to
     node [sloped,below] {\small }     (leftl); 
     \draw[kernel1] (leftr) to
     node [sloped,below] {\small }     (left);  
      \draw[kernel1] (leftlc) to
     node [sloped,below] {\small }     (leftl); 
     \end{tikzpicture} 
 \in C_2, \qquad
   \begin{tikzpicture}[scale=0.2,baseline=0.2cm]
          \node at (-5,4)  [dot,red] (leftll) {};
          \node at (-5,2)  [dot,red] (leftlr) {};
      \node at (-4,0)  [dot] (left) {};
         \node at (-3,2)  [dot,red] (leftr) {};
        
        \draw[kernel1,red] (leftll) to     node [sloped,below] {\small }     (leftlr);
        \draw[kernel1] (leftlr) to     node [sloped,below] {\small }     (left);
     \draw[kernel1] (leftr) to     node [sloped,below] {\small }     (left);  
     \end{tikzpicture} 
\begin{tikzpicture}[scale=0.2,baseline=0.2cm]
          \node at (-5,2)  [dot,blue] (leftl) {};
          \node at (-4,4)  [dot,blue] (leftlr) {};
          \node at (-6,4)  [dot] (leftlc) {};
      \node at (-4,0)  [dot,color=blue] (left) {};
         \node at (-3,2)  [dot,blue] (leftr) {};
        
        \draw[kernel1,blue] (leftl) to
     node [sloped,below] {\small }     (left);
     \draw[kernel1,blue] (leftlr) to
     node [sloped,below] {\small }     (leftl); 
     \draw[kernel1,blue] (leftr) to
     node [sloped,below] {\small }     (left);  
      \draw[kernel1] (leftlc) to
     node [sloped,below] {\small }     (leftl); 
     \end{tikzpicture} 
 \notin C_2.
  \end{equs}
  
The action of $\Delta_1$ on $\BB_i$, $i=1,2$, can be described on $\Vec(C_i)$ as the action of $(\CK_1\otimes\CK_i)\Delta_1$, namely: on a coloured forest $(F,\hat F)\in C_i$, one chooses a subforest $B$ of $F$ which contains $\hat F_1$
and is disjoint from $\hat F_2$, which is empty if $i=1$ and equal to the set of roots of $F$ if $i=2$; then one has
$(B,\hat F\restr B)\in C_1$ and $\CK_i(F,F\cup_1 B)\in C_i$. Summing over all possible $B$ of this form, we find
\[
(\CK_1\otimes\CK_i)\Delta_1 (F,\hat F)=\sum_B \CK_1(B,\hat F\restr B)\otimes \CK_i(F,F\cup_1 B) \in \BB_1\otimes\BB_i.
\]
This describes the coproduct of $\BB_1$ if $i=1$ and the coaction on $\BB_2$ if $i=2$. In both cases, we have a 
contraction/extraction operator of subforests: indeed, in $(B,\hat F\restr B)$ we have the extracted subforest $B$, with
colouring inherited from $\hat F$, while in $\CK_i(F,F\cup_1 B) $ we have extended the red colour to $B$ and then
contracted $B$ to a family of red single nodes. For instance, using Example~\ref{ex:colours2}
\[
(\CK_1\otimes\CK_2)\Delta_1 \begin{tikzpicture}[scale=0.2,baseline=0.2cm]
        \node at (0,0)  [dot,blue,label={[label distance=-0.2em]left: \scriptsize  $ a $} ] (root) {};
          \node at (-2,2)  [dot,red,,label={[label distance=-0.2em]left: \scriptsize  $ b $} ] (leftl) {};
         \node at (-3,4)  [dot,label={[label distance=-0.2em]above: \scriptsize  $ h $} ] (leftll) {};
          \node at (-1,4)  [dot,label={[label distance=-0.2em]above: \scriptsize  $ i $}] (leftlc) {};
           \node at (1,2)  [dot,red,label={[label distance=-0.2em]above: \scriptsize  $ k $} ] (rightll) {};
         \node at (3,2)  [dot,label={[label distance=-0.2em]above: \scriptsize  $ p $} ] (rightrr) {};
        
       \draw[kernel1] (leftl) to node [sloped,below] {\small } (root); ;
        \draw[kernel1] (leftll) to node [sloped,below] {\small } (leftl); ;
        \draw[kernel1] (leftlc) to
     node [sloped,below] {\small }     (leftl);
     \draw[kernel1] (rightll) to
     node [sloped,below] {\small }     (root); 
     \draw[kernel1] (rightrr) to
     node [sloped,below] {\small }     (root);
     \end{tikzpicture}   
=\ldots + 
\begin{tikzpicture}[scale=0.2,baseline=0.2cm]
           \node at (-2,2)  [dot,label={[label distance=-0.2em]above: \scriptsize  $ i $}] (leftlc) {};
      \node at (-2,0)  [dot,color=red] (left) {};
         \node at (0,0)  [dot,label={[label distance=-0.2em]above: \scriptsize  $ p $}] (o) {};
        
        \draw[kernel1] (leftlc) to
     node [sloped,below] {\small }     (left);
     \end{tikzpicture} 
      \otimes   
\begin{tikzpicture}[scale=0.2,baseline=0.2cm]
        \node at (0,0)  [dot,blue] (root) {};
          \node at (-2,2)  [dot,red,label=left:] (leftl) {};
         \node at (-2,4)  [dot,label={[label distance=-0.2em]above: \scriptsize  $ h $} ] (leftll) {};
             \node at (0,2)  [dot,red,label={[label distance=-0.2em]above: \scriptsize  $ k $} ] (rightll) {};
         \node at (2,2)  [dot,red,label={[label distance=-0.2em]above: \scriptsize  $ p $} ] (rightrr) {};
        
       \draw[kernel1] (leftl) to node [sloped,below] {\small } (root); ;
        \draw[kernel1] (leftll) to node [sloped,below] {\small } (leftl); ;
      \draw[kernel1] (rightll) to
     node [sloped,below] {\small }     (root); 
     \draw[kernel1] (rightrr) to
     node [sloped,below] {\small }     (root);
     \end{tikzpicture}
+ \ldots
\]
since by \eqref{e:cki} the red node labelled $k$ on the left side of the tensor product is killed by $\CK_1$.

The action of $\Delta_2$ on $\BB_2$ can be described on $\Vec(C_2)$ as the action of $(\CK_2\otimes\CK_2)\Delta_2$, namely: on a coloured tree $(F,\hat F)\in C_2$, one chooses a subtree $A$ of $F$ which contains the root of
$F$; then one has
$(A,\hat F\restr A)\in C_2$ and $\CK_2(F,F\cup_2 A)\in C_2$. Summing over all possible $A$ of this form, we find
\[
(\CK_2\otimes\CK_2)\Delta_2 (F,\hat F)=\sum_A \CK_2(A,\hat F\restr A)\otimes \CK_2(F,F\cup_2 A) \in \BB_2\otimes\BB_2.
\]
If $(F,\hat F)=\tau\in C_2$ is a coloured forest, one decomposes $\tau$ in connected components, and then uses
the above description and the multiplicativity of the coproduct. 
This describes the coproduct of $\BB_2$ as a contraction/extraction operator of rooted subtrees.
For instance, using Example~\ref{ex:colours2}
\[
(\CK_2\otimes\CK_2)\Delta_2 \begin{tikzpicture}[scale=0.2,baseline=0.2cm]
        \node at (0,0)  [dot,blue,label={[label distance=-0.2em]left: \scriptsize  $ a $} ] (root) {};
          \node at (-2,2)  [dot,red,,label={[label distance=-0.2em]left: \scriptsize  $ b $} ] (leftl) {};
         \node at (-3,4)  [dot,label={[label distance=-0.2em]above: \scriptsize  $ h $} ] (leftll) {};
          \node at (-1,4)  [dot,label={[label distance=-0.2em]above: \scriptsize  $ i $}] (leftlc) {};
           \node at (1,2)  [dot,red,label={[label distance=-0.2em]above: \scriptsize  $ k $} ] (rightll) {};
         \node at (3,2)  [dot,label={[label distance=-0.2em]above: \scriptsize  $ p $} ] (rightrr) {};
        
       \draw[kernel1] (leftl) to node [sloped,below] {\small } (root); ;
        \draw[kernel1] (leftll) to node [sloped,below] {\small } (leftl); ;
        \draw[kernel1] (leftlc) to
     node [sloped,below] {\small }     (leftl);
     \draw[kernel1] (rightll) to
     node [sloped,below] {\small }     (root); 
     \draw[kernel1] (rightrr) to
     node [sloped,below] {\small }     (root);
     \end{tikzpicture}   
=\ldots + 
\begin{tikzpicture}[scale=0.2,baseline=0.1cm]
        \node at (-1,0)  [dot,blue,label={[label distance=-0.2em]left: \scriptsize  $ a $} ] (root) {};
          \node at (-1,2)  [dot,red,label={[label distance=-0.2em]left: \scriptsize  $ b $}] (left) {};
        
        \draw[kernel1] (left) to node [sloped,below] {\small } (root); 
     \end{tikzpicture} \
       \otimes  
   \begin{tikzpicture}[scale=0.2,baseline=0.1cm]
        \node at (0,0)  [dot,blue] (root) {};
          \node at (-3,2)  [dot,label={[label distance=-0.2em]above: \scriptsize  $ h $} ] (leftll) {};
          \node at (-1,2)  [dot,label={[label distance=-0.2em]above: \scriptsize  $ i $}] (leftlc) {};
           \node at (1,2)  [dot,red,label={[label distance=-0.2em]above: \scriptsize  $ k $} ] (rightll) {};
         \node at (3,2)  [dot,label={[label distance=-0.2em]above: \scriptsize  $ p $} ] (rightrr) {};
        
        \draw[kernel1] (leftll) to node [sloped,below] {\small } (root); ;
        \draw[kernel1] (leftlc) to
     node [sloped,below] {\small }     (root);
     \draw[kernel1] (rightll) to
     node [sloped,below] {\small }     (root); 
     \draw[kernel1] (rightrr) to
     node [sloped,below] {\small }     (root);
     \end{tikzpicture}   
+ \ldots
\]
The operators $\{\Delta_1,\Delta_2\}$ on the spaces $\{\CH_1,\CH_2\}$ act in the same way on the coloured subforests, 
and add the action on the decorations.

\subsection{Joining roots}
\label{sec:join}

While the product given by ``disjoint unions'' considered so far is very natural when
considering forests, it is much less natural when considering spaces of trees. There,
the more natural thing to do is to join trees together by their roots.  
Given a typed forest $F$, we then define the typed tree $\JJ(F)$ by joining all the 
roots of $F$ together. In other words, we set $\JJ(F) = F / \sim$, where 
$\sim$ is the equivalence relation on nodes in $N_F$ given by 
$x \sim y$ if and only if either $x=y$ or both $x$ and $y$ belong to the set $\rho_F$ of nodes of $F$. For example
\[
F=
\begin{tikzpicture}[scale=0.2,baseline=0.2cm]
      \node at (-7,0)  [dot] (1t) {};
      \node at (-7,2)  [dot] (2t) {};
        \draw[kernel1] (1t) to     node [sloped,below] {\small }     (2t);
          \node at (-5,4)  [dot] (leftll) {};
          \node at (-5,2)  [dot] (leftlr) {};
      \node at (-4,0)  [dot] (left) {};
         \node at (-3,2)  [dot] (leftr) {};

        \draw[kernel1] (leftll) to     node [sloped,below] {\small }     (leftlr);
        \draw[kernel1] (leftlr) to     node [sloped,below] {\small }     (left);
     \draw[kernel1] (leftr) to     node [sloped,below] {\small }     (left);  
\end{tikzpicture} \ 
\Longrightarrow \ \JJ(F)=
\begin{tikzpicture}[scale=0.2,baseline=0.2cm]
 %     \node at (-5,0)  [dot] (1t) {};
      \node at (-5,2)  [dot] (2t) {};
          \node at (-3,4)  [dot] (leftll) {};
          \node at (-3,2)  [dot] (leftlr) {};
      \node at (-3,0)  [dot] (left) {};
         \node at (-1,2)  [dot] (leftr) {};

        \draw[kernel1] (left) to     node [sloped,below] {\small }     (2t);
        \draw[kernel1] (leftll) to     node [sloped,below] {\small }     (leftlr);
        \draw[kernel1] (leftlr) to     node [sloped,below] {\small }     (left);
     \draw[kernel1] (leftr) to     node [sloped,below] {\small }     (left);  
\end{tikzpicture} 
\]
When considering coloured or decorated trees as we do here, such an operation cannot in general be performed 
unambiguously since different trees may have roots of different colours. For example, if
\[
(F,\hat F)=
\begin{tikzpicture}[scale=0.2,baseline=0.2cm]
      \node at (-7,0)  [dot,blue] (1t) {};
      \node at (-7,2)  [dot,blue] (2t) {};
        \draw[kernel1,blue] (1t) to     node [sloped,below] {\small }     (2t);
          \node at (-5,4)  [dot,red] (leftll) {};
          \node at (-5,2)  [dot,red] (leftlr) {};
      \node at (-4,0)  [dot,red] (left) {};
         \node at (-3,2)  [dot,red] (leftr) {};

        \draw[kernel1,red] (leftll) to     node [sloped,below] {\small }     (leftlr);
        \draw[kernel1,red] (leftlr) to     node [sloped,below] {\small }     (left);
     \draw[kernel1,red] (leftr) to     node [sloped,below] {\small }     (left);  
\end{tikzpicture} 
\] 
then we do not know how to define a colouring of $\JJ(F)$ which is compatible
with $\hat F$. This justifies the
definition of the subset $\CD_i(\JJ) \subset \Tra$\label{defDJJ} as
the set of all forests $(F,\hat F,\Labn,\Labo,\Labe)$ such that 
$\hat F(\rho) \in \{0,i\}$ for every root $\rho$ of $F$.
We also write $\CD(\JJ) = \bigcup_{i \ge 0} \CD_i(\JJ)$ and $\hat\CD_i(\JJ) \subset \CD_i(\JJ)$\label{defDJJhat}
for the set of forests such that \textit{every} root has colour $i$.
\begin{example} Using as usual red for $1$ and blue for $2$, we have
\[
\begin{tikzpicture}[scale=0.2,baseline=0.2cm]
      \node at (-7,0)  [dot] (1t) {};
      \node at (-7,2)  [dot] (2t) {};
        \draw[kernel1] (1t) to     node [sloped,below] {\small }     (2t);
          \node at (-5,4)  [dot,blue] (leftll) {};
          \node at (-5,2)  [dot,blue] (leftlr) {};
      \node at (-4,0)  [dot,red] (left) {};
         \node at (-3,2)  [dot,red] (leftr) {};

        \draw[kernel1,blue] (leftll) to     node [sloped,below] {\small }     (leftlr);
        \draw[kernel1] (leftlr) to     node [sloped,below] {\small }     (left);
     \draw[kernel1] (leftr) to     node [sloped,below] {\small }     (left);  
\end{tikzpicture} 
     \in\CD_1(\JJ), \qquad 
\begin{tikzpicture}[scale=0.2,baseline=0.2cm]
      \node at (-7,0)  [dot,red] (1t) {};
      \node at (-7,2)  [dot] (2t) {};
        \draw[kernel1] (1t) to     node [sloped,below] {\small }     (2t);
          \node at (-5,4)  [dot,blue] (leftll) {};
          \node at (-5,2)  [dot,blue] (leftlr) {};
      \node at (-4,0)  [dot,red] (left) {};
         \node at (-3,2)  [dot,red] (leftr) {};

        \draw[kernel1,blue] (leftll) to     node [sloped,below] {\small }     (leftlr);
        \draw[kernel1] (leftlr) to     node [sloped,below] {\small }     (left);
     \draw[kernel1] (leftr) to     node [sloped,below] {\small }     (left);  
\end{tikzpicture} 
     \in\hat\CD_1(\JJ), \qquad 
\begin{tikzpicture}[scale=0.2,baseline=0.2cm]
      \node at (-7,0)  [dot,blue] (1t) {};
      \node at (-7,2)  [dot] (2t) {};
        \draw[kernel1] (1t) to     node [sloped,below] {\small }     (2t);
          \node at (-5,4)  [dot,red] (leftll) {};
          \node at (-5,2)  [dot,red] (leftlr) {};
      \node at (-4,0)  [dot,blue] (left) {};
         \node at (-3,2)  [dot,blue] (leftr) {};

        \draw[kernel1,red] (leftll) to     node [sloped,below] {\small }     (leftlr);
        \draw[kernel1] (leftlr) to     node [sloped,below] {\small }     (left);
     \draw[kernel1,blue] (leftr) to     node [sloped,below] {\small }     (left);  
\end{tikzpicture} 
     \in\hat\CD_2(\JJ).
\]
\end{example}
We can then extend $\JJ$ to $\CD(\JJ)$ in a natural way as follows.

\begin{definition}\label{def:productTrees}
For $\tau = (F,\hat F,\Labn,\Labo,\Labe) \in \CD(\JJ)$, we define the 
decorated tree $\JJ(\tau)\in\Tra$ by
\begin{equ}
\JJ(\tau) = (\JJ(F),[\hat F],[\Labn],[\Labo],\Labe)\;,
\end{equ}
where $[\Labn](x) = \sum_{y \in x} \Labn(y)$, $[\Labo](x) = \sum_{y \in x} \Labo(y)$,
and $[\hat F](x) = \sup_{y \in x} \hat F(y)$.
\end{definition}

\begin{example} The following coloured forests belong to $\CD_2(\JJ)$
\[
 \tau_1=  
\begin{tikzpicture}[scale=0.2,baseline=0.2cm]
          \node at (0,4)  [dot,red] (leftll) {};
          \node at (0,2)  [dot,red] (leftlr) {};
       \node at (0,0)  [dot] (left) {};
 
        \draw[kernel1,red] (leftll) to     node [sloped,below] {\small }     (leftlr);
        \draw[kernel1] (leftlr) to     node [sloped,below] {\small }     (left);
     \end{tikzpicture} 
\qquad \tau_2=   
\begin{tikzpicture}[scale=0.2,baseline=0.2cm]
          \node at (-5,4)  [dot] (leftll) {};
          \node at (-5,2)  [dot,blue] (leftlr) {};
      \node at (-4,0)  [dot,blue] (left) {};
         \node at (-3,2)  [dot,blue] (leftr) {};

        \draw[kernel1] (leftll) to     node [sloped,below] {\small }     (leftlr);
        \draw[kernel1,blue] (leftlr) to     node [sloped,below] {\small }     (left);
     \draw[kernel1,blue] (leftr) to     node [sloped,below] {\small }     (left);  
     \end{tikzpicture} 
     \qquad \tau_1\cdot \tau_2=
\begin{tikzpicture}[scale=0.2,baseline=0.2cm]
          \node at (0,4)  [dot,red] (leftll) {};
          \node at (0,2)  [dot,red] (leftlr) {};
       \node at (0,0)  [dot] (left) {};
 
        \draw[kernel1,red] (leftll) to     node [sloped,below] {\small }     (leftlr);
        \draw[kernel1] (leftlr) to     node [sloped,below] {\small }     (left);
     \end{tikzpicture} 
\begin{tikzpicture}[scale=0.2,baseline=0.2cm]
          \node at (-5,4)  [dot] (leftll) {};
          \node at (-5,2)  [dot,blue] (leftlr) {};
      \node at (-4,0)  [dot,blue] (left) {};
         \node at (-3,2)  [dot,blue] (leftr) {};

        \draw[kernel1] (leftll) to     node [sloped,below] {\small }     (leftlr);
        \draw[kernel1,blue] (leftlr) to     node [sloped,below] {\small }     (left);
     \draw[kernel1,blue] (leftr) to     node [sloped,below] {\small }     (left);  
     \end{tikzpicture} 
\qquad \JJ(\tau_1\cdot \tau_2) = 
\begin{tikzpicture}[scale=0.2,baseline=0.2cm]
          \node at (-2,4)  [dot,red] (21) {};
          \node at (-2,2)  [dot,red] (11) {};
       \node at (0,0)  [dot,blue] (root) {};
 
        \draw[kernel1,red] (21) to     node [sloped,below] {\small }     (11);
        \draw[kernel1] (11) to     node [sloped,below] {\small }     (root);
          \node at (0,4)  [dot] (22) {};
          \node at (0,2)  [dot,blue] (12) {};
         \node at (2,2)  [dot,blue] (13) {};

        \draw[kernel1] (22) to     node [sloped,below] {\small }     (12);
        \draw[kernel1,blue] (12) to     node [sloped,below] {\small }     (root);
     \draw[kernel1,blue] (13) to     node [sloped,below] {\small }     (root);  
     \end{tikzpicture} 
\]
The following coloured forests belong to $\hat\CD_2(\JJ)$
\[
 \tau_1=  
\begin{tikzpicture}[scale=0.2,baseline=0.2cm]
          \node at (0,4)  [dot,red] (leftll) {};
          \node at (0,2)  [dot,red] (leftlr) {};
       \node at (0,0)  [dot,blue] (left) {};
 
        \draw[kernel1,red] (leftll) to     node [sloped,below] {\small }     (leftlr);
        \draw[kernel1] (leftlr) to     node [sloped,below] {\small }     (left);
     \end{tikzpicture} 
\qquad \tau_2=   
\begin{tikzpicture}[scale=0.2,baseline=0.2cm]
          \node at (-5,4)  [dot] (leftll) {};
          \node at (-5,2)  [dot,blue] (leftlr) {};
      \node at (-4,0)  [dot,blue] (left) {};
         \node at (-3,2)  [dot,blue] (leftr) {};

        \draw[kernel1] (leftll) to     node [sloped,below] {\small }     (leftlr);
        \draw[kernel1,blue] (leftlr) to     node [sloped,below] {\small }     (left);
     \draw[kernel1,blue] (leftr) to     node [sloped,below] {\small }     (left);  
     \end{tikzpicture} 
     \qquad \tau_1\cdot \tau_2=
\begin{tikzpicture}[scale=0.2,baseline=0.2cm]
          \node at (0,4)  [dot,red] (leftll) {};
          \node at (0,2)  [dot,red] (leftlr) {};
       \node at (0,0)  [dot,blue] (left) {};
 
        \draw[kernel1,red] (leftll) to     node [sloped,below] {\small }     (leftlr);
        \draw[kernel1] (leftlr) to     node [sloped,below] {\small }     (left);
     \end{tikzpicture} 
\begin{tikzpicture}[scale=0.2,baseline=0.2cm]
          \node at (-5,4)  [dot] (leftll) {};
          \node at (-5,2)  [dot,blue] (leftlr) {};
      \node at (-4,0)  [dot,blue] (left) {};
         \node at (-3,2)  [dot,blue] (leftr) {};

        \draw[kernel1] (leftll) to     node [sloped,below] {\small }     (leftlr);
        \draw[kernel1,blue] (leftlr) to     node [sloped,below] {\small }     (left);
     \draw[kernel1,blue] (leftr) to     node [sloped,below] {\small }     (left);  
     \end{tikzpicture} 
\qquad \JJ(\tau_1\cdot \tau_2) = 
\begin{tikzpicture}[scale=0.2,baseline=0.2cm]
          \node at (-2,4)  [dot,red] (21) {};
          \node at (-2,2)  [dot,red] (11) {};
       \node at (0,0)  [dot,blue] (root) {};
 
        \draw[kernel1,red] (21) to     node [sloped,below] {\small }     (11);
        \draw[kernel1] (11) to     node [sloped,below] {\small }     (root);
          \node at (0,4)  [dot] (22) {};
          \node at (0,2)  [dot,blue] (12) {};
         \node at (2,2)  [dot,blue] (13) {};

        \draw[kernel1] (22) to     node [sloped,below] {\small }     (12);
        \draw[kernel1,blue] (12) to     node [sloped,below] {\small }     (root);
     \draw[kernel1,blue] (13) to     node [sloped,below] {\small }     (root);  
     \end{tikzpicture} 
\]
\end{example}
It is clear that the $\CD_i$'s are closed under multiplication and that 
one has 
\begin{equ}[e:multJJ]
\JJ(\tau\cdot \bar \tau) = \JJ\bigl(\tau\cdot \JJ(\bar \tau)\bigr)\;,\qquad \tau, \bar \tau \in \CD_i(\JJ)
\end{equ}
for every $i \ge 0$. Furthermore, $\JJ$ is idempotent and preserves our bigrading.
The following fact is also easy to verify, where $\CK$, $\hat\CK_i$, $\Phi_i$, $\hat\Phi_i$ and $\hat P_i$ were defined in Section~\ref{sec:Hopf}.

\begin{lemma}\label{lem:JJ0}
For $i \ge 0$, the sets $\CD_i(\JJ)$ and $\hat\CD_i(\JJ)$ are invariant under $\CK$, $\Phi_i$, $\hat P_i$ and $\JJ$. 
Furthermore, $\JJ$ commutes with both $\CK$ and $\hat P_i$
on $\CD_i(\JJ)$ and satisfies the identity 
\begin{equ}[e:PhiJJ]
\hat\CK_i \JJ = \hat\CK_i \JJ \hat\CK_i\;, \qquad \text{on \ $\hat \CD_i(\JJ)$.}
\end{equ}
In particular $\hat\CK_i \JJ$ is idempotent on $\hat \CD_i(\JJ)$.
\end{lemma}
\begin{proof}
The spaces $\CD_i(\JJ)$ and $\hat\CD_i(\JJ)$ are invariant under 
$\CK$, $\Phi_i$ and $\hat P_i$ because these operations never change the colours
of the roots. The invariance under $\JJ$ follows in a similar way.

The fact that $\JJ$ commutes with $\CK$ is obvious.
The reason why it commutes with $\hat P_i$ is that $\Labo$ 
vanishes on colourless nodes by the definition of $\Tra$.
Regarding \eqref{e:PhiJJ}, since $\hat\CK_i = \hat P_i \Phi_i \CK$,
and all three operators are idempotent and commute with each other, 
we have
\begin{equ}
\hat\CK_i \JJ = \Phi_i \hat P_i \JJ \CK \;,\qquad
\hat\CK_i \JJ \hat\CK_i = \Phi_i \hat P_i \JJ  \Phi_i \CK
\end{equ}
so that it suffices to show that
\begin{equ}[e:idenJJ]
\hat P_i \JJ  \CK = \hat P_i \JJ \Phi_i \CK\;.
\end{equ}
For this, consider an element $\tau \in \hat \CD_i(\JJ)$ and write
$\tau = \mu\cdot \nu$ as in \eqref{Phi}. By the definition of this decomposition and
of $\CK$, there exist $k \ge 0$ and labels $n_j \in \N^d$, $o_j \in \Z^d \oplus \Z(\Lab)$ 
with $j \in \{1,\ldots,k\}$ such that
\begin{equ}
\CK \tau = (\CK \mu)\cdot x_{n_1,o_1}^{(i)}\cdots x_{n_k,o_k}^{(i)}\;,
\end{equ}
where $x_{n,o}^{(i)} = (\bullet,i,n,o,0)$. It follows that 
\begin{equ}[e:PhiK]
\Phi_i \CK \tau = (\CK \mu)\cdot x_{n,0}^{(i)}
\end{equ}
with $n = \sum_{j=1}^k n_j$. On the other hand, by \eqref{e:multJJ}, one has
\begin{equ}
\JJ \CK \tau = \JJ \bigl((\CK \mu)\cdot x_{n,o}^{(i)}\bigr)\;,
\end{equ}
with $o$ defined from the $o_i$ similarly to $n$. Comparing this to \eqref{e:PhiK}, 
it follows that 
$\JJ \CK \tau $ differs from $\JJ \Phi_i \CK \tau $ only by its $\Labo$-decoration 
at the root of one of its connected components in the sense of Remark~\ref{rem:connected}. 
Since these are set to $0$ by $\hat \Phi_i$, \eqref{e:idenJJ} follows.
\end{proof}

Finally, we show that the operation of joining roots is well adapted to the definitions
given in the previous subsection. In particular, we assume from now on that the $\Adm_i$ for $i=1,2$ are given by Definition~\ref{def:admspecific}. Our definitions guarantee that 
\begin{itemize}
\item $\Tra_1 \subset \CD_1(\JJ)$ 
\item $\Tra_2 \subset \hat\CD_2(\JJ)$.
\end{itemize}
We then have the following, where $\JJ$ is extended to the relevant spaces as a triangular map.

\begin{proposition}\label{prop:propJJ}
One has the identities
\begin{equs}[2]
\Delta_2 \JJ &= 
(\JJ \otimes \JJ)\Delta_2 =
(\JJ \otimes \JJ)\Delta_2 \JJ \;,&\quad &\text{on \quad $\CD(\JJ)$,}\\
\Delta_1 \JJ &= (\id \otimes \JJ)\Delta_1= (\id \otimes \JJ)\Delta_1\JJ\;,&\quad &\text{on \quad $\Tra_2$.}
\end{equs}
\end{proposition}

\begin{proof}
Extend $\JJ$ to coloured trees by $\JJ(F,\hat F) = (\JJ(F),[\hat F])$ with $[\hat F]$
as in Definition~\ref{def:productTrees}.
The first identity then follows from the following facts. By the definition of $\Adm_2$,
one has
\begin{equ}[e:idenAdm2]
\Adm_2(\JJ(F,\hat F)) = \{\JJ_F A\,:\, A \in \Adm_2(F,\hat F)\}\;,
\end{equ}
where $\JJ_F A$ is the subforest of $\JJ F$ obtained by the image of 
the subforest $A$ of $F$ under the quotient map. The map
$\JJ_F$ is furthermore injective on $\Adm_2(F,\hat F)$, thus yielding
a bijection between $\Adm_2(\JJ(F,\hat F))$ and $\Adm_2(F,\hat F)$.
Finally, as a consequence of
the fact that each connected component of $A$ contains a root of $F$, 
there is a natural tree isomorphism between $\JJ_F A$ and $\JJ A$.
Combining this with an application of the Chu-Vandermonde identity on 
the roots allows to conclude.

The identity \eqref{e:idenAdm2} fails to be true for $\Adm_1$ in general.
However, if $(F, \hat F, \Labn, \Labo, \Labe) \in \Tra_2$, then each of the roots of $F$ is covered by $\hat F^{-1}(2)$, so that \eqref{e:idenAdm2} with $\Adm_2$ replaced by $\Adm_1$
does hold in this case. Furthermore, one then has a natural forest
isomorphism between $\JJ_F A$ and $A$ (as a consequence of the fact that $A$ does not
contain any of the roots of $F$), so that the second identity follows immediately.
\end{proof}

We now use the ``root joining'' map $\JJ$ to define
\begin{equ}[eq:hatCH2] 
\hat \CH_2 \eqdef \scal{\Tra_2} / \ker (\JJ\hat \CK_2) \simeq \CH_2 / \ker (\JJ\hat P_2) \;.
\end{equ}
Note here that $\JJ \hat P_2$ is well-defined on $\CH_2$ by \eqref{e:PhiJJ}, so that 
the last identity makes sense. The identity \eqref{e:PhiJJ} also implies that
$\ker (\JJ\hat \CK_2) = \ker (\hat \CK_2\JJ)$, so the order in which the two operators appear here
does not matter. 
 We define also
\begin{equ}[eq:hatBB2] 
\hat \BB_2 \eqdef \Vec({\C_2}) / \ker (\JJ\CK_2) \simeq \BB_2 / \ker (\JJ) \;,
\end{equ}
where $\JJ:\C_2\to\C_2$ is defined by $(\JJ(F),\hat F)$, which makes sense since all roots in $F$ have the same
(blue) colour. 

Finally, we define the {\it tree product} for $i\geq 0$
\begin{equ}[odot]
\CD_i(\JJ)\times\CD_i(\JJ)\ni
(\tau, \bar \tau) \mapsto \tau \bar \tau \eqdef \JJ(\tau \cdot \bar \tau)\;
\end{equ}

Then we have the following complement to Corollary~\ref{lem:preCEFM}
\begin{proposition}
\label{lem:treeprod}
Denoting by $\hat\CM$ the tree product \eqref{odot}, 
\begin{enumerate}
\item $(\hat\CH_2,\hat\CM,\Delta_2,\one_2,\one^\star_2)$ is a Hopf algebra and a comodule bialgebra over the Hopf algebra $(\CH_1,\CM,\Delta_1,\one_1,\one^\star_1)$ with coaction $\Delta_1$ and counit $\one^\star_1$.
\item $(\hat\BB_2,\hat\CM,\Delta_2,\one_2,\one^\star_2)$ is a Hopf algebra and a comodule bialgebra over the Hopf algebra $(\BB_1,\CM,\Delta_1,\one_1,\one^\star_1)$ with coaction $\Delta_1$ and counit $\one^\star_1$.
\end{enumerate}
\end{proposition}
\begin{proof}
The Hopf algebra structure of $\CH_2$
turns $\hat \CH_2$ into a Hopf algebra as well by the first part of Proposition~\ref{prop:propJJ}
and \eqref{e:multJJ}, combined with
\cite[Thm~1 (iv)]{Quotients}, which states that if $H$ is a Hopf algebra over a field and $I$ a bi-ideal of $H$ such that $H/I$ is commutative, then $H/I$ is a Hopf algebra. 
For $\hat\BB_2$, the same proof holds.
\end{proof}

The second assertion in Proposition~\ref{lem:treeprod} is in fact the same result, just written differently, 
as \cite[Thm~8]{MR2803804}. Indeed,
our space $\BB_2$ is isomorphic to the Connes-Kreimer Hopf algebra ${\mathcal H}_{\rm CK}$, and $\BB_1$ is 
isomorphic to an extension of
the extraction/contraction Hopf algebra ${\mathcal H}$. The difference between our $\BB_1$ and ${\mathcal H}$ in 
\cite{MR2803804} is that we allow extraction of arbitrary subforests, including with connected components reduced to
single nodes; a subspace of $\BB_1$ which turns out to be exactly isomorphic to $\mathcal H$ is the linear space 
generated by coloured forests $(F,\hat F)\in C_1$ such that $N_F\subset \hat F_1$.

\subsection{Algebraic renormalisation}

We set
\begin{equ}[eq:CHcirc] 
\Tra_\circ \eqdef\{(F,\hat F, \Labn, \Labo, \Labe) \in \Tra: \hat F \le 1, \ 
F \text{ is a tree}\}\;,\quad 
\CH_\circ \eqdef \scal{\Tra_\circ} / \ker (\CK)\;.
\end{equ}
Then, $\CH_\circ$ is an algebra when endowed with the {\it tree product} \eqref{odot}
in the special case $i=1$.
Note that this product is well-defined on $\CH_\circ$ since $\CK$ is multiplicative and
$\JJ$ commutes with $\CK$.  Furthermore, one has $\tau \cdot \bar \tau \in \CD_1(\JJ)$
for any $\tau, \bar \tau \in \Tra_\circ$.
As a consequence of \eqref{e:multJJ} and the fact that $\cdot$ is
associative, we see that the tree product is associative, thus turning $\CH_\circ$ into
a commutative algebra with unit $(\bullet,0,0,0,0)$.

\begin{remark}
The main reason why we do not define $\CH_\circ$ similarly to $\hat \CH_2$ by 
setting $\CH_\circ = \scal{\Tra_1} / \ker (\JJ\CK)$ is that $\Delta_1$ is not 
well-defined on that quotient space, while it is well-defined on $\CH_\circ$ as
given by \eqref{eq:CHcirc}, see Proposition~\ref{prop:alg}.
\end{remark}

\begin{remark}\label{rem:H_circ}
Using Lemma~\ref{lem:PP=P} as in Remark~\ref{rem:H_i}, we have canonical isomorphisms 
\begin{equs}
\CH_\circ &\simeq \scal{H_\circ}, \qquad  H_\circ\eqdef \{\CF\in\Tra_\circ: \CK\CF=\CF\}\;,\\
\CH_1 &\simeq \scal{H_1}, \qquad  H_1 \eqdef \{\CF\in\Tra_1: \CK_1\CF=\CF\}\;,\label{e:isoH}\\
\hat \CH_2 &\simeq \scal{\hat H_2}, \qquad  \hat H_2 \eqdef \{\CF\in\Tra_2: \JJ\hat \CK_2\CF=\CF\}\;.
\end{equs}
In particular, we can view $\CH_\circ$ and $\hat \CH_2$ as spaces of decorated trees rather than forests.
In both cases, the original forest product $\cdot$ can (and will) be interpreted as the tree product \eqref{odot} with, respectively, $i=1$ and $i=2$. 
\end{remark}
We denote by $\hat\CG_2$\label{charGhat2} the group of characters of $\hat\CH_2$ and by $\CG_1$  the group of characters of $\CH_1$.

Combining all the results we obtained so far, we see that we have constructed the following structure.
\begin{proposition}\label{prop:alg}
We have
\begin{enumerate}
\item $\CH_\circ$ is a left comodule over $\CH_1$ with coaction $\Delta_1$ and counit $\one^\star_1$.
\item $\hat\CH_2$ is a left comodule over $\CH_1$ with coaction $\Delta_1$ and counit $\one^\star_1$.
\item $\CH_\circ$ is a right comodule algebra over $\hat\CH_2$ with coaction $\Delta_2$ and counit $\one^\star_2$.
\item Let $\CH\in\{\CH_\circ,\hat\CH_2\}$. We define a left action of $\CG_1$ on $\CH^*$ by
\[
g h(\tau)\eqdef(g\otimes h)\Delta_1\tau, \qquad g\in \CG_1, \ h\in \CH^*, \ \tau\in\CH,
\]
and a right action of $\hat\CG_2$ on $\CH^*$ by
\[
hf(\tau)\eqdef(h\otimes f)\Delta_2\tau, \qquad f\in \hat \CG_2, \ h\in\CH^*, \ \tau\in\CH.
\]
Then we have
\begin{equs}\label{??}
g(hf) = (g h)(g f)\;, \qquad g\in \CG_1, \ f\in \hat \CG_2, \ h\in \CH^*.
\end{equs}
\end{enumerate}
\end{proposition}
\begin{proof}
The first, the second and the third assertions follow from the coassociativity of $\Delta_1$, respectively $\Delta_2$, proved in Proposition~\ref{prop:coassoc}, combined with Proposition~\ref{prop:propJJ} to show
that these maps are well-defined on the relevant quotient spaces. 
The multiplicativity of $\Delta_2$ with respect to the tree product \eqref{odot} follows from
the first identity of Proposition~\ref{prop:propJJ}, combined with the fact that 
$\hat \CH_2$ is a quotient by $\ker \JJ$.

In order to prove the last assertion, we show first that the above definitions yield indeed actions, since by the 
coassociativity of $\Delta_1$ and $\Delta_2$ proved in Proposition~\ref{prop:coassoc}
\begin{equs}
g_1(g_2h) & =(g_1\otimes(g_2\otimes h)\Delta_1)\Delta_1=(g_1\otimes g_2\otimes h)(\id\otimes\Delta_1)\Delta_1
\\ & = (g_1\otimes g_2\otimes h)(\Delta_1\otimes\id)\Delta_1 = ((g_1\otimes g_2)\Delta_1\otimes h)\Delta_1 = (g_1g_2)h,
\end{equs}
and
\begin{equs}
(hf_1)f_2 & =((h\otimes f_1)\Delta_2\otimes f_2)\Delta_2 = (h\otimes f_1\otimes f_2)(\Delta_2\otimes\id)\Delta_2
\\ & = (h\otimes f_1\otimes f_2)(\id\otimes\Delta_2)\Delta_2
= (h\otimes(f_1\otimes f_2)\Delta_2)\Delta_2
=h(f_1f_2).
\end{equs}
Following \eqref{actsleft}, the natural definition is for $(g,f)\in \CG_1\times \hat \CG_2$ and $h\in \CH_\circ^*$
\[
(g,f)h\eqdef(g h)f^{-1}=(g h\otimes f\CA_2)\Delta_2=(g\otimes h\otimes f\CA_2)(\Delta_1\otimes\id)\Delta_2.
\]
We prove now \eqref{??}. By the definitions, we have
\begin{equs}
g(hf) & = (g\otimes (h\otimes f)\Delta_2)\Delta_1 = (g\otimes h\otimes f)(\id\otimes \Delta_2)\Delta_1
\\ & = (g\otimes h\otimes f)(\id\otimes \Delta_2)\Delta_1,
\end{equs}
while
\begin{equs}
(g h)(g f) & = ((g\otimes h)\Delta_1\otimes (g\otimes f)\Delta_1)\Delta_2
\\ & = (g\otimes h \otimes g\otimes f)(\Delta_1\otimes\Delta_1)\Delta_2
\\ & =(g\otimes h\otimes f)\CM^{(13)(2)(4)}(\Delta_1\otimes\Delta_1)\Delta_2.
\end{equs}
and we conclude by Proposition~\ref{prop:doublecoass}.
\end{proof}

Proposition~\ref{prop:alg} and its direct descendant, Theorem~\ref{def:charpm}, are
crucial in the renormalisation procedure below, see Theorem~\ref{theo:algebra} and in particular \eqref{e:defPig}.

By Proposition~\ref{left-right} and \eqref{??}, we obtain from \eqref{??} that
$\CH_\circ$ is a left comodule over the Hopf algebra $\hat\CH_{12}=\CH_1\ltimes \hat\CH_2 = (\CH_1\ltimes \CH_2)/\ker(\id\otimes \JJ)$, with counit $\one^\star_{12}$ and coaction
\[
\Delta_{\circ}:\CH_\circ\to\hat\CH_{12}\hattimes\CH_\circ, \qquad
\Delta_{\circ} \eqdef \sigma^{(132)} (\Delta_1\otimes \hat\CA_2)\Delta_2
\]
where $\sigma^{(132)}(a\otimes b\otimes c)\eqdef a\otimes c\otimes b$ and 
$\hat\CA_2$ is the antipode of $\hat\CH_2$. Equivalently,  the semi-direct product $\CG_1\ltimes \hat\CG_2$ acts on the left on the dual space $\CH_\circ^*$ by the formula 
\[
(\ell,g)h(\tau)\eqdef(\ell\otimes h \otimes g\hat\CA_2)(\Delta_1\otimes \id)\Delta_2\tau,
\]
for $\ell\in \CG_1$, $g\in \hat \CG_2$, $h\in \CH_\circ^*$, $\tau\in\CH_\circ$. In other words, with this 
action $\CH_\circ^*$ is a left module on $\CG_1\ltimes \hat \CG_2$, see Proposition~\ref{left-right}.

\begin{remark}
The action of  $\Delta_1$ on $\hat\CH_2$ differs from the action on $\{\CH_\circ,\CH_1\}$ because of the following
detail: $\hat\CH_2$ is generated (as bigraded space) by a basis of rooted trees whose root is blue; since $\Delta_1$ acts
by extraction/contraction of subforests which contain $\hat F_1$ and are disjoint from $\hat F_2$, such subforests
can never contain the root. Since on the other hand in $\CH_\circ$ and $\CH_1$ one has coloured forests with empty 
$\hat F_2$, no such restriction applies to the action of $\Delta_1$ on these spaces.
\end{remark}

\subsection{Recursive formulae}\label{sec:recurs}

We now show how the formalism developed so far in this article links to the 
one developed in \cite[Sec.~8]{reg}. 
For that, we use the canonical identifications  
\[
\CH_\circ=\scal{H_\circ}, \qquad \CH_1=\scal{H_1}, \qquad \hat\CH_2=\scal{\hat H_2}, 
\]
given in Remarks~\ref{rem:H_i} and~\ref{rem:H_circ}. We furthermore
introduce the following notations. 
\begin{enumerate}
\item For $ k \in \N^d$, we write $X^k$\label{defX} as a shorthand for $(\bullet,0)_{0}^{k,0} \in H_\circ$. 
We also interpret this as an element of $\hat \CH_2$, although
its canonical representative there is $(\bullet,2)_{0}^{k,0} \in \hat H_2$.
As usual, we also write $\one$ instead of $X^0$, 
and we write $X_i$ with $i \in \{1,\ldots,d\}$ as a shorthand for $X^k$ with $k$ equal to the
$i$-th canonical basis element of $\N^d$.
\item For every type $ \Labhom \in \Lab$ and every
 $ k \in \N^d $, we define the linear operator 
\begin{equ}[e:CI]
\CI^{\Labhom}_{k}\colon \CH_\circ \to \CH_\circ
\end{equ} 
in the following way.
Let $\tau = (F,\hat F)_\Labe^{\Labn,\Labo} \in H_\circ$, so that 
we can assume that $F$ consists of a single tree with root $\rho$.
Then, $\CI^{\Labhom}_{k}(\tau) = (G,\hat G)_{\bar\Labe}^{\bar \Labn,\bar \Labo}\in H_\circ$ is 
given by
\begin{equ}
N_G = N_F \sqcup \{\rho_G\}\;,\qquad E_G = E_F \sqcup \{(\rho_G,\rho)\}\;,
\end{equ}
the root of $G$ is $\rho_G$, the type of the edge $(\rho_G,\rho)$ is $\Labhom$. For instance
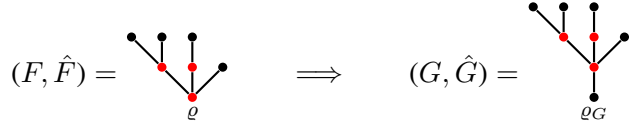
\[
(F,\hat F)=
\begin{tikzpicture}[scale=0.2,baseline=0.2cm]
        \node at (0,0)  [dot,red,label={[label distance=-0.2em]below: \scriptsize  $ \rho $} ] (root) {};
         \node at (-4,4)  [dot] (leftll) {};
          \node at (-2,4)  [dot] (leftlc) {};
      \node at (-2,2)  [dot,color=red] (left) {};
          \node at (0,2)  [dot,red ] (rightll) {};
          \node at (0,4)  [dot ] (rightll2) {};
         \node at (2,2)  [dot ] (rightrr) {};
        
        \draw[kernel1] (left) to node [sloped,below] {\small } (root); ;
     \draw[kernel1] (leftll) to
     node [sloped,below] {\small }     (left);
      \draw[kernel1] (leftlc) to
     node [sloped,below] {\small }     (left); 
     \draw[kernel1] (rightrr) to
     node [sloped,below] {\small }     (root);
     \draw[kernel1] (rightll2) to
     node [sloped,below] {\small }     (rightll);
     \draw[kernel1] (rightll) to
     node [sloped,below] {\small }     (root);
     \end{tikzpicture} \qquad
\Longrightarrow \qquad (G,\hat G)=
\begin{tikzpicture}[scale=0.2,baseline=0.2cm]
        \node at (0,0)  [dot,label={[label distance=-0.2em]below: \scriptsize  $ \rho_G $} ] (star) {};
        \node at (0,2)  [dot,red] (root) {};
          \node at (-4,6)  [dot] (leftll) {};
          \node at (-2,6)  [dot] (leftlc) {};
      \node at (-2,4)  [dot,color=red] (left) {};
          \node at (0,4)  [dot,red ] (rightll) {};
          \node at (0,6)  [dot ] (rightll2) {};
         \node at (2,4)  [dot ] (rightrr) {};
       
        \draw[kernel1] (root) to node [sloped,below] {\small } (star); ;
        \draw[kernel1] (left) to node [sloped,below] {\small } (root); ;
     \draw[kernel1] (leftll) to
     node [sloped,below] {\small }     (left);
      \draw[kernel1] (leftlc) to
     node [sloped,below] {\small }     (left); 
     \draw[kernel1] (rightrr) to
     node [sloped,below] {\small }     (root);
     \draw[kernel1] (rightll2) to
     node [sloped,below] {\small }     (rightll);
     \draw[kernel1] (rightll) to
     node [sloped,below] {\small }     (root);
     \end{tikzpicture} 
\]
The decorations of $\CI^{\Labhom}_{k}(\tau)$, as well as $\hat G$, 
coincide with those of $\tau$, except on the newly added edge / vertex where $\hat G$, 
$\bar \Labn$ and $\bar \Labo$ vanish, while $\bar \Labe(\rho_G,\rho) = k$. This gives a triangular operator and $\CI^{\Labhom}_{k}\colon \CH_\circ \to \CH_\circ$ is therefore well defined.
\item Similarly, we define operators 
\begin{equ}[e:CJ]
\hat \CJ^{\Labhom}_{k}\colon \CH_\circ \to \hat\CH_2
\end{equ}
in exactly the same way as the operators $\CI^{\Labhom}_{k}$ defined in \eqref{e:CI}, 
except that the 
root of $\hat \CJ^{\Labhom}_{k}(\tau)$ is coloured with the colour $2$, for instance
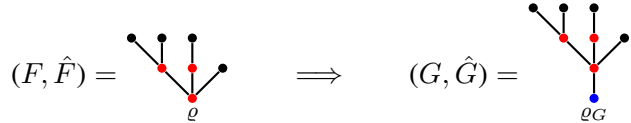
\[
(F,\hat F)=
\begin{tikzpicture}[scale=0.2,baseline=0.2cm]
        \node at (0,0)  [dot,red,label={[label distance=-0.2em]below: \scriptsize  $ \rho $} ] (root) {};
         \node at (-4,4)  [dot] (leftll) {};
          \node at (-2,4)  [dot] (leftlc) {};
      \node at (-2,2)  [dot,color=red] (left) {};
          \node at (0,2)  [dot,red ] (rightll) {};
          \node at (0,4)  [dot ] (rightll2) {};
         \node at (2,2)  [dot ] (rightrr) {};
        
        \draw[kernel1] (left) to node [sloped,below] {\small } (root); ;
     \draw[kernel1] (leftll) to
     node [sloped,below] {\small }     (left);
      \draw[kernel1] (leftlc) to
     node [sloped,below] {\small }     (left); 
     \draw[kernel1] (rightrr) to
     node [sloped,below] {\small }     (root);
     \draw[kernel1] (rightll2) to
     node [sloped,below] {\small }     (rightll);
     \draw[kernel1] (rightll) to
     node [sloped,below] {\small }     (root);
     \end{tikzpicture} \qquad
\Longrightarrow \qquad (G,\hat G)=
\begin{tikzpicture}[scale=0.2,baseline=0.2cm]
        \node at (0,0)  [dot,blue,label={[label distance=-0.2em]below: \scriptsize  $ \rho_G $} ] (star) {};
        \node at (0,2)  [dot,red] (root) {};
          \node at (-4,6)  [dot] (leftll) {};
          \node at (-2,6)  [dot] (leftlc) {};
      \node at (-2,4)  [dot,color=red] (left) {};
          \node at (0,4)  [dot,red ] (rightll) {};
          \node at (0,6)  [dot ] (rightll2) {};
         \node at (2,4)  [dot ] (rightrr) {};
       
        \draw[kernel1] (root) to node [sloped,below] {\small } (star); ;
        \draw[kernel1] (left) to node [sloped,below] {\small } (root); ;
     \draw[kernel1] (leftll) to
     node [sloped,below] {\small }     (left);
      \draw[kernel1] (leftlc) to
     node [sloped,below] {\small }     (left); 
     \draw[kernel1] (rightrr) to
     node [sloped,below] {\small }     (root);
     \draw[kernel1] (rightll2) to
     node [sloped,below] {\small }     (rightll);
     \draw[kernel1] (rightll) to
     node [sloped,below] {\small }     (root);
     \end{tikzpicture} 
\]
\item
For $\alpha \in \Z^d\oplus\Z(\Lab)$, 
we define linear triangular maps $\CR_\alpha\colon \CH_\circ \to \CH_\circ$
in such a way that if $\tau = (T,\hat T)_\Labe^{\Labn,\Labo}\in H_\circ$ with root $\rho \in N_T$,
then $\CR_\alpha(\tau)\in H_\circ$ coincides with $\tau$, except for $\Labo(\rho)$ to which we
add $\alpha$ and $\hat T(\rho)$ which is set to $1$. 
In particular, one has $\CR_\alpha \circ \CR_\beta = \CR_{\alpha+\beta}$. 
\end{enumerate}
\begin{remark}\label{rem:A}
With these notations, it follows from the definition of the sets $H_\circ$, $H_1$ and $\hat H_2$
that they can be constructed as follows.
\begin{itemize}
\item Every element of $H_\circ\setminus\{\one\}$ can be obtained from
elements of the type $X^k$ by successive applications of the 
maps $\CI^{\Labhom}_{k}$, $\CR_\alpha$, and the tree product \eqref{odot}.
\item Every element of $H_1$ is the forest product of a finite number of elements of $H_\circ$.
\item Every element of $\hat H_2$ is of the form
\begin{equ}[e:formH2]
X^k \prod_i \hat \CJ_{k_i}^{\Labhom_i}(\tau_i)\;,
\end{equ}
for some finite collection of elements $\tau_i \in H_\circ\setminus\{\one\}$, $\Labhom_i \in \Lab$ and $k_i \in \N^d$. 
\end{itemize}
\end{remark}
Then, one obtains a simple recursive description of the coproduct $\Delta_2$.
\begin{proposition}
With the above notations, the operator 
$\Delta_2\colon \CH_\circ \to \CH_\circ \hotimes \hat\CH_2$ is multiplicative,
satisfies the identities
\begin{equs}
\Delta_2 X_i &= X_i \otimes \one + \one \otimes X_i \;,\qquad 
\Delta_2 \one  = \one \otimes \one\;, \\
\Delta_2 \CI^{\Labhom}_k(\tau) &  = \left( \CI^{\Labhom}_k \otimes \id \right)  \Delta_2 \tau
+ \sum_{\ell}  \frac{X^{\ell}}{\ell !}  \otimes   \hat \CJ^{\Labhom}_{k+ \ell}(\tau), \label{e:def_rec_2}\\
\Delta_2 \CR_\alpha(\tau) &= (\CR_\alpha \otimes \id) \Delta_2 \tau\;
\end{equs}
and it is completely determined by these properties. 
Likewise, $\Delta_2 \colon \hat\CH_2 \to \hat\CH_2 \hotimes \hat\CH_2$ is multiplicative, 
satisfies the identities on the first line of \eqref{e:def_rec_2} 
and 
\begin{equ}[e:defDeltaH2]
\Delta_2 \hat \CJ^{\Labhom}_k(\tau)  = \left( \hat \CJ^{\Labhom}_k \otimes \id \right)  \Delta_2 \tau
+ \sum_{\ell}  \frac{X^{\ell}}{\ell !}  \otimes   \hat \CJ^{\Labhom}_{k+ \ell}(\tau)\;
\end{equ}
and it is completely determined by these properties. 
\end{proposition}

\begin{proof} 
The operator $\Delta_2$ is multiplicative on $\CH_\circ$ as a consequence of the first identity
of Proposition~\ref{prop:propJJ} and
its action on $X^k$ was already mentioned in \eqref{e:DeltaX}.
It remains to verify that the recursive identities hold as well. 

We first consider $\Delta_2 \sigma$ with $\sigma = \CI_k^{\Labhom}(\tau)$
and $\tau = (T,\hat T)^{\Labn,\Labo}_\Labe$.
We write $\sigma = (F,\hat F)^{\Labn,\Labo}_{\Labe + k \one_e}$, where $e$ is the 
``trunk'' of type $\Labhom$ created by $\CI_k^{\Labhom}$ and $\rho$ is the root of $F$; 
moreover we extend $\Labn$ to $N_F$ and $\Labo$ to $N_{\hat F}$ by setting 
$\Labn(\rho)=\Labo(\rho)=0$.
It follows from the definitions that
\begin{equs}
\Adm_2(F,\hat F) = \{\{\rho\}\} \cup \{ A \cup \{\rho,e\}\,:\, A \in \Adm_2(T,\hat T)\}\;.
\end{equs}
Indeed, if $e$ does not belong to an element $A$ of $\Adm_2(F,\hat F)$ then, since $A$ has to 
contain $\rho$ and be connected, one necessarily has $A = \{\rho\}$. If on the other hand
$e \in A$, then one also has $\rho \in A$ and the remainder of $A$ is necessarily a connected 
subtree of $T$ containing its root, namely an element of $\Adm_2(T,\hat T)$. 

Given $A \in \Adm_2(T, \hat T)$, since the root-label of $\sigma$ is $0$, the set of 
all possible node-labels $\Labn_A$ for $\sigma$ appearing in \eqref{def:Deltabar} for 
$\Delta_2 \sigma$ coincides with those appearing 
in the expression for 
$ \Delta_2 \tau$, 
so that we have the identity
\begin{equs}
 \Delta_2 \sigma &= (\CI^{\Labhom}_k \otimes \id) \Delta_2 \tau
+ \sum_{\varepsilon_{\rho}^F,\Labn_{\rho}} {1\over \varepsilon_{\rho}^F!}\binom{\Labn }{\Labn_{\rho}}
(\bullet,0,\Labn_{\rho} + \pi \varepsilon_{\rho}^F,0,0 ) \\
& \qquad \otimes (F,\hat F + 2 \un{\rho}, \Labn -  \Labn_{\rho},\Labo,\Labe+k\un{e}+\varepsilon_{\rho}^F)      \\ 
& = (\CI^{\Labhom}_k \otimes \id)  \Delta_2 \tau
+  \sum_{\ell} {1\over \ell !} 
X^{\ell}  \otimes 
  \hat \CJ^{\Labhom}_{k+\ell}(\tau)\;.
\end{equs}
This is because $\Labn(\rho) =0$, so that
the sum over $\Labn_{\rho}$ contains only the zero term. Since $\Delta_2 \colon\CH_\circ \to 
\CH_\circ \hotimes \hat\CH_2$, we are implicitly applying the appropriate contraction $\CK\otimes\JJ\hat \CK_2$, see \eqref{eq:hatCH2}-\eqref{eq:CHcirc}.

We now consider $\Delta_2 \sigma$ with $\sigma = \CR_\alpha(\tau)$.
In this case, we write $\tau = (T,\hat T)^{\Labn,\Labo}_\Labe$ so that,
denoting by $\rho$ the root of $T$, one has
$\sigma = (T,\hat T \vee \un{\rho},\Labn,\Labo + \alpha \un{\rho},\Labe)$.
We claim that in this case one has
\begin{equ}
\Adm_2(T,\hat T) = \Adm_2(T,\hat T \vee \un{\rho})\;.
\end{equ}
This is non-trivial only in the case $\hat T(\rho) = 0$. In this case however,
it is necessarily the case that $\hat T(e) = 0$ for every edge $e$ incident to the
root. This in turn guarantees that the family $\Adm_2(T,\hat T)$ remains unchanged by the 
operation of colouring the root. 
This implies that one has 
\begin{equ}
\Delta_2 \CR_\alpha(\tau) = (\CR_\alpha \otimes \CR_\alpha) \Delta_2 \tau\;.
\end{equ}
This appears slightly different from the desired identity, but the latter then
follows by observing that, for every $\bar \tau \in \hat \CH_2$,
one has $\CR_\alpha \bar \tau = \bar \tau$ as elements of $\hat \CH_2$, thanks to the 
fact that we quotiented by the kernel of $\hat \CK_2$ which sets the
value of $\Labo$ to $0$ on the root. 
\end{proof}

We finally have the following results on the antipode of $\hat\CH_2$:
\begin{proposition}\label{prop:CA2}
Let $\hat\CA_2:\hat\CH_2\to\hat\CH_2$ be the antipode of $\hat\CH_2$. Then 
\begin{itemize}
\item The algebra morphism $\hat\CA_2:\hat\CH_2\to\hat\CH_2$ is defined uniquely by the fact that $\hat\CA_2 X_i = - X_i$ and for all $\hat \CJ^{\Labhom}_{k}(\tau)\in \hat\CH_2$ with $\tau \in \CH_\circ$
\begin{equ}[e:defA+0]
\hat\CA_2 \hat \CJ_k^{\Labhom}(\tau) = -\sum_{\ell\in \N^d}{(-X)^\ell \over \ell!} \CM \bigl(\hat \CJ_{k+\ell}^{\Labhom}\otimes \hat\CA_2\bigr)\Delta_2 \tau\;,
\end{equ}
where $\CM \colon \hat\CH_2\hotimes \hat\CH_2\to \hat\CH_2$ denotes the (tree) product.
\item On $\hat\CH_2$, one has the identity
\begin{equ}[e:propWanted20]
\Delta_1  \hat\CA_2 = (\id \otimes  \hat\CA_2)\Delta_1\;.
\end{equ}
\end{itemize}
\end{proposition}
\begin{proof}
By \eref{e:formH2} and by induction over the number of edges in $\tau$, 
this uniquely determines a morphism 
$\hat\CA_2$ of $\hat\CH_2$, so it only remains to show that
\[
\CM(\id\otimes\hat\CA_2)\Delta_2 \tau = \one_{\hat\CH_2}\one^\star_{\hat\CH_2}(\tau)\;.
\]
The formula is true for $\tau= X^k$, so that, since both sides are multiplicative, it is enough to consider elements of the form $\hat \CJ_k^\Labhom(\tau)$ for some $\tau \in \CH_\circ$.
Exploiting the identity \eqref{e:defA+0}, one then has
\begin{equs}
{} &
\CM \bigl(\id \otimes \hat\CA_2 \bigr) \Delta_2 \hat \CJ^\Labhom_k(\tau)  = 
\\ & =  \CM \bigl(\id \otimes \hat\CA_2\bigr) \left[ 
\left( \hat \CJ^{\Labhom}_k \otimes \id \right)  \Delta_2 \tau
+ \sum_{\ell}  \frac{X^{\ell}}{\ell !}  \otimes  \hat \CJ^{\Labhom}_{k+ \ell}(\tau) \right]
\\ & = \CM \left[ 
\left( \hat \CJ^{\Labhom}_k \otimes  \hat\CA_2 \right)  
- \sum_{\ell,i}  \frac{X^{\ell}}{\ell !}  \otimes  \frac{(-X)^{i}}{i !} \CM(\hat \CJ^{\Labhom}_{k+ \ell+i}\otimes\hat\CA_2)\right]\Delta_2 \tau
\\ & = \CM \left[ 
\left( \hat \CJ^{\Labhom}_k \otimes  \hat\CA_2 \right)  
- \sum_{\ell} \frac{(X-X)^{\ell}}{\ell !}\, \CM(\hat \CJ^{\Labhom}_{k+ \ell}\otimes\hat\CA_2)\right]\Delta_2 \tau
\\ & = \left[ \CM 
\left( \hat \CJ^{\Labhom}_k \otimes  \hat\CA_2 \right)  
- \CM(\hat \CJ^{\Labhom}_{k}\otimes\hat\CA_2)\right]\Delta_2 \tau  = 0\;,
\end{equs}
as required.

A similar proof by induction yields \eqref{e:propWanted20}: see the proof of Lemma~\ref{commutation_antipode} for an analogous argument. Note that \eqref{e:propWanted20} is also a direct consequence of Proposition~\ref{prop:doublecoass} and more precisely of the fact that the bialgebras $\CH_1$ and $\hat\CH_2$ are in cointeraction, as follows from Remark~\ref{cointera}: see \cite[Prop.~2]{2016arXiv160508310F} for a proof. Having this property, the antipode $\hat\CA_2$ is a morphism of the $\CH_1$-comodule $\hat\CH_2$.
\end{proof}

In this section we have shown several useful recursive formulae that characterize 
$\Delta_2$, see also Section~\ref{sec:reduced} below. The paper \cite{YB} explores in greater
detail this recursive approach to Regularity Structures, and includes a recursive formula
for $\Delta_1$, which is however more complex than that for $\Delta_2$.

\section{Rules and associated Regularity Structures}

\label{sec5}

We recall the definition of a regularity structure from \cite[Def. 2.1]{reg}
\begin{definition}\label{def:regStruct}
A \textit{regularity structure} $\TT = (A, T, G)$ consists of the following elements:
\begin{itemize}
\item An index set $A \subset \R$ such that $A$ is bounded from below, and $A$ is locally finite.
\item A \textit{model space} $T$, which is a graded vector space $T = \bigoplus_{\alpha \in A} T_\alpha$,
with each $T_\alpha$ a Banach space. 
\item A \textit{structure group} $G$ of linear operators acting on $T$ such that, for every $\Gamma \in G$, every $\alpha \in A$,
and every $a \in T_\alpha$, one has
\begin{equ}[e:coundGroup]
\Gamma a - a \in \bigoplus_{\beta < \alpha} T_\beta\;.
\end{equ}
\end{itemize}
\end{definition}

The aim of this section is to relate the construction of the previous section to the
theory of regularity structures as exposed in \cite{reg,CDM}. For this, we 
first assign real-valued degrees to each element of $\Tra$.

\begin{definition}\label{def:scaling}
A \textit{scaling} is a map $\s:\{1,\ldots d\} \to [1,\infty)$ and a \textit{degree assignment} is
a map $|\cdot|_\s \colon \Lab \to \R\setminus \{0\}$.
By additivity, we then assign a \textit{degree} to each $(k,v) \in\Z^d\oplus\Z(\Lab)$ by
setting
\begin{equ}[e:grading]
|(k,v)|_\s\eqdef |k|_\s+|v|_\s\in\R, \qquad |k|_\s \eqdef  \sum_{i=1}^d k_i \s_i , \qquad
|v|_\s\eqdef  \sum_{\Labhom\in \Lab} v_\Labhom\,|\Labhom|_\s, 
\end{equ}
if $v=\sum_{\Labhom\in \Lab} v_\Labhom \Labhom$ with $v_\Labhom\in\Z$. 
\end{definition}

\begin{definition} \label{homog}
Given a scaling $\s$ as above, for  $\tau = (F,\hat F, \Labn,\Labo,\Labe) \in \Tra_2$,
we define two different notions of degree $|\tau|_-, |\tau|_+ \in\R$ by
\begin{equs}
|\tau|_- &= \sum_{e \in E_F \setminus \hat E} \bigl(|\Labhom(e)|_{\s} - |\Labe(e)|_\s\bigr) + \sum_{x \in N_F} |\Labn(x)|_\s \;,\\
|\tau|_+ &= \sum_{e \in E_F\setminus {\hat E_2}} \bigl(|\Labhom(e)|_{\s} - |\Labe(e)|_\s\bigr) + \sum_{x \in N_F} |\Labn(x)|_\s + \sum_{x \in N_F \setminus {\hat N_2}} |\Labo(x)|_\s \;,
\end{equs}
where we recall that $\Labo$ takes values in $\Z^d\oplus\Z(\Lab)$ and $\Labhom:E_F\to\Lab$ is the map assigning to an edge its type in $F$, see Section~\ref{sec:rooted}.
\end{definition}

Note that both of these degrees are compatible with the contraction operator $\CK$ of Definition~\ref{CKop}, as well as the operator $\JJ$, in the sense that 
$|\tau|_\pm = |\bar \tau|_\pm$ if and only if
$|\CK\tau|_\pm = |\CK \bar \tau|_\pm$ and similarly for $\JJ$. 
In the case of $|\cdot|_+$, this is true thanks
to the definition \eqref{e:defhatn}, while the coloured part of the tree is simply ignored by $|\cdot|_-$.
We furthermore have

\begin{lemma}\label{lem:grading}
The degree $|\cdot|_-$ is compatible with the operators $\CK_i$ and $\hat \CK_i$ of \eref{e:bi-ideal},
while $|\cdot|_+$ is compatible with $\CK_2$ and $\hat\CK_2$. Furthermore, both degrees are
compatible with $\JJ$ and $\CK$, so that in particular $\CH_1$ is $|\cdot|_-$-graded and
$\hat\CH_2$ and $\CH_\circ$ are both $|\cdot|_-$ and $|\cdot|_+$-graded.
\end{lemma}

\begin{proof}
The first statement is obvious since $|\cdot|_-$ ignores the coloured part of the tree, except for
the labels $\Labn$ whose total sum is preserved by all these operations.
For the second statement, we need to verify that $|\cdot|_+$ is compatible
with $\hat \Phi_2$ as defined just below \eqref{Phi}.
 which is the case when acting on a tree with $\rho \in \hat F_2$ since the $\Labo$-decoration of nodes in $\hat F_2$ does not contribute to the definition of $|\cdot|_+$.
\end{proof}

As a consequence, $|\cdot|_-$ yields a grading for $\CH_1$, $|\cdot|_+$ yields
a grading for $\hat \CH_2$, and both of them yield gradings for $\CH_\circ$. 
With these definitions, we see that we obtain a structure resembling a regularity
structure by taking $\CH_\circ$ to be our model space, with grading given by 
$|\cdot|_+$ and structure group
given by the character group $\hat \CG_2$ of $\hat \CH_2$ acting on $\CH_\circ$
via 
\begin{equ}
\Gamma_g \colon \scal{\Tra_\circ} \to \scal{\Tra_\circ}\;,\quad
\Gamma_g \tau = (\id \otimes g) \Delta_2 \tau\;.
\end{equ}
The second statement of Proposition~\ref{prop:alg} then guarantees that this action is multiplicative
with respect to the tree product \eqref{odot} on $\CH_\circ$, so that we are in the 
context of \cite[Sec.~4]{reg}. There are however two conditions that are not met:
\begin{enumerate}
\item The action of $\hat\CG_2$ on $\CH_\circ$ is not of the form ``identity plus
terms of strictly lower degree'', as required for regularity structures.
\item The possible degrees appearing in $\CH_\circ$ have no lower bound
and might have accumulation points. 
\end{enumerate}

We will fix the first problem by encoding in our context what we mean by considering
a ``subcritical problem''. Such problems will allow us to prune our structure in
a natural way so that we are left with a subspace of $\CH_\circ$ that has the 
required properties. The second problem will then be addressed by quotienting 
a suitable subspace of $\hat \CH_2$ by the terms of negative degree. The
group of characters of the resulting Hopf algebra will then turn out to 
act on $\CH_\circ$ in the desired way.

\subsection{Trees generated by rules}

From now to Section~\ref{sec:SPDERules} included, the colourings and the labels $\Labo$ will be ignored. 
It is therefore convenient to consider the space
\begin{equ}[def:Trees]
\Trees\eqdef\{(T,\hat T,\Labn,\Labo,\Labe)\in\Tra: \, T \ \text{is a tree}, \ \hat T\equiv 0, \ \Labo\equiv 0\}.
\end{equ}
In order to lighten notations, we
 write elements of $\Trees$ as $(T,\Labn,\Labe) = T_\Labe^{\Labn}$ with 
$T$ a typed tree (for some set of types $\Lab$) and $\Labn\colon N_T \to \N^d$, 
$\Labe\colon E_T \to \N^d$ as above. Similarly to before, $\Trees$ is a monoid for the {tree product} \eqref{odot}.
Again, this product is associative and commutative, with unit $(\bullet,0,0)$.

\begin{definition}\label{def:elementary}
We say that an element $T_\Labe^{\Labn}\in \Trees$ is 
\textit{trivial} if $T$ consists of a single node $\bullet$. It is
\textit{planted} if $T$ has exactly one edge incident to its root $\rho$ and furthermore $\Labn(\rho) = 0$. 
\end{definition}
In other words, a planted $T_\Labe^{\Labn}\in \Trees$ is necessarily of the form $\CI_k^\Labhom(\tau)$ with $\tau\in\Trees$, see \eqref{e:CI}. For example, 
\[
\text{a planted tree:} \quad 
\begin{tikzpicture}[scale=0.2,baseline=0.2cm]
          \node at (1,4)  [dot] (leftll) {};
        \node at (-1,4)  [dot] (leftup) {};
         \node at (0,2)  [dot] (leftlr) {};
       \node at (0,0)  [dot] (left) {};
 
        \draw[kernel1] (leftll) to     node [sloped,below] {\small }     (leftlr);
        \draw[kernel1] (leftup) to     node [sloped,below] {\small }     (leftlr);
        \draw[kernel1] (leftlr) to     node [sloped,below] {\small }     (left);
     \end{tikzpicture}
\qquad \text{and a non-planted tree:} \quad    
\begin{tikzpicture}[scale=0.2,baseline=0.2cm]
          \node at (-5,4)  [dot] (leftll) {};
          \node at (-5,2)  [dot] (leftlr) {};
      \node at (-4,0)  [dot] (left) {};
         \node at (-3,2)  [dot] (leftr) {};

        \draw[kernel1] (leftll) to     node [sloped,below] {\small }     (leftlr);
        \draw[kernel1] (leftlr) to     node [sloped,below] {\small }     (left);
     \draw[kernel1] (leftr) to     node [sloped,below] {\small }     (left);  
     \end{tikzpicture} .
     \]

With this definition, each $\tau \in \Trees$ has by \eqref{e:formH2} a unique (up to permutations) factorisation with respect to the tree product \eqref{odot}
\begin{equ}[e:factorisationTrees]
\tau = \bullet_n \tau_1 \tau_2\cdots  \tau_k\;,
\end{equ}
for some $n \in \N^d$, where each $\tau_i$ is planted and $\bullet_n$ denotes the trivial
element $(\bullet,n,0)\in\Trees$.

In order to define a suitable substructure of the structure described in 
Proposition~\ref{prop:alg}, we introduce the notion of ``rules''. Essentially, a ``rule''
describes what behaviour we allow for a tree in the vicinity of any one of
its nodes.

In order to formalise this,
we first define the set of \textit{edge types} $\CE$ and the set of \textit{node types} $\CN$ by
\begin{equ}[e:defCN]
\CE = \Lab \times \N^d \;,\qquad 
\CN = \hat\CP(\CE) \eqdef \bigcup_{n \ge 0} [\CE]^{n}\;,
\end{equ} 
where $[\CE]^n$ denotes the set of \textit{unordered} $\CE$-valued $n$-uples, namely
$[\CE]^n = \CE^n / S_n$, with the natural action of the symmetric group $S_n$ on $\CE^n$.
In other words, given any set $A$, 
$\hat \CP(A)$ consists of all finite multisets whose elements are elements of $A$.
\begin{remark}
The fact that 
we consider multisets and not just $n$-uples is a reflection of the fact that we always
consider the situation where the tree product \eqref{odot} is commutative. This condition could in principle
be dropped, thus leading us to consider forests consisting of planar trees instead, but this would lead to additional complications and does not seem to bring any advantage.
\end{remark}
Given two sets $A \subset B$, we have a natural inclusion $\hat \CP(A) \subset \hat \CP(B)$.
We will usually write elements of $[\CE]^{n}$ as $n$-uples with the understanding that this is
just an arbitrary representative of an equivalence class. In particular, we write 
$()$ for the unique element of $[\CE]^0$.

Given any $T_\Labe^{\Labn} \in \Trees$, we then associate to each node
$x \in N_T$ a node type $\CN(x) \in \CN$ by
\begin{equ}[e:defNx]
\CN(x) = \bigl(s(e_1),\ldots,s(e_n)\bigr), \qquad s(e)\eqdef
(\Labhom(e),\Labe(e))\in\CE, \quad e\in E_T,
\end{equ}
where $(e_1,\ldots,e_n)$ denotes the collection of edges leaving $x$, i.e.\ edges
of the form $(x,y)$ for some node $y$.
We will sometimes use set-theoretic notations. 
In particular, given $N = (s_1,\ldots, s_n)\in \CN$ and $M = (r_1,\ldots,r_\ell) \in \CN$, we write
\begin{equ}
M \sqcup N \eqdef (r_1,\ldots,r_\ell,s_1,\ldots, s_n)\;,
\end{equ} 
and we say that $M \subset N$ if there exists $\bar N$ such that $N = M \sqcup \bar N$.
%In this situation, we then write $N \setminus M$ for $\bar N$. 
When we write a sum of the type $\sum_{M \subset N}$, we take multiplicities into
account. For example $(a,b)$ is contained twice in $(a,b,b)$, so that such a sum
always contains $2^{n}$ terms if $N$ is an $n$-tuple.
Similarly, we write $t \in N$ if $(t) \subset N$ and we also count sums
of the type $\sum_{t \in N}$ with the corresponding multiplicities.

\begin{definition}\label{def:rule}
Denoting by $\CP \CN$ the powerset of $\CN$, a \textit{rule} is a 
map $R \colon \Lab \to \CP \CN\setminus\{\emptyset\}$.
A rule is said to be \textit{normal} if, whenever $M \subset N \in R(\Labhom)$, one also has 
$M \in R(\Labhom)$.
\end{definition}

For example we may have $\Lab=\{\Labhom_1,\Labhom_2\}$ and 
\begin{equation}\label{e:ru}
R\bigg( \begin{tikzpicture}[scale=0.2,baseline=0.32cm]
          \node at (0,0) [dot] (k) {}; %k
           \node at (0,5)  (l) {}; %l
    \draw[kernel1] (l) -- node [rect1] {\tiny$\Labhom_1$}   (k)  ;
\end{tikzpicture}  
\bigg)
=R\bigg( \begin{tikzpicture}[scale=0.2,baseline=0.32cm]
          \node at (0,0) [dot] (k) {}; %k
           \node at (0,5)  (l) {}; %l
    \draw[kernel1] (l) -- node [rect1] {\tiny$\Labhom_2$}   (k)  ;
\end{tikzpicture}  
\bigg)
=\bigg\{ (), \ 
\bigg( \begin{tikzpicture}[scale=0.2,baseline=0.32cm]
          \node at (0,0) [dot] (k) {}; %k
           \node at (0,5)  (l) {}; %l
    \draw[kernel1] (l) -- node [rect1] {\tiny$\Labhom_1,\Labe_1$}   (k)  ;
\end{tikzpicture}  
\bigg),
\bigg( \begin{tikzpicture}[scale=0.2,baseline=0.32cm]
          \node at (0,0) [dot] (k) {}; %k
           \node at (0,5)  (l) {}; %l
    \draw[kernel1] (l) -- node [rect1] {\tiny$\Labhom_2,\Labe_2$}   (k)  ;
\end{tikzpicture}  
\bigg),
\bigg( \begin{tikzpicture}[scale=0.2,baseline=0.32cm]
          \node at (0,0) [dot] (k) {}; %k
           \node at (0,5)  (l) {}; %l
    \draw[kernel1] (l) -- node [rect1] {\tiny$\Labhom_1,\Labe_1$}   (k)  ;
\end{tikzpicture}  ,
\begin{tikzpicture}[scale=0.2,baseline=0.32cm]
          \node at (0,0) [dot] (k) {}; %k
           \node at (0,5)  (l) {}; %l
    \draw[kernel1] (l) -- node [rect1] {\tiny$\Labhom_2,\Labe_2$}   (k)  ;
\end{tikzpicture}  
\bigg)
\bigg\}\;.
\end{equation}
Then, according to the rule $R$, an edge of type $\Labhom_1$ or $\Labhom_2$ can be followed
in a tree by, respectively, no edge, or a single edge of type $\Labhom_i$ with decoration $\Labe_i$
with $i\in\{1,2\}$, or by two edges, one of type $\Labhom_1$ with decoration $\Labe_1$ and
one of type $\Labhom_2$ with decoration $\Labe_2$. We do not expect however to find two edges both of type
$\Labhom_1$ (or $\Labhom_2$) sharing a node which is not the root. 

\begin{definition}\label{def:conform}
Let $R$ be a rule and $\tau = T_\Labe^{\Labn} \in \Trees$. We say that
\begin{itemize}
\item $\tau$ \textit{conforms to $R$ at the vertex $x$} if either $x$ is the root and 
there exists $\Labhom\in\Lab$ such that $\CN(x) \in R(\Labhom)$
or one has $\CN(x) \in R(\Labhom(e))$, where $e$ is the unique edge linking $x$ to its parent in $T$.
\item $\tau$ \textit{conforms to $R$} if it conforms to $R$ at every vertex $x$, except possibly its root.
\item $\tau$ \textit{strongly conforms to $R$} if it conforms to $R$
at every vertex $x$.
\end{itemize}
\end{definition}
In particular, the trivial tree $\bullet$ strongly conforms to every normal rule since,
as a consequence of Definition~\ref{def:rule}, there exists at least one 
$\Labhom \in \Lab$ with $() \in R(\Labhom)$.

\begin{example} Consider $R$ as in \eqref{e:ru} and the trees
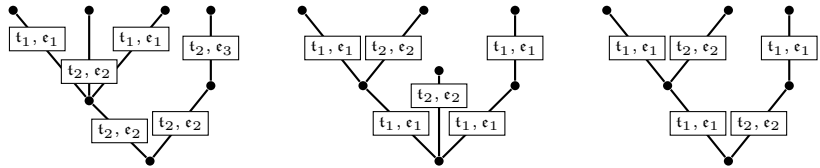
\[
\begin{tikzpicture}[scale=0.2,baseline=2cm]
         \node at (-9,10) [dot] (i) {}; %h
          \node at (-4,4) [dot] (d) {}; %d
          \node at (1,10) [dot] (j) {}; %j
          \node at (-4,10) [dot] (h) {}; %i
      \node at (0,0) [dot] (b) {}; %b
         \node at (4,5) [dot] (e) {}; %e
         \node at (4,10) [dot] (f) {}; %f
 
          \draw[kernel1] (f)  -- node [rect1] {\tiny$\Labhom_2,\Labe_3$}   (e);
    \draw[kernel1] (i)  -- node [rect1,pos=0.3] {\tiny$\Labhom_1,\Labe_1$}   (d);
     \draw[kernel1] (j)  --  node [rect1,pos=0.3]{\tiny $ \Labhom_1,\Labe_1 $ } (d) ;

     \draw[kernel1] (b) -- node [rect1] {\tiny$\Labhom_2,\Labe_2$}   (d);
      \draw[kernel1,black] (d)  -- node [rect1,pos=0.3] {\tiny$\Labhom_2,\Labe_2$}   (h);
     \draw[kernel1] (b) -- node [rect1] {\tiny$\Labhom_2,\Labe_2$}  (e) ;

\end{tikzpicture} \qquad
\begin{tikzpicture}[scale=0.2,baseline=2cm]
         \node at (-9,10) [dot] (i) {}; %h
          \node at (-5,5) [dot] (d) {}; %d
          \node at (-1,10) [dot] (j) {}; %j
          \node at (0,6) [dot] (h) {}; %i
      \node at (0,0) [dot] (b) {}; %b
         \node at (5,5) [dot] (e) {}; %e
         \node at (5,10) [dot] (f) {}; %f
 
%      \draw[kernel1,black] (d)  -- node [rect1,near end] {\tiny$\Labhom_2,\Labe_2$}   (h);
          \draw[kernel1] (f)  -- node [rect1] {\tiny$\Labhom_1,\Labe_1$}   (e);
    \draw[kernel1] (i)  -- node [rect1] {\tiny$\Labhom_1,\Labe_1$}   (d);

     \draw[kernel1] (b) -- node [rect1,near end] {\tiny$\Labhom_2,\Labe_2$}   (h);
     \draw[kernel1] (b) -- node [rect1] {\tiny$\Labhom_1,\Labe_1$}   (d);
     \draw[kernel1] (j)  --  node [rect1]{\tiny $ \Labhom_2,\Labe_2 $ } (d) ;
     \draw[kernel1] (b) -- node [rect1] {\tiny$\Labhom_1,\Labe_1$}  (e) ;

\end{tikzpicture} \qquad
\begin{tikzpicture}[scale=0.2,baseline=2cm]
         \node at (-8,10) [dot] (i) {}; %h
          \node at (-4,5) [dot] (d) {}; %d
          \node at (0,10) [dot] (j) {}; %j
  %        \node at (-4,15) [dot] (h) {}; %i
      \node at (0,0) [dot] (b) {}; %b
         \node at (4,5) [dot] (e) {}; %e
         \node at (4,10) [dot] (f) {}; %f
 
%      \draw[kernel1,black] (d)  -- node [rect1,near end] {\tiny$\Labhom_2,\Labe_2$}   (h);
          \draw[kernel1] (f)  -- node [rect1] {\tiny$\Labhom_1,\Labe_1$}   (e);
    \draw[kernel1] (i)  -- node [rect1] {\tiny$\Labhom_1,\Labe_1$}   (d);

     \draw[kernel1] (b) -- node [rect1] {\tiny$\Labhom_1,\Labe_1$}   (d);
     \draw[kernel1] (j)  --  node [rect1]{\tiny $ \Labhom_2,\Labe_2 $ } (d) ;
     \draw[kernel1] (b) -- node [rect1] {\tiny$\Labhom_2,\Labe_2$}  (e) ;

\end{tikzpicture} 
\]
The first tree does not conform to the rule $R$ since the bottom left edge of type $ \Labhom_2$ 
is followed by three edges. 
The second tree conforms to $R$ but not strongly, since the root is incident to three edges. The third
tree strongly conforms to $R$. If we call $\rho_i$ the root of the $i$-th tree, then we have
$\CN(\rho_1)=\{(\Labhom_2,\Labe_2),(\Labhom_2,\Labe_2)\}$,
$\CN(\rho_2)=\{(\Labhom_1,\Labe_1),(\Labhom_1,\Labe_1),(\Labhom_2,\Labe_2)\}$,
$\CN(\rho_3)=\{(\Labhom_1,\Labe_1),(\Labhom_2,\Labe_2)\}$, see \eqref{e:defNx}. Finally, note that $R$ is normal.
\end{example}

\begin{remark}\label{nouseless}
If $R$ is a normal rule, then by Definition~\ref{def:rule} we have in particular that $()\in R(\Labhom)$ for every $\Labhom \in \Lab$. This guarantees that $\Lab$ contains no \textit{useless}
labels in the sense that, for every $\Labhom \in \Lab$, there exists a 
tree conforming to $R$ containing an edge of type $\Labhom$: it suffices to consider a rooted tree with a single
edge $e=(x,y)$ of type $\Labhom$; in this case, $\CN(y)=\{()\}\in R(\Labhom)$.
More importantly, this also guarantees that we can build any tree conforming to $R$
from the root upwards (start with an edge of type $\Labhom$, add 
to it a node of some type in $R(\Labhom)$, then restart the construction for 
each of the outgoing edges of that node) in finitely many steps.
\end{remark}

\begin{remark}\label{G(R)}
A rule $R$ can be represented by a directed bipartite multigraph 
$\CG(R) = (V(R), E(R))$ as follows.
Take as the vertex set $V(R) = \CE \sqcup \CN$. Then, connect $N \in \CN$ to $t \in \CE$ 
if $t \in N$. If $t$ is contained in $N$ multiple times, repeat the connection
the corresponding number of times. Conversely, connect $(\Labhom,k) \in \CE$ to $N \in \CN$ 
if $N \in R(\Labhom)$. The conditions then guarantee that $() \in \CN$ can be reached from
every vertex in the graph.
Given a tree $\tau \in \Trees$, every edge of $\tau$ corresponds to an element of $\CE$ and every 
node corresponds to an element of $\CN$ via the map $x \mapsto \CN(x)$ defined above.
A tree then conforms to $R$ if, for every path joining the root to one of the leaves, 
the corresponding path in $V$ always follows directed edges in $\CG(R)$. It strongly conforms
to $R$ if the root corresponds to a vertex in $V$ with at least one incoming edge.
\end{remark}

\begin{definition}
Given $\s$ as in Definition~\ref{def:scaling}, 
we assign a degree $|\tau|_\s$ to
any $\tau \in \Trees$ by setting
\begin{equ}[e:defDegree]
|T_\Labe^\Labn|_\s = \sum_{e \in E_T} \bigl(|\Labhom(e)|_{\s} - |\Labe(e)|_\s\bigr) + \sum_{x \in N_T} |\Labn(x)|_\s \;.
\end{equ}
\end{definition}
This definition is compatible with both notions of degree given in Definition~\ref{homog}, since we view $\Trees$ as a subset of $\Tra$ with $\hat F$ and $\Labo$ identically $0$. This also allows us to give the following definition.

\begin{definition} \label{def_space}
Given a rule $R$, we write 
\begin{itemize}
\item $\Trees_\circ(R) \subset \Trees$ for the set of trees that strongly
conform to $R$ 
\item $\Trees_1(R) \subset \Tra$ for the submonoid of $\Tra$ (for the forest product) generated by $\Trees_\circ(R)$
\item $\Trees_2(R) \subset \Trees$ for the set of trees that conform to $ R $.
\end{itemize}
Moreover, we write $\Trees_-(R) \subset \Trees_\circ(R)$ for the set of trees $\tau = T_\Labe^\Labn$ such 
that 
\begin{itemize}
\item $|\tau|_\s < 0$, $\Labn(\rho_\tau) = 0$, 
\item if $\tau$ is planted, namely $\tau=\CI_k^\Labhom(\bar\tau)$ with $\bar\tau\in\Trees$, see \eqref{e:CI}, then $|\Labhom|_\s< 0$.
\end{itemize}
\end{definition}
The second restriction on the definition of $\tau\in\Trees_-(R)$ is related to the definition~\eqref{e:defJ} of the Hopf algebra $\CT^\ex_-$ and of its characters group $\CG^\ex_-$, that we call the {\it renormalisation group} and which plays a fundamental role in the theory, see e.g.\ Theorem~\ref{theo:algebra}.

\subsection{Subcriticality}

Given a map $\reg\colon \Lab \to \R$ we will henceforth interpret it as maps 
$\reg\colon \CE \to \R$ and $\reg\colon \CN \to \R$ as follows: 
for $(\Labhom,k)\in\CE$ and $N\in\CN$
\begin{equ}[convention]
\reg(\Labhom,k) \eqdef \reg(\Labhom)-|k|_\s, \qquad
\reg(N) \eqdef \sum_{(\Labhom,k) \in N} \reg(\Labhom,k),
\end{equ}
with the convention that the sum over the empty word $()\in\CN$ is $0$.
\begin{definition}\label{def:subcritical}
A rule $R$ is \textit{subcritical} with respect to a fixed scaling $\s$ if there exists a map 
$\reg\colon \Lab \to \R$ such that
\begin{equ}[e:defreg]
\reg(\Labhom) < |\Labhom|_\s + \inf_{N \in R(\Labhom)} 
\reg(N)
\;, \qquad \forall \, \Labhom \in \Lab,
\end{equ}
where we use the notation \eqref{convention}.
\end{definition}

We will see in Section~\ref{sec:SPDERules} below that classes of stochastic PDEs generate 
rules. In this context, the notion of subcriticality given here formalises the one 
given somewhat informally  in \cite{reg}. In particular, we have the following result which is 
essentially a reformulation of \cite[Lem.~8.10]{reg} in this context.

\begin{proposition}\label{prop:finite}
If $R$ is a subcritical rule, then, for every $\gamma \in \R$, the set $\{\tau \in \Trees_\circ(R)\,:\, |\tau|_\s \le \gamma\}$ is finite.
\end{proposition}

\begin{proof}
Fix $\gamma \in \R$ and let $T_\Labe^\Labn \in \Trees_\circ(R)$ with $|T_\Labe^\Labn|_\s \le \gamma$.
Since there exists $c>0$ such that 
\begin{equ}
|T_\Labe^\Labn|_\s \ge |T_\Labe^0|_\s + c|\Labn|
\end{equ}
and there exist only finitely many trees in $\Trees_\circ(R)$ of the type $|T_\Labe^0|$ for
a given number of edges,
it suffices to show that the number $|E_T|$ of edges of $T$ is bounded by some constant depending only
on $\gamma$. 

Since the set $\Lab$ is finite, \eqref{e:defreg} implies that there exists a constant $\kappa > 0$ such that the bound
\begin{equ}[e:regkappa]
\reg(\Labhom) + \kappa \le |\Labhom|_\s  + \inf_{N \in R(\Labhom)}\reg(N)\;,
\end{equ}
holds for every $\Labhom \in \Lab$ with the notation \eqref{convention}. We claim that for every planted $T_\Labe^\Labn\in\Trees_\circ(R)$ such that the edge type of its trunk $e=(\rho,x)$ is $(\Labhom,k)\in\CE$, we have
\begin{equ}[reg1]
\reg(\Labhom,k)\leq |T_\Labe^\Labn|_\s - \kappa |E_T|.
\end{equ}
We denote the space of such planted trees by $\Trees^{(\Labhom,k)}_\circ(R)  $.
We verify \eqref{reg1} by induction on the number of edges $|E_T|$ of $T$. If $|E_T|=1$, namely the unique element of $E_T$ is the trunk $e=(\rho,x)$, then $\CN(x)=()\in R(\Labhom)$ in the notation of \eqref{e:defNx} and by \eqref{e:regkappa}
\[
\reg(\Labhom) + \kappa \le |\Labhom|_\s \quad \Longrightarrow \quad
\reg(\Labhom,k)\leq |\Labhom|_\s-|k|_\s- \kappa\leq|T_\Labe^\Labn|_\s- \kappa.
\]
For a planted $T_\Labe^\Labn\in\Trees_\circ(R)$ with $|E_T|>1$, then $\CN(x)=(s(e_1),\ldots,s(e_n))\in R(\Labhom)$ and by \eqref{e:regkappa} and the induction hypothesis
\[
\reg(\Labhom)-|k|_\s + \kappa \le |\Labhom|_\s-|k|_\s  + 
\sum_{i=1}^n \left[\reg(\Labhom_i)-|k_i|_\s\right]\leq |T_\Labe^\Labn|_\s-\kappa(|E_T|-1)
\;,
\]
where $s(e_i)=(\Labhom_i,k_i)$. Therefore \eqref{reg1} is proved for planted trees.

Given an arbitrary tree $T_\Labe^\Labn$ of degree at most $\gamma$
strongly conforming to the rule $R$, there exists $\Labhom_0\in\Lab$ such that $e\in\CN(\rho_T)=R(\Labhom_0)$. We can therefore consider the planted tree
$\bar T_\Labe^\Labn$ containing a trunk of type $\Labhom_0$ connected to the root of $T$,
and with vanishing labels on the root and trunk respectively. It then follows
that 
\begin{equs}
\kappa |E_T| < \kappa |E_{\bar T}|
&\le |\bar T_\Labe^\Labn|_\s-\reg(\Labhom_0)
= |T_\Labe^\Labn|_\s + |\Labhom_0|_\s -\reg(\Labhom_0)\\
&\le \gamma + \inf_{\Labhom \in \Lab} \bigl(|\Labhom|_\s -\reg(\Labhom)\bigr)\;,
\end{equs}
and the latter expression is finite since $\Lab$ is finite. The claim follows at once.
\end{proof}

\begin{remark} 
The inequality \eqref{e:defreg} encodes the fact that we would like to be able to assign a regularity
$\reg(\Labhom)$ to each component $ u_{\Labhom} $ of our SPDE in such a way that the ``na\"\i ve regularity'' of the 
corresponding right hand side obtained by a power-counting argument
is strictly better than $\reg(\Labhom) - |\Labhom|$. Indeed, $\inf_{N \in R(\Labhom)} 
\reg(N) $ is precisely the regularity one would like to assign to $ F_{\Labhom}(u,\nabla u,\xi) $.
Note that if the inequality in \eqref{e:defreg} is not strict,
then the conclusion of Proposition~\ref{prop:finite} may fail to hold.
\end{remark}

\begin{remark}\label{min_plus}
Assuming that there exists a map $\reg$ satisfying \eqref{e:regkappa} for a
given $\kappa > 0$, one can find a map $  \reg_\kappa$ that is optimal in the sense that it saturates 
the bound \eqref{reg1}:
\begin{equs}
\reg_\kappa(\Labhom,k) = \min_{T_\Labe^\Labn \in \Trees^{(\Labhom,k)}_\circ(R) } \left( |T_\Labe^\Labn|_\s - \kappa |E_T| \right)
\end{equs} 
 where $  (\Labhom,k) \in \CE $. We proceed as follows.
Set $\reg_\kappa^0(\Labhom) = +\infty$ for every $\Labhom  \in \Lab$ and then 
define recursively
\begin{equ}[e:iterationReg]
\reg_\kappa^{n+1}(\Labhom) = |\Labhom|_\s - \kappa + \inf_{N \in R(\Labhom)} \reg_\kappa^n(N)\;.
\end{equ}
By recurrence we show that $n\mapsto\reg_\kappa^n(\Labhom)$ is decreasing and $\reg\leq\reg_\kappa^n$; then the limit
\[
\reg_\kappa(\Labhom) = \lim_{n \to \infty} \reg_\kappa^n(\Labhom)
\]
exists and has the required properties. If we extend $\reg_\kappa^n$ to
$\CE \sqcup \CN$ by \eqref{convention}, the iteration \eqref{e:iterationReg} 
can be interpreted as a
min-plus network on the graph $\CG(R)$ with arrows reversed, see Remark~\ref{G(R)}. 
\end{remark}

\subsection{Completeness}

Given an arbitrary rule (subcritical or not), there is no reason in general
to expect that the actions of the analogues of the groups $\CG_1$ and $\hat\CG_2$
constructed in Section~\ref{sec4} leave the linear span of $\Trees_\circ(R)$ invariant.
We now introduce a notion of completeness, which will guarantee later on
that the actions of $\CG_1$ and $\hat\CG_2$ do indeed leave the span of $\Trees_\circ(R)$ (or rather an 
extension of it involving again labels $\Labo$ on nodes) invariant. 
This eventually allows us to build, for large classes of subcritical stochastic PDEs,
regularity structures allowing to formulate them, endowed with a large enough group
of automorphisms to perform the renormalisation procedures required to give them
canonical meaning.

\begin{definition}\label{partial^mN}
Given $N = ((\Labhom_1,k_1),\ldots,(\Labhom_n,k_n)) \in \CN$ and $m \in \N^d$, we define
$\d^m N \subset \CN$ as the set of all $n$-tuples of the form
$((\Labhom_1,k_1+m_1),\ldots,(\Labhom_n,k_n+m_n))$ where the $m_i\in\N^d$ are such
that $\sum_i m_i = m$.
\end{definition}
Furthermore, we introduce the following \textit{substitution} operation on $\CN$.
Assume that we are given $N \in \CN$, $M \subset N$ and an element $\tilde M \in \hat \CP(\CN)$ 
which has the same size as $M$. In other words, if $M = (r_1,\ldots,r_\ell)$,
one has
$\tilde M = (\tilde M_1,\ldots,\tilde M_\ell)$ with $\tilde M_i \in \CN$. Then, writing
$N = M \sqcup \bar N$, we define
\begin{equ}[e:subst]
\CR_M^{\tilde M} N \eqdef \bar N \sqcup \tilde M_1 \sqcup \ldots \sqcup \tilde M_\ell\;.
\end{equ}

\begin{definition}\label{barCN}
Given a rule $R$, for any tree $T_\Labe^\Labn \in \Trees_\circ(R)$ we associate to each edge $e\in E_T$ a set $\bar \CN(e) \subset \CN$ in the following recursive way.
If $e = (x,y)$ and $y$ is a leaf, namely the node-type $\CN(y)$ of the vertex $y$ is equal to the empty word $()\in\CN$, then we set
\begin{equ}
\bar \CN(e) \eqdef R(\Labhom(e))\;.
\end{equ}
Otherwise, writing $(e_1,\ldots,e_\ell)$ the incoming edges of $y$, namely $e_i=(y,v_i)$, we define 
\begin{equs}
\bar \CN(e) \eqdef \{\CR_{\CN(y)}^M N\,:& \ \CN(y) \subset N \in R(\Labhom(e)),
\ M\in \bar\CN(e_1)\times\cdots\times\bar\CN(e_\ell)\}\;. 
\end{equs}
Finally, we define for every node $y\in N_T$ a set $\CM(y)\subset\hat \CP(\CN)$ by $\CM(y)\eqdef\{()\}$ if $y$ is a leaf, and
\[
\CM(y)\eqdef\bar\CN(e_1)\times\cdots\times\bar\CN(e_\ell)
\]
if $(e_1,\ldots,e_\ell)$ are the outgoing edges of $y$.
\end{definition}
It is easy to see that, if we explore the tree \textit{from the leaves down}, this 
specifies $\bar \CN(e)$ and $\CM(y)$ uniquely for all edges and nodes of $T$.

\begin{definition}\label{-compl}
A rule $R$ is $\ominus$-\textit{complete} with respect to a fixed scaling $\s$ if, whenever $\tau \in \Trees_-(R)$ and $\Labhom \in \Lab$ are
such that there exists $N \in R(\Labhom)$ with
$\CN(\rho_\tau) \subset N$, one also has
\begin{equ}
\d^{m} \bigl(\CR_{\CN(\rho_\tau)}^M N\bigr) \subset R(\Labhom)\;,
\end{equ}
for every $M \in \CM(\rho_\tau)$ and for every multiindex $m$
with $|m|_\s + |\tau|_\s < 0$.
\end{definition}

At first sight, the notion of $\ominus$-completeness might seem rather 
tedious to verify and potentially quite restrictive. Our next result shows
that this is fortunately not the case, at least when we are in the subcritical situation.

\begin{proposition}\label{prop:completion}
Let $R$ be a normal subcritical rule. Then, there exists a normal subcritical rule
$\bar R$ which is $\ominus$-complete and extends $R$ in the sense that
$R(\Labhom) \subset \bar R(\Labhom)$ for every $\Labhom \in \Lab$.
\end{proposition}

\begin{proof}
Given a normal subcritical rule $R$, we define a new rule $\CQ R$ by setting
\begin{equ}[e:defR-]
\bigl(\CQ R\bigr)(\Labhom) = R(\Labhom) \cup \bigcup_{\tau \in \Trees_-(R)} R_-(\Labhom;\tau)\;,
\end{equ} 
where $R_-(\Labhom;\tau)$ is the union of all collections of node types of the type
\begin{equ}
\hat N \in \d^{m} \bigl(\CR_{\CN(\rho_\tau)}^M N\bigr)\;,
\end{equ}
for some $N \in R(\Labhom)$ with $\CN(\rho_\tau) \subset N$,
some $M \in \CM(\rho_\tau)$, and some multiindex $m$ with
$|m|_\s + |\tau|_\s \le 0$. 
Since $ \bigl(\CQ R\bigr) (\Labhom) \supset R(\Labhom)$ and $\Trees_-(R)$ is finite by
Proposition~\ref{prop:finite}, this is again a valid rule.
Furthermore, by definition, a rule $R$ is $\ominus$-complete
if and only if $\CQ R = R$.

We claim that the desired rule $\bar R$ can be obtained by setting
\begin{equ}
\bar R(\Labhom) = \bigcup_{n \ge 0} \bigl(\CQ^n R\bigr)(\Labhom)\;. 
\end{equ}
It is straightforward to verify that $\bar R$ is $\ominus$-complete.
(This follows from the fact that the sequence of rules $\CQ^n R$ is
increasing and $\CQ$ is closed under increasing limits.)

It remains to show that $\bar R$ is again normal and subcritical. 
To show normality, we note that if $R$ is normal, then $\CQ R$ is again
normal. This is because, by Definition~\ref{barCN}, the sets $\bar \CN(e)$ used to build 
$\CM(\rho_\tau)$ also 
have the property that if $N \in \bar \CN(e)$ and $M \subset N$, then
one also has $M\in \bar \CN(e)$.
As a consequence, $\CQ^n R$ is normal for every $n$, from which the normality
of $\bar R$ follows.

To show that $\bar R$ is subcritical,  we 
first recall that by Remark~\ref{min_plus}, for $\kappa$ as in \eqref{e:regkappa},
we can find a maximal function $\reg_\kappa \colon \Lab \to \R$ such that
\begin{equ}[e:regkappa2]
\reg_\kappa(\Labhom) = |\Labhom|_\s - \kappa + \inf_{N \in R(\Labhom)} \reg_\kappa(N)\;.
\end{equ}
Furthermore, the extension of $\reg_\kappa$ to node types given by  
\eqref{convention} is such that, for every node type $N$ and every multiindex $m$, one has
\begin{equ}[e:regmm]
\reg_\kappa(\d^m N)=\reg_\kappa(N)- |m|_\s\;.
\end{equ}
(We used a small abuse of notation here since $\d^m N$ is really a collection of node
types. Since $\reg_\kappa$ takes the same value on each of them, this creates no
ambiguity.)

We claim that the same function $\reg_\kappa$ also satisfies \eqref{e:defreg} 
for the larger rule $\CQ R$. In view of \eqref{e:regkappa2} and of the definition \eqref{e:defR-} of $\CQ R$, it is enough to prove that
\begin{equ}[e:defreg3]
\reg_\kappa(\Labhom) \leq |\Labhom|_\s - \kappa +\reg(N), \qquad \forall \, N\in  \bigcup_{\tau \in \Trees_-(R)} R_-(\Labhom;\tau).
\end{equ}

Arguing by induction as in the proof of \eqref{reg1}, one can first show the following. Let $\sigma \in \Trees_\circ(R)$ any every planted tree whose trunk $e$ has edge type $(\Labhom,k)$. Then one has the bound
\begin{equ}[e:boundReg]
\reg_\kappa(\Labhom,k) \le |\sigma|_\s + \reg_\kappa(G)\;, \qquad \forall \, G\in\bar\CN(e).
\end{equ}
Indeed, if $e$ is the only edge of $\sigma$, then $\bar\CN(e)=R(\Labhom)$ and 
by \eqref{e:regkappa2} 
\[
\reg_\kappa(\Labhom,k)\leq |\Labhom|_\s-|k|_\s+\reg_\kappa(G)=|\sigma|_\s+\reg_\kappa(G).
\]
If now $e=(x,y)$ and $(e_1,\ldots,e_\ell)$ are the outgoing edges of $y$, then $\bar\CN(e)$ is the set of all $\CR_{\CN(y)}^M N$ with $\CN(y) \subset N \in R(\Labhom(e))$ and $M=(M_1,\ldots,M_\ell)$  with $M_i \in \bar\CN(e_i)$. By the induction hypothesis,
\[
\reg_\kappa(\CN(y))\leq \sum_{i=1}^\ell \Big[|\sigma_i|_\s+\reg_\kappa(M_i)\Big]
\]
where $\sigma_i$ is the largest planted subtree of $\sigma$ with trunk $e_i$. 
Then
\begin{equ}
\reg_\kappa(\CR_{\CN(y)}^M N) =  \reg_\kappa(N)-\reg_\kappa(\CN(y))+\sum_{i=1}^\ell \reg_\kappa(M_i) \geq \reg_\kappa(N)-\sum_{i=1}^\ell|\sigma_i|_\s.
\end{equ}
Combining this with \eqref{e:regkappa2} we obtain, since $|\Labhom|_\s - |k|_\s+\sum_{i=1}^\ell|\sigma_i|_\s=|\sigma|_\s$,
\[
\reg_\kappa(\Labhom,k) \leq |\Labhom|_\s - |k|_\s + \reg(N)\leq |\sigma|_\s+\reg_\kappa(\CR_{\CN(y)}^M N)
\]
and \eqref{e:boundReg} is proved.

We prove now \eqref{e:defreg3}.  Let $\tau\in\Trees_-(R)$, $N \in R(\Labhom)$ with $\CN(\rho_\tau) \subset N$, $M=(M_1,\ldots,M_\ell) \in \CM(\rho_\tau)$, and $m\in\N^d$ with $|m|_\s + |\tau|_\s \le 0$. Let $\tau = \tau_1\ldots \tau_\ell$ be the decomposition of $\tau$ into planted trees.
Recalling \eqref{e:regmm} and Definitions~\ref{barCN} and~\ref{partial^mN}, 
we have
\begin{equs}
\reg_\kappa\big(\d^{m} \bigl(\CR_{\CN(\rho_\tau)}^M N\bigr)\big) 
& \ = \reg_\kappa\left(\CR_{\CN(\rho_\tau)}^M N\right) - |m|_\s
\\ & \ =\reg_\kappa(N)+\sum_{i=1}^\ell \Big[\reg_\kappa(M_i)-\reg_\kappa(s_i)\Big] - |m|_\s\;,
\end{equs}
where $s_i$ is the edge type of the trunk of $\tau_i$.
Combining this with \eqref{e:boundReg} yields
\begin{equs}
\reg_\kappa\big(\d^{m} \bigl(\CR_{\CN(\rho_\tau)}^M N\bigr)\big) 
\ge \reg_\kappa(N) - |m|_\s - |\tau|_\s \ge \reg_\kappa(N) \;,
\end{equs}
with the last inequality a consequence of the condition $|m|_\s + |\tau|_\s \le 0$. This proves \eqref{e:defreg3}. 

We conclude that \eqref{e:regkappa2} also holds when considering $N \in (\CQ R)(\Labhom)$,
thus yielding the desired claim. Iterating this, we conclude
that $\reg_\kappa$ satisfies \eqref{e:defreg}
for each of the rules $\CQ^n R$ and therefore also for $\bar R$ as required.
\end{proof}

\begin{definition}\label{def:complete}
We say that a subcritical rule $R$ is \textit{complete} (with respect to a fixed scaling $\s$) if it is both normal and $\ominus$-complete. If $R$ is only normal, we call the rule $\bar R$ constructed in the proof of Proposition~\ref{prop:completion} the completion of $R$.
\end{definition}

\subsection{Three prototypical examples}
\label{sec:SPDERules}

Let us now show how, concretely, a given stochastic PDE (or system thereof) gives rise
to a rule in a natural way.
Let us start with a very simple example, the KPZ equation formally given by
\begin{equ}
\partial_{t} u = \Delta u + \left( \partial_x u \right)^{2} + \xi\;. 
\end{equ}
One then chooses the set $\Lab$ so that it has one element for each noise process and one
for each convolution operator appearing in the equation. In this case, using the variation of constants
formula, we rewrite the equation in integral form as
\begin{equ}
 u = P u_0 + P * \one_{t > 0}\bigl(\left( \partial_x u \right)^{2} + \xi\bigr)\;,
\end{equ}
where $P$ denotes the heat kernel and $*$ is space-time convolution.
We therefore need two types in $\Lab$ in this case, which we call 
$\{\Xi,\CI\}$ in order to be consistent with \cite{reg}.

We assign degrees to these types 
just as in \cite{reg}. In our example, the underlying space-time dimension is $d=2$
and the equation is parabolic, so we fix the parabolic scaling
$\s = (2,1)$ and then assign to $\Xi$ a degree just below the exponent of self-similarity 
of white noise under the scaling $\s$,
namely $|\Xi|_\s = -{3\over 2} - \kappa$ for some small $\kappa > 0$. We also assign to each
type representing a convolution operator the degree corresponding to the amount by which
it improves regularity in the sense of \cite[Sec.~4]{reg}. 
In our case, this is given by $|\CI|_\s = 2$.

It then seems natural to assign to such an equation a rule $\tilde R$ by
\begin{equs} 
  \tilde R(\Xi)  = \lbrace () \rbrace, \; \quad 
    \tilde R( \CI) = \{(\Xi), (\CI_1,\CI_1)\}\;,
\end{equs}
where $\CI_1$ is a shorthand for the edge type $(\CI,(0,1))$ and we simply write $\Labhom$ as
a shorthand for the edge type $(\Labhom,0)$.
In other words, for every noise type $\Labhom$, we set $\tilde R(\Labhom) = \{()\}$
and for every kernel type $\Labhom$ we include one node type into $\tilde R(\Labhom)$ 
for each of the monomials in our equation that are convolved with the corresponding kernel.
The problem is that such a rule is not normal. Therefore we define rather
\begin{equs} 
  R(\Xi)  = \lbrace () \rbrace, \; \quad 
    R( \CI) = \{(), (\Xi), (\CI_1), (\CI_1,\CI_1)\}\;,
\end{equs} 
which turns out to be normal and complete. It is simple to see that the function $\reg_\kappa:\{\Xi,\CI\}\to\R$
\[
\reg_\kappa(\Xi)=-{3\over 2}-2\kappa, \qquad \reg_\kappa(\CI)={1\over 2}-3\kappa,
\]
makes $R$ subcritical for sufficiently small $\kappa>0$.
%We can see that we obtain $R$ from $\tilde R$ by applying the 
%Let us show how the algorithm given in Remark~\ref{min_plus} easily shows that this
%rule is indeed subcritical for sufficiently small values of $\kappa$. 
%Applying this algorithm, we obtain for $n = 1,2$
%\begin{equs}[2]
%\reg_\kappa^1(\Xi) &= -{3\over 2}-2\kappa\;,&\quad  \reg_\kappa^1(\CI) &= 2-\kappa\;,\\
%\reg_\kappa^2(\Xi) &= -{3\over 2}-2\kappa\;,&\quad  \reg_\kappa^2(\CI) &= {1\over 2}-3\kappa\;,
%\end{equs}
%and then $\reg_\kappa^n = \reg_\kappa^2$ for $n \ge 2$, provided that $\kappa \le {1\over 8}$. 
%The fact that the algorithm stops after
%finitely many steps shows that the rule $R$ is indeed subcritical.

One can also consider systems of equations. Consider for example the system of coupled  KPZ  equations
formally given by
\begin{equs}
\partial_{t} u_1 & = \Delta u_1 + \left( \partial_x u_1 \right)^{2} + \xi_1\;, \\
\partial_{t} u_2 & = \nu\Delta u_2 + \left( \partial_x u_2 \right)^{2} + \Delta u_1 + \xi_2\;. 
\end{equs}
In this case, we have two noise types $\Xi_{1,2}$ as well as two kernel types,
which we call $\CI$ for the heat kernel with diffusion constant $1$ and $\CI^\nu$
for the heat kernel with diffusion constant $\nu$. There is some ambiguity in this case
whether the term $\Delta u_1$ appearing in the second equation should be considered
part of the linearisation of the equation or part of the nonlinearity. In this
case, it turns out to be more convenient to consider this term as part of the nonlinearity,
and we will see that the corresponding rule is still subcritical thanks to the triangular
structure of this system. 

Using the same notations as above, the normal and complete rule $R$ naturally associated with this system of equations is given by
\begin{equs}[0]
R(\Xi_i) = \{()\}, \; \quad  R( \CI) = \lbrace (), (\Xi^1) , (\CI_1),  (\CI_1,\CI_1) \rbrace \\
R(\CI^{\nu}) = \lbrace (), (\Xi_2), (\CI_1^{\nu}),  (\CI_1^\nu,\CI_1^\nu), \; (\CI_2) \rbrace.
\end{equs}
In this case, we see that $R$ is again subcritical for sufficiently small $\kappa>0$ with
\begin{equ}
\reg_{\kappa}(\Xi_i) = -3/2 - 2 \kappa\;,\quad 
\reg_{\kappa}(\CI) = 1/2 - 3\kappa\;,\quad
\reg_{\kappa}(\CI^\nu) = 1/2 - 4\kappa\;.
\end{equ}

Our last example is given by the following generalisation of the KPZ equation:
\begin{equs}
\partial_{t} u = \Delta u + g(u) \left( \partial_x u \right)^2 + h(u) \partial_{x} u + k(u) + f(u) \xi\;,
\end{equs}
which is motivated by \eqref{e:chris} above, see \cite{proc}.
In this case, the set $ \Lab $ is again given by $\{\Xi, \; \CI\}$, just as in the
case of the standard KPZ equation.
Writing $[\CI]_{\ell} $ as a shorthand for $\CI,...,\CI$ where $ \CI $ is repeated $ \ell $ times,
the rule $ R $ associated to this equation is given by
\begin{equs}
R(\Xi) = \lbrace () \rbrace, \; R(\CI) = \lbrace ([\CI]_{\ell}), \;   ([\CI]_{\ell},\CI_1), \;    ([\CI]_{\ell},\CI_1,\CI_1), \;  ([\CI]_{\ell},\Xi), \; \ell \in \N \rbrace\;.
\end{equs}
Again, it is straightforward to verify that $ R $ is subcritical and that one can use the same
map $\reg_\kappa$ as in the case of the standard KPZ equation. Even though in this case there 
are infinitely many node types appearing in $R(\CI)$, this is not a problem 
because $\reg_\kappa(\CI) > 0$,
so that repetitions of the symbol $\CI$ in a node type only increase the corresponding 
degree.

\subsection{Regularity structures determined by rules}
\label{sec:quotient}

Throughout this section, we assume that we are given 
\begin{itemize}
\item a \textit{finite} type set $\Lab$
together with a scaling $\s$ and degrees $|\cdot|_\s$ as in 
Definition~\ref{def:scaling},
\item a normal rule $R$ for $\Lab$ which is both subcritical and complete, in the sense 
of Definition~\ref{def:complete},
\item the integer $d\geq 1$ which has been fixed at the beginning of the paper.
\end{itemize}

We show that the above choices, when combined with the structure built in
Sections~\ref{sec3} and~\ref{sec4}, 
yield a natural substructure with the same algebraic properties (the
only exception being that the subspace of $\CH_\circ$ we consider is not an algebra in general),
but which is sufficiently small to yield a regularity structure. Furthermore,
this regularity structure contains a very large group of automorphisms, unlike
the slightly smaller structure described in \cite{reg}. The reason for this is the additional flexibility
granted by the presence of the decoration $\Labo$, 
which allows to keep track of the degrees of the subtrees contracted by the action of $\CG_1$.

\begin{definition}\label{def:D}
We define for every $\tau=(G,\Labn',\Labe')\in\Trees$  
and every node $x\in N_G$ 
a set $D(x,\tau) \subset \Z^d \oplus \Z(\Lab)$ by postulating that $\alpha\in D(x,\tau)$ 
if there exist 
\begin{itemize}
\item $\sigma=(F,\Labn,\Labe)\in\Trees$
\item $A\subset F$ is a subtree such that $\sigma$ conforms to the rule $R$ at every node $y \in A$
\item functions $\Labn_A \colon N_A \to \N^d$ with $\Labn_A\leq \Labn\restr N_A$ and $\eps_A^F\colon \d(A,F)\to \N^d$
\end{itemize}
such that $(A,0,\Labn_A+\pi\varepsilon_A^F,0,\Labe)\in\Trees_-(R)$ (see Definition~\ref{def_space})
and 
\begin{equ}[e:reprLambda]
(G,\un{\{x\}},\Labn',\alpha\un{\{x\}},\Labe')=\CK_1(F,\un{A},\Labn - \Labn_A, \Labn_A + \pi(\varepsilon_A^F-\Labe_\emptyset^A), \Labe_A^F + \varepsilon_A^F)
\end{equ}
and in particular 
\[
\alpha=\sum_{N_A}\left(\Labn_A + \pi(\varepsilon_A^F-\Labe_\emptyset^A)\right).
\]
%
%
%\item $\tau = (F,\Labn,\Labe) \in \Trees_-(R)$,
%\item $M \in R(\Labhom)$ with $\CN(\rho_\tau) \subset M$ and $\bar M \in \CM(\rho_\tau)$,
%\item $m \in \N^d$ with $|\tau|_\s + |m|_\s < 0$,
%\end{itemize}
%such that
%\begin{equ}
%N \in \d^{m} \bigl(\CR_{\CN(\rho_\tau)}^{\bar M} M\bigr)\;
%\qquad\text{and}\qquad
%\alpha = \sum_{e \in E_F} \bigl(\Labhom(e) - \Labe(e)\bigr) + \sum_{x \in N_F}\Labn(x) + m\;.
%\end{equ}
\end{definition}
We define $\CS:\Tra\to \Trees \subset \Tra$ by $\CS(F,\hat F, \Labn,\Labo,\Labe)\eqdef
(F , \Labn,\Labe)$.
\begin{definition}\label{D(labhom,N)}
We denote by $\Lambda=\Lambda(\Lab,R,\s,d)$ the set of all $\tau = (F,\hat F, \Labn,\Labo,\Labe)\in\Tra$ such that
$\tau = \CK_1\tau$ and,
for all $x\in N_F$, exactly one of the following two mutually exclusive statements holds.
\begin{itemize}
\item One has $\hat F(x) \in \{0,2\}$ and $\Labo(x)=0$.
\item One has $\hat F(x) = 1$ and $\Labo(x)\in D(x,\CS\tau)$.
%there exists a node $y \in N_F$ and an
%edge $e = (y,x) \in E_F$ such that $\Labo(x) \in D(\Labhom(e),\CN(x))$, 
%where $\CN(x)$ is as in Definition~\ref{def:conform}.
%\item One has $\hat F(x) = 1$, $x$ is a root of $F$, and there exists $\Labhom \in \Lab$ such that
%$\Labo(x) \in D(\Labhom,\CN(x))$.
\end{itemize}
\end{definition}

\begin{lemma}\label{lem:contraction}
Let $\sigma = (F,\hat F, \Labn,\Labo,\Labe)\in \Lambda$ and $A \in\Adm_1(F,\hat F)$ 
be a subforest
such that $\sigma$ conforms to the rule $R$ at every vertex $x \in A$ and  fix
functions $\Labn_A \colon N_A \to \N^d$ with $\Labn_A\leq \Labn\restr N_A$ and $\eps_A^F\colon \d(A,F)\to \N^d$. 
Assume furthermore that for each connected component $B$ of $A$, we have 
$(B,0,\Labn_A+\pi\varepsilon_A^F,0,\Labe)\in\Trees_-(R)$, see Definition~\ref{def_space}.
Then
the element
\begin{equ}[e:reprLambda2]
\tau = \CK_1 (F, \hat F \cup_1 A,\Labn - \Labn_A, \Labn_A + \pi(\varepsilon_A^F-\Labe_\emptyset^A), \Labe_A^F + \varepsilon_A^F)
\end{equ}
also belongs to $\Lambda$.

Conversely, every element $\tau$ of $\Lambda$ is of the form \eqref{e:reprLambda2} for 
an element $\sigma$ with $\hat F(x) \in \{0,2\}$ and $\Labo \equiv 0$.
\end{lemma}

\begin{proof}
Let us start by showing the last assertion. Let $\tau=(G,\hat G,\Labn',\Labo',\Labe')\in\Lambda$ and $\{x_1,\ldots,x_n\}\subset N_G$ all nodes is such that $\hat G(x_i)=1$. Let us argue by
recurrence over  $i\in\{1,\ldots,n\}$. By Definition \ref{def:D} one can write 
\begin{equs}
\ &(G,\hat G\un{\{x_1,\ldots,x_i\}},\Labn',\Labo'\un{\{x_1,\ldots,x_i\}},\Labe')=\CK_1\sigma_i
\\ & = \CK_1(F_i,\un{A_i},\Labn - \Labn_{A_i}, \Labn_{A_i} + \pi(\varepsilon_{A_i}^{F_i}-\Labe_\emptyset^{A_i}), \Labe_{A_i}^{F_i} + \varepsilon_{A_i}^{F_i})
\end{equs}
as in \eqref{e:reprLambda}. Setting $F=F_n$ and $A=A_n$ we have the required representation. 

Now the first assertion follows easily from the second one.
\end{proof}

We now define spaces of coloured forests $\tau = (F,\hat F, \Labn,\Labo,\Labe)$ such that $(F,0,\Labn,0,\Labe)$ is compatible with the rule $R$ in a suitable sense, and 
such that $\tau \in \Lambda$.
\begin{definition}\label{def:CT}
Recalling Definition~\ref{def_space} and Remark~\ref{rem:H_circ}, we define the bigraded spaces
\begin{equs}[2]
\hat\CT_{+}^{\ex}& = \scal{B_+} \subset \hat \CH_2 \;,&\qquad B_+ &\eqdef \{\tau\in\hat H_2 \, : \,  \tau\in{\Lambda}\ \ \& \ \ \CS \tau\in \Trees_2(R) \} \;,\\
\hat\CT_{-}^{\ex}&= \scal{B_-}\subset \CH_1 \;,&\qquad B_- &\eqdef \{\tau\in H_{1}\, : \,  \tau\in{\Lambda}\ \ \& \ \ \CS \tau \in \Trees_1(R)\} \;,\\
\CT^{\ex} &= \scal{B_\circ}\subset \CH_\circ\;,&\qquad B_\circ& \eqdef \{\tau\in H_\circ\, : \, \tau\in{\Lambda}\ \ \& \ \ \CS \tau\in \Trees_\circ(R)\}\;.
\end{equs}
\end{definition}

\begin{remark}\label{extended}
The superscript ``ex'' stands for ``extended'', see Section~\ref{sec:reduced} below for an explanation
of the reason why we choose this terminology.
The identification of these spaces as suitable subspaces of $\hat \CH_2$, $\CH_1$ and $\CH_\circ$
is done via the canonical basis \eqref{e:isoH}.
\end{remark}

Note that both $\hat \CT_{-}^{\ex}$  and $\hat \CT_{+}^{\ex}$ are algebras for the products
inherited from $\CH_1$ and $\hat \CH_2$ respectively. On the other hand, $\CT^{\ex} $ is in general 
not an algebra anymore. 

\begin{lemma}\label{lem:5.28}
We have 
\begin{equ}
\Delta_1\colon \CT^{\ex} \to \hat \CT_{-}^{\ex}\hotimes \CH_{\circ}\;,\quad
\Delta_1\colon \hat \CT_{-}^{\ex} \to \hat \CT_{-}^{\ex}\hotimes \CH_{1}\;,\quad
\Delta_1\colon \hat \CT_{+}^{\ex} \to \hat \CT_{-}^{\ex}\hotimes \hat\CH_{2}\;,
\end{equ}
as well as
$\Delta_2\colon \CH \to \CH\hotimes \hat \CT_{+}^{\ex}$ for $\CH \in \{\CT^\ex, \hat \CT_+^\ex\}$.
Moreover, $\hat \CT_{+}^{\ex}$ is a Hopf subalgebra of $\hat\CH_{2}$ and $\CT^{\ex}$ is a right Hopf-comodule over $\hat \CT_{+}^{\ex}$ with coaction $\Delta_2$.
\end{lemma}

\begin{proof}
By the normality of the rule $R$, if a tree conforms to $R$ then any of its subtrees does too. On the other hand, contracting subforests can generate non-conforming trees in the case of $\Delta_1$, while, since $\Delta_2$ extracts only subtrees at the root, completeness of the rule implies that this can not happen in the case of $\Delta_2$, thus showing that the maps $\Delta_i$ do indeed
behave as claimed.

The fact that $\hat \CT_{+}^{\ex}$ is in fact a Hopf algebra, namely that the antipode $\hat\CA_2$ of $\hat\CH_2$ leaves $\hat \CT_{+}^{\ex}$ invariant, can be shown by induction using \eqref{e:defA+0} and Remark~\ref{rem:A}.
\end{proof}
Note that $\hat \CT_{-}^{\ex}$ is a sub-algebra but in general {\it not} a sub-coalgebra of $\CH_{1}$
(and \textit{a fortiori} not a Hopf algebra).
Recall also that, by Lemma~\ref{lem:grading}, the grading $|\cdot|_-$ of 
Definition~\ref{homog} is well defined on $\hat\CT_{-}^{\ex}$ and on $\CT^{\ex}$,  and 
that $|\cdot|_+$ is well defined on both $\hat\CT_{+}^{\ex}$ and $\CT^{\ex} $. 
Furthermore, these gradings are preserved by the corresponding products and coproducts.

\begin{definition}\label{CT^ex_pm}
Let $\CJ_\mp \subset \hat\CT_\pm^\ex$ be the ideals given by 
\begin{equs}[e:defJ]
\CJ_- = &\scal{\{\tau \in B_+\,:\, \tau = \JJ\hat \CK_2(\sigma \cdot \bar\sigma) \;,\quad \sigma,\bar\sigma \in B_+,\ \sigma \neq \one_2,\  |\sigma|_+ \le 0\}}\;,
\\ \CJ_+ = &\langle\{\tau \in B_-\,:\, \tau = \CK_1(\sigma \cdot \bar\sigma) \;,\quad \sigma,\bar\sigma \in B_-,\ 
 \qquad\qquad
\\ & \qquad\qquad(\sigma \neq \one_1\quad \& \quad |\sigma|_- \geq 0) \quad \text{or} \quad (\sigma = \CI_k^\Labhom(\sigma')\quad \& \quad  |\Labhom|_\s>0)\}\rangle  \;.
\end{equs}
Then, we set
\begin{equ}\label{eq:Hopf}
\CT_{-}^{\ex} \eqdef \hat \CT_{-}^{\ex} / \CJ_{+}\;, \qquad
\CT_{+}^{\ex} \eqdef \hat \CT_{+}^{\ex} / \CJ_{-}\;,
\end{equ}
with canonical projections $\proj^\ex_\pm:\hat\CT_{\pm}^{\ex}\to \CT_{\pm}^{\ex}$. Moreover, we define the operator $ \CJ^{\Labhom}_k : \CT^{\ex} \rightarrow  \CT^{\ex}_+ $ as $  \CJ^{\Labhom}_k =  \proj^\ex_+ \circ \hat \CJ^{\Labhom}_k $.
\end{definition}
With these definitions at hand, it turns out that the map $(\proj_-^\ex\otimes\id)\Delta_1$ is much
better behaved. Indeed, we have the following.
\begin{lemma}
The map $\Deltam_\ex = (\proj_-^\ex\otimes\id)\Delta_1$ satisfies
\begin{equ}
\Deltam_\ex\colon \CH \to \CT_{-}^{\ex}\hotimes \CH\;,\qquad
\text{for $\CH \in \{\hat \CT_-^\ex, \CT^\ex, \hat\CT_+^\ex\}$.} 
\end{equ}
\end{lemma}
\begin{proof}
This follows immediately from Lemma~\ref{lem:5.28}, combined with 
the fact that completeness of $R$ has beed defined in Definition \ref{-compl}
in terms of extraction of $\tau\in\Trees_-(R)$, which in particular means that
$|\tau|_\s=|\tau|_-<0$.
%Lemma~\ref{lem:contraction}.
\end{proof}

Analogously to Lemma~\ref{lem:biideal} we have

\begin{lemma}\label{lem:biideal2}
We have
\begin{equ}[e:projIden]
(\proj_-^\ex\otimes \proj_-^\ex)\Delta_1 \CJ_+ = 0\;,\quad
(\proj_-^\ex\otimes \proj_+^\ex)\Delta_1 \CJ_- = 0\;,\quad
(\proj_+^\ex\otimes \proj_+^\ex)\Delta_2 \CJ_- = 0\;.
\end{equ}
\end{lemma}
\begin{proof}
We note that the degrees $|\cdot|_\pm$ have the following compatibility properties with the operators $\Delta_i$. For $0<i\leq j\leq 2$, $\tau\in\Tra_j$ and $\Delta_i\tau=\sum \tau^{(1)}_i\otimes\tau^{(2)}_i$ 
(with the summation variable suppressed), one has
\begin{equ}[e:degreeDelta]
|\tau^{(1)}_1|_- +|\tau^{(2)}_1|_-=|\tau|_-\;, \quad |\tau^{(2)}_1|_+=|\tau|_+\;,\quad
|\tau^{(1)}_2|_+ +|\tau^{(2)}_2|_+=|\tau|_+ \;.
\end{equ}
The first identity of \eqref{e:projIden} then follows from the 
first identity of \eqref{e:degreeDelta} and from the following remark: if $B_-\ni\tau = \CI_k^\Labhom(\sigma)$,
then for each term appearing in the sum over $A\in\Adm_1$ in the expression \eqref{def:Deltabar} for $\Delta_1\tau$, one has two possibilities: 
\begin{itemize}
\item either $A$ does not contain the edge incident
to the root of $\tau$, and then the second factor is a tree with only one edge incident to its root,
\item or $A$ does contain the edge incident
to the root, in which case the first factor contains one connected component of that type.
\end{itemize}
The second identity of \eqref{e:projIden} follows from the second identity
of \eqref{e:degreeDelta} combined with the fact that, for $\tau  \in \Tra_2$,
$\Delta_1 \tau$ contains no term of the form $\sigma\otimes \one_2$, even when 
quotiented by $\ker(\JJ\hat \CK_2)$. The third identity of \eqref{e:projIden}
finally follows from the third identity of \eqref{e:degreeDelta}, combined with the fact
that if $\tau \in B_+\setminus\{\one_2\}$ with $|\tau|_+\le 0$, 
then the term $\one_2 \otimes \one_2$ does not appear in the
expansion for $\Delta_2 \tau$.
\end{proof}

As a corollary, we have the following.

\begin{corollary}\label{cor:domains}
The operator $\Deltam_\ex = (\proj_-^\ex\otimes \id)\Delta_1$ is well-defined as a map
\begin{equ}
\Deltam_\ex \colon \CH \to \CT_-^\ex \hotimes \CH\;,\qquad
\text{for $\CH \in \{\hat \CT_-^\ex, \CT_-^\ex, \CT^\ex, \CT_+^\ex, \hat\CT_+^\ex\}$.} \end{equ}
Similarly, the operator $\Deltap_\ex = (\id \otimes \proj_+^\ex ) \Delta_2$ is well-defined as a map
\begin{equ}
\Deltap_\ex \colon \CH \to \CH \hotimes \CT_+^\ex\;,\qquad
\text{for $\CH \in \{\CT^\ex, \CT_+^\ex, \hat\CT_+^\ex\}$.} 
\end{equ}
\end{corollary}
\begin{remark}
The operators $\Delta_\ex^{\pm}$ of Corollary~\ref{cor:domains} are now given by {\it finite sums}
so that for all of these choices of $\CH$, the operators $\Deltam_\ex$ and $\Deltap_\ex$ actually
map $\CH$ into $\CT_-^\ex \otimes \CH$ and $\CH \otimes \CT_+^\ex$ respectively. 
\end{remark}

\begin{proposition}\label{prop:Hopf+}
There exists an algebra morphism $\CA^{\ex}_+ \colon \CT^{\ex}_+ \to \CT^{\ex}_+$ so that 
$(\CT^{\ex}_+,\CM,\br\Deltapp_{\ex},\one_2,\one^\star_2,\CA^{\ex}_+)$, where $\CM$ is the tree product \eqref{odot}, is a Hopf algebra. Moreover the map $\Deltap_{\ex}:\CT^{\ex} \to \CT^{\ex} \otimes \CT^{\ex}_+$, turns
$\CT^{\ex}$ into a right comodule for $\CT^{\ex}_+$ with counit $\one^\star_2$.
\end{proposition}
\begin{proof}
We already know that $\hat\CT^{\ex}_+$ is a Hopf sub-algebra of $\hat\CH_2$ with antipode $\hat\CA_2$ satisfying \eqref{e:defA+0}. Since $\CJ_-$ is a bialgebra ideal by Lemma~\ref{lem:biideal2}, the 
first claim follows from \cite[Thm~1.(iv)]{Quotients}.

The fact that $\Deltap_{\ex}:\CT^{\ex} \to \CT^{\ex} \otimes \CT^{\ex}_+$ is a co-action and turns $\CT^{\ex}$ into a right comodule for $\CT^{\ex}_+$ follows from the coassociativity of $\Delta_2$.
\end{proof}

\begin{proposition}\label{prop:Hopf-}
There exists an algebra morphism $\CA^{\ex}_- \colon \CT^{\ex}_- \to \CT^{\ex}_-$ so that 
$(\CT^{\ex}_-, \cdot,\br \Deltamm_{\ex},\one_1,\one^\star_1,\CA^{\ex}_-)$ is a Hopf algebra. Moreover 
the map $\Deltam_{\ex}:\CT^{\ex} \to \CT^{\ex}_- \otimes \CT^{\ex}$ turns
$\CT^{\ex}$ into a left comodule for $\CT^{\ex}_-$ with counit $\one^\star_1$.
\end{proposition}
\begin{proof}
One difference between $\CT^{\ex}_-$ and $\CT^{\ex}_+$ is that $\hat\CT^{\ex}_-$ is not in general a sub-coalgebra of $\CH_1$ and therefore it does not possess an antipode. However we can see that the antipode $\CA_1$ of $\CH_1$ satisfies for all $\tau\ne\one$
\[
\CA_1\tau= -\tau-\CM(\CA_1\otimes\id)(\Delta_1\tau-\tau\otimes \one-\one\otimes\tau),
\]
where $\CM$ is the product map.
By the second formula of \eqref{e:degreeDelta}, it follows that if $|\tau|_->0$ then 
$\CA_1\tau\in\CJ_+$ and therefore, since $\CA_1$ is an algebra morphism, $\CA_1(\CJ_+)\subseteq\CJ_+$. We obtain that $\CA_1$ defines a unique algebra morphism $\CA^{\ex}_- \colon \CT^{\ex}_- \to \CT^{\ex}_-$ which is an antipode for $\CT^{\ex}_-$.
\end{proof}

\begin{definition}\label{def:charpm}
We call $\CG_\pm^\ex$ the character group of $\CT^{\ex}_\pm$.
\end{definition}
We have therefore obtained the following analogue of Proposition~\ref{prop:alg}:

\begin{theorem}\label{CTpmexHopf}
$ $
\begin{enumerate}
\item  On $\CT^{\ex} $, we have the identity
\begin{equ}[e:propWanted1]
\CM^{(13)(2)(4)}(\Deltam_{\ex} \otimes \Deltam_{\ex} ) \Deltap_{\ex}  = (\id \otimes \Deltap_{\ex} ) \Deltam_{\ex} \;,
\end{equ}
holds, with $\CM^{(13)(2)(4)}$ as in \eqref{e:def1324}. 
The same is also true on $\CT^{\ex}_+$.
\item Let $\CH\in\{\CT^\ex,\CT^\ex_+\}$. We define a left action of $\CG_-^\ex$ on $\CH^*$ by
\[
g h(\tau)\eqdef(g\otimes h)\Deltam_\ex\tau, \qquad g\in \CG_-^\ex, \ h\in \CH^*, \ \tau\in\CH,
\]
and a right action of $\CG_+^\ex$ on $\CH^*$ by
\[
hf(\tau)\eqdef(h\otimes f)\Deltap_\ex\tau, \qquad f\in \CG_+^\ex, \ h\in\CH^*, \ \tau\in\CH.
\]
Then we have
\begin{equs}\label{e:semidirect}
g(hf) = (g h)(g f)\;, \qquad g\in \CG_-^\ex, \ f\in \CG_+^\ex, \ h\in \CH^*.
\end{equs}
\end{enumerate}
\end{theorem}
\begin{proof}
By the second identity of \eqref{e:degreeDelta}, the action of $ \Deltam_{\ex} $ preserves the degree $ |\cdot|_{+} $. In particular we have
\begin{equ}[e:crucoin]
\Deltam_{\ex} \proj_+^{\ex} = \left( \id \otimes \proj_{+}^{\ex} \right) \Deltam_{\ex}.
\end{equ}
From this property, one has:
\begin{equs}
\CM^{(13)(2)(4)}(\Deltam_{\ex} \otimes \Deltam_{\ex} ) \Deltap_{\ex} & = \CM^{(13)(2)(4)}(\Deltam_{\ex} \otimes \left(  \id \otimes \proj_{+}^{\ex} \right) \Deltam_{\ex} ) \Delta_{2} 
\\
& = \left( \proj_{-}^{\ex} \otimes \id \otimes \proj_{+}^{\ex} \right) \CM^{(13)(2)(4)}(\Delta_1 \otimes \Delta_1 ) \Delta_2
\end{equs}
and we conclude by applying the Proposition~\ref{prop:doublecoass}. Now the proof of \eqref{e:semidirect} is the same as that of 
\eqref{??} above.
\end{proof}
Formula \eqref{e:propWanted1} yields the cointeraction property see 
Remark \ref{cointera}.
\begin{remark}\label{explanation}
We can finally see here the role played by the decoration $\Labo$: were it not included, the cointeraction property \eqref{e:propWanted1} of Theorem~\ref{CTpmexHopf} would fail, since it is based upon \eqref{e:crucoin}, which itself depends on the second identity of \eqref{e:degreeDelta}. Now recall that $|\cdot|_+$ takes the decoration $\Labo$ into account, and this is what makes the second identity of \eqref{e:degreeDelta} true. See also Remark~\ref{rem:fails} below.
\end{remark}
As in the discussion following Proposition~\ref{prop:alg}, we see that 
$\CT^\ex$ is a left comodule over the Hopf algebra $\hat\CT_{12}^\ex\eqdef\CT_-^\ex\ltimes \CT_+^\ex$, with coaction
\[
\Delta_{\circ}:\CT^\ex\to\hat\CT_{12}^\ex\hattimes\CT^\ex, \qquad
\Delta_{\circ} \eqdef \sigma^{(132)} (\Deltam_\ex\otimes \CA_+^\ex)\Deltap_\ex
\]
where $\sigma^{(132)}(a\otimes b\otimes c)\eqdef a\otimes c\otimes b$ and 
$\CA_+^\ex$ is the antipode of $\CT_+^\ex$.

We define $A^\ex\eqdef\{|\tau|_+:\tau\in B_\circ\}$, where $\CT^\ex=\scal{B_\circ}$ as in Definition~\ref{def:CT}.
\begin{proposition}\label{TTex}
The above construction yields a regularity structure $\TT^{\ex}=(A^\ex,\CT^\ex,\CG_+^\ex)$ in the sense of Definition~\ref{def:regStruct}.
\end{proposition}
\begin{proof}
By %Lemma~\ref{lem:contraction}
the definitions, every element $\tau \in B_\circ$ has a representation
of the type \eqref{e:reprLambda2} for some $\sigma = (T,0,\Labn,0,\Labe) \in \Trees$. 
Furthermore, it follows from the definitions of $|\cdot|_+$ and $|\cdot |_\s$ that 
one has $|\tau|_+ = |\sigma|_\s$.
The fact that, for all $\gamma\in\R$, the set $\{a\in A^\ex: a\leq \gamma\}$ is finite then
follows from Proposition~\ref{prop:finite}.

The space $\CT^\ex$ is graded by $|\cdot|_+$
and $\CG_+^\ex$ acts on it by $\Gamma_g\eqdef (\id\otimes g)\Deltap_\ex$. 
The property \eqref{e:coundGroup} then follows from the fact $\Deltap_\ex$ preserves
the total $|\cdot|_+$-degree by the third identity in \eqref{e:degreeDelta} and all 
terms appearing in the second factor of  $\Deltap_\ex \tau - \tau \otimes \one$ 
have strictly positive $|\cdot|_+$-degree by Definition~\ref{CT^ex_pm}.
\end{proof}

\begin{remark}
Since $\CT^\ex_-$ is finitely generated as an algebra (though infinite-dimen\-sional as a vector space), its
character group $\CG_-^\ex$ is a finite-dimensional Lie group. In contrast, $\CG_+^\ex$ is
not finite-dimensional but can be given the structure of an infinite-dimensional Lie 
group, see \cite{Liegroup}.
\end{remark}

%\lorenzoText{We should comment on the fact that $\CT^\ex$ has two gradings,
%$|\cdot|_+$ and $|\cdot|_-$. We use the former to define $A^\ex$ but in Theorem~\ref{main:renormalisation} the "negative" elements are defined using the latter. 
%Since on $\CT^\ex$ we have $|\cdot|_+\leq|\cdot|_-$, there are less $|\cdot|_-$-negative
%trees than $|\cdot|_+$-negative trees. Does this have any impact on other results, like
%Theorem 10.7 in \cite{reg} where one uses negative trees? }
%\martinText{
%That theorem is really a theorem about $\CT$ on which the two degrees agree...
%It exploits the fact that everything is generated by
%integration and products which isn't really true for $\CT^\ex$.
%Of course, the sort of models we consider are determined by their values on $\CT$ so then
%it applies...
%}

\section{Renormalisation of models}
\label{sec6}

We now show how the construction of the previous sections can be 
applied to the theory of regularity structures to show that the
``contraction'' operations one would like to perform in order to renormalise
models are ``legitimate'' in the sense that they give rise to 
automorphisms of the regularity structures built in Section~\ref{sec:quotient}. Throughout this section, we are in the framework set at the beginning of Section~\ref{sec:quotient}. We furthermore impose the additional constraint
that, writing $\Lab = \Lab_- \sqcup \Lab_+$ with $\Labhom \in \Lab_+$ if and only if $|\Labhom|_\s > 0$, one has
\begin{equ}[e:oldnormal]
\Labhom \in \Lab_- \quad\Rightarrow\quad R(\Labhom) = \{()\}\;.
\end{equ}

\begin{remark}
Labels in $\Lab_+$ represent ``kernels'' while labels in $\Lab_+$ represent ``noises'', which naturally
leads to \eqref{e:oldnormal}. (We could actually have defined $\Lab_-$ by $\Lab_- = \{\Labhom \,:\, R(\Labhom) = \{()\}\}$.)
The condition that elements of $\Lab_-$ are of negative degree and those in $\Lab_+$ are of positive degree
is also natural in this context. It could in principle be weakened, which corresponds to allowing kernels with a non-integrable
singularity at the origin. This would force us to slightly modify Definition~\ref{def:modelmap} below in order to interpret these
kernels as distributions but would not otherwise lead to any additional complications. 
\end{remark}

Note now that we have a natural identification of $\CT_\pm^\ex$ with
the subspaces 
\begin{equ}
\scal{\{\tau \in B_\pm \,:\, \tau \not \in  \CJ_{\mp}^{\ex} \}} \subset \hat \CT_\pm^\ex\;.
\end{equ}
Denote by $\inj_\pm^\ex \colon \CT_\pm^\ex \to \hat \CT_\pm^\ex$\label{inj} the corresponding
inclusions, so that we have direct sum decompositions
\begin{equ}[hatandwithouthat]
\hat \CT_{\pm}^{\ex}=\CT_{\pm}^{\ex}\oplus\CJ_{\mp}^{\ex}\;.
\end{equ}
For instance, with this identification, the map $\hat\CJ^{\Labhom}_{k}\colon \CT^{\ex} \to \hat\CT_+^{\ex}$ defined in \eqref{e:CJ} associates to $\tau\in\CT^{\ex}$ an element $\hat\CJ^{\Labhom}_{k}(\tau)\in \hat\CT_+^{\ex}$ which 
can be viewed as $\CJ^{\Labhom}_{k}(\tau)\in\CT_+^{\ex}\setminus\{0\}$ if and only if its degree $|\CJ^{\Labhom}_{k}(\tau)|_+$ is positive, namely $|\tau|_+ + |\Labhom|_\s - |k|_\s > 0$.

\begin{proposition} 
Let $\CA^{\ex}_+:\CT^\ex_+\to\CT^\ex_+$ be the antipode of $\CT^\ex_+$. Then 
\begin{itemize}
\item $\CA^{\ex}_+$ is defined uniquely by the fact that $\CA^{\ex}_+ X_i = - X_i$ and for all $\CJ^{\Labhom}_{k}(\tau)\in \CT_+^{\ex}$
\begin{equ}[e:defA+]
\CA^{\ex}_+ \CJ^{\Labhom}_k(\tau) = -\sum_{\ell\in \N^d}{(-X)^\ell \over \ell!} \CM^{\ex}_+ \bigl(\CJ^{\Labhom}_{k+\ell}\otimes \CA_+^{\ex}\bigr)\Deltap_{\ex} \tau\;,
\end{equ}
where $\CM^{\ex}_+ \colon \CT^{\ex}_+\otimes \CT^{\ex}_+ \to \CT^{\ex}_+$ denotes the (tree) product and $\Deltap_\ex\colon \CT^{\ex} \to \CT^{\ex}\otimes \CT_{+}^{\ex}$.
\item On $\CT^{\ex}_+$, one has the identity
\begin{equ}[e:propWanted2]
\Deltam_{\ex}  \CA^{\ex}_+ = (\id \otimes  \CA^{\ex}_+)\Deltam_{\ex}\;.
\end{equ}
\end{itemize}
\end{proposition}
\begin{proof}
The claims follow easily from Propositions~\ref{prop:CA2} and~\ref{prop:Hopf+}.
\end{proof}

\subsection{Twisted antipodes}

We define now the operator $P_+:\hat\CT_+^{\ex}\to\hat\CT_+^{\ex}$
given on $\tau \in B_+$ by
\begin{equ}
P_+(\tau)\eqdef \left\{ \begin{array}{ll} \tau \qquad & \text{if $|\tau|_+ > 0$,}
\\ 0 & \text{otherwise.}
\end{array}\right.
\end{equ}
Note that this is quite different from the projection $\inj_+^\ex\circ \proj_+^\ex$.
However, for elements of the form $\hat \CJ_k^\Labhom(\tau)\in\hat\CT_+^{\ex}$ for some $\tau \in \CT^\ex$, we have $P_+ \hat \CJ_k^\Labhom(\tau)= (\inj_+^\ex\circ \proj_+^\ex)\bigl(\hat \CJ_k^\Labhom(\tau)\bigr)$. The difference is that $\inj_+^\ex\circ \proj_+^\ex$ is multiplicative under the tree product, while $P_+$ is not.
\begin{proposition}\label{prop:twisted+}
There exists a unique algebra morphism $\tilde\CA_+^{\ex}\colon \CT_+^{\ex} \to \hat \CT_+^{\ex}$, 
which we call the ``positive twisted antipode'', such that $\tilde\CA_+^{\ex} X_i = - X_i$ and furthermore for all $\CJ^{\Labhom}_{k}(\tau)\in \CT_+^{\ex}$
\begin{equ}[e:pseudoant]
\tilde\CA_+^{\ex} \CJ^\Labhom_k(\tau) = -\sum_{\ell\in\N^d}
{(-X)^\ell \over \ell!} P_+ \hat \CM^{\ex}_+ \bigl(\hat \CJ_{k+\ell}^\Labhom\otimes \tilde\CA_+^{\ex}\bigr)\Deltap_{\ex} \tau\;,
\end{equ}
where $\hat \CJ^\Labhom_k:\CT^{\ex} \to \hat \CT_+^{\ex}$ is defined in \eqref{e:CJ}, similarly to above $\hat \CM^{\ex}_+$ is the product in $\hat\CT^{\ex}_+$ and $\Deltap_\ex\colon \CT^{\ex} \to \CT^{\ex}\otimes \CT_{+}^{\ex}$ is as in Corollary~\ref{cor:domains}.
\end{proposition}
\begin{proof}
Proceeding by induction over the number of edges appearing in $\tau$, one easily
verifies that such a map exists and is uniquely determined by the above properties.
\end{proof}
Comparing this to the recursion for $\CA^{\ex}_+$ given in \eqref{e:defA+}, we see that
they are very similar, but the projection $\proj_+^\ex$ in \eqref{e:defA+} is {\it inside} the multiplication $\CM^{\ex}_+$, while $P_+$ in \eqref{e:pseudoant} is {\it outside} $\hat \CM^{\ex}_+$.

We recall now that the antipode $\CA_+^{\ex}$ is characterised among algebra-morphisms of $\CT^\ex_+$ by the 
identity
\begin{equs}[e:defAhat++]
\CM^{\ex}_+ \bigl(\id \otimes \CA_+^{\ex}\bigr) \Deltapp_{\ex}  = 
%\one_{\CT_+^{\ex}} \one^{\star}_{ \CT_+^{\ex}}\;.
\one_{2} \one^{\star}_{2} \qquad \text{on} \quad \CT_+^{\ex}\; ,
\end{equs}
where $\Deltap_\ex\colon \CT^{\ex}_+ \to \CT^{\ex}_+\otimes \CT_{+}^{\ex}$ is as in Corollary~\ref{cor:domains}.
The following result shows that $\tilde\CA_+^{\ex}$ satisfies a property close to \eqref{e:defAhat++},
which is where the name ``twisted antipode'' comes from.

\begin{proposition}
The map $\tilde\CA_+^{\ex}:\CT_+^{\ex}\to\hat \CT_+^{\ex}$ satisfies the equation
\begin{equs}[e:defAhat+]
\hat \CM^{\ex}_+ \bigl(\id \otimes \tilde\CA_+^{\ex}\bigr) \Deltap_{\ex} \inj_+^\ex  = 
%\one_{\hat \CT_+^{\ex}} \one^{\star}_{ \CT_+^{\ex}}\;,
\one_{2} \one^{\star}_{ 2} \qquad \text{on} \quad \CT_+^{\ex}\; ,
\end{equs}
where  
%$ \one_{\hat \CT_+^{\ex}} $ is the unit of $ \hat \CT_+^{\ex} $, $ \one^{\star}_{ \CT_+^{\ex}} $ is the counit of $\CT_+^{\ex} $ and 
$\Deltap_\ex\colon \hat\CT^{\ex}_+ \to \hat\CT^{\ex}_+\otimes \CT_{+}^{\ex}$ is as in Corollary~\ref{cor:domains}.
\end{proposition}

\begin{proof}
%We note that the map $\Deltap_\ex$ which appears in \eqref{e:defAhat+} should not be confused with $\Deltap_\ex\colon \CT^{\ex} \to \CT^{\ex}\otimes \CT_{+}^{\ex}$ in \eqref{e:defA+} and \eqref{e:pseudoant}.
Since both sides of \eqref{e:defAhat+} are multiplicative and since the identity obviously holds when applied to elements of the type $X^k$, we only need to verify that the left hand side vanishes when applied to elements of the form $ \CJ_k^\Labhom(\tau)$ for some $\tau \in \CT^\ex$ with 
$|\tau|_+ + |\Labhom|_\s - |k|_\s > 0$, and then use Remark~\ref{rem:A}.
Similarly to the proof of \eqref{e:defA+0}, we have
\begin{equs}
{} &
\hat \CM^{\ex}_+ \bigl(\id \otimes \tilde\CA_+^{\ex} \bigr) \Deltap_{\ex}  \hat \CJ^\Labhom_k(\tau)  = 
\\ & = \hat \CM^{\ex}_+ \bigl(\id \otimes \tilde\CA_+^{\ex}\bigr) \left[ 
\left( \hat \CJ^{\Labhom}_k \otimes \id \right)  \Deltap_{\ex} \tau
+ \sum_{\ell}  \frac{X^{\ell}}{\ell !}  \otimes   \CJ^{\Labhom}_{k+ \ell}(\tau) \right]
\\ & = \hat \CM^{\ex}_+ \left[ 
\left( \hat \CJ^{\Labhom}_k \otimes  \tilde\CA_+^{\ex} \right)\Deltap_{\ex} \tau  
- \sum_{\ell,m}  \frac{X^{\ell}}{\ell !}  \otimes  \frac{(-X)^{m}}{m !} P_+\hat\CM^{\ex}_+(\hat \CJ^{\Labhom}_{k+ \ell+m}\otimes\tilde\CA_+^{\ex})\Deltap_{\ex} \tau  \right]
\\ & = \left[ \hat \CM^{\ex}_+ 
\left( \hat \CJ^{\Labhom}_k \otimes  \tilde\CA_+^{\ex} \right)  
- P_+\hat\CM^{\ex}_+(\hat \CJ^{\Labhom}_{k}\otimes\tilde\CA_+^{\ex})\right]\Deltap_{\ex} \tau  = 0\;,
\end{equs}
since $|\hat\CM^{\ex}_+(\hat \CJ^{\Labhom}_{k}\otimes\tilde\CA_+^{\ex})\Deltap_{\ex} \tau|_+=| \CJ^{\Labhom}_{k}(\tau)|_+>0$.
\end{proof}

A very useful property of the positive twisted antipode $\tilde\CA_+^{\ex}$ is that 
its action is intertwined with that of $\Deltam_\ex$ in the following way.

\begin{lemma}\label{commutation_antipode}
The identity
\begin{equ}
\Deltam_{\ex} \tilde \CA_{+}^{\ex} = \big( \id \otimes \tilde \CA^{\ex}_+ \big)
\Deltam_{\ex}
\end{equ}
holds between linear maps from $\CT_+^\ex$ to $\CT_-^\ex \otimes \hat \CT_+^\ex$.
\end{lemma}
\begin{proof}
Since both sides of the identity are multiplicative, by using Remark~\ref{rem:A} it is enough to prove the result on $X_i$ and on elements of the form $\CJ_k(\tau)\in\CT_+^\ex$. The identity clearly holds on the linear span of $X^k$ since
$\Deltam_\ex$ acts trivially on them and $\tilde \CA_{+}^{\ex}$ preserves that subspace.

Using the recursion \eqref{e:pseudoant} for $\tilde \CA^{\ex}_+$, the 
identity $\Deltam_{\ex} P_+ =(\id \otimes P_+) \Deltam_{\ex} $
on $\hat\CT^\ex_+$, followed by the fact that $\Deltam_\ex$ is multiplicative,
we obtain
\begin{equs}
{} & \Deltam_{\ex} \tilde \CA^{\ex}_+ \CJ^{\Labhom}_k(\tau)  = -\sum_{\ell\in\N^d}\,\Big(\id \otimes {(-X)^\ell \over \ell!}\Big) \Deltam_{\ex} P_+ \hat \CM^{\ex}_+ \bigl(\hat \CJ^{\Labhom}_{k+\ell}\otimes \tilde \CA_+^{\ex}\bigr)\Deltap_{\ex} \tau
\\ & = -\sum_{\mathclap{\ell\in\N^d}}\,
\Big(\id \otimes{(-X)^\ell \over \ell!} P_+  \hat \CM^{\ex}_+ \Big)  \CM^{(13)(2)(4)} \bigl(\Deltam_{\ex} \hat \CJ^{\Labhom}_{k+\ell}\otimes \Deltam_{\ex} \tilde \CA_+^{\ex}\bigr)\Deltap_{\ex} \tau\;.
\end{equs}
Using the fact that $\Deltam_{\ex}  \hat\CJ^{\Labhom}_k=\left( \id \otimes \hat\CJ^{\Labhom}_{k} \right) \Deltam_{\ex}$, as well as \eqref{e:propWanted1}, we have 
\begin{equs}
\Deltam_{\ex} \tilde \CA^{\ex}_+ \CJ^{\Labhom}_k(\tau) & = -\sum_{\mathclap{\ell\in\N^d}}
\Big(\id \otimes {(-X)^\ell \over \ell!}P_+ \hat \CM^{\ex}_+ \Big)  \\
&\quad \times \CM^{(13)(2)(4)} \bigl( ( \id \otimes \hat \CJ^{\Labhom}_{k+\ell} ) \Deltam_{\ex}  \otimes (\id \otimes \tilde \CA_+^{\ex} ) \Deltam_{\ex} \bigr)\Deltap_{\ex} \tau
\\
& = -\sum_{\mathclap{\ell\in\N^d}}\,
\Big(\id \otimes {(-X)^\ell \over \ell!}P_+ \hat \CM^{\ex}_+( \hat \CJ^{\Labhom}_{k + \ell} \otimes \tilde \CA_+^{\ex})\Big) (\id \otimes   \Deltap_{\ex} )  \Deltam_{\ex} \tau 
\\ &  = (  \id \otimes \tilde \CA_{+}^{\ex}\CJ^{\Labhom}_{k})   \Deltam_{\ex} \tau
  = (  \id \otimes \tilde \CA_{+}^{\ex} ) \Deltam_{\ex} \CJ^{\Labhom}_{k}\tau\;.
\end{equs} 
Here, the passage from the penultimate to the last line crucially relies on the
fact that the action of $\CG_{\ex}^{-}$ onto $\CT_+^\ex$ preserves
the $|\cdot|_+$-degree, i.e. on the second formula in \eqref{e:degreeDelta}.
\end{proof}

We have now a similar construction of a {\it negative twisted antipode}.
\begin{proposition}\label{prop:twisted-}
There exists a unique algebra morphism $\tilde\CA_-^{\ex}\colon \CT_-^{\ex} \to \hat \CT_-^{\ex}$, that we call the ``negative twisted antipode'', such that for $\tau\in\CT_-^{\ex}\cap\ker\one^\star_1$
%\begin{equ}[e:pseudoant-]
%\tilde\CA_-^{\ex} \tau = -\tau-\hat\CM_-^\ex(\tilde\CA_-^{\ex}\otimes\id)(\Deltam_\ex\tau-\tau\otimes \one_1-\one_1\otimes\tau).
%\end{equ}
\begin{equ}[e:pseudoant-]
\tilde\CA_-^{\ex} \tau = -\hat\CM_-^\ex(\tilde\CA_-^{\ex}\otimes\id)(\Deltam_\ex \inj_-^\ex\tau-\tau\otimes \one_1).
\end{equ}
Similarly to \eqref{e:defAhat+}, the 
morphism $\tilde\CA_-^{\ex}\colon \CT_-^{\ex} \to \hat \CT_-^{\ex}$ satisfies
\begin{equs}[e:defAhat-]
\hat \CM^{\ex}_- \bigl(\tilde\CA_-^{\ex} \otimes \id\bigr) \Deltam_{\ex}\inj_-^\ex  = 
%\one_{\hat \CT_-^{\ex}} \one^{\star}_{ \CT_-^{\ex}}\;,
\one_{1} \one^{\star}_{1} \qquad \text{on} \quad \CT_-^{\ex}\; ,
\end{equs}
where  %$ \one_{\hat \CT_-^{\ex}} $ is the unit of $ \hat \CT_-^{\ex} $, $ \one^{\star}_{ \CT_-^{\ex}} $ is the counit of $\CT_-^{\ex} $ and 
$\Deltam_{\ex}\colon \hat\CT^{\ex}_- \to \CT^{\ex}_-\otimes \hat\CT_{-}^{\ex}$ is as in Corollary~\ref{cor:domains}.
\end{proposition}
\begin{proof}
Proceeding by induction over the number of colourless edges appearing in $\tau$, one easily
verifies that such a morphism exists and is uniquely determined by \eqref{e:pseudoant-}. The property \eqref{e:defAhat-} is a trivial consequence of \eqref{e:pseudoant-}.
\end{proof}

\subsection{Models}

We now recall (a simplified version of) the definition of a model for a regularity structure given in \cite[Def. 2.17]{reg}.
Given a scaling $\s$ as in Definition~\ref{def:scaling} and interpreting 
our constant $d\in\N$ as a space(-time) dimension, we define 
a metric $d_\s$ on $\R^d$ by
\begin{equ}[e:defds]
\|x-y\|_\s \eqdef \sum_{i=1}^d |x_i - y_i|^{1/\s_i}\;.
\end{equ}
Note that $\|\cdot\|_\s$ is not a norm since it is not 
$1$-homogeneous, but it is still a distance function since $\s_i \ge 1$.
It is also homogeneous with respect to the (inhomogeneous) scaling in which the 
$i$th component is multiplied by $\lambda^{\s_i}$.

\begin{definition}\label{def:model}
A \textit{smooth model} for a given regularity structure $\TT = (A,T,G)$ on $\R^d$ 
with scaling $\s$ consists of the following elements:
\begin{itemize}
\item A map $\Gamma \colon \R^d\times \R^d \to G$ such that $\Gamma_{xx} = \id$, the identity operator, and 
such that $\Gamma_{xy}\, \Gamma_{yz} = \Gamma_{xz}$ for every $x,y,z$ in $\R^d$.
\item A collection of continuous linear maps $\Pi_x \colon T \to \CC^\infty(\R^d)$ such that $\Pi_y = \Pi_x \circ \Gamma_{xy}$
for every $x,y \in \R^d$.
\end{itemize}
Furthermore, for every $\ell \in A$ and every compact set $\K \subset \R^d$, we assume the existence of a constant $C_{\ell,\K}$ such that the bounds
\begin{equ}[e:boundPi]
|\Pi_x \tau(y)| \le C_{\ell,\K} \|\tau\|_\ell \, \|x-y\|_\s^\ell, \qquad
\|\Gamma_{xy} \tau\|_m \le C_{\ell,\K} \|\tau\|_\ell \, \|x-y\|_\s^{\ell-m}\;, 
\end{equ}
hold uniformly over all $x,y \in \K$, all $m\in A$ with $m < \ell$ and all $\tau \in T_\ell$. 
\end{definition}

Here, recalling that the space $T$ in Definitions~\ref{def:regStruct} and~\ref{def:model} is a direct 
sum of Banach spaces $(T_\alpha)_{\alpha\in A}$,
the quantity $\|\sigma\|_m$ appearing in \eqref{e:boundPi} denotes the norm of the component of 
$\sigma\in T$ in the Banach space $T_m$ for $m\in A$.
We also note that Definition~\ref{def:model} does not include the general framework 
of \cite[Def.~2.17]{reg}, where $\Pi_x$ takes 
values in $\CD'(\R^d)$ rather than $\CC^\infty(\R^d)$; however this simplified setting is sufficient for our purposes, 
at least for now. The condition \eqref{e:boundPi} on $\Pi_x$ is of course relevant only for $\ell>0$ 
since $\Pi_x \tau(\cdot)$ is assumed to be a smooth function at this stage.

Recall that we fixed a label set $\Lab = \Lab_- \sqcup \Lab_+$.
We also fix a collection of kernels  
$\{K_{\Labhom}\}_{\Labhom \in \Lab_+}$, $K_{\Labhom}:\R^d\setminus\{0\}\to\R$, satisfying the conditions of
\cite[Ass.~5.1]{reg} with $\beta = |\Labhom|_\s$. 
We use extensively the notations of Section~\ref{sec:recurs}.

\begin{definition}\label{def:modelmap}
Given a linear map $\PPi\colon \CT^{\ex} \to \CC^\infty$, we define for all $z,\bar z\in\R^d$ 
\begin{itemize}
\item a character $g^+_z(\PPi)\colon \hat \CT^{\ex}_+ \to \R$ by extending multiplicatively 
\[
g^+_z(\PPi)X_i = \bigl(\PPi X_i\bigr)(z), \qquad 
g^+_z(\PPi)\CJ_k^\Labhom(\tau) = (D^k K_\Labhom* \PPi\tau)(z)
\]
for $\Labhom \in \Lab_+$ and setting $g^+_z(\PPi)\CJ_k^\noise(\tau) = 0$ for $\noise \in \Lab_-$.
\item a linear map $\Pi_z\colon \CT^{\ex} \to \CC^\infty$ and a character
$f_z\in \CG^{\ex}_+$ by
\begin{equ}[e:defModel1]
\Pi_z = \bigl(\PPi \otimes f_z\bigr)\Deltap_{\ex} \;,
\qquad f_z =  g_z^+(\PPi)\tilde\CA^{\ex}_+\;,
\end{equ}
where $\tilde\CA^\ex_+$ is the positive twisted antipode defined in \eqref{e:pseudoant}
\item a linear map $\Gamma_{\!z\bar z}\colon \CT^{\ex} \to \CT^{\ex}$ and a character $\gamma_{\!z\bar z}\in \CG^{\ex}_+$ by
\begin{equ}[e:defModel2]
\Gamma_{\!z\bar z} = \left( \id \otimes \gamma_{\!z\bar z} \right) \Deltap_{\ex}, \quad
 \gamma_{\!z\bar z} =  \bigl( f_z\CA_{+}^{\ex} \otimes f_{\bar z}\bigr)\Deltap_{\ex} \;.
\end{equ}
\end{itemize}
Finally, we write 
$\CZ^{\ex} \colon \PPi \mapsto (\Pi,\Gamma)$ for the map
given by \eqref{e:defModel1} and \eqref{e:defModel2}. 
\end{definition}

We do not want to consider arbitrary maps $\PPi$ as above, but we want them
to behave in a ``nice'' way with respect to the natural operations we have on $\CT^\ex$.
We therefore introduce the following notion of admissibility.
For this, we note that, as a consequence of \eqref{e:oldnormal}, the only basis vectors of the 
type $\CI_k^{\Labhom}(\tau)$\label{noise} with $\Labhom \in \Lab_-$ belonging to $\CT^\ex$ are those
with $\tau = X^\ell$ for some $\ell \in \N^d$, so we give them a special name
by setting $\Xi_{k,\ell}^\noise = \CI_k^{\noise}(X^\ell)$ and $\Xi^\noise = \Xi_{0,0}^\noise$.

\begin{definition}\label{def:admissible}
Given a linear map $\PPi\colon \CT^{\ex} \to\CC^\infty$, we set 
$\xi_\noise \eqdef \PPi \Xi^\noise$ for $\noise \in \Lab_-$.
We then say that $\PPi$ is \textit{admissible} if it satisfies
\begin{equs}[2][e:admissible]
\PPi \one &= 1\;,&\quad
\PPi X^k\tau &= x^k \PPi\tau\;,\\
\PPi \CI_k^{\Labhom}(\tau) &=  D^k K_{\Labhom} * \PPi\tau\;,&\qquad
\PPi \Xi_{k,\ell}^{\noise} &= D^k \bigl(x^\ell \xi_{\noise}\bigr)\;,
\end{equs}
for all $\tau\in\CT^\ex$, $k,\ell\in\N^d$, $\Labhom\in\Lab_+$, $\noise \in \Lab_-$, where $\CI_k^{\Labhom}:\CT^{\ex} \to \CT^{\ex}$ is defined by \eqref{e:CI}, $*$ is the distributional convolution in $\R^d$, and we use the notation
\[
D^k=\prod_{i=1}^d\frac{\partial^{k_i}}{\partial y_i^{k_i}}, \qquad x^k:\R^d\to\R, \quad x^k(y)\eqdef \prod_{i=1}^dy_i^{k_i}.
\]
Note that this definition guarantees that the identity $\PPi \CI_k^{\Labhom}(\tau) = D^k \PPi \CI_0^{\Labhom}(\tau)$
always holds, whether $\Labhom$ is in $\Lab_-$ or in $\Lab_+$. 
\end{definition}

It is then simple to check that, with these definitions, $\Pi_z \Gamma_{\!z\bar z} = \Pi_{\bar z}$ and $(\Pi,\Gamma)$ 
satisfies the algebraic requirements of Definition~\ref{def:model}. However, $(\Pi,\Gamma)$ does not necessarily 
satisfy the analytical bounds \eqref{e:boundPi}, although one has the following.

\begin{lemma}\label{lem:algpropr}
If $\PPi$ is admissible then, for every $ \CI_k^{\Labhom}(\tau) \in \CT^\ex$ with $\Labhom \in \Lab_+$, we have
\begin{equs}
f_z(\CJ_k^{\Labhom}(\tau))  &= -  \sum_{| \ell |_{\s} < | \CJ_k^{\Labhom}(\tau) |_{+} }{(-z)^\ell \over \ell!}   \big( D^{k + \ell} K_{\Labhom} * \Pi_{z} \tau \big)(z)\;, \label{e:exprfz}\\
\left( \Pi_z \CI_k^{\Labhom}(\tau) \right)(\bar z) &= 
\big( D^k K_{\Labhom} * \Pi_{z} \tau \big) (\bar z) - \sum_{|\ell |_{s} < |\CI_k^{\Labhom}(\tau)|_{+} } \frac{(\bar z - z)^{\ell}}{\ell ! } \big( D^{k+\ell} K_{\Labhom} * \Pi_{z} \tau \big) ( z)\;.
\end{equs}
\end{lemma}
\begin{proof}
%In view of the identifications \eqref{e:identifcations}, this is essentially the same as \cite[Eqns~(8.19)--(8.20)]{reg}.
%One has:
%\begin{equs}
%f_z(\CJ_k^{\Labhom}(\tau)) & =  \left(g_z^+(\PPi)\hat\CA^{\ex}_+ \right)(\CJ^{\Labhom}_k(\tau)) \\
%& = -  g_z^+(\PPi)\left( \sum_{\ell \in\N^d}{(-X)^\ell \over \ell!} P_{+} \hat \CM^{\ex}_+ \bigl(\CJ^{\Labhom}_{k+\ell}\otimes \hat\CA_+^{\ex}\bigr)\Deltap_{\ex} \tau \right)
%\\ & = -  \sum_{| \ell |_{\s} < | \CI_k^{\Labhom}(\tau) |_{+} }{(-z)^\ell \over \ell!}  g_z^+(\PPi) \hat \CM_+^{\ex} \bigl(\CJ_{k+\ell}^{\Labhom}\otimes \hat\CA_+^{\ex}\bigr)\Deltap_{\ex} \tau
%\\ &  =  -  \sum_{| \ell |_{\s} < | \CI_k^{\Labhom}(\tau) |_{+} }{(-z)^\ell \over \ell!}   \left( \left( D^{k + \ell} K_{\Labhom} * \PPi \cdot \right)(z) \otimes g_z^+(\PPi)  \hat\CA_+^{\ex}\right)\Deltap_{\ex} \tau
%\\ &  = -  \sum_{| \ell |_{\s} < | \CI_k^{\Labhom}(\tau) |_{+} }{(-z)^\ell \over \ell!}   \left( D^{k + \ell} K_{\Labhom} * \Pi_{z} \tau \right)(z).
%\end{equs}
%Alternatively, i
It follows immediately from 
\eqref{e:defDeltaH2} and the admissibility of $\PPi$ that 
$\Pi_z \CI_k^\Labhom(\tau) - D^k K_\Labhom * \Pi_z \tau$ is a polynomial
of degree $|\CI_k^{\Labhom}(\tau)|_{+}$.
On the other hand, it follows from \eqref{e:defAhat+} that $\Pi_z \CI_k^\Labhom(\tau)$
and its derivatives up to the required order (because taking derivatives
commutes with the action of the structure group) vanish at $z$, so there is
no choice of what that polynomial is, thus yielding the second identity.
The first identity then follows by comparing the second formula to \eqref{e:defModel1}. 
\end{proof}

\begin{remark}\label{rem:remainders}
Lemma~\ref{lem:algpropr} shows that the positive twisted antipode $\tilde\CA^\ex_+$ is intimately related to Taylor remainders, see Remark~\ref{rem:intuition2} and \eqref{e:defModel1}.
\end{remark}

Lemma~\ref{lem:algpropr} shows that $(\Pi,\Gamma)$ satisfies the analytical property \eqref{e:boundPi} 
on planted trees of the form $\CI_k^{\Labhom}(\tau)\in\CT^\ex$. However this is not necessarily 
the case for products of such trees, since neither $\PPi$ nor $\Pi_z$ are assumed to be 
multiplicative under the tree product \eqref{odot}. 
If, however, we also assume that $\PPi$ is multiplicative, then the map $\CZ^{\ex}$ always produces
a \textit{bona fide} model.

\begin{proposition}\label{prop:model}
If $\PPi\colon \CT^{\ex} \to \CC^\infty$ is admissible and such that, for all $\tau, \bar \tau\in\CT^\ex$ with $\tau \bar \tau\in\CT^\ex$ and all  $\alpha \in \Z^d\oplus\Z(\Lab)$, we have
\begin{equ}\label{e:canonical}
\PPi (\tau \bar \tau) 
=  (\PPi \tau)\cdot (\PPi \bar \tau)\;, \qquad \PPi \CR_\alpha(\tau) =  \PPi\tau\;,
\end{equ}
then $\CZ^{\ex}(\PPi)$ is a model for $\TT^{\ex}$. 
\end{proposition}
\begin{proof}
The proof of the algebraic properties follows immediately from
\eqref{e:defModel2}. Regarding the analytical bound \eqref{e:boundPi} on $\Pi_z \sigma$,
it immediately follows from Lemma~\ref{lem:algpropr} in the case when $\sigma$ is 
of the form $\CI_k^{\Labhom}(\tau)$. For products of such elements, it follows immediately from
the multiplicative property of $\PPi$ combined with the multiplicativity of the action of 
$\Deltap_\ex$ on $\CT^\ex$, which imply that 
\begin{equ}
\Pi_x (\sigma \bar \sigma) 
=  (\Pi_x \sigma)\cdot (\Pi_x \bar \sigma)\;.
\end{equ}
Regarding vectors of the type $\sigma = \CR_\alpha(\tau)$, it follows immediately from 
the last identity in \eqref{e:def_rec_2} combined with \eqref{e:canonical} that
$\Pi_x \CR_\alpha(\tau) = \Pi_x \tau$.

The proof of the second bound in \eqref{e:boundPi} for $\Gamma_{xy}$ is virtually identical to 
the one given in \cite[Prop.~8.27]{reg}, combined with Lemma~\ref{lem:algpropr}. 
Formally, the main difference comes from the change of basis \eqref{e:tildeJ} mentioned in Section~\ref{sec:reduced},
but this does not affect the relevant bounds since it does not mix basis vectors of different
$|\cdot|_+$-degree.
\end{proof}

\begin{remark}\label{rem:canonical}
If a map $\PPi\colon \CT^{\ex} \to \CC^\infty$ is admissible and 
furthermore satisfies  \eqref{e:canonical}, then it is uniquely determined by
the functions $\xi_\noise \eqdef \PPi \Xi^\noise$ for $\noise \in \Lab_-$.
In this case, we call $\PPi$ the \textit{canonical lift} of the functions
$\xi_\noise$.
\end{remark}

\subsection{Renormalised Models}

We now use the structure built in this article to provide a large
class of renormalisation procedures, which in particular includes those 
used in \cite{reg,wong,2015arXiv150701237H}. For this, we first need
a topology on the space of all models for a given regularity structure. 
Given two smooth models $(\Pi,\Gamma)$ and $(\bar\Pi,\bar\Gamma)$, for all $\ell\in A$ and $\K\subset\R^d$ a compact set, we define the pseudo-metrics
\begin{equ}[e:defDistModel]
\$(\Pi,\Gamma) ; (\bar\Pi,\bar\Gamma)\$_{\ell;\K} \eqdef  
\|\Pi - \bar \Pi\|_{\ell;\K} + \|\Gamma - \bar \Gamma\|_{\ell;\K}\;,
\end{equ}
where
\begin{equs}
\|\Pi - \bar \Pi\|_{\ell;\K}&\eqdef \sup\left\{ \frac{|\scal{(\Pi_x - \bar \Pi_x) \tau,\phi_x^\lambda}|}{\|\tau\|\,\lambda^\ell} :
\ x \in \K, |\tau|_+=\ell, \lambda \in (0,1], \phi \in \CB\right\} \;,\\
\|\Gamma - \bar \Gamma\|_{\ell;\K}&\eqdef \sup\left\{\frac{\|\Gamma_{xy}\tau -\bar\Gamma_{xy}\tau\|_m}{\|\tau\|\,\|x-y\|_\s^{m-\ell}} :  x,y \in \K, x\ne y,  |\tau|_+=\ell, m<\ell\right\}. 
\end{equs}
Here, the set $\CB \in \CC_0^\infty(\R^d)$ denotes the set of test functions
with support in the centred ball of radius one and all derivatives up to oder $1 + |\inf A|$
bounded by $1$. Given $\phi \in \CB$, $\phi_x^\lambda:\R^d\to\R$ denotes the translated and rescaled
function 
\[
\phi_x^\lambda(y)\eqdef \lambda^{-(\s_1+\cdots+\s_d)}\, \varphi\Big( \bigl((y_i-x_i)\lambda^{-\s_i}\bigr)_{i=1}^d\Big), \qquad y\in\R^d,
\]
for $x\in\R^d$ and $\lambda>0$ as in \cite{reg}. Finally, $\scal{\cdot,\cdot}$ is the usual $L^2$ scalar product.

\begin{definition}\label{def:modelspace}
We denote by $\MM^\ex_{\infty}$ the space of all smooth models of the form $\CZ^{\ex}(\PPi)$ for some admissible linear map $\PPi\colon \CT^{\ex} \to\CC^\infty$ in the sense of Definition~\ref{def:admissible}. We endow $\MM^\ex_{\infty}$ with the system of pseudo-metrics $(\$\cdot ; \cdot\$_{\ell;\K})_{\ell;\K}$ and we denote by $\MM^\ex_0$ the completion of this metric space. 
\end{definition}
We refer to \cite[Def.~2.17]{reg} for the definition of the space $\MM^\ex$ of models of a fixed regularity structure. With that definition, $\MM^\ex_0$ is nothing but the closure of $\MM^\ex_{\infty}$ in $\MM^\ex$.

In many singular SPDEs, one is naturally led to a sequence of models 
$\CZ(\PPi^{(\eps)})$ which do not converge as $\eps \to 0$. One would 
then like to be able to ``tweak'' this model in such a way that it remains an admissible
model but has a chance of converging as $\eps \to 0$.
A natural way of ``tweaking'' $\PPi^{(\eps)}$ is to compose it with some
linear map $M^{\ex}\colon \CT^{\ex} \to \CT^{\ex}$.
This naturally leads to the following question: what are the linear maps
$M^{\ex}$ which are such that if $\CZ^{\ex}(\PPi)$ is an admissible
model, then $\CZ^{\ex}(\PPi M^{\ex})$ is also a model? 
We then give the following definition.

\begin{definition}\label{def:renormalize}
A linear map $M \colon \CT^{\ex} \to \CT^{\ex}$ is an \textit{admissible renormalisation procedure} if 
\begin{itemize}
\item for every admissible $\PPi\colon \CT^{\ex} \to\CC^\infty$ such that $\CZ^{\ex}(\PPi)\in\MM^\ex_\infty$, $\PPi M$ is admissible and $\CZ^{\ex}(\PPi M)\in\MM^\ex_\infty$
\item the map  $\MM^\ex_\infty\ni\CZ^{\ex}(\PPi) \mapsto \CZ^{\ex}(\PPi M)\in\MM^\ex_\infty$ extends to a continuous map from $\MM^\ex_0$ to $\MM^\ex_0$.
\end{itemize}
\end{definition}

We define a right action of $\CG^{\ex}_-$ onto
$\CH$, with $\CH\in\{\CT^\ex,\CT^\ex_+\}$, by $g \mapsto M^{\ex}_g$ with
\begin{equ}[e:defMg]
M^{\ex}_g\colon\CH\to\CH, \qquad M^{\ex}_g \tau = (g \otimes \id)\Deltam_{\ex} \tau,
\qquad g\in\CG^{\ex}_-, \ \tau\in\CH.\;
\end{equ}
The following Theorem is one of the main results of this article.

\begin{theorem}\label{theo:algebra}
For every $g \in \CG^{\ex}_-$, the map $M^{\ex}_g\colon \CT^\ex \to \CT^\ex$ is an admissible renormalisation procedure. Moreover the renormalised model $ \CZ^{\ex}(\PPi M^{\ex}_g) = (\Pi^{g},\Gamma^{g}) $ is described by:
\begin{equ}[e::nice_action]
\Pi_{z}^g  = \Pi_z M^{\ex}_g, \quad \gamma^{g}_{z \bar z} =  \gamma_{z \bar z}  M^{\ex}_{g}\;.
\end{equ}
\end{theorem}
\begin{proof}
Let us fix $g \in \CG^{\ex}_-$ and an admissible linear map $\PPi$ such that
$\CZ^{\ex}(\PPi) = (\Pi,\Gamma)$ is a model and set $\PPi^g \eqdef \PPi M^{\ex}_g$. We check first that $\PPi^g$ is admissible, namely that it satisfies \eqref{e:admissible}. First, we note that, in the sum over $A$ in \eqref{def:Deltabar} defining $\Delta_1\CI_k^\Labhom(\tau)$, we have two mutually excluding possibilities:
\begin{enumerate}
\item $A$ is a subforest of $\tau$ 
\item $A$ contains the edge of type $\Labhom$ added by the operator $\CI_k^\Labhom$ or the root of $\CI_k^\Labhom(\tau)$ as an isolated node (which has however positive degree and is therefore killed by the projection $\proj^\ex_-$ in $\Deltam_\ex$). 
\end{enumerate}
When we apply $g\proj^\ex_-$ to the terms corresponding to case 2, the result is $0$ since 
$A$ contains one planted tree (with same root as that of $\CI_k^\Labhom(\tau)$) and 
$\proj^\ex_-\CI_k^\Labhom=0$ by the definition \eqref{e:defJ} of $\CJ_+$.
Therefore we have 
\begin{equs}
{} & (g\otimes \id)\Deltam_\ex  \CI_k^\Labhom(\tau) %= (g\proj^\ex_-\otimes\CI_k^\Labhom)\Delta_1\tau
= (g\otimes\CI_k^\Labhom)\Deltam_\ex\tau.
\end{equs}
Therefore
\begin{equs}
{} \PPi^g \CI_k^\Labhom(\tau) & = (g\otimes \PPi)\Deltam_\ex  \CI_k^\Labhom(\tau) =
(g\otimes \PPi \CI_k^\Labhom)\Deltam_\ex  \tau 
\\ & =(g\otimes D^kK_\Labhom*\PPi )\Deltam_\ex  \tau = D^kK_\Labhom*\PPi^g\tau.
\end{equs}
Since $X^k$ has positive degree, with a similar computation we obtain
\begin{equs}
{} \PPi^gX^k\tau & = (g\otimes \PPi)\Deltam_\ex X^k\tau =
(g\otimes \PPi X^k)\Deltam_\ex  \tau 
\\ & =(g\otimes x^k\PPi )\Deltam_\ex  \tau = x^k\PPi^g\tau
\end{equs}
and this shows that $\PPi^g$ is admissible.

Now we verify that, writing $M^{\ex}_g$ as before and $\CZ^{\ex}(\PPi^g) = (\Pi^g, \Gamma^g)$, we have
\begin{equ}
\gamma_{z\bar z}^g = (g \otimes \gamma_{z\bar z})\Deltam_{\ex}\;,\qquad 
\Pi_z^g  = (g \otimes \Pi_z)\Deltam_{\ex}\;.
\end{equ}
To show this, one first uses \eqref{e:propWanted2} to show that 
$f_{z}^g = (g \otimes f_z)\Deltam_{\ex}$,
where $f$ and $f^g$ are defined from $\PPi$ and $\PPi^g$ as in \eqref{e:defModel1}.
 Indeed, one has
\begin{equs}
f^{g}_{z}  & =   g_z^+(\PPi M^{\ex}_g)\tilde\CA^{\ex}_+
  = \left( g \otimes g_z^+(\PPi)\right) \Deltam_{\ex} \tilde\CA^{\ex}_+ \\
  & = \big( g \otimes g_z^+(\PPi) \tilde \CA^{\ex}_+ \big) \Deltam_{\ex} =  \left( g \otimes f_z \right) \Deltam_{\ex}  = f_z M_g^\ex \;.
\end{equs} 
One then uses \eqref{e:propWanted1} on $\CT^{\ex}$ to show that the required identity 
\eqref{e::nice_action} for $\Pi_z^g$ holds. Indeed, it follows that
\begin{equs}[e:defPig]
\Pi_{z}^{g} & = \left(  \PPi^{g} \otimes f^{g}_z \right) \Deltap_{\ex}  = \left(   g \otimes \PPi \otimes g
\otimes f_z \right) \left( \Deltam_{\ex} \otimes \Deltam_{\ex} \right) \Deltap_{\ex}
\\
& =\left( g\otimes \PPi \otimes f_z \right) \left(  \id \otimes \Deltap_{\ex} \right) \Deltam_{\ex}
= (g \otimes \Pi_z) \Deltam_{\ex}.
\end{equs}
In other words, we have applied \eqref{e:semidirect} for $(g,f,h)=(g,f_z,\PPi)$. 
Regarding $ \gamma_{z \bar z} $, we have analogously
\begin{equs}[e:defGammag]
\gamma_{z \bar z}^{g} 
 & = \left( f^g_z \CA_{+}^{\ex}   \otimes 
f^g_{\bar z}   \right) \Deltap_{\ex} 
  =\left( f_z M^{\ex}_{g}\CA_{+}^{\ex}  \otimes 
f_{\bar z} M^{\ex}_g \right) \Deltap_{\ex} 
\\ & =  (f_z \CA_{+}^{\ex} \otimes f_{\bar z}) \left(   M^{\ex}_g \otimes  M^{\ex}_g \right) \Deltap_{\ex}
 =  (f_z \CA_{+}^{\ex} \otimes f_{\bar z})  \Deltap_{\ex}  M^{\ex}_g
 \\ & = (g \otimes \gamma_{z \bar z} ) \Deltam_{\ex}\;.
\end{equs}
Note now that, at the level of the character $\gamma_{z\bar z}$, the bound \eqref{e:boundPi}
reads $|\gamma_{z\bar z}(\tau)| \le \|z-\bar z\|_\s^{|\tau|_+}$ as a consequence of the fact that 
$\Deltap_\ex$ preserves the sum of the $|\cdot|_+$-degrees of each factor.
On the other hand, for every character $g$ of $\CT_-^\ex$ and any $\tau$ belonging to either 
$B_\circ$ or $B_+$ (see Definition~\ref{def:CT}), the element $(g\otimes \id)\Deltam_\ex \tau$ 
is a linear combination of terms with the \textit{same} $|\cdot|_+$-degree as $\tau$.
As a consequence, it is immediate that if a given model $(\Pi,\Gamma)$ satisfies the
bounds \eqref{e:boundPi}, then the renormalised model $(\Pi^g,\Gamma^g)$ satisfies
the same bounds, albeit with different constants, depending on $g$.
We conclude that indeed for every admissible $\PPi\colon \CT^{\ex} \to\CC^\infty$ such that 
$\CZ^{\ex}(\PPi)\in\MM^\ex_\infty$, $\PPi^g$ is admissible and $\CZ^{\ex}(\PPi^g)\in\MM^\ex_\infty$.

The exact same argument also shows that if we extend the action of $\CG_-^\ex$ to all of $\MM^\ex$
by \eqref{e:defPig} and \eqref{e:defGammag}, then this yields a continuous action, which 
in particular leaves $\MM^\ex_0$ invariant as required by Definition~\ref{def:renormalize}.
\end{proof}

Note now that the group $\R^d$ acts on admissible (in the 
sense of Definition~\ref{def:admissible}) linear maps
$\PPi \colon \CT^{\ex} \to \CC^\infty$ in two different ways.
First, we have the natural action by translations $T_h$, $h \in \R^d$ given by
\begin{equ}
\bigl(T_h(\PPi) \tau\bigr)(z) \eqdef \bigl(\PPi\tau\bigr)(z-h)\;.
\end{equ}
However, $\R^d$ can also be viewed as a subgroup of $\CG_+^\ex$ 
by setting 
\begin{equ}[e:defgh]
g_h(X_i) = -h_i\;,\qquad g_h(\CJ_k^\Labhom(\tau)) = 0\;.
\end{equ}
This also acts on admissible linear maps by setting
\begin{equ}[e:actTilde]
\bigl(\tilde T_h(\PPi) \tau\bigr)(z) \eqdef  \bigl((\PPi \otimes g_h)\Deltap_\ex \tau\bigr)(z) \;.
\end{equ}
Note that if $\PPi$ is admissible, then 
one has $T_h(\PPi) X^k = \tilde T_h(\PPi)X^k$ for every $k \in \N^d$ and every $h \in \R^d$.

\begin{definition}\label{stationary}
We say that a \textit{random} linear map $\PPi \colon \CT^\ex \to \CC^\infty$ 
is \textit{stationary} if, for every (deterministic) element $h \in \R^d$, the random linear maps $T_h(\PPi)$
and $\tilde T_h(\PPi)$ are equal in law. We also assume that $\PPi$ and its derivatives, computed at $0$ have moments of all orders.
\end{definition}

By Definition~\ref{def:CT} and Remark~\ref{rem:A}, 
$\hat \CT_-^\ex$ can be identified canonically with the free algebra generated by $B_\circ$. We write
\[
\iota_\circ\colon \CT^\ex=\scal{B_\circ} \to \hat \CT_-^\ex
\] 
for the associated canonical injection.

Every random stationary map $\PPi \colon \CT^\ex \to \CC^\infty$ in the sense of Definition~\ref{stationary} 
then naturally determines a (deterministic) character $g^-(\PPi)$ of $\hat \CT_-^\ex$ by setting
\begin{equ}[e:defBPHZCharacter]
g^-(\PPi)(\iota_\circ\tau) \eqdef 
\E \left(\PPi \tau\right)(0)\;,
\end{equ}
for $\tau \in B_\circ$, where the symbol $\E$ on the right hand side denotes expectation
over the underlying probability space. This is
extended multiplicatively to all of $\hat\CT_-^\ex$. 
Then we can define a {\it renormalised} map $\hPPi\colon \CT^{\ex} \to \CC^\infty$ by
\begin{equ}[e:defHeps]
\hPPi \tau = \bigl(g^-(\PPi)\tilde\CA_-^\ex\otimes \PPi\bigr)\Deltam_\ex\tau\;,
\end{equ}
where $\tilde\CA_-^\ex\colon \CT_-^{\ex} \to \hat \CT_-^{\ex}$ is the negative twisted antipode defined in \eqref{e:pseudoant-} and satisfying \eqref{e:defAhat-}.
 
Let us also denote by $B_{\circ}^-$\label{defB0min} the (finite!) set of basis vectors $\tau\in B_\circ$ such that $|\tau|_- < 0$.
The specific choice of $g = g^-(\PPi)\tilde\CA_-^\ex$ used to define $\hPPi$ is very natural 
and canonical in the following sense. 
\begin{theorem}\label{main:renormalisation}
Let $\PPi\colon \CT^\ex \to \CC^\infty$ be stationary and admissible such that $\CZ^{\ex}(\PPi)$ is a model in $\MM^\ex_\infty$. Then, among all random functions $\PPi^g\colon \CT^\ex \to \CC^\infty $ of the form 
\[
\PPi^g=\PPi M^{\ex}_g= (g \otimes \PPi)\Deltam_{\ex} \tau, \qquad g\in\CG^\ex_-,
\] 
with $M^{\ex}_g$ as in \eqref{e:defMg}, 
$\hPPi$ is the only one such that, for all $h \in \R^d$, we have
\begin{equ}[e:claimPPi]
\E \big(\hPPi \tau\big)(h) = 0\;, \qquad \forall \, \tau\in B_\circ^-\;.
\end{equ}
We call $\hPPi$ the \emph{BPHZ renormalisation} of $\PPi$.
\end{theorem}

%Theorem~\ref{main:renormalisation} is an immediate consequence of Proposition~\ref{prop:BPHZ}.
%
%\begin{proposition}\label{prop:BPHZ}
%Let $\PPi\colon \CT^\ex \to \CC^\infty$ be stationary %and admissible
% and let $\hat \PPi$ be defined by \eqref{e:defHeps}. 
%Then, among all functions $\PPi^g\colon \CT^\ex \to \CC^\infty $ of the form 
%\[
%\PPi^g=\PPi M^{\ex}_g= (g \otimes \PPi)\Deltam_{\ex} \tau, \qquad g\in\CG^\ex_-,
%\] 
%with $M^{\ex}_g$ as in \eqref{e:defMg}, 
%$\hPPi$ is the only one such that, for all $\tau \in B_\circ^-$ and all $h \in \R^d$, we have
%\begin{equ}[e:claimPPi]
%\E \big(\hPPi \tau\big)(h) = 0\;.
%\end{equ}
%\end{proposition}
\begin{proof}
We first show that $\hPPi$ does indeed have the desired property.
We first consider $h = 0$ and we write $\PPi_0\colon \CT^\ex \to \R$ for the map (not to be confused with $\Pi_0$)
\begin{equ}
\PPi_0 \tau = \E (\PPi \tau)(0)\;.
\end{equ}
Let us denote by $B_\circ^\sharp$\label{defBsharp} the set of $\tau \in B_\circ^{-}$ which are {\it not} of the form $\CI_k^\Labhom(\sigma)$ with $|\Labhom|_\s>0$.
The main point now is that, thanks to the definitions of $g^-(\PPi)$ and $\Deltam_\ex$, we
have the identity
\begin{equ}
\bigl(\id \otimes \PPi_0\bigr)\Deltam_\ex  = 
\bigl(\id \otimes g^-(\PPi)\bigr)\Deltam_\ex \iota_\circ\;, \qquad \text{on} \quad \CT^\ex.
\end{equ}
Combining this with \eqref{e:defHeps}, we obtain for all $\tau \in B_\circ^\sharp$
\begin{equs}
\E \big(\hPPi \tau\big)(0) &= 
\bigl(g^-(\PPi)\tilde\CA_-^\ex\otimes \PPi_0\bigr)\Deltam_\ex \tau=
\bigl(g^-(\PPi)\tilde\CA_-^\ex\otimes g^-(\PPi)\bigr)\Deltam_\ex \iota_\circ \tau\\
&=
g^-(\PPi)\hat \CM_-^{\ex}(\tilde\CA_-^\ex\otimes\id)\Deltam_\ex \iota_\circ\tau = 0\;,
\end{equs} 
by the defining property \eqref{e:defAhat-} of the negative twisted antipode, 
since $\iota_\circ \tau$ belongs both to the image of $\inj_-^\ex$ and
to the kernel of $\one^\star_{1}$. 

Let now $\tau\in B_\circ^-$ be of the form $\CI_k^\Labhom(\sigma)$ with $|\Labhom|_\s>0$, 
i.e. $\tau\in B_\circ^-\setminus B_\circ^\sharp$. Arguing as in the proof of Theorem 
\ref{theo:algebra} we see that
\begin{equ}
\Deltam_\ex \iota_\circ\CI_k^\Labhom(\sigma)  = (\id \otimes \iota_\circ\CI_k^\Labhom)\Deltam_\ex \sigma\;.
\end{equ}
It then follows that
\begin{equs}
\E \big(\hPPi \tau\big)(0) 
%&= \hat\CM_-^\ex(g^-(\PPi)\tilde\CA_-^{\ex}\otimes g^-(\PPi))\Deltam_\ex \iota_\circ\CI_k^\Labhom(\sigma) \\ & 
= \hat\CM_-^\ex(g^-(\PPi)\tilde\CA_-^{\ex}\otimes g^-(\PPi)\iota_\circ\CI_k^\Labhom)\Deltam_\ex \sigma\;.
\end{equs}
The definition of $g^-(\PPi)$ combined with the fact that 
$\PPi$ is admissible and the definition of $\hPPi$ now implies that 
\begin{equ}
\E \big(\hPPi \tau\big)(0)
= \int_{\R^d} D^k K_\Labhom(-y) \E (\hPPi \sigma)(y)\,dy\;,
\end{equ}
where $D^k K_\Labhom$ should be interpreted in the sense of distributions. In particular,
one has
\begin{equ}[e:integral]
\E \big(\hPPi \tau\big)(0)
= (-1)^{|k|}\int_{\R^d} K_\Labhom(-y) D^k \E (\hPPi \sigma)(y)\,dy\;.
\end{equ}
For $\sigma = (F,\hat F, \Labn,\Labo,\Labe)$ and $\bar \Labn \colon N_F \to \N^d$ with $\bar \Labn \le \Labn$,
we now write $L_{\bar \Labn}\sigma = (F,\hat F, \Labn-\bar \Labn,\Labo,\Labe)$ and we note that 
for $g_h$ as in \eqref{e:defgh} one has the identity
\begin{equ}
(\id \otimes g_h)\Deltap_\ex \sigma = \sum_{\bar \Labn} \binom{\Labn}{\bar \Labn} (-h)^{\Sigma \bar \Labn} L_{\bar \Labn}\sigma\;,
\end{equ}
so that the stationarity of $\PPi$ implies that
\begin{equ}
\E (\hPPi \sigma)(y) = \E (\tilde T_{-y}\hPPi \sigma)(0)
=  \sum_{\bar \Labn} \binom{\Labn}{\bar \Labn} y^{\Sigma \bar \Labn} \E (\hPPi L_{\bar \Labn}\sigma)(0)\;.
\end{equ}
Plugging this into \eqref{e:integral}, we conclude that 
the terms for which there exists $i$ with $k_i > (\Sigma \bar \Labn)_i$ vanish.
If on the other hand one has $k_i \le (\Sigma \bar \Labn)_i$ for every $i$, then 
$|k|_\s \le |\Sigma \bar \Labn|_\s$ and one has
\begin{equ}
|L_{\bar \Labn}\sigma|_\s = |\sigma|_\s - |\Sigma \bar \Labn|_\s \le |\sigma|_\s - |k|_\s
\le |\sigma|_\s - |k|_\s + |\Labhom|_\s = |\CI_k^\Labhom(\sigma)|_\s < 0\;,
\end{equ}
so that $L_{\bar \Labn}\sigma\in B_\circ^-$ and has strictly less colourless edges than $\tau=\CI_k^\Labhom(\sigma)$. If $\sigma$ has only one colourless edge, then $\sigma$ belongs to $B_\circ^\sharp$; therefore the proof follows by induction over the number of colourless edges of $\tau$. 

Let us now turn to the case $h \neq 0$. First, we claim that, setting $\hPPi_h = \tilde T_h(\hPPi)$, one has
\begin{equ}[e:almostWanted]
\E \big(\hPPi_h \tau\big)(h) = 0\;.
\end{equ} 
This follows from the fact that $\hat \PPi$ is 
stationary since the action $\tilde T$ commutes with that
of $\CG_-^\ex$ as a consequence of \eqref{e:propWanted1}, combined
with the fact that $(f \otimes g_h) \Deltam_\ex \tau = g_h(\tau)$
for every $f \in \CG_-^\ex$, every $\tau \in \CT_+^\ex$
and every $g_h$ of the form \eqref{e:defgh}.

On the other hand, we have
\begin{equ}
\hPPi \tau = \tilde T_{-h}(\hPPi_h) \tau\;.
\end{equ}
It follows immediately from the
expression for the action of $\tilde T$ that
$\hPPi \tau$ is a deterministic linear combination of terms
of the form $\hPPi_h \sigma$ with $|\sigma|_- \le |\tau|_-$, so that the
claim \eqref{e:claimPPi} follows from \eqref{e:almostWanted}.

It remains to show that $\hat \PPi$ is the only function of the type $\PPi^g$ with
this property. For this, note that every such function is also
of the form $\hPPi^g$ for some different $g \in \CG^{\ex}_-$, so that we only need to show that
for every element $g$ different from the identity, there exists $\tau$ such that 
$\E \big(\hPPi^g \tau\big)(0) \neq 0$.

Using Definitions~\ref{def:CT} and~\ref{CT^ex_pm}, Remark~\ref{rem:A} and the 
identification \eqref{hatandwithouthat}, $\CT_-^\ex$ can be canonically identified with 
the free algebra generated by $B_\circ^\sharp$. Therefore the character $g$ is completely characterised by its evaluation on $B_\circ^\sharp$ and it is the identity if and only if this evaluation vanishes identically.
Fix now such a $g$ different from the identity and let $\tau \in B_\circ^\sharp$ be such that $g(\tau) \neq 0$, and
such that $g(\sigma) = 0$ for all $\sigma \in B_\circ^\sharp$ with the property that either $|\sigma|_- < |\tau|_-$
or $|\sigma|_- = |\tau|_-$, but $\sigma$ has strictly less colourless edges than $\tau$.
Since $B_\circ^\sharp$ is finite and $g$ doesn't vanish identically, such a $\tau$ exists.

We can then also view $\tau$ as an element of $\CT^\ex$ and we write
\begin{equ}
\Deltam_\ex \tau = \tau \otimes \one_1 + \sum_i \tau_i^{(1)} \otimes \tau_i^{(2)}\;,
\end{equ}
so that 
\begin{equ}[e:actPPig]
\hPPi^g \tau = g(\tau) + \sum_i g(\tau_i^{(1)}) \hPPi \tau_i^{(2)}\;.
\end{equ}
Note now that $\Deltam_\ex$ preserves the $|\cdot|_-$-degree so that 
for each of the term in the sum it is either the case that 
$|\tau_i^{(1)}|_- < |\tau|_-$ or that $|\tau_i^{(2)}| \le 0$.
In the former case, the corresponding term in \eqref{e:actPPig} vanishes identically by the definition of $\tau$.
In the latter case, its expectation vanishes at the origin if $|\tau_i^{(2)}| < 0$ by \eqref{e:claimPPi}.
If $|\tau_i^{(2)}| = 0$ then, since $\tau_i^{(2)}$ is not proportional to $\one_1$ (this is the first
term which was taken out of the sum explicitly), $\tau_i^{(2)}$ must contain at least one colourless edge.
Since $\Deltam_\ex$ also preserves the number of colourless edges, this implies that again 
$g(\tau_i^{(1)}) = 0$ by our construction of $\tau$. We conclude that
one has indeed $\E(\hPPi^g \tau)(0) = g(\tau) \neq 0$, as required.
\end{proof}

\begin{remark}\label{rem:6.19}
The rigidity apparent in \eqref{e:claimPPi} suggests that for a large class of random admissible maps
$\PPi^{(\eps)}\colon \CT^{\ex} \to \CC^\infty$ built from some stationary processes 
$\xi_\Labhom^{(\eps)}$ by \eqref{e:admissible} and \eqref{e:canonical}, the corresponding collection of models
built from $\hPeps$ defined as in \eqref{e:defHeps} should converge to 
a limiting model, provided that the  $\xi_\Labhom^{(\eps)}$ converge in a suitable
sense as $\eps\to 0$. This is indeed the case, as shown in the companion ``analytical'' article
\cite{Ajay}. It is also possible to verify that the renormalisation procedures
that were essentially ``guessed'' in \cite{KPZ,reg,wong,2015arXiv150701237H} 
are precisely of BPHZ type, see Section~\ref{example} and Section~\ref{sec:6.5} below.
\end{remark}

\begin{remark}
One immediate consequence of Theorem~\ref{main:renormalisation} is that, for any $g \in \CG_-^\ex$ and any
admissible $\PPi$, if we set $\PPi^g = (g \otimes \PPi)\Deltam_\ex$ as in Theorem~\ref{theo:algebra}, then the BPHZ renormalisation of $\PPi^g$ is $\hPPi$. In particular, the BPHZ renormalisation of the canonical lift of a collection 
of stationary processes
$\{\xi_\noise\}_{\noise \in \Lab_-}$ as in Remark~\ref{rem:canonical} is identical to that 
of the centred collection $\{\tilde \xi_\noise\}_{\noise \in \Lab_-}$ where $\tilde \xi_\noise = \xi_\noise - \E \xi_\noise(0)$.
\end{remark}

\begin{remark}
Although the map $\PPi \mapsto \hPPi$ selects a ``canonical'' representative in the class of 
functions of the form $\PPi^g$, this does not necessarily mean that every stochastic PDE 
in the class described by the underlying rule $R$ can be renormalised in a canonical way.
The reason is that the kernels $K_\Labhom$ are typically some \textit{truncated} version of 
the heat kernel and not simply the heat kernel itself. Different choices of the kernels
$K_\Labhom$ may then lead to different choices of the renormalisation constants for the
corresponding SPDEs.
\end{remark}

\subsection{The reduced regularity structure}
\label{sec:reduced}

In this section we study the relation between the regularity structure $\TT^\ex$ introduced in this paper and the one originally constructed in \cite[Sec. 8]{reg}. 

\begin{definition}
Let us call an admissible map $\PPi\colon \CT^\ex \to \CC^\infty$ {\it reduced}
if the second identity in \eqref{e:canonical} holds, namely $\PPi \CR_\alpha(\tau) =  \PPi\tau$ for all $\tau\in\CT^\ex$ and $\alpha \in \Z^d\oplus\Z(\Lab)$. 
We also define the idempotent map $\CQ_1\colon \Tra \to \Tra$ by
\begin{equ}
\CQ_1 \colon (F,\hat F,\Labn,\Labo,\Labe) \mapsto (F,\dd_1 \circ \hat F,\Labn,0,\Labe)\;,
\end{equ} 
with $\dd_1\colon \N\to\N$, $\dd_1(n)=n\un{(n\ne 1)}$,
and set $\CQ = \CQ_1 \CK$. 
\end{definition}
For example
\[
\CQ_1\left(
\begin{tikzpicture}[scale=0.1,baseline=0.4cm]
          \node at (-6,10) (d) {}; %d
      \node at (0,0) (b) {}; %b
         \node at (6,10)  (f) {}; %f
          \draw[kernel1,black] (f) node[red,round2] {\tiny $ \Labn,\Labo $}  -- node [rect1] {\tiny$\Labhom,\Labe$}   (b);
     \draw[kernel1,black] (b)node[blue,round3] {\tiny $ \Labn,\Labo $ } -- node [rect1] {\tiny$\Labhom,\Labe$}   (d);
     \draw[kernel1] (d) node[rect2] {\tiny $ \Labn $};

\end{tikzpicture} 
\right) = 
\begin{tikzpicture}[scale=0.1,baseline=0.4cm]
          \node at (-6,10) (d) {}; %d
      \node at (0,0) (b) {}; %b
         \node at (6,10)  (f) {}; %f
          \draw[kernel1,black] (f) node[rect2] {\tiny $ \Labn $}  -- node [rect1] {\tiny$\Labhom,\Labe$}   (b);
     \draw[kernel1,black] (b)node[blue,round3] {\tiny $ \Labn $ } -- node [rect1] {\tiny$\Labhom,\Labe$}   (d);
     \draw[kernel1] (d) node[rect2] {\tiny $ \Labn $};

\end{tikzpicture} 
\]
An admissible map is reduced if and only if $\PPi \tau = \PPi \CQ \tau$ for every
$\tau \in \CT^\ex$. Moreover  $\CQ$ commutes with the maps $\CK_i$, $\hat \CK_i$ and $\JJ$, and preserves
the $|\cdot|_-$-degree, so that
it is in particular also well-defined on $\CT^\ex$, $\hat \CT_+^\ex$, $\hat \CT_-^\ex$ and  $\CT_-^\ex$.
It does however \textit{not} preserve the $|\cdot|_+$-degree so that it is
\textit{not} well-defined on $\CT_+^\ex$! Indeed,
the $|\cdot|_+$-degree depends on the $\Labo$ decoration, which is set to $0$ by $\CQ$, 
see Definition~\ref{homog}.

\begin{definition} \label{def:6.23}
Let $\CT$ and $\hat \CT_+$ respectively be the subspaces of $\CT^{\ex}$ and 
$\hat \CT_+^{\ex}$ given by 
\begin{equ}
\CT \eqdef \{\tau \in \CT^\ex \,:\, \CQ \tau = \tau\}\;,\quad
\hat \CT_+ \eqdef \{\tau \in \hat \CT_+^\ex \,:\, \CQ \tau = \tau\}\;, \quad 
\end{equ}
We also set $\CT_+ = \proj_+^\ex \hat \CT_+$, where $\proj^\ex_\pm:\hat\CT_{\pm}^{\ex}\to \CT_{\pm}^{\ex}$ is defined after \eqref{eq:Hopf}. 
\end{definition}

The reason why we define $\CT_+$ in this slightly more convoluted way 
instead of setting it equal to $\{\tau \in \CT_+^\ex \,:\, \CQ \tau = \tau\}$
is that although $\CQ$ is well-defined on $\hat \CT_+^\ex$, it is not well-defined on 
$\CT_+^\ex$ since it does not preserve the $|\cdot|_+$-degree, as already mentioned above. 
Since $\CQ$ is multiplicative, $\CT_+$ is a subalgebra of $\CT_+^\ex$. We set
\begin{equ}[e:Deltarestr]
\Delta\eqdef\Deltap_\ex\colon\CT\to\CT\otimes\CT_+ \;,\quad
\Deltap\eqdef\Deltap_\ex\colon\CT_+\to\CT_+\otimes\CT_+
\;.
\end{equ} 
Looking at the recursive definition \eqref{e:defA+} of 
the antipode $\CA_+^\ex$, it is clear that it also maps $\CT_+$ into itself,
so that $\CT_+$ is a Hopf subalgebra of $\CT_+^\ex$. Moreover $\Delta$ turns $\CT$
into a co-module over $\CT_+$.

We can therefore define $\CG_+$ as the characters group of $\CT_+$ and introduce the 
action of $\CG_+$ on $\CT$:
\[
\CG_+\ni f\to \Gamma_f:\CT\to\CT, \qquad \Gamma_f\tau\eqdef(\id\otimes f)\Delta\tau,
\qquad \tau\in\CT.
\]
If we grade $\CT$ by $|\cdot|_+$ and we define $\TT=(A,\CT,\CG_+)$ 
where $A\eqdef\{|\tau|_+:\tau\in B_\circ, \ \tau=\CQ\tau\}$ and $\CT^\ex=\scal{B_\circ}$ 
as in Definition~\ref{def:CT}, then arguing as in the proof of Proposition~\ref{TTex}, 
we see that the action of $\CG_+$ on $\CT$ satisfies \eqref{e:coundGroup}. Therefore $\TT$ 
is a regularity structure as in Definition~\ref{def:regStruct}.

We set now $\tilde\CJ_k^{\Labhom}:\CT^\ex\to\CT^\ex_+$ and $\tilde\CJ_k^{\Labhom}:\CT\to\CT_+$,
\begin{equ}[e:tildeJ]
\tilde\CJ_k^{\Labhom}(\tau)= \sum_{|m |_{\s} < | \CJ_k^{\Labhom}(\tau) |_{+} } \frac{(-X)^m}{m!} \CJ_{k+m}^{\Labhom}(\tau)\;.
\end{equ}
Suppose that $\{\Labhom,\ii\}\subseteq\Lab$ with $|\Labhom|_\s>0$ and $|\ii|_\s<0$. We set $\Xi^\ii:=\CI^\ii_0(\one)$. Then we have by \eqref{e:def_rec_2} and \eqref{e:defDeltaH2} for all $\tau\in\CT$
\begin{equs}[0]
   \Delta \one =  \one \otimes \one\;,\quad
  \Delta \Xi^\ii = \Xi^\ii \otimes \one\;,\quad \Delta X_i = X_i  \otimes \one + \one \otimes X_i\;,
     \\ 
 \Delta \CI^\Labhom_k(\tau) = (\CI^\Labhom_k \otimes 1)\Delta \tau + \sum_{\ell,m} {X^\ell \over \ell!} \otimes {X^m \over m!}
\tilde\CJ^\Labhom_{k+\ell+m}(\tau)\;,\label{e:recursiveDelta}
\end{equs}
as well as
\begin{equs}[0]
 \Deltap \one =  \one \otimes \one \;,\quad
   \Deltap X_i = X_i  \otimes \one + \one \otimes X_i\;,\\
\Deltap \tilde\CJ^\Labhom_{k}(\tau) = 1 \otimes \tilde\CJ^\Labhom_{k}\tau   \label{e:recursiveDeltap}
    +  \sum_{\ell}  \Big( \tilde\CJ^\Labhom_{k+\ell} \otimes \frac{(-X)^{\ell}}{\ell!} \Big)  \Delta \tau\;,
\end{equs}
with the additional property that both maps are multiplicative with respect to the tree product. 

We see therefore that the operators $\Delta \colon \CT \to \CT \otimes \CT_+$ and
$\Deltap\colon \CT_+ \to \CT_+\otimes \CT_+$ are isomorphic to those defined in \cite[Eq.~(8.8)--(8.9)]{reg}.
%As far as $\Deltap$ is concerned, exploiting the multiplicativity of $\Deltap$ and the
%fact that $\Deltap {X^k \over k!} = \sum_{m+n = k} {X^m \over m!} \otimes {X^n \over n!}$, one has the identity
 This shows that the regularity structure $\TT$, associated to a subcritical complete rule $R$, is 
isomorphic to the regularity structure associated to a subcritical equation constructed in 
\cite[Sec.~8]{reg}, modulo a simple change of coordinates. Note that this change of coordinates is ``harmless''
as far as the link to the analytical part of \cite{reg} is concerned since it 
does not mix basis vectors of different degrees. 

As explained in Remark~\ref{extended}, the superscript `$\ex$' stands for extended: the 
reason is that the regularity structure $\TT^\ex$ is an extension of $\TT$ in the sense 
that $\TT \subset \TT^\ex$ with the inclusion interpreted as in \cite[Sec.~2.1]{reg}. 
By contrast, we call $\TT$ the {\it reduced} regularity structure. 

By the definition 
of $\CQ$, the extended structure $\CT^\ex$ encodes more information since we keep
track of the effect of the action of $\CG_-$ by storing the (negative) homogeneity of the 
contracted subtrees in the decoration $\Labo$ and by colouring the corresponding nodes; 
both these details are lost when we apply $\CQ$ and therefore
in the reduced structure $\CT$.

Note that if $\PPi \colon \CT^\ex \to \CC^\infty$ is such that $ \CZ^{\ex}(\PPi) = (\Pi, \Gamma) $ is a model of $\TT^\ex$,
then the restriction $\CZ(\PPi)$ of $\CZ^{\ex}(\PPi)$ to $\TT$ is automatically again a model. 
This is always the case, irrespective of whether $\PPi$ is reduced or not, since the action of $\CG_+^\ex$ leaves $\CT$
invariant. This allows to give the following
definition.

\begin{definition}\label{def:modelspace2}
We denote by $\MM_{\infty}$ the space of all smooth models for $\TT$, in the sense of Definition~\ref{def:model}, obtained by restriction to $\CT$ of
$\CZ^{\ex}(\PPi)$ for some reduced admissible linear map $\PPi\colon \CT^{\ex} \to\CC^\infty$. 
We endow $\MM_{\infty}$ with the system of pseudo-metrics \eqref{e:defDistModel} and we denote by $\MM_0$ the completion of this metric space. 
\end{definition}

\begin{remark}
The restriction that $\PPi$ be reduced may not seem very natural in view of the discussion
preceding the definition. It follows however from Theorem~\ref{theo:models} below that 
lifting this restriction makes no difference whatsoever since it implies in particular that
every smooth admissible model on $\TT$ is of the form $\CZ(\PPi)$ for some reduced $\PPi$.
\end{remark}

\begin{remark}\label{rem:fails}
By restriction of $ \CZ^{\ex}(\PPi M_g) $ to $\TT$ for $g\in\CG_-^\ex$, we get a renormalised model $ \CZ(\PPi M_g) $ which covers all the examples treated so far in singular SPDEs. 
It is however not clear a priori whether we really have an action of a 
suitable subgroup of $\CG_-^\ex$ onto $\MM_{\infty}$ or $\MM_0$. This is because the coaction
of $\Deltam_\ex$ on $\CT^\ex$ and $\CT_+^\ex$ fails to leave the reduced sector invariant.
If on the other hand we tweak this coaction by setting $\Deltam = (\id\otimes \CQ)\Deltam_\ex$, 
then unfortunately $\Deltap$ and $\Deltam$ do
\textit{not} have the cointeraction property \eqref{e:intertwine}, which was crucial for our construction, see Remark~\ref{explanation}. See Corollary \ref{final} below for more on $\Deltam$.
\end{remark}

\begin{remark}\label{rem:freduced}
In accordance with \cite[Formula (8.20)]{reg}, it follows from \eqref{e:exprfz} and the binomial identity that,
for all $\CJ_k^{\Labhom}(\tau)\in\CT^\ex_+$ with $|\CJ_k^{\Labhom}(\tau)|_+>0$
\begin{equ}
f_z(\tilde\CJ_k^{\Labhom}(\tau))   = 
- \big( D^{k } K_{\Labhom} * \Pi_{z} \tau \big)(z)\;.
\end{equ}
\end{remark}

\begin{remark}\label{after}
The negative twisted antipode 
$\tilde\CA_-^{\ex}\colon \CT_-^{\ex} \to \hat \CT_-^{\ex}$ of Proposition 
\ref{prop:twisted-} satisfies the
identity $\CQ \tilde \CA_-^\ex  = \CQ \tilde \CA_-^\ex \CQ$. This follows from
the induction \eqref{e:pseudoant-}, the multiplicativity of $\CQ$, and the formula
\begin{equ}[e:intertwiningQDelta]
(\CQ \otimes \CQ)\Deltam_\ex \CQ = 
(\CQ \otimes \CQ)\Deltam_\ex\;,
\end{equ}
where $\Deltam_\ex\colon \CT^\ex \to \CT_-^\ex \otimes \CT^\ex$.
Therefore, if a stationary admissible $\PPi$ is (almost surely) reduced, then the character $g^-(\PPi)$
is also reduced in the sense that $g^-(\PPi)(\CQ\tau) = g^-(\PPi)(\tau)$.
Using again \eqref{e:intertwiningQDelta}, it follows immediately that 
$\hPPi$ as given by \eqref{e:defHeps} is again reduced, so that the 
class of reduced models is preserved by the BPHZ renormalisation procedure.
\end{remark}

There turn out to be two natural subgroups of $\CG_-^\ex$ that are determined 
by their values on $\CQ \CT_-^\ex$:
\begin{itemize}
\item We set $ \CG_-\eqdef\{ g\in \CG^{\ex}_-: 
g( \tau)= g( \CQ \tau) , \, \forall\, \tau \in \CT_{-}^{\ex}  \}$. This is the most natural 
subgroup of $\CG_-^\ex$ since it contains the characters $ 
g^{-}(\PPi) \tilde \CA_{-}^{\ex} $ used for the definition of
$ \hat \PPi $ in \eqref{e:defHeps}, as soon as $ \PPi = \PPi\CQ $. 
The fact that $\CG_-$ is a subgroup follows from the property \eqref{e:intertwiningQDelta}. 
\item We set $\CG_{-}^{a} \eqdef \lbrace  g \in \CG^{\ex}_{-} \; 
: \;  g(\tau) = 0, \, \forall \tau \in  
\CT^{c}_{-} \rbrace  $ where $ \CT^{c}_{-} $ is the bialgebra ideal of $ \CT_{-}^{\ex} $ generated by 
$\lbrace \tau \in B_-,\; \CQ \tau \neq \tau \rbrace$. Then one can identify $  \CG_{-}^{a} $   with the group of characters of the Hopf algebra  $ \left( \CT_{-}^{\ex} / \CT^{c}_{-}, \Deltam_{\ex} \right) $. 
It turns out that this is simply the polynomial Hopf algebra with generators 
$\{\tau \in B_-\,:\, |\tau|_- < 0,\,\CQ \tau = \tau\}$, so that $ \CG_{-}^{a} $
is abelian.
\end{itemize}
We then have the following result.

\begin{theorem}\label{thm:renorred}
There is a continuous action $R$ of $\CG_-$ onto $\MM_0$ with the property that, for every $g \in \CG_-$
and every reduced and admissible $\PPi \colon \CT^\ex \to \CC^\infty$ with $\CZ^\ex(\PPi) \in \MM_0^\ex$, one has
$R_g \CZ(\PPi) = \CZ(\PPi M_g)$.
\end{theorem}

\begin{proof}
We already know by Theorem~\ref{theo:algebra} that $\CG_-$ acts continuously onto $\MM_0^\ex$.
Furthermore, by the definition of $\CG_-$, it preserves the subset $\MM_0^r\subset \MM_0^\ex$
of reduced models, i.e. the closure in $\MM_0^\ex$ of all models of the form $\CZ^\ex(\PPi)$
for $\PPi$ admissible and reduced. 
Since $\TT \subset \TT^\ex$, we already mentioned that we have a natural projection 
$\pi^\ex\colon \MM_0^\ex \to \MM_0$ given by restriction
(so that $\CZ(\PPi) = \pi^\ex \CZ^\ex(\PPi)$), and it is straightforward to see that 
$\pi^\ex$ is injective on $\MM_0^r$. It therefore suffices to show that there is a continuous
map $\iota^\ex \colon \MM_0 \to \MM_0^\ex$ which is a right inverse to $\pi^\ex$, and this is 
the content of Theorem~\ref{theo:models} below.
\end{proof}

\begin{remark}
We'll show in Section~\ref{sec:6.5} below that the action of 
$\CG_-$ onto $\MM_0$ is given by elements of the 
``renormalisation group'' defined in \cite[Sec.~8.3]{reg}.
\end{remark}

\subsubsection{An example}
\label{example}

We consider the example of the stochastic quantization given in dimension $ 3 $ by:
\begin{equs}
\partial_t u = \Delta u + u^{3} + \xi.
\end{equs}
This equation has been solved first in \cite{reg} with regularity structures and then in \cite{CatCh}. One tree needed for its resolution reveals the importance of the extended decoration.
Using the symbolic notation, it is given by $ \tau = \CI(\Xi)^2 \CI( \CI(\Xi)^3 )  $.  Then we use the following representation: 
   \begin{equs}
   \CI(\Xi) = \begin{tikzpicture}[scale=0.15,baseline=0.1cm]
        \node at (0,0)  [dot] (root) {\mbox{\tiny $ $}};
         \node at (0,3)  [circ] (center) {};
     \draw[kernel1] (center) to
     node [sloped,below] {\small }     (root);
     \end{tikzpicture}\;, \quad
   \CR_\alpha \CI_{e_i} =  \begin{tikzpicture}[scale=0.15,baseline=0.1cm]
        \node at (0,0)  [circ1] (root) {\mbox{\tiny $a$}};
         \node at (0,5)  [dot] (center) {};
     \draw[kernel1] (center) to
     node [rect] {\mbox{\tiny $ i $} }     (root);
     \end{tikzpicture}\;, \quad
   X_i = \begin{tikzpicture}[scale=0.15,baseline=0.13cm]
        \node at (0,1.5)  [circ] (root) {\mbox{\tiny $i$}};
     \end{tikzpicture}\;, \quad 
    \CJ = \begin{tikzpicture}[scale=0.15,baseline=0.1cm]
        \node at (0,0)  [dot,blue] (root) {};
         \node at (0,3)  [dot] (center) {};
     \draw[kernel1] (center) to
     node [sloped,below] {\small }     (root);
     \end{tikzpicture}\;, \quad \tau = \IXitwoIIXithree, 
    \end{equs}
    where $ e_i $ is the
$i$th canonical basis element of $\N^d$ and $ a $ belongs to $ \lbrace \alpha, \beta, \gamma \rbrace $ with $ \alpha = 2 \CI + 2 \Xi  $, $ \beta =  2 \CI + 2 \Xi + 1 $ and $ \gamma = 5 \CI + 4 \Xi $. Then we have 
  \begin{equs}
 \Deltam_{ \ex}  & \IXitwoIIXithree  =  \IXitwoIIXithree  
      \otimes \one_{1} + \one_{1} 
     \otimes \IXitwoIIXithree + 3 \IXitwo \otimes
     \begin{tikzpicture}[scale=0.15,baseline=0.1cm]
        \node at (0,0)  [dot,label=below:$  $] (root) {};
         \node at (2,2)  [circ] (right) {\mbox{\tiny $ $}};
         \node at (0,3)  [circ1] (center) {\mbox{\tiny $\alpha$}};
          \node at (2,5)  [circ] (centerr) {\mbox{\tiny $ $}};
         \node at (-2,2)  [circ] (left) {\mbox{\tiny $ $}};
            \draw[kernel1] (right) to
     node [sloped,below] {\small }     (root); \draw[kernel1] (left) to
     node [sloped,below] {\small }     (root);
     \draw[kernel1] (center) to
     node [sloped,below] {\small }     (root);
      \draw[kernel1] (centerr) to
     node [sloped,below] {\small }     (center);
     \end{tikzpicture} + 3  \XIXitwo \otimes \begin{tikzpicture}[scale=0.15,baseline=0.1cm]
        \node at (0,0)  [dot,label=below:$  $] (root) {};
         \node at (2,2)  [circ] (right) {\mbox{\tiny $ $}};
         \node at (0,3)  [circ1] (center) {\mbox{\tiny $\beta$}};
          \node at (2,8)  [circ] (centerr) {\mbox{\tiny $ $}};
         \node at (-2,2)  [circ] (left) {\mbox{\tiny $ $}};
            \draw[kernel1] (right) to
     node [sloped,below] {\small }     (root); \draw[kernel1] (left) to
     node [sloped,below] {\small }     (root);
     \draw[kernel1] (center) to
     node [sloped,below] {\small }     (root);
      \draw[kernel1] (centerr) to
     node [rect] {\mbox{\tiny $ i $} }     (center);
     \end{tikzpicture} \\ & +  \IXitwo \otimes 
     \begin{tikzpicture}[scale=0.15,baseline=0.1cm]
        \node at (0,0)  [circ1] (root) {\mbox{\tiny $\alpha$}};
         \node at (0,3)  [dot] (center) {};
          \node at (0,6)  [circ] (centerc) {\mbox{\tiny $ $}};
          \node at (-2,5)  [circ] (centerl) {\mbox{\tiny $ $}};
          \node at (2,5)  [circ] (centerr) {\mbox{\tiny $ $}};
     \draw[kernel1] (center) to
     node [sloped,below] {\small }     (root);
     \draw[kernel1] (centerc) to
     node [sloped,below] {\small }     (center);
      \draw[kernel1] (centerr) to
     node [sloped,below] {\small }     (center);
      \draw[kernel1] (centerl) to
     node [sloped,below] {\small }     (center);
     \end{tikzpicture} +  \XIXitwo \otimes
     \begin{tikzpicture}[scale=0.15,baseline=0.1cm]
        \node at (0,0)  [circ1] (root) {\mbox{\tiny $\beta$}};
         \node at (0,5)  [dot] (center) {};
          \node at (0,8)  [circ] (centerc) {\mbox{\tiny $ $}};
          \node at (-2,7)  [circ] (centerl) {\mbox{\tiny $ $}};
          \node at (2,7)  [circ] (centerr) {\mbox{\tiny $ $}};
     \draw[kernel1] (center) to
     node [rect] {\mbox{\tiny $ i $} }     (root);
     \draw[kernel1] (centerc) to
     node [sloped,below] {\small }     (center);
      \draw[kernel1] (centerr) to
     node [sloped,below] {\small }     (center);
      \draw[kernel1] (centerl) to
     node [sloped,below] {\small }     (center);
     \end{tikzpicture} + 
     3 \IXitwo \! \! \! \IXitwo \otimes \begin{tikzpicture}[scale=0.15,baseline=0.1cm]
        \node at (0,0)  [circ1] (root) {\mbox{\tiny $\alpha$}};
         \node at (0,3)  [circ1] (center) {\mbox{\tiny $\alpha$}};
          \node at (2,5)  [circ] (centerr) {\mbox{\tiny $ $}};
     \draw[kernel1] (center) to
     node [sloped,below] {\small }     (root);
      \draw[kernel1] (centerr) to
     node [sloped,below] {\small }     (center);
     \end{tikzpicture} + 3 \IXitwo  \! \! \! \XIXitwo \otimes \begin{tikzpicture}[scale=0.15,baseline=0.1cm]
        \node at (0,0)  [circ1] (root) {\mbox{\tiny $\alpha$}};
         \node at (0,3)  [circ1] (center) {\mbox{\tiny $\beta$}};
          \node at (2,8)  [circ] (centerr) {\mbox{\tiny $ $}};
     \draw[kernel1] (center) to
     node [sloped,below] {\small }     (root);
      \draw[kernel1] (centerr) to
     node [rect] {\mbox{\tiny $ i $}}     (center);
     \end{tikzpicture} \\ & + 3 \IXitwo  \! \! \! \XIXitwo \otimes \begin{tikzpicture}[scale=0.15,baseline=0.1cm]
        \node at (0,0)  [circ1] (root) {\mbox{\tiny $\beta$}};
         \node at (0,5)  [circ1] (center) {\mbox{\tiny $\alpha$}};
          \node at (2,7)  [circ] (centerr) {\mbox{\tiny $ $}};
     \draw[kernel1] (centerr) to
     node [sloped,below] {\small }     (center);
      \draw[kernel1] (center) to
     node [rect] {\mbox{\tiny $ i $}}     (root);
     \end{tikzpicture} + 3 \XIXitwo  \! \! \! \begin{tikzpicture}[scale=0.15,baseline=0.1cm]
        \node at (0,0)  [circ,label=below:$  $] (root) {\mbox{\tiny $ j $}};
         \node at (1,2)  [circ] (right) {\mbox{\tiny $ $}};
         \node at (-1,2)  [circ] (left) {\mbox{\tiny $ $}};
            \draw[kernel1] (right) to
     node [sloped,below] {\small }     (root); \draw[kernel1] (left) to
     node [sloped,below] {\small }     (root);
     \end{tikzpicture} \otimes \begin{tikzpicture}[scale=0.15,baseline=0.1cm]
        \node at (0,0)  [circ1] (root) {\mbox{\tiny $\beta$}};
         \node at (0,5)  [circ1] (center) {\mbox{\tiny $\beta$}};
          \node at (2,10)  [circ] (centerr) {\mbox{\tiny $ $}};
     \draw[kernel1] (centerr) to
     node [rect] {\mbox{\tiny $ j $}}     (center);
      \draw[kernel1] (center) to
     node [rect] {\mbox{\tiny $ i $}}     (root);
     \end{tikzpicture} + \begin{tikzpicture}[scale=0.15,baseline=0.1cm]
        \node at (0,0)  [dot,label=below:$  $] (root) {};
         \node at (2,2)  [circ] (right) {\mbox{\tiny $ $}};
         \node at (0,3)  [dot] (center) {};
          \node at (-2,5)  [circ] (centerl) {\mbox{\tiny $ $}};
          \node at (2,5)  [circ] (centerr) {\mbox{\tiny $ $}};
         \node at (-2,2)  [circ] (left) {\mbox{\tiny $ $}};
            \draw[kernel1] (right) to
     node [sloped,below] {\small }     (root); \draw[kernel1] (left) to
     node [sloped,below] {\small }     (root);
     \draw[kernel1] (center) to
     node [sloped,below] {\small }     (root);
      \draw[kernel1] (centerr) to
     node [sloped,below] {\small }     (center);
      \draw[kernel1] (centerl) to
     node [sloped,below] {\small }     (center);
     \end{tikzpicture} \otimes \begin{tikzpicture}[scale=0.15,baseline=0.1cm]
        \node at (0,0)  [circ1] (root) {\mbox{\tiny $ \gamma $}};
         \node at (0,3)  [circ] (center) {};
     \draw[kernel1] (center) to
     node [sloped,below] {\small }     (root);
     \end{tikzpicture} + (...)
  \end{equs}
with summation over $i$ and $j$ implied. In $ (...) $, we omit terms of the form
$ \tau^{(1)} \otimes \tau^{(2)} $ where $  \tau^{(1)} $ may contain planted trees or where $ \tau^{(2)} $ has an edge of type $ \CI $ finishing on a leaf.
The planted trees will disappear by applying an element of $ \CG^{\ex}_- $ and the others are put to zero through the evaluation of the smooth model  $ \PPi $ see \cite[Ass.~5.4]{reg} where the kernels $  \{K_{\Labhom}\}_{\Labhom \in \Lab_+}$ are chosen such that they integrate polynomials to zero up to a certain fixed order. If $ g \in \CG^{\ex}_-  $ is the character associated to the BPHZ renormalisation 
for a Gaussian driving noise with a covariance that is symmetric under spatial reflections, we obtain 
\begin{equs}
  M^{\ex}_g \tau  & = (g \otimes \id)\Deltam_{\ex} \tau \\
  & = \IXitwoIIXithree + 3 C_1
     \begin{tikzpicture}[scale=0.15,baseline=0.1cm]
        \node at (0,0)  [dot,label=below:$  $] (root) {};
         \node at (2,2)  [circ] (right) {\mbox{\tiny $ $}};
         \node at (0,3)  [circ1] (center) {\mbox{\tiny $\alpha$}};
          \node at (2,5)  [circ] (centerr) {\mbox{\tiny $ $}};
         \node at (-2,2)  [circ] (left) {\mbox{\tiny $ $}};
            \draw[kernel1] (right) to
     node [sloped,below] {\small }     (root); \draw[kernel1] (left) to
     node [sloped,below] {\small }     (root);
     \draw[kernel1] (center) to
     node [sloped,below] {\small }     (root);
      \draw[kernel1] (centerr) to
     node [sloped,below] {\small }     (center);
     \end{tikzpicture}
      + C_1
     \begin{tikzpicture}[scale=0.15,baseline=0.1cm]
        \node at (0,0)  [circ1] (root) {\mbox{\tiny $\alpha$}};
         \node at (0,3)  [dot] (center) {};
          \node at (0,6)  [circ] (centerc) {\mbox{\tiny $ $}};
          \node at (-2,5)  [circ] (centerl) {\mbox{\tiny $ $}};
          \node at (2,5)  [circ] (centerr) {\mbox{\tiny $ $}};
     \draw[kernel1] (center) to
     node [sloped,below] {\small }     (root);
     \draw[kernel1] (centerc) to
     node [sloped,below] {\small }     (center);
      \draw[kernel1] (centerr) to
     node [sloped,below] {\small }     (center);
      \draw[kernel1] (centerl) to
     node [sloped,below] {\small }     (center);
     \end{tikzpicture} + 
     3 C_1^2 \begin{tikzpicture}[scale=0.15,baseline=0.1cm]
        \node at (0,0)  [circ1] (root) {\mbox{\tiny $\alpha$}};
         \node at (0,3)  [circ1] (center) {\mbox{\tiny $\alpha$}};
          \node at (2,5)  [circ] (centerr) {\mbox{\tiny $ $}};
     \draw[kernel1] (center) to
     node [sloped,below] {\small }     (root);
      \draw[kernel1] (centerr) to
     node [sloped,below] {\small }     (center);
     \end{tikzpicture}  + 3C_2\begin{tikzpicture}[scale=0.15,baseline=0.1cm]
        \node at (0,0)  [circ1] (root) {\mbox{\tiny $ \gamma $}};
         \node at (0,3)  [circ] (center) {};
     \draw[kernel1] (center) to
     node [sloped,below] {\small }     (root);
     \end{tikzpicture} 
\end{equs} 
where 
\begin{equ}
 C_1 = - g^-(\PPi)\Big[ \IXitwo  \Big]\;,\quad C_2 = - g^-(\PPi)\Big[\begin{tikzpicture}[scale=0.1,baseline=0.15cm]
        \node at (0,0)  [dot,label=below:$  $] (root) {};
         \node at (2,2)  [circ] (right) {\mbox{\tiny $ $}};
         \node at (0,3)  [dot] (center) {};
          \node at (-2,5)  [circ] (centerl) {\mbox{\tiny $ $}};
          \node at (2,5)  [circ] (centerr) {\mbox{\tiny $ $}};
         \node at (-2,2)  [circ] (left) {\mbox{\tiny $ $}};
            \draw[kernel1] (right) to
     node [sloped,below] {\small }     (root); \draw[kernel1] (left) to
     node [sloped,below] {\small }     (root);
     \draw[kernel1] (center) to
     node [sloped,below] {\small }     (root);
      \draw[kernel1] (centerr) to
     node [sloped,below] {\small }     (center);
      \draw[kernel1] (centerl) to
     node [sloped,below] {\small }     (center);
     \end{tikzpicture} \Big]\;,
\end{equ}
and all other renormalisation constants vanish. Applying $ \CQ $, we indeed recover the renormalisation map
given in \cite[Sec 9.2]{reg}.
 The main interest of the extended decorations is to shorten some Taylor expansions which allows us to get the co-interaction between the two renormalisations. In the computation below, we show the difference between a term having extended decoration and the same without:
    
  \begin{equs}
  \Deltap_{\ex} \begin{tikzpicture}[scale=0.15,baseline=0.1cm]
        \node at (0,0)  [circ1] (root) {\mbox{\tiny $\alpha$}};
         \node at (0,3)  [circ1] (center) {\mbox{\tiny $\alpha$}};
          \node at (2,5)  [circ] (centerr) {\mbox{\tiny $ $}};
     \draw[kernel1] (center) to
     node [sloped,below] {\small }     (root);
      \draw[kernel1] (centerr) to
     node [sloped,below] {\small }     (center);
     \end{tikzpicture} & = \begin{tikzpicture}[scale=0.15,baseline=0.1cm]
        \node at (0,0)  [circ1] (root) {\mbox{\tiny $\alpha$}};
         \node at (0,3)  [circ1] (center) {\mbox{\tiny $\alpha$}};
          \node at (2,5)  [circ] (centerr) {\mbox{\tiny $ $}};
     \draw[kernel1] (center) to
     node [sloped,below] {\small }     (root);
      \draw[kernel1] (centerr) to
     node [sloped,below] {\small }     (center);
     \end{tikzpicture} \otimes \one + 
     \begin{tikzpicture}[scale=0.15,baseline=-0.1cm]
        \node at (0,0)  [circ1] (root) {\mbox{\tiny $\alpha$}};
     \end{tikzpicture} \otimes \begin{tikzpicture}[scale=0.15,baseline=0.1cm]
        \node at (0,0)  [dot,blue] (root) {};
         \node at (0,3)  [circ1] (center) {\mbox{\tiny $\alpha$}};
          \node at (2,5)  [circ] (centerr) {\mbox{\tiny $ $}};
     \draw[kernel1] (center) to
     node [sloped,below] {\small }     (root);
      \draw[kernel1] (centerr) to
     node [sloped,below] {\small }     (center);
     \end{tikzpicture} \\ 
  \Deltap_{\ex} \begin{tikzpicture}[scale=0.15,baseline=0.1cm]
        \node at (0,0)  [dot] (root) {};
         \node at (0,3)  [dot] (center) {};
          \node at (2,5)  [circ] (centerr) {\mbox{\tiny $ $}};
     \draw[kernel1] (center) to
     node [sloped,below] {\small }     (root);
      \draw[kernel1] (centerr) to
     node [sloped,below] {\small }     (center);
     \end{tikzpicture} & = \begin{tikzpicture}[scale=0.15,baseline=0.1cm]
        \node at (0,0)  [dot] (root) {};
         \node at (0,3)  [dot] (center) {};
          \node at (2,5)  [circ] (centerr) {\mbox{\tiny $ $}};
     \draw[kernel1] (center) to
     node [sloped,below] {\small }     (root);
      \draw[kernel1] (centerr) to
     node [sloped,below] {\small }     (center);
     \end{tikzpicture} \otimes \one + \one \otimes \begin{tikzpicture}[scale=0.15,baseline=0.1cm]
        \node at (0,0)  [dot,blue] (root) {};
         \node at (0,3)  [dot] (center) {};
          \node at (2,5)  [circ] (centerr) {\mbox{\tiny $ $}};
     \draw[kernel1] (center) to
     node [sloped,below] {\small }     (root);
      \draw[kernel1] (centerr) to
     node [sloped,below] {\small }     (center);
     \end{tikzpicture} + X_i \otimes \begin{tikzpicture}[scale=0.15,baseline=0.1cm]
        \node at (0,0)  [dot,blue] (root) {};
         \node at (0,5)  [dot] (center) {};
          \node at (2,7)  [circ] (centerr) {\mbox{\tiny $ $}};
     \draw[kernel1] (center) to
     node [rect] { \mbox{\tiny $ i $}}     (root);
      \draw[kernel1] (centerr) to
     node [sloped,below] {\small }     (center);
     \end{tikzpicture}. 
  \end{equs}

\subsubsection{Construction of extended models}

In general if, for some sequence $\PPi^{(n)} \colon \CT^\ex \to \CC^\infty$, $\CZ^\ex(\PPi^{(n)})\in\MM_\infty^\ex$
converges to a limiting model in $\MM_0^\ex$, it does \textit{not} follow that the 
characters $g_+(\PPi^{(n)})$ of $\hat\CT^\ex_+$ converge to a limiting character. However, we claim that 
the characters $f_x^{(n)}$ of $\CT^\ex_+$ given by \eqref{e:defModel1} \textit{do} converge, which is 
not so surprising since our definition of convergence implies that 
the characters $\gamma_{xy}^{(n)}$ of $\CT^\ex_+$ given by \eqref{e:defModel2}
do converge. More surprising is that the convergence of the characters 
$f_x^{(n)}$ follows already from a seemingly much weaker type of convergence.
Writing $\CD'$ for the space of distributions on $\R^d$, we have the following.

\begin{proposition}\label{prop:conv}
Let $\PPi^{(n)} \colon \CT^\ex \to \CC^\infty$ be an admissible linear map with 
\[
\CZ^\ex(\PPi^{(n)}) = (\Pi^{(n)},\Gamma^{(n)}) \in \MM_\infty^\ex
\]
and assume that there exist linear maps $\Pi_x \colon \CT^\ex \to \CD'(\R^d)$ such that, with the notation
of \eqref{e:defDistModel},
$\|\Pi^{(n)} - \Pi\|_{\ell,\K} \to 0$ for every $\ell \in \R$ and every compact set $\CK$. 
Then, the characters $f_x^{(n)}$ defined as in \eqref{e:defModel1} converge to a limit $f_x$. Furthermore,
defining $\Gamma_{xy}$ by \eqref{e:defModel2}, one has
$\CZ = (\Pi,\Gamma) \in \MM_0^\ex$ and $\CZ^\ex(\PPi^{(n)}) \to \CZ$ in $\MM_0^\ex$. 

Finally,
one has $\PPi \colon \CT^\ex \to \CD'(\R^d)$ such that $\Pi_x = (\PPi \otimes f_x)\Deltap_\ex$ and such that
$\PPi^{(n)}\tau \to \PPi\tau$ in $\CD'(\R^d)$ \, for every $\tau \in \CT^\ex$.
\end{proposition}

\begin{proof}
The convergence of the $f_x^{(n)}$ follows immediately from the formula given in 
Lemma~\ref{lem:algpropr}, combined with the convergence of the $\Pi_x^{(n)}$ and
\cite[Lem.~5.19]{reg}. The fact that $(\Pi,\Gamma)$ satisfies the algebraic identities
required for a model follows immediately from the fact that this is true for every $n$. 
The convergence of the $\Gamma_{xy}^{(n)}$ and the analytical bound on the limit then follow 
from \cite[Sec.~5.1]{reg}.
\end{proof}

\begin{remark}
This relies crucially on the fact that the maps $\PPi$ under consideration are admissible and
that the kernels $K_\Labhom$ satisfy the assumptions of \cite[Sec.~5]{reg}. If one considers
different notions of admissibility, as is the case for example in \cite{woKP}, then the
conclusion of Proposition~\ref{prop:conv} may fail.
\end{remark}

For a linear $\PPi \colon \CT \to \CC^\infty$ we define $\PPi^\ex \colon \CT^\ex \to \CC^\infty$ 
by simply setting 
$\PPi^\ex = \PPi \CQ$. Then we say that $\PPi$ is admissible if $\PPi^\ex$ is.
We have the following crucial fact

\begin{theorem}\label{theo:models}
If $\PPi \colon \CT \to \CC^\infty$ is admissible and $\CZ(\PPi^\ex)$ belongs to $\MM_\infty$, 
then $\CZ^\ex(\PPi^\ex)$ belongs to $\MM_\infty^\ex$. Furthermore, the
map $\CZ(\PPi^\ex) \mapsto \CZ^\ex(\PPi^\ex)$ extends to a continuous map from $\MM_0$ to $\MM_0^\ex$.
\end{theorem}
Before proving this Theorem, we define a linear map $L:\CT^\ex\to\CT\otimes\CT_+$ such that 
\begin{equ}
L \Xi_{k,\ell}^\noise = \Xi_{k,\ell}^\noise\otimes\one\;,\qquad
L X^k = X^k\otimes\one\;,
\end{equ}
and then recursively
\begin{equ}
L \CR_\alpha(\tau) = L\tau\;,\qquad
L(\tau \bar \tau) = L(\tau) L(\bar \tau)\;,
\end{equ}
as well as
\begin{equ}[e::defL]
L \CI_k^\Labhom(\tau) 
=  (\CI_k^\Labhom \otimes \id)L \tau 
- \sum_{|m|_\s \ge |\CI_{k}^\Labhom\tau|_+}
{X^m \over m!} \otimes 
\CM_+ \bigl(  \tilde\CJ_{k+m}^\Labhom \otimes \id\bigr)L\tau\;,
\end{equ}
where $\CM_+$ is the tree product \eqref{odot} on $\CT_+$ and $\tilde\CJ$ is as in \eqref{e:tildeJ}.

Moreover $L_+:\CT^\ex_{+}\to\CT_+$ is the algebra morphism such that 
$L_+ X^k = X^k$ and for $\CJ_k^\Labhom(\tau) \in\CT^\ex_+$ with $|\CJ_k^\Labhom(\tau)|_+>0$
\begin{equ}[e::defL+]
L_+ \tilde\CJ_k^\Labhom(\tau) = 
\CM_+\left(  \tilde\CJ_k^\Labhom \otimes \id \right) L \tau \;. 
\end{equ}
The reason for these definitions is that these map will provide the required injection
$\MM_0 \to \MM_0^\ex$ by \eqref{e:extended} below.
Before we proceed to show this, we state the following preliminary identity. 
\begin{lemma}\label{lem:L}
On $\CT^\ex$
\begin{equ}[e:mainL]
(\id\otimes\CM_+)(\Delta\otimes\id) L  =  (\CQ   \otimes  L_+ )\Deltap_\ex \; .
\end{equ}
\end{lemma}
\begin{proof}
We prove \eqref{e:mainL} by recursion. Both maps  in \eqref{e:mainL} agree on elements of the form 
$\Xi_{k,\ell}^\noise$ or $X^k$ and both maps are
multiplicative for the tree product. Consider  now a tree of the form $ \CI_k^\Labhom(\tau) $ and
assume that \eqref{e:mainL} holds when applied to $\tau$. Then we have by \eqref{e:recursiveDelta}
\begin{equs}
(\id&\otimes\CM_+)(\Delta\otimes\id) L \CI_k^\Labhom(\tau)  = (\id\otimes\CM_+)(\Delta\CI_k^\Labhom \otimes \id)L \tau
\\ & \quad - (\id\otimes\CM_+)\sum_{|m|_\s \ge |\CI_{k}^\Labhom\tau|_+}
{\Delta X^m \over m!} \otimes 
\CM_+ \bigl( \tilde\CJ_{k+m}^\Labhom \otimes \id\bigr)L\tau
\\ & = \left( \CI^{\Labhom}_k \otimes \CM_+ \right)(\Delta \otimes \id) L\tau
+ \sum_{\ell,m}  \frac{X^{\ell}}{\ell !}  \otimes  {X^{m} \over m!}    \CM_+( \tilde\CJ^{\Labhom}_{k+ \ell+m} \otimes \id)L\tau
\\ & \quad - \sum_{|\ell+m|_\s \ge |\CI_{k}^\Labhom\tau|_+}
{X^\ell \over \ell!}\otimes {X^{m} \over m!}   
\CM_+ \bigl( \tilde\CJ_{k+\ell+m}^\Labhom \otimes \id\bigr)L\tau
\\ &= \left( \CI^{\Labhom}_k \otimes \CM_+ \right)(\Delta \otimes \id) L\tau
+ \sum_{|\ell+m|_\s < |\CI_{k}^\Labhom\tau|_+}
{X^\ell \over \ell!}\otimes {X^{m} \over m!}   
\CM_+ \bigl( \tilde\CJ_{k+\ell+m}^\Labhom \otimes \id\bigr)L\tau\;.
\end{equs}
On the other hand
\begin{equs}
(\CQ &  \otimes  L_+ )\Deltap_\ex \CI_k^\Labhom(\tau) = 
\\ & = \left( \CQ\CI^{\Labhom}_k \otimes L_+ \right)  \Deltap_\ex \tau
+ \sum_{|\ell+m|_\s < |\CI_{k}^\Labhom\tau|_+}  \frac{X^{\ell}}{\ell !}  \otimes {X^{m} \over m!} L_+\tilde\CJ^{\Labhom}_{k+ \ell+m}(\tau)
\\ & = \left( \CI^{\Labhom}_k\,\CQ \otimes L_+ \right)  \Deltap_\ex \tau
+ \sum_{|\ell+m|_\s < |\CI_{k}^\Labhom\tau|_+}
{X^\ell \over \ell!}\otimes {X^{m} \over m!}   
\CM_+ \bigl( \tilde\CJ_{k+\ell+m}^\Labhom \otimes \id\bigr)L\tau.
\end{equs}
Comparing both right hand sides and using the induction hypothesis, we
conclude that \eqref{e:mainL} does indeed hold as claimed.
\end{proof}
\begin{proof}[of Theorem~\ref{theo:models}]
Let $\PPi\colon \CT \to \CC^\infty$ be such that $\CZ(\PPi^\ex) = (\Pi,\Gamma)$ is a model of $\TT$
and write $(\Pi^\ex, \Gamma^\ex) = \CZ^\ex(\PPi^\ex)$. In accordance with \eqref{e:defModel1} and \eqref{e:defModel2}, 
we set
\[
f^\ex_z \eqdef g^+_z(\PPi^\ex)\tilde\CA_+^\ex, \qquad \gamma_{z\bar z}^\ex \eqdef
\bigl( f_z^\ex\CA_{+}^{\ex} \otimes f_{\bar z}^\ex\bigr)\Deltap_{\ex}\;,
\]
so that one has
\begin{equ}
\Pi_z^\ex = \bigl(\PPi^\ex \otimes f_z^\ex\bigr)\Deltap_{\ex} \;, \qquad
\Gamma_{\!z\bar z}^\ex = \left( \id \otimes \gamma_{\!z\bar z}^\ex \right) \Deltap_{\ex} \;.
\end{equ}
With the notations introduced in \eqref{e:Deltarestr}, the model $(\Pi, \Gamma) = \CZ(\PPi^\ex)$
is then given by  
\[
\Pi_z\tau=\bigl(\PPi \otimes f_z\bigr)\Delta\tau \;, \qquad
\Gamma_{\!z\bar z}\tau = \left( \id \otimes \gamma_{z\bar z} \right) \Delta\tau \;, \qquad 
\tau\in\CT,
\]
where $f_z=f_z^\ex\restr\CT_+$ and similarly for $\gamma_{z\bar z}$.
Define $\hat\Pi_z:\CT^\ex\to\CC^\infty$ and $\hat f_z\in\CG_+^\ex$ by
\begin{equ}[e:extended]
\hat\Pi_z \eqdef (\Pi_z\otimes f_z)L, \qquad \hat f_z\eqdef f_z L_+\;,
\end{equ}
where $L,L_+$ are defined in \eqref{e::defL}-\eqref{e::defL+}.
We want to show that $\PPi^\ex = (\hat\Pi_z\otimes \hat f_z\CA_+^\ex)\Deltap_\ex$ for all $z$. By the definitions 
\begin{equs}
(\hat\Pi_z\otimes \hat f_z\CA_+^\ex)\Deltap_\ex &= (\Pi_z\otimes f_z\otimes \hat f_z\CA_+^\ex)
(L\otimes \id)\Deltap_\ex
\\ & = (\PPi\otimes f_z\CM_+\otimes \hat f_z\CA_+^\ex)((\Delta\otimes\id)L\otimes \id)\Deltap_\ex.
\end{equs}
By \eqref{e:mainL}
\begin{equs}
(\hat\Pi_z\otimes \hat f_z\CA_+^\ex)\Deltap_\ex &= (\PPi\otimes f_z\otimes \hat f_z\CA_+^\ex)((\CQ\otimes L_+)\Deltap_\ex\otimes \id)\Deltap_\ex
\\ & = (\PPi\CQ\otimes \hat f_z\otimes \hat f_z\CA^\ex_+)(\Deltap_\ex\otimes\id)\Deltap_\ex
\\ & = (\PPi^\ex\otimes \hat f_z\otimes \hat f_z\CA^\ex_+)(\id\otimes\Deltap_\ex)\Deltap_\ex = \PPi^\ex.
\end{equs}
We want now to show that $\hat f_z\equiv f^\ex_z$ on $\CT^\ex_+$. By Remark~\ref{rem:freduced}, for $\CJ_k^\Labhom(\sigma)\in\CT_+$ 
with $|\CJ_k^\Labhom(\sigma)|_+>0$ we have
\begin{equ}
f_z(\tilde\CJ_k^{\Labhom}(\sigma))  = -   \big( D^{k} K_{\Labhom} * \Pi_{z} \sigma \big)(z)\;.
\end{equ}
Therefore, by the definitions of $\hat f_z$ and $L_+$, for all $\CJ_k^\Labhom(\tau)\in\CT_+^\ex$ with $|\CJ_k^\Labhom(\tau)|_+>0$
\begin{equs}
\hat f_z (\tilde\CJ_k^\Labhom(\tau)) \,&= \left( f_z\tilde\CJ_k^\Labhom \otimes f_z \right) L \tau =
-   \big( D^{k} K_{\Labhom} * (\Pi_{z} \otimes f_z)L\tau \big)(z)
\\ & =- \big( D^{k} K_{\Labhom} * \hat\Pi_{z} \tau \big)(z)\;,
\end{equs}
%Setting $(\id\otimes f_z)L\tau=\sum_i c_i \sigma_i$, we then obtain
%\begin{equs}
%\hat f_z& (\tilde\CJ_k^\Labhom(\tau)) =\sum_i c_i f_z\tilde\CJ_k^\Labhom(\sigma_i) = -\sum_i c_i \big( D^{k} K_{\Labhom} * \Pi_{z} \sigma_i \big)(z)
%\\ & = - \big( D^{k } K_{\Labhom} * (\Pi_{z}\otimes f_z)L\tau \big)(z)
% =- \big( D^{k} K_{\Labhom} * \hat\Pi_{z} \tau \big)(z)\;,
%\end{equs}
which is equal to $f^\ex_z(\tilde\CJ_k^\Labhom(\tau))$ by Lemma~\ref{lem:algpropr} and Remark~\ref{rem:freduced}. 
Since $\hat f_z$ and $f^\ex_z$ are multiplicative 
linear functionals on $\CT^\ex_+$ and they coincide on a set which generates $\CT^\ex_+$ as an algebra, we conclude 
that $\hat f_z\equiv f^\ex_z$ on $\CT^\ex_+$ and therefore that
$\hat \Pi_z\equiv\Pi^\ex_z$ on $\CT^\ex$. Finally, we can prove by recurrence that for all $\tau\in\CT^\ex$ and $\bar\tau\in\CT_+^\ex$
\[
L\tau = \CQ\tau\otimes\one + \sum_i \tau^{(1)}_i\otimes\tau^{(2)}_i,
\qquad L_+\bar\tau=\CQ\bar\tau + \sum_i \bar\tau_i,
\]
with $|\tau^{(1)}_i|_+ \ge |\tau|_+$ and $|\bar\tau^{(1)}_i|_+ \ge |\bar\tau|_+$. This implies the required 
analytical estimates for $(\Pi^\ex,\Gamma^\ex)$.
\end{proof}

\subsubsection{Renormalisation group of the reduced structure}\label{sec:6.5}

In this section, we show that the action of the renormalisation group $\CG_-$ on $\MM_0$ 
given by Theorem~\ref{thm:renorred} is indeed given by elements of the ``renormalisation group'' 
${\mathfrak R}$ as defined in \cite[Sec. 8.3]{reg}. This shows in particular 
that the BPHZ renormalisation procedure given in Theorem~\ref{main:renormalisation}
does always fit into the framework developed there.

We recall that, by \cite[Lem.~8.43, Thm~8.44]{reg} and \cite[Thm~B.1]{woKP}, 
${\mathfrak R}$ is the set of linear operators $M:\CT\to\CT$ satisfying the following properties.
\begin{itemize}
\item One has $\CI_k^\Labhom M\tau =M \CI_k^\Labhom\tau$ and $MX^k \tau=X^k M \tau$
for all $\Labhom\in\Lab_+$, $k\in\N^d$, and $\tau \in \CT$.
\item Consider the (unique) linear operators $\DeltaM:\CT\to\CT\otimes\CT_+$ and $\hat M:\CT_+\to\CT_+$ 
such that $\hat M$ is an algebra morphism,  $\hat M X^k=X^k$ for all $k$, and such that,
for every $\tau \in \CT$ and every $\sigma \in \CT$ and $k \in \N^d$ with $|\CJ^\Labhom_k(\sigma)|_+>0$,
\begin{equs}
\hat M \tilde\CJ^\Labhom_k(\sigma) =\CM_+(\tilde\CJ^\Labhom_k&\otimes\id)\DeltaM\sigma\;,  \label{e:hatM} \\
(\id \otimes\CM_+)(\Delta\otimes\id)\DeltaM \tau &=(M\otimes\hat M)\Delta\tau\;, \label{e:hatM2}
\end{equs}
where $\tilde\CJ^\Labhom_k \colon \CT \to \CT_+$ is defined by \eqref{e:tildeJ}.
Then, for all $\tau\in\CT$, one can write
$\DeltaM \tau= \sum \tau^{(1)}\otimes\tau^{(2)}$
with $| \tau^{(1)}|_+ \ge | \tau|_+$.
\end{itemize}

\begin{remark}
Despite what a cursory inspection may suggest, the condition \eqref{e:hatM} is 
\textit{not} equivalent to the same expression with 
$\tilde\CJ^\Labhom_k$ replaced by $\CJ^\Labhom_k$. This is because \eqref{e:hatM} will typically
fail to hold when $|\CJ^\Labhom_k(\sigma)|_+ \le 0$.
\end{remark}

%We recall that for every $M\in{\mathfrak R}$ and every model $\CZ(\PPi)=(\Pi,\Gamma)\in\MM_0$ for some admissible
%$\PPi \colon \CT \to \CC^\infty$, one can construct a model $(\Pi^M,\Gamma^M)\in\MM_0$ by setting
%\[
%\Pi^M_z:=(\Pi_z\otimes f_z)\DeltaM, \qquad \Gamma^M:=(\gamma_{xy}\otimes f_y)\DeltahM,
%\]
%where $\DeltahM:\CT_+\to\CT_+\otimes\CT_+$ is uniquely determined by $\{M,\hat M,\DeltaM\}$, see 
%\cite[Eq.~8.38]{reg} and \cite[Prop.~B.5]{woKP}.
%\[
%(\CA_+\hat M\CA_+\otimes\hat M)\Deltap=(\id\otimes\CM)(\Deltap\otimes\id)\hat \Delta^{\! M}.
%\]
We recall that the group $\CG_-\eqdef\{ g\in \CG^{\ex}_-: 
g( \tau)= g( \CQ \tau) , \, \forall\, \tau \in \CT_{-}^{\ex}  \}$ has beed defined after Remark \ref{after}.
\begin{theorem}\label{thm:cg-r}
Given $g\in\CG_-$, define $M_g^\ex$ on $\CT^\ex$ and $\CT_+^\ex$
as in \eqref{e:defMg} and let $M_g\colon \CT\to\CT$ be given by
$M_g = \CQ M_g^\ex$.
Then $M_g \in \RR$, $g \mapsto M_g$ is a group homomorphism, and one has the identities
\begin{equ}[e:iden]
\hat M_g =  L_+  M_g^\ex \colon \CT_+\to\CT_+\;,\qquad
\DeltaMg = L M^{\ex}_g\colon \CT\to\CT\otimes\CT_+\;, 
\end{equ}
where the maps $L,L_+$ are given in \eqref{e::defL}--\eqref{e::defL+}.
\end{theorem}
\begin{proof}
In order to check \eqref{e:hatM}, it suffices by \eqref{e:iden} to use \eqref{e::defL+} and the fact that $M_g^\ex$ preserves the
$|\cdot|_+$-degree.
It remains to check \eqref{e:hatM2}. We have on $\CT$ that
\begin{equs}
(\id\otimes\CM_+)(\Delta\otimes\id)\DeltaMg &= (\id\otimes\CM_+)(\Delta\otimes\id) L M^{\ex}_g \;,\\
(M_g\otimes\hat M_g)\Delta  &=(\CQ M^{\ex}_g  \otimes  L_+ M^{\ex}_g)\Delta =
\\ & =(\CQ M^{\ex}_g  \otimes  L_+ M^{\ex}_g)\Deltap_\ex  =  (\CQ   \otimes  L_+ )\Deltap_\ex M^{\ex}_g\;,
\end{equs}
where we have used the co-interaction property in the last line. It follows from \eqref{e:mainL}
that these two terms are indeed equal.
The triangularity of $L$ and $M^{\ex}_g$, combined with \eqref{e:iden},
implies the triangularity of $\DeltaMg$.

The homomorphism property follows from \eqref{e:intertwiningQDelta} and the definition of $\CG_-$ since
\begin{equs}
M_{\bar g} M_g &=  \CQ M_{\bar g}^\ex\CQ M_g^\ex 
= (\bar g \otimes \CQ)\Deltam_\ex \CQ M_g^\ex 
= (\bar g \CQ \otimes \CQ)\Deltam_\ex \CQ M_g^\ex \\
&= (\bar g \CQ \otimes \CQ)\Deltam_\ex M_g^\ex 
= (\bar g \otimes \CQ)\Deltam_\ex M_g^\ex 
= \CQ M_{\bar g}^\ex M_g^\ex
= \CQ M_{\bar g g}^\ex
= M_{\bar g g}\;,
\end{equs}
as required.
\end{proof}

\begin{corollary}\label{final}
The space $\CT_- \eqdef \CT_-^\ex / \ker \CQ$ inherits from $\CT_-^\ex$ a Hopf algebra structure
and its group of characters is isomorphic to $\CG_-$. Furthermore, the map
\begin{equ}
\Deltam \colon \CT \to \CT_- \otimes \CT\;,\qquad
\Deltam \eqdef (\id \otimes \CQ)\Deltam_\ex\;,
\end{equ}
turns $\CT$ into a left comodule for $\CT_-$.
\end{corollary}

\begin{proof}
This follows immediately from \eqref{e:intertwiningQDelta}, Theorem~\ref{thm:cg-r}, the 
definition of $\CG_-$, the fact that $\CQ$ is an algebra morphism on $\CT_-^\ex$, and the
same argument as in the proof of Proposition \ref{lem:treeprod}.
\end{proof}

By the Remarks \ref{rem:6.19} and \ref{after}, the renormalisation procedures
of \cite{KPZ,reg,wong,2015arXiv150701237H} 
can be described in this framework.

%In particular we obtain on $\CT$
%\[
%\Pi^M_x = \Pi_{x}^{M^{\ex}_g} = \Pi^{\ex}_{x} M^{\ex}_g =
%\left( \Pi_x \otimes f_x \right) L M^{\ex}_g.
%\]
%
%
%The renormalisation procedures performed in \cite{KPZ,reg,wong,2015arXiv150701237H} 
%are covered by the above construction with $g\in\CG_-^\ex$ the BPHZ functional, see Remark~\ref{rem:6.19} above.

\appendix
\section{Spaces and canonical basis vectors}
\label{sec:diagrams}

The following diagram summarises the relations between the main spaces
appearing in this article.
\begin{equ}
\xymatrix{
&  \langle\Tra_1\rangle \ar@{_{(}->}[dl]  \ar@{->>}[r]^{\CK_1} &  
\CH_{1} &  \hat \CT^{\ex}_{-} \ar@{_{(}->}[l]
\ar@{->>}@/_/[r]_{\proj_{-}^\ex}
&  \CT^{\ex}_{-} \rlap{$\,=\hat \CT_-^\ex / \CJ_+$} \ar@{^{(}->}@/_/[l]_{\inj_{-}^\ex} 
\\
\langle\Tra \rangle & \langle\Tra_{\circ}\rangle \ar@{->>}[r]^{\CK} \ar@{_{(}->}[l]&
\CH_{\circ} & \CT^{\ex} \ar@{_{(}->}[l] \ar@{^{(}->}[u]_{\iota_\circ}
\\
& \langle\Tra_2\rangle \ar@{_{(}->}
[ul] \ar@{->>}[r]^{\JJ \hat \CH_2} & 
\hat \CH_{2}&  \hat \CT^{\ex}_{+} \ar@{_{(}->}[l]\ar@{->>}@/_/[r]_{\proj_{+}^\ex} 
&  \CT^{\ex}_{+}\rlap{$\,=\hat \CT_+^\ex / \CJ_-$} \ar@{^{(}->}@/_/[l]_{\inj_{+}^\ex} 
 }
\end{equ}
The next diagram similarly shows the relations between various
sets of trees / forests. The first four columns in this diagram 
show the canonical basis vectors for the spaces appearing in the
first four columns of the previous diagram.
\begin{equ}
\xymatrix{
&  \Tra_1 \ar@{_{(}->}[dl]  \ar@{->>}[r]^{\CK_1} &  
H_{1} &  B_- \ar@{_{(}->}[l]
\ar@{->>}[r]_{\CS}
&  \Trees_1(R) 
\\
\Tra & \Tra_{\circ} \ar@{->>}[r]^{\CK} \ar@{_{(}->}[l]&
H_{\circ} & B_\circ \ar@{_{(}->}[l] \ar@{^{(}->}[u]_{\iota_\circ}
\ar@{->>}[r]_{\CS}
&  \Trees_\circ(R)  \ar@{^{(}->}[u] \ar@{_{(}->}[d] &  \Trees_-(R) \ar@{_{(}->}[l]
\\
& \Tra_2 \ar@{_{(}->}
[ul] \ar@{->>}[r]^{\JJ \hat \CH_2} & 
\hat H_{2}&  B_{+} \ar@{_{(}->}[l]\ar@{->>}[r]_{\CS} 
&  \Trees_2(R)  
 }
\end{equ}

\section{Symbolic index}
\label{sec:index}

Here, we collect some of the most used symbols of the article, together
with their meaning and the page where they were first introduced.

\begin{center}
\renewcommand{\arraystretch}{1.1}
\begin{longtable}{lll}
\toprule
Symbol & Meaning & Page\\
\midrule
\endfirsthead
\toprule
Symbol & Meaning & Page\\
\midrule
\endhead
\bottomrule
\endfoot
\bottomrule
\endlastfoot
$|\cdot|_\bi$ & Bigrading on coloured decorated forests & \pageref{e:grading0} \\
$|\cdot|_-$ & Degree not taking into account the label $\Labo$ & \pageref{homog} \\
$|\cdot|_+$ & Degree taking into account the label $\Labo$ & \pageref{homog} \\
$\Adm_i$ & Subforests appearing in the definition of $\Delta_i$ & \pageref{ass:1} \\
$\CA_i$ & Antipode of $\CH_i$ & \pageref{prop:Hopf}\\
$\CA_\pm^\ex$ & Antipode of $\CT_\pm^\ex$ & \pageref{prop:Hopf+} \\
$\tilde \CA_\pm^\ex$ & Twisted antipode $\CT_\pm^\ex \to \hat\CT_\pm^\ex$ & \pageref{e:pseudoant} \\
$\hat\CA_2$ & Antipode of $\hat\CH_2$ & \pageref{prop:alg}\\
$B_\circ$ & Elements of $H_\circ$ strongly conforming to the rule $R$ & \pageref{def:CT}\\
$B_\circ^-$ & Elements of $B_\circ$ of negative degree & \pageref{defB0min}\\
$B_\circ^\sharp$ & Elements of $B_\circ^-$ that are not planted & \pageref{defBsharp}\\
$B_-$ & Elements of $H_1$ strongly conforming to the rule $R$ & \pageref{def:CT}\\
$B_+$ & Elements of $\hat H_2$ conforming to the rule $R$ & \pageref{def:CT}\\
$\BB_i$ & Hopf algebra of coloured forests & \pageref{prop:combi}\\
$\mathfrak{C}$ & All coloured forests $(F,\hat F)$ & \pageref{def:colour}\\
$\mathfrak{C}_i$ & All coloured forests compatible with $\Adm_i$ & \pageref{def:tra_i}\\
$\CD_i(\JJ)$ & All roots of colour in $\{0,i\}$ & \pageref{defDJJ}\\
$\hat \CD_i(\JJ)$ & All roots of colour $i$ & \pageref{defDJJhat}\\
$\Delta_i$ & Coproduct on $\scal{\Tra}$ turning the $\scal{\Tra_i}$ into bialgebras & \pageref{def:Deltabar}\\
$\CE$ & Edge types given by $\CE = \Lab \times \N^d$ & \pageref{e:defCN} \\
$f \restr A$ & Restriction of the function $f$ to the set $A$ & \pageref{frestr} \\
$\Tra$ & All decorated forests $(F,\hat F, \Labn,\Labo,\Labe)$ & \pageref{def:decoration} \\
$\Tra_i$ & All decorated forests compatible with $\Adm_i$ & \pageref{def:tra_i}\\
$\Tra_\circ$ & Trees with colours in $\{0,1\}$ & \pageref{eq:CHcirc} \\
$\Phi_i$ & Collapse of factors in $\M_i$ & \pageref{e:bi-ideal}\\
$g^+_z(\PPi)$ & Character on $\hat \CT^{\ex}_+$ defined by $\PPi$ & \pageref{def:modelmap}\\
$g^-_z(\PPi)$ & Character on $\hat \CT^{\ex}_-$ defined by $\PPi$ & \pageref{e:defHeps}\\
$\CG_i$ & Characters of $\CH_i$ & \pageref{e:defG}\\
$\CG_\pm^\ex$ & Character group of $\CT_\pm^\ex$ & \pageref{def:charpm} \\
$\hat\CG_2$ & Characters of $\hat\CH_2$ & \pageref{charGhat2}\\
$\CH_\circ$ & Algebra given by $\scal{\Tra_\circ} / \ker \CK$ & \pageref{eq:CHcirc} \\
$\hat \CH_2$ & Hopf algebra $\CH_2 / \ker (\JJ\hat P_2)$ & \pageref{eq:hatCH2} \\
$\CH_i$ & Hopf algebra $\scal{\Tra_i} / \CI_i$ & \pageref{e:defHi}\\
$H_\circ$ & Representative of $\CH_\circ$ given by $H_\circ = \CK \Tra_\circ$ & \pageref{e:isoH}\\
$\hat H_2$ & Representative of $\hat\CH_2$ given by $\hat H_2 = \JJ\hat \CK_2\Tra_2$ & \pageref{e:isoH}\\
$H_i$ & Representative of $\CH_i$ given by $H_i = \CK_i\Tra_i$ & \pageref{rem:H_i}\\
$\inj_\pm^\ex$ & Canonical injection $\CT_\pm^\ex \hookrightarrow \hat \CT_\pm^\ex$ & \pageref{inj}\\
$\CI_i$ & Kernel of $\CK_i$ & \pageref{lem:biideal}\\
$\hat\CI_i$ & Kernel of $\hat \CK_i$ & \pageref{lem:biideal}\\
$\CI_k^\Labhom$ & Abstract integration map in $\CH_\circ$ & \pageref{e:CI}\\
$\JJ$ & Joins the root of all trees together & \pageref{sec:join}\\
$\CJ_k^\Labhom$ & Abstract integration map $\CH_\circ \to \hat \CH_2$ & \pageref{e:CJ}\\
$\CJ_+$ & Subspace of terms in $\hat \CT_-^\ex$ with a factor of positive degree & \pageref{CT^ex_pm} \\ 
$\CJ_-$ & Subspace of terms in $\hat \CT_+^\ex$ with a factor of negative degree & \pageref{CT^ex_pm} \\ 
$\CK$ & Contraction of coloured portions & \pageref{CKop} \\
$\CK_i$ & Defined by $\CK_i = \Phi_i \circ \CK$ & \pageref{e:bi-ideal}\\
$\hat \CK_i$ & Defined by $\hat \CK_i = P_i \circ \Phi_i \circ \CK$ & \pageref{e:bi-ideal}\\
$|k|$ & Unscaled length of a multi-index $k$ & \pageref{e:grading0} \\
$|k|_\s$ & Scaled length of a multi-index $k$ & \pageref{def:scaling} \\
$\Lab$ & Set of all types & \pageref{lab} \\
$\M_i$ & Elements of $\Tra_i$ completely coloured with $i$ & \pageref{Mi} \\
$\MM$ & Space of all models & \pageref{def:modelspace} \\
$\MM_0$ & Closure of smooth models & \pageref{def:modelspace} \\
$\MM_\infty$ & Space of all smooth models & \pageref{def:modelspace} \\
$\CN$ & Node types given by $\CN = \hat \CP(\CE)$ & \pageref{e:defCN} \\
$\CN(x)$ & Type of the node $x$ & \pageref{e:defNx} \\
$\hat P_i$ & Sets $\Labo$-decoration to $0$ on $i$-coloured roots & \pageref{e:bi-ideal}\\
$\CP(A)$ & Powerset of the set $A$ & \pageref{def:rule} \\
$\hat \CP(A)$ & Multisets with elements from the set $A$ & \pageref{e:defCN} \\
$\proj_\pm^\ex$ & Canonical projection $\hat \CT_\pm^\ex \to \CT_\pm^\ex$ & \pageref{CT^ex_pm} \\
$\PPi$ & Linear map $\CT^\ex \to \CC^\infty$ specifying a model & \pageref{def:modelmap} \\
$R$ & Rule determining a class of trees & \pageref{def:rule}\\
$\CR_\alpha$ & Operator adding $\alpha$ to $\Labo$ at the root & \pageref{e:CJ}\\
$\s$ & Scaling of $\R^d$ & \pageref{def:scaling} \\ 
$\Trees$ & Simple decorated trees & \pageref{def:Trees} \\
$\Trees_\circ(R)$ & Trees strongly conforming to the rule $R$ & \pageref{def_space}\\
$\Trees_1(R)$ & Forests strongly conforming to the rule $R$ & \pageref{def_space}\\
$\Trees_2(R)$ & Trees conforming to the rule $R$ & \pageref{def_space}\\
$\Trees_-(R)$ & Trees strongly conforming to $R$ of negative degree & \pageref{def_space}\\
$\hat\CT_+^\ex$ & Subspace of $\hat\CH_2$ determined by a rule $R$ & \pageref{def:CT} \\
$\hat\CT_-^\ex$ & Subspace of $\CH_1$ determined by a rule $R$ & \pageref{def:CT} \\
$\CT^\ex$ & Subspace of $\CH_\circ$ determined by a rule $R$ & \pageref{def:CT} \\
$\CT_+^\ex$ & Quotient space $\hat \CT_+^\ex / \CJ_-$ & \pageref{CT^ex_pm} \\
$\CT_-^\ex$ & Quotient space $\hat \CT_-^\ex / \CJ_+$ & \pageref{CT^ex_pm} \\
$\Units_i$ & Units of $\scal{\Tra_i}$ & \pageref{e:counit} \\
$\scal{V}$ & Bigraded space generated from a bigraded set $V$ & \pageref{e:defForests}\\
$X^k$ & Shorthand for $(\bullet,i)_0^{k,0}$ with $i \in \{0,2\}$ depending on context & \pageref{defX}\\
$\Xi^\noise$ & Element $\CI_0^\noise(\one)$ representing the noise & \pageref{noise} \\
$\|z\|_\s$ & Scaled distance & \pageref{e:defds}\\
$\CZ^\ex$ & Map turning $\PPi$ into a model & \pageref{e:defModel2} \\
\end{longtable}
\end{center}

\endappendix

\bibliographystyle{Martin}

\bibliography{these}

\begin{thebibliography}{CEFM11}
\expandafter\ifx\csname url\endcsname\relax
  \def\url#1{\texttt{#1}}\fi
\expandafter\ifx\csname urlprefix\endcsname\relax\def\urlprefix{URL }\fi
\expandafter\ifx\csname href\endcsname\relax
  \def\href#1#2{#2}\fi
\expandafter\ifx\csname burlalt\endcsname\relax
  \def\burlalt#1#2{\href{#2}{\texttt{#1}}}\fi

\bibitem[AR91]{AlbRock91}
\textsc{S.~Albeverio} and \textsc{M.~R{\"o}ckner}.
\newblock Stochastic differential equations in infinite dimensions: solutions
  via {D}irichlet forms.
\newblock \emph{Probab. Theory Related Fields} \textbf{89}, no.~3, (1991),
  347--386.
\newblock
  \burlalt{doi:10.1007/BF01198791}{http://dx.doi.org/10.1007/BF01198791}.

\bibitem[BCCH17]{BCCH}
\textsc{Y.~{Bruned}}, \textsc{A.~{Chandra}}, \textsc{I.~{Chevyrev}}, and
  \textsc{M.~{Hairer}}.
\newblock {Renormalising SPDEs in regularity structures}.
\newblock \emph{ArXiv e-prints} (2017).
\newblock \burlalt{arXiv:1711.10239}{http://arxiv.org/abs/1711.10239}.

\bibitem[BP57]{BP}
\textsc{N.~N. Bogoliubow} and \textsc{O.~S. Parasiuk}.
\newblock \"{U}ber die {M}ultiplikation der {K}ausalfunktionen in der
  {Q}uantentheorie der {F}elder.
\newblock \emph{Acta Math.} \textbf{97}, (1957), 227--266.
\newblock
  \burlalt{doi:10.1007/BF02392399}{http://dx.doi.org/10.1007/BF02392399}.

\bibitem[Bru18]{YB}
\textsc{Y.~Bruned}.
\newblock Recursive formulae in regularity structures.
\newblock \emph{Stoch. Partial Differ. Equ. Anal. and Comput.} \textbf{6},
  no.~4, (2018), 525--564.
\newblock \burlalt{arXiv:1710.10634}{http://arxiv.org/abs/1710.10634}.
\newblock
  \burlalt{doi:10.1007/s40072-018-0115-z}{http://dx.doi.org/10.1007/s40072-018-0115-z}.

\bibitem[BS18]{Liegroup}
\textsc{G.~{Bogfjellmo}} and \textsc{A.~{Schmeding}}.
\newblock The geometry of characters of {H}opf algebras.
\newblock vol.~13 of \emph{Abel Symp.} Springer, 2018.
\newblock \burlalt{arXiv:1704.01099}{http://arxiv.org/abs/1704.01099}.
\newblock
  \burlalt{doi:10.1007/978-3-030-01593-0}{http://dx.doi.org/10.1007/978-3-030-01593-0}.

\bibitem[Car07]{Cartier}
\textsc{P.~Cartier}.
\newblock A primer of {H}opf algebras.
\newblock In \emph{Frontiers in number theory, physics, and geometry. {II}},
  537--615. Springer, Berlin, 2007.
\newblock
  \burlalt{doi:10.1007/978-3-540-30308-4_12}{http://dx.doi.org/10.1007/978-3-540-30308-4_12}.

\bibitem[CC18]{CatCh}
\textsc{R.~Catellier} and \textsc{K.~Chouk}.
\newblock Paracontrolled distributions and the 3-dimensional stochastic
  quantization equation.
\newblock \emph{Ann. Probab.} \textbf{46}, no.~5, (2018), 2621--2679.
\newblock \burlalt{arXiv:1310.6869}{http://arxiv.org/abs/1310.6869}.
\newblock
  \burlalt{doi:10.1214/17-AOP1235}{http://dx.doi.org/10.1214/17-AOP1235}.

\bibitem[CEFM11]{MR2803804}
\textsc{D.~Calaque}, \textsc{K.~Ebrahimi-Fard}, and \textsc{D.~Manchon}.
\newblock Two interacting {H}opf algebras of trees: a {H}opf-algebraic approach
  to composition and substitution of {B}-series.
\newblock \emph{Adv. in Appl. Math.} \textbf{47}, no.~2, (2011), 282--308.
\newblock \burlalt{arXiv:0806.2238}{http://arxiv.org/abs/0806.2238}.
\newblock
  \burlalt{doi:10.1016/j.aam.2009.08.003}{http://dx.doi.org/10.1016/j.aam.2009.08.003}.

\bibitem[CH16]{Ajay}
\textsc{A.~Chandra} and \textsc{M.~Hairer}.
\newblock An analytic {BPHZ} theorem for {R}egularity {S}tructures.
\newblock \emph{ArXiv e-prints} (2016).
\newblock \burlalt{arXiv:1612.08138}{http://arxiv.org/abs/1612.08138}.

\bibitem[Che54]{Chen54}
\textsc{K.-T. Chen}.
\newblock Iterated integrals and exponential homomorphisms.
\newblock \emph{Proc. London Math. Soc. (3)} \textbf{4}, (1954), 502--512.

\bibitem[Che57]{Chen57}
\textsc{K.-T. Chen}.
\newblock Integration of paths, geometric invariants and a generalized
  {B}aker-{H}ausdorff formula.
\newblock \emph{Ann. of Math. (2)} \textbf{65}, (1957), 163--178.

\bibitem[Che58]{Chen58}
\textsc{K.-T. Chen}.
\newblock Integration of paths---a faithful representation of paths by
  non-commutative formal power series.
\newblock \emph{Trans. Amer. Math. Soc.} \textbf{89}, (1958), 395--407.

\bibitem[Che71]{Chen71}
\textsc{K.-T. Chen}.
\newblock Algebras of iterated path integrals and fundamental groups.
\newblock \emph{Trans. Amer. Math. Soc.} \textbf{156}, (1971), 359--379.

\bibitem[CHV05]{CHV}
\textsc{P.~Chartier}, \textsc{E.~Hairer}, and \textsc{G.~Vilmart}.
\newblock {A substitution law for B-series vector fields}.
\newblock Research Report RR-5498, {INRIA}, 2005.
\newblock \urlprefix\url{https://hal.inria.fr/inria-00070509}.

\bibitem[CHV10]{MR2657947}
\textsc{P.~{Chartier}}, \textsc{E.~{Hairer}}, and \textsc{G.~{Vilmart}}.
\newblock Algebraic structures of {B}-series.
\newblock \emph{Found. Comput. Math.} \textbf{10}, no.~4, (2010), 407--427.
\newblock
  \burlalt{doi:10.1007/s10208-010-9065-1}{http://dx.doi.org/10.1007/s10208-010-9065-1}.

\bibitem[CK98]{CK}
\textsc{A.~Connes} and \textsc{D.~Kreimer}.
\newblock Hopf algebras, renormalization and noncommutative geometry.
\newblock \emph{Comm. Math. Phys.} \textbf{199}, no.~1, (1998), 203--242.
\newblock \burlalt{arXiv:hep-th/9808042}{http://arxiv.org/abs/hep-th/9808042}.
\newblock
  \burlalt{doi:10.1007/s002200050499}{http://dx.doi.org/10.1007/s002200050499}.

\bibitem[CK99]{1126-6708-1999-09-024}
\textsc{A.~Connes} and \textsc{D.~Kreimer}.
\newblock Renormalization in quantum field theory and the {R}iemann-{H}ilbert
  problem.
\newblock \emph{J. High Energy Phys.} \textbf{1999}, no.~9, (1999), Paper 24,
  8.
\newblock \burlalt{arXiv:hep-th/9909126}{http://arxiv.org/abs/hep-th/9909126}.
\newblock
  \burlalt{doi:10.1088/1126-6708/1999/09/024}{http://dx.doi.org/10.1088/1126-6708/1999/09/024}.

\bibitem[CK00]{CKI}
\textsc{A.~Connes} and \textsc{D.~Kreimer}.
\newblock Renormalization in quantum field theory and the {R}iemann-{H}ilbert
  problem {I}: the {H}opf algebra structure of graphs and the main theorem.
\newblock \emph{Commun. Math. Phys.} \textbf{210}, (2000), 249--73.
\newblock \burlalt{arXiv:hep-th/9912092}{http://arxiv.org/abs/hep-th/9912092}.
\newblock
  \burlalt{doi:10.1007/s002200050779}{http://dx.doi.org/10.1007/s002200050779}.

\bibitem[CK01]{CKII}
\textsc{A.~Connes} and \textsc{D.~Kreimer}.
\newblock Renormalization in quantum field theory and the {R}iemann-{H}ilbert
  problem. {II}. {T}he {$\beta$}-function, diffeomorphisms and the
  renormalization group.
\newblock \emph{Comm. Math. Phys.} \textbf{216}, no.~1, (2001), 215--241.
\newblock \burlalt{arXiv:hep-th/0003188}{http://arxiv.org/abs/hep-th/0003188}.
\newblock
  \burlalt{doi:10.1007/PL00005547}{http://dx.doi.org/10.1007/PL00005547}.

\bibitem[DPD02]{DPD}
\textsc{G.~Da~Prato} and \textsc{A.~Debussche}.
\newblock Two-dimensional {N}avier-{S}tokes equations driven by a space-time
  white noise.
\newblock \emph{J. Funct. Anal.} \textbf{196}, no.~1, (2002), 180--210.
\newblock
  \burlalt{doi:10.1006/jfan.2002.3919}{http://dx.doi.org/10.1006/jfan.2002.3919}.

\bibitem[DPD03]{DPD2}
\textsc{G.~Da~Prato} and \textsc{A.~Debussche}.
\newblock Strong solutions to the stochastic quantization equations.
\newblock \emph{Ann. Probab.} \textbf{31}, no.~4, (2003), 1900--1916.
\newblock
  \burlalt{doi:10.1214/aop/1068646370}{http://dx.doi.org/10.1214/aop/1068646370}.

\bibitem[EFGK04]{Birkhoff1}
\textsc{K.~Ebrahimi-Fard}, \textsc{L.~Guo}, and \textsc{D.~Kreimer}.
\newblock Spitzer's identity and the algebraic {B}irkhoff decomposition in
  p{QFT}.
\newblock \emph{J. Phys. A} \textbf{37}, no.~45, (2004), 11037 --11052.
\newblock \burlalt{arXiv:hep-th/0407082}{http://arxiv.org/abs/hep-th/0407082}.
\newblock
  \burlalt{doi:10.1088/0305-4470/37/45/020}{http://dx.doi.org/10.1088/0305-4470/37/45/020}.

\bibitem[FH14]{Peter}
\textsc{P.~K. Friz} and \textsc{M.~Hairer}.
\newblock \emph{A course on rough paths}.
\newblock Universitext. Springer, Cham, 2014,  xiv+251.
\newblock With an introduction to regularity structures.
\newblock
  \burlalt{doi:10.1007/978-3-319-08332-2}{http://dx.doi.org/10.1007/978-3-319-08332-2}.

\bibitem[FMRS85]{FMRS85}
\textsc{J.~Feldman}, \textsc{J.~Magnen}, \textsc{V.~Rivasseau}, and
  \textsc{R.~S{\'e}n{\'e}or}.
\newblock Bounds on {R}enormalized {F}eynman graphs.
\newblock \emph{Comm. Math. Phys.} \textbf{100}, no.~1, (1985), 23--55.
\newblock
  \burlalt{doi:10.1007/BF01212686}{http://dx.doi.org/10.1007/BF01212686}.

\bibitem[{Foi}16]{2016arXiv160508310F}
\textsc{L.~{Foissy}}.
\newblock {Commutative and non-commutative bialgebras of quasi-posets and
  applications to Ehrhart polynomials}.
\newblock \emph{ArXiv e-prints} (2016).
\newblock \burlalt{arXiv:1605.08310}{http://arxiv.org/abs/1605.08310}.

\bibitem[GIP15]{Gub}
\textsc{M.~Gubinelli}, \textsc{P.~Imkeller}, and \textsc{N.~Perkowski}.
\newblock Paracontrolled distributions and singular {PDE}s.
\newblock \emph{Forum Math. Pi} \textbf{3}, (2015), e6, 75.
\newblock \burlalt{arXiv:1210.2684}{http://arxiv.org/abs/1210.2684}.
\newblock
  \burlalt{doi:10.1017/fmp.2015.2}{http://dx.doi.org/10.1017/fmp.2015.2}.

\bibitem[GP17]{2015arXiv150803877G}
\textsc{M.~Gubinelli} and \textsc{N.~Perkowski}.
\newblock K{PZ} reloaded.
\newblock \emph{Comm. Math. Phys.} \textbf{349}, no.~1, (2017), 165--269.
\newblock \burlalt{arXiv:1508.03877}{http://arxiv.org/abs/1508.03877}.
\newblock
  \burlalt{doi:10.1007/s00220-016-2788-3}{http://dx.doi.org/10.1007/s00220-016-2788-3}.

\bibitem[Gub04]{Gubinelli200486}
\textsc{M.~Gubinelli}.
\newblock Controlling rough paths.
\newblock \emph{Journal of Functional Analysis} \textbf{216}, no.~1, (2004), 86
  -- 140.
\newblock
  \burlalt{doi:10.1016/j.jfa.2004.01.002}{http://dx.doi.org/10.1016/j.jfa.2004.01.002}.

\bibitem[Gub10]{Gubinelli2010693}
\textsc{M.~Gubinelli}.
\newblock Ramification of rough paths.
\newblock \emph{Journal of Differential Equations} \textbf{248}, no.~4, (2010),
  693 -- 721.
\newblock
  \burlalt{doi:10.1016/j.jde.2009.11.015}{http://dx.doi.org/10.1016/j.jde.2009.11.015}.

\bibitem[Guo10]{Birkhoff2}
\textsc{L.~Guo}.
\newblock Algebraic {B}irkhoff decomposition and its applications.
\newblock In \emph{Automorphic forms and the {L}anglands program}, vol.~9 of
  \emph{Adv. Lect. Math. (ALM)},  277--319. Int. Press, Somerville, MA, 2010.
\newblock \burlalt{arXiv:0807.2266}{http://arxiv.org/abs/0807.2266}.

\bibitem[{Hai}13]{KPZ}
\textsc{M.~{Hairer}}.
\newblock {Solving the KPZ equation.}
\newblock \emph{{Ann. Math. (2)}} \textbf{178}, no.~2, (2013), 559--664.
\newblock \burlalt{arXiv:1109.6811}{http://arxiv.org/abs/1109.6811}.
\newblock
  \burlalt{doi:10.4007/annals.2013.178.2.4}{http://dx.doi.org/10.4007/annals.2013.178.2.4}.

\bibitem[Hai14]{reg}
\textsc{M.~Hairer}.
\newblock A theory of regularity structures.
\newblock \emph{Invent. Math.} \textbf{198}, no.~2, (2014), 269--504.
\newblock \burlalt{arXiv:1303.5113}{http://arxiv.org/abs/1303.5113}.
\newblock
  \burlalt{doi:10.1007/s00222-014-0505-4}{http://dx.doi.org/10.1007/s00222-014-0505-4}.

\bibitem[Hai16a]{proc}
\textsc{M.~Hairer}.
\newblock The motion of a random string.
\newblock \emph{ArXiv e-prints} (2016).
\newblock Proceedings of the XVIII ICMP, to appear.
\newblock \burlalt{arXiv:1605.02192}{http://arxiv.org/abs/1605.02192}.

\bibitem[Hai16b]{CDM}
\textsc{M.~Hairer}.
\newblock Regularity structures and the dynamical $\phi^4_3$ model.
\newblock \emph{Current Developments in Mathematics} \textbf{2014}, (2016),
  1--50.
\newblock \burlalt{arXiv:1508.05261}{http://arxiv.org/abs/1508.05261}.

\bibitem[Hep69]{Hepp}
\textsc{K.~Hepp}.
\newblock On the equivalence of additive and analytic renormalization.
\newblock \emph{Comm. Math. Phys.} \textbf{14}, (1969), 67--69.
\newblock
  \burlalt{doi:10.1007/BF01645456}{http://dx.doi.org/10.1007/BF01645456}.

\bibitem[HL15]{Cyril2}
\textsc{M.~Hairer} and \textsc{C.~Labb{\'e}}.
\newblock A simple construction of the continuum parabolic {A}nderson model on
  {${\bf R}^2$}.
\newblock \emph{Electron. Commun. Probab.} \textbf{20}, (2015), no. 43, 11.
\newblock \burlalt{arXiv:1501.00692}{http://arxiv.org/abs/1501.00692}.
\newblock
  \burlalt{doi:10.1214/ECP.v20-4038}{http://dx.doi.org/10.1214/ECP.v20-4038}.

\bibitem[HL18]{Cyril3}
\textsc{M.~{Hairer}} and \textsc{C.~{Labb{\'e}}}.
\newblock {Multiplicative stochastic heat equations on the whole space}.
\newblock \emph{{J. Eur. Math. Soc. }} \textbf{20}, no.~4, (2018), 1005--1054.
\newblock \burlalt{arXiv:1504.07162}{http://arxiv.org/abs/1504.07162}.
\newblock \burlalt{doi:10.4171/JEMS/781}{http://dx.doi.org/10.4171/JEMS/781}.

\bibitem[Hos16]{2016arXiv160204570H}
\textsc{M.~Hoshino}.
\newblock K{PZ} equation with fractional derivatives of white noise.
\newblock \emph{Stoch. Partial Differ. Equ. Anal. Comput.} \textbf{4}, no.~4,
  (2016), 827--890.
\newblock \burlalt{arXiv:1602.04570}{http://arxiv.org/abs/1602.04570}.
\newblock
  \burlalt{doi:10.1007/s40072-016-0078-x}{http://dx.doi.org/10.1007/s40072-016-0078-x}.

\bibitem[HP15]{wong}
\textsc{M.~Hairer} and \textsc{{\'E}.~Pardoux}.
\newblock A {W}ong-{Z}akai theorem for stochastic {PDE}s.
\newblock \emph{J. Math. Soc. Japan} \textbf{67}, no.~4, (2015), 1551--1604.
\newblock \burlalt{arXiv:1409.3138}{http://arxiv.org/abs/1409.3138}.
\newblock
  \burlalt{doi:10.2969/jmsj/06741551}{http://dx.doi.org/10.2969/jmsj/06741551}.

\bibitem[HQ18]{woKP}
\textsc{M.~Hairer} and \textsc{J.~Quastel}.
\newblock A class of growth models rescaling to {KPZ}.
\newblock \emph{Forum Math. Pi} \textbf{6}, (2018), e3.
\newblock \burlalt{arXiv:1512.07845}{http://arxiv.org/abs/1512.07845}.
\newblock
  \burlalt{doi:10.1017/fmp.2018.2}{http://dx.doi.org/10.1017/fmp.2018.2}.

\bibitem[HS16]{2014arXiv1409.5724H}
\textsc{M.~Hairer} and \textsc{H.~Shen}.
\newblock The dynamical sine-{G}ordon model.
\newblock \emph{Comm. Math. Phys.} \textbf{341}, no.~3, (2016), 933--989.
\newblock \burlalt{arXiv:1409.5724}{http://arxiv.org/abs/1409.5724}.
\newblock
  \burlalt{doi:10.1007/s00220-015-2525-3}{http://dx.doi.org/10.1007/s00220-015-2525-3}.

\bibitem[HS17]{2015arXiv150701237H}
\textsc{M.~Hairer} and \textsc{H.~Shen}.
\newblock A central limit theorem for the {KPZ} equation.
\newblock \emph{Ann. Probab.} \textbf{45}, no.~6B, (2017), 4167--4221.
\newblock \burlalt{arXiv:1507.01237}{http://arxiv.org/abs/1507.01237}.
\newblock
  \burlalt{doi:10.1214/16-AOP1162}{http://dx.doi.org/10.1214/16-AOP1162}.

\bibitem[JLM85]{Jona}
\textsc{G.~Jona-Lasinio} and \textsc{P.~K. Mitter}.
\newblock On the stochastic quantization of field theory.
\newblock \emph{Comm. Math. Phys.} \textbf{101}, no.~3, (1985), 409--436.

\bibitem[Kre98]{Kreimer}
\textsc{D.~Kreimer}.
\newblock On the {H}opf algebra structure of perturbative quantum field
  theories.
\newblock \emph{Adv. Theor. Math. Phys.} \textbf{2}, no.~2, (1998), 303--334.
\newblock \burlalt{arXiv:q-alg/9707029}{http://arxiv.org/abs/q-alg/9707029}.
\newblock
  \burlalt{doi:10.4310/ATMP.1998.v2.n2.a4}{http://dx.doi.org/10.4310/ATMP.1998.v2.n2.a4}.

\bibitem[Kup16]{Antti}
\textsc{A.~Kupiainen}.
\newblock Renormalization group and stochastic {PDE}s.
\newblock \emph{Ann. Henri Poincar\'e} \textbf{17}, no.~3, (2016), 497--535.
\newblock \burlalt{arXiv:1410.3094}{http://arxiv.org/abs/1410.3094}.
\newblock
  \burlalt{doi:10.1007/s00023-015-0408-y}{http://dx.doi.org/10.1007/s00023-015-0408-y}.

\bibitem[Lyo98]{Lyons98}
\textsc{T.~J. Lyons}.
\newblock Differential equations driven by rough signals.
\newblock \emph{Rev. Mat. Iberoamericana} \textbf{14}, no.~2, (1998), 215--310.
\newblock \burlalt{doi:10.4171/RMI/240}{http://dx.doi.org/10.4171/RMI/240}.

\bibitem[Mol77]{Molnar}
\textsc{R.~K. Molnar}.
\newblock Semi-direct products of {H}opf algebras.
\newblock \emph{J. Algebra} \textbf{47}, no.~1, (1977), 29--51.
\newblock
  \burlalt{doi:10.1016/0021-8693(77)90208-3}{http://dx.doi.org/10.1016/0021-8693(77)90208-3}.

\bibitem[Nic78]{Quotients}
\textsc{W.~D. Nichols}.
\newblock Quotients of {H}opf algebras.
\newblock \emph{Comm. Algebra} \textbf{6}, no.~17, (1978), 1789--1800.
\newblock
  \burlalt{doi:10.1080/00927877808822321}{http://dx.doi.org/10.1080/00927877808822321}.

\bibitem[Sch87]{MR914660}
\textsc{W.~R. Schmitt}.
\newblock Antipodes and incidence coalgebras.
\newblock \emph{J. Combin. Theory Ser. A} \textbf{46}, no.~2, (1987), 264--290.
\newblock
  \burlalt{doi:10.1016/0097-3165(87)90006-9}{http://dx.doi.org/10.1016/0097-3165(87)90006-9}.

\bibitem[Sch94]{schmitt99}
\textsc{W.~R. Schmitt}.
\newblock Incidence {H}opf algebras.
\newblock \emph{J. Pure Appl. Algebra} \textbf{96}, no.~3, (1994), 299--330.
\newblock
  \burlalt{doi:10.1016/0022-4049(94)90105-8}{http://dx.doi.org/10.1016/0022-4049(94)90105-8}.

\bibitem[SX18]{2016arXiv160105724S}
\textsc{H.~Shen} and \textsc{W.~Xu}.
\newblock {Weak universality of dynamical $\Phi^4_3$: non-Gaussian noise}.
\newblock \emph{Stoch. Partial Differ. Equ. Anal. Comput.} \textbf{6}, no.~2,
  (2018), 211--254.
\newblock \burlalt{arXiv:1601.05724}{http://arxiv.org/abs/1601.05724}.
\newblock
  \burlalt{doi:10.1007/s40072-017-0107-4}{http://dx.doi.org/10.1007/s40072-017-0107-4}.

\bibitem[Zim69]{Zimmermann}
\textsc{W.~Zimmermann}.
\newblock Convergence of {B}ogoliubov's method of renormalization in momentum
  space.
\newblock \emph{Comm. Math. Phys.} \textbf{15}, (1969), 208--234.
\newblock
  \burlalt{doi:10.1007/BF01645676}{http://dx.doi.org/10.1007/BF01645676}.

\bibitem[ZZ15]{Zhu}
\textsc{R.~Zhu} and \textsc{X.~Zhu}.
\newblock Three-dimensional {N}avier-{S}tokes equations driven by space-time
  white noise.
\newblock \emph{J. Differential Equations} \textbf{259}, no.~9, (2015),
  4443--4508.
\newblock \burlalt{arXiv:1406.0047}{http://arxiv.org/abs/1406.0047}.
\newblock
  \burlalt{doi:10.1016/j.jde.2015.06.002}{http://dx.doi.org/10.1016/j.jde.2015.06.002}.

\end{thebibliography}

\end{document}